\documentclass[fleqn]{report}
\author{Scott Hudson}
\date{mathematics.scott@gmail.com \endgraf \bigskip July 2026}
\title{A Paper on Calculating the Height and Relational Complexity of the Primitive Actions of $ PSL_{2} (q) $ and $ PGL_{2} (q) $}
\usepackage[
top    = 2cm,
bottom = 2cm,
left   = 2cm,
right  = 2cm,
marginparsep = 0pt,
marginparwidth=0pt,
]{geometry}
\usepackage{lastpage}
\usepackage{fancyhdr}
\fancypagestyle{plain}{%
\fancyhf{} 

\pagestyle{fancy}

\fancyhead{}
\cfoot{\thepage}
\lfoot{}
\rfoot{}}
\usepackage{titling}
\usepackage{blkarray}
\usepackage{bm}
\usepackage{amsmath}
\usepackage[toc]{appendix}
\usepackage{amsfonts}
\usepackage{mathtools}
\usepackage[dvipsnames]{xcolor}
\usepackage{amsthm}
\usepackage{amssymb}
\usepackage{sectsty} 
\chaptertitlefont{\huge}
\usepackage{etoolbox}
\usepackage{mathrsfs}
\usepackage{listings}
\usepackage{setspace}
\theoremstyle{plain}
\newtheorem{thrm}{Theorem}[section]

\theoremstyle{plain}
\newtheorem{deft}[thrm]{Definition}
\newtheorem{exmp}[thrm]{Example}
\newtheorem{lema}[thrm]{Lemma}
\newtheorem{corl}[thrm]{Corollary}
\newtheorem{tble}[thrm]{Table}
\newtheorem{cjct}[thrm]{Conjecture}
\usepackage{hyperref}

\begin{document}

\maketitle

\newpage
\pagenumbering{arabic}
\begin{center}
\vspace*{\fill}
\subsection*{Abstract}
For a finite group acting on a finite set, a statistic called relational complexity can be calculated for the action. This notion was defined by Gregory Cherlin and motivated by considerations in model theory. Another related statistic is the height of the action, which provides an upper bound for relational complexity. In this paper, both concepts are defined and some basic results proved. The main focus later on is examining the primitive actions of $ PSL_{2} (q) $ and $ PGL_{2} (q) $ and computing both the height and relational complexity for each one.
\vspace*{\fill}
\end{center}

\tableofcontents

\newpage
\chapter{Introduction}
Given an action of a finite group $G$ on a finite set $ \Omega $, a number can be determined known as the \textbf{relational complexity} of the action. This is defined below, with greater detail provided in the main text later.

\medskip
Let $ r \in \mathbb{N} $. Let $ I , J \in \Omega ^{r} $. For $ 1 \leq s \leq r $, an $ s $-subtuple of $ I $ is an $ s $-tuple whose entries are taken from $ I $ and placed in the same order as they were in $ I $. The notation $ I \sim _{s} J $ is used to mean for any $ s $-subtuple $ I ^{ \prime } $ of $ I $ and $ s $-subtuple $ J ^{ \prime } $ formed from entries in corresponding places of $ J $, there exists $ g \in G $ such that $ ( I ^{ \prime } ) ^{g} =  J ^{ \prime } $.

\medskip
The relational complexity of the action, denoted as $ RC(G, \Omega ) $, is the least $k \in \mathbb{N} $ such that $k \geq 2$ and whenever the following conditions are met,
\begin{itemize}
\item $ m \in \mathbb{N} $ where $k \leq m$,
\item $ I, J \in \Omega ^{m} $,
\item $ I \sim _{k} J $,
\end{itemize}
then $I \sim _{m} J $.

\medskip
Relational complexity was first defined in a paper by Cherlin, Martin and Saracino, see \cite{CHERLIN2}, but was known as \textbf{arity}. The term relational complexity started being used in \cite{CHERLIN1}, again by Cherlin.

\medskip
Model theory is where interest in this area came from and, although the model theoretic side of relational complexity will not be discussed, this is what has motivated group theorists to look into the group theoretic side of the subject. This connection between both sides is explored by Cherlin in \cite{CHERLIN3}.

\medskip
Relational complexity is a relatively unexplored topic for which little is yet known. Some recent papers and books have been produced however. For example Gill, Liebeck and Spiga classify the primitive binary permutation groups in \cite{GILL4} (binary actions being those with relational complexity $ 2 $). This built on previous work by Cherlin in \cite{CHERLIN1}, Gill, Hunt and Spiga in \cite{GILL1}, Dalla Volta, Gill and Spiga in \cite{GILL5} and Wiscons in \cite{WISCONS3}. A general upper bound for relational complexity of the actions of primitive permutation groups is found in \cite{@KELSEY} by Kelsey and Roney-Dougal.

\medskip
Another central topic of this paper is the \textbf{height} of a group action, written as $ Ht(G , \Omega ) $. Height is defined as being the size of the largest \textbf{independent set} in $ \Omega $. An independent set $ \Delta $ is a non-empty subset of $ \Omega $ such that the point stabilizer of any non-empty proper subset of $ \Delta $ is not equal to the point stabilizer of $ \Delta $.

\medskip
Height can often be easier to calculate than relational complexity and provides an upper bound for the latter; $ RC(G, \Omega ) \leq Ht(G , \Omega ) + 1 $.

\medskip
As height has largely been studied alongside relational complexity, most of what has been written down can be found in the above papers mentioned for relational complexity. Others texts include \cite{GILL2}, where Gill, Lod{\'a} and Spiga discover an upper bound for the height of actions of finite permutation groups, and \cite{GILL6} where Gill and Lod{\'a} look at the height of $ S_{n} $ acting on $ k $-subsets of $ \{ 1, \dots , n \} $ when $ 1 \leq k \leq \tfrac{n}{2} $.

\medskip
In ultimate goal of this paper is to calculate both the height and relational complexity for the primitive actions of $ PSL_{2} (q) $ and $ PGL_{2} (q) $. The final result being the theorems below. The chapter numbers given are where the corresponding results are proved. Since primitive actions are equivalent to the action on the right cosets of some maximal subgroup, the actions are listed by the names/structure of the maximal subgroups in $ PSL_{2} (q) $ and $ PGL_{2} (q) $. The information below is summarized in more detail at the start of each of the relevant chapters.
\begin{thrm}\label{thermintrofrst}
For $ q \geq 11 $, the height and relational complexity of the primitive actions of $ PSL_{2} (q) $ are:

\medskip
\begin{tabular}{|l|l|l|l|}
\hline
Chapter Number & Action Name & Height & Relational Complexity \\
\hline
\ref{chapfive} & Borel & 3 & 4 \\
\ref{chapsix} & Dihedral & 3 & 3 \\
\ref{chapseven} & $ A_{4} $ & 2 & 3 \\
\ref{chapeight} & $ S_{4} $ & 3 & 3 or 4 \\
\ref{chapnine} & $ A_{5} $ & 3 & 4 \\
\ref{chapten} & Subfield (point stabilizer isomorphic to $ PSL_{2} (3) $\upshape{)} & 2 & 3 \\
\ref{chapten} & Subfield (point stabilizer not isomorphic to $ PSL_{2} (3) $\upshape{)} & 3 & 4 \\
\hline
\end{tabular}
\end{thrm}
\begin{thrm}\label{thermintroscnd}
For $ q \geq 11 $, the height and relational complexity of the primitive actions of $ PGL_{2} (q) $ are:

\medskip
\begin{tabular}{|l|l|l|l|}
\hline
Chapter Number & Action Name & Height & Relational Complexity \\
\hline
\ref{chapfour} & $ PSL_{2} (q) $ & 1 & 2 \\
\ref{chapfive} & Borel & 3 & 4 \\
\ref{chapsix} & Dihedral & 3 & 3 \\
\ref{chapeight} & $ S_{4} $ & 3 & 3 or 4 \\
\ref{chapten} & Subfield & 3 & 4 \\
\hline
\end{tabular}
\end{thrm}
For $ 4 \leq q < 11 $ the above statistics are also calculated for $ PSL_{2} (q) $ and $ PGL_{2} (q) $, but often require dealing with special cases that do not fit the general pattern, so are discussed further at the beginning of each of the above chapters.

\medskip
You may have spotted the relational complexity for these actions is never less than height. The tables in \cite{WISCONS1} list the primitive actions of degree at most $ 100 $ along with their height and relational complexity, calculated using the GAP \cite{GAP4} code provided by Wiscons at \cite{WISCONS2}. The majority of the time in those tables it can be seen the relational complexity is at least equal to the height. But not always.

\medskip
If $ Ht(G , \Omega ) \leq RC( G , \Omega ) \leq Ht(G , \Omega ) + 1 $ then the possibility $ RC( G , \Omega ) = Ht(G , \Omega ) + 1 $ can be ruled out if there do not exist any \textbf{almost independent sets} in $ \Omega $. These are subsets of $ \Omega $ that are not independent, but all proper subsets are independent. Almost independent sets are a new concept that have been defined and looked at for the first time is in this paper and play a part in later calculations for the above theorems.

\medskip
This text is structured in the following way.

\medskip
The second chapter covers the definition of relational complexity. In the second section of the second chapter, some basic results are given and upper bounds are found for $ RC (G, \Omega ) $ that are dependent on the size of $ \Omega $. The relational complexity of actions on some small sets is calculated as well. Unfortunately the best general upper bound based on set size is quite large and not always so useful. In the third section the height of an action is given to provide a more useful bound. Almost independent sets and their uses are introduced in the fourth section of the chapter. The fifth section looks at some results for multiply transitive actions that will be made use of in later chapters.

\medskip
Most of the second chapter is not new material.

\medskip
The third chapter gives a description of the maximal subgroups of $ PSL_{2} (q) $ and $ PGL_{2} (q) $, so that we can look at their primitive actions. Many basic results for these groups are provided, that will be needed later on. The scene is then set for the final chapters. None of this is new research.

\medskip
The remaining six chapters of this paper are devoted to calculating the height and relational complexity of these primitive actions. In the theorems above, each of the statistics in the table are calculated exactly, with the exception of the relational complexity for the $ S_{4} $ action. We will see later that for the $ S_{4} $ action, the relational complexity can be $ 3 $ or $ 4 $ depending on the value of $ q $ and that trying to find out when each possibility happens would likely take a huge amount of casework. So, based on evidence collected in Chapter 7, a conjecture is made on when the relational complexity is $ 4 $.

\medskip
When dealing with the dihedral actions in the fifth chapter, the main theorems are general enough to be applied to other simple groups acting on maximal dihedral subgroups. An example of this is given with the Suzuki groups acting on maximal dihedral subgroups.

\medskip
Calculating the height and relational complexity for these primitive actions in the final six chapters is new research.

\medskip
Notation is generally going to be defined as it appears, but to clear up some ambiguity the following conventions will be used throughout the text: $ \mathbb{N} $ is used for the set of positive integers $\{ 1, 2, 3, \dots \} $. For a subset that is not necessarily proper the notation $ \subseteq $ is used. A strictly proper subset will be written using $ \subset $.

\newpage
\chapter{Relational Complexity Definition and General Results}
This chapter is largely devoted to finding upper bounds for relational complexity. The first section defines relational complexity. The second section develops some general lemmas and gives some bounds based on the size of the set being acted on.

\medskip
In the third section another statistic that can be calculated for an action is introduced, called height. Finding the height of various actions is one of the aims of this paper, so results are proved to assist with this. We will also see how height presents us with another upper bound for relational complexity. Most of the material in these first three sections is not new research. Although not collected in a single text, or even necessarily written down before, some of the lemmas and examples can be found in \cite{GILL4} and \cite{GILL2}.

\medskip
Almost independent sets are the focus of the fourth section. These are subsets of the set being acted on with properties that can help us narrow down the relational complexity of an action. This section consists of new research. Particularly, the notion of almost independent sets are defined and studied for the first time.

\medskip
The final section of the chapter is on bounds for the relational complexity of $ m $-transitive actions. Both upper and lower bounds are provided, which will be needed for the Borel actions of $ PSL_{2} (q) $ and $ PGL_{2} (q) $ in later chapters.
\section{The Definition of Relational Complexity}
Throughout this text all groups will be finite, as will sets being acted on.  Unless stated otherwise, $G$ is a group and $ \Omega $ will be a set, usually acted on by $ G $. Throughout this text group elements will act on the right. An element $g \in G $ acting on an element $ \omega \in \Omega $ will be denoted as $ \omega ^{g} $. When a specific group action $ \phi : \Omega \times G \rightarrow \Omega $ needs referring to, the notation $ ( \omega ,   g ) \phi $ may also be used instead of $ \omega ^{g} $ and the action will be called $ \phi $.

\medskip
The kernel of a group action $ \phi $ will be written as $ \ker ( \phi ) = \{ g \in G : \omega ^{g} = \omega \text{ for all } \omega \in \Omega \} $. The notation $ \Omega ^{m} $ will be used for the set of all $m$-tuples with entries in $ \Omega $.

\medskip
The definition of relational complexity formulated in this section uses subtuples.
\begin{deft}[Subtuple] Let $m \in \mathbb{N} $ and let $ I \in \Omega ^{m} $, where $ I = ( I_{1} , I_{2} , \dots , I_{m} ) $. For $r \leq m $, an \textbf{$r$-subtuple} of $I$ is an $r$-tuple $( I_{p_{1} } , \dots , I_{ p_{r} } ) $ where $1 \leq p_{1}  <  \dots < p_{r} \leq m $.
\end{deft}
\begin{exmp}
Let $I = (1, 2, 3, 4) \in \mathbb{Z} ^{4} $. A $3$-subtuple of $I$ is $(1,3,4) $ and another is $(2,3,4) $. An example of a $2$-subtuple of $I$ is $(2,3) $.
\end{exmp}
If a group $G$ acts on a set $ \Omega $ with group action $\phi _{1} $, then another group action $\phi _{2} $ on the set of $m$-tuples $ \Omega ^{m} $ can be defined such that for each $I  = ( I_{1} , I_{2} , \dots , I_{m} ) \in \Omega ^{m} $ and $g \in G$ we have $ I ^{g} = ( I_{1} ^{ g} , I_{2} ^{g } , \dots , I_{m} ^{g } ) $. This action on $ \Omega ^{m} $ allows us to explain what it means for a pair for $m$-tuples to be $k$-subtuple complete.
\begin{deft}[$k$-subtuple complete] Let $ m \in \mathbb{N} $ and $ k \in \{ 1, \dots , m \} $. Let $I, J \in \Omega ^{m} $, where $I =  ( I_{1}  , \dots , I_{m} )  $ and  $J =  ( J_{1} , \dots , J_{m} )  $. The pair $(I , J ) $ is $k$-\textbf{subtuple complete} if and only if for each corresponding pair of $k$-subtuples $ ( I_{p_{1} }  , \dots , I_{ p_{k} } ) $ and $( J_{p_{1} }  , \dots , J_{ p_{k} } ) $ there exists $g \in G$ such that $ ( I_{p_{1} } , \dots , I_{ p _{k} } ) ^{g} = ( J_{p_{1} } ,  \dots , J_{ p_{k} } ) $, that is $I_{ p_{i} } ^{g} = J_{ p_{i} } $ for all $i \in \{ 1, \dots , k \} $. The notation $ I \sim _{k} J $ is used to say that $(I,J)$ is $k$-subtuple complete.
\end{deft}
It is straightforward to show that being $k$-subtuple complete is an equivalence relation.
\begin{lema}\label{qvrltn}
For any positive integer $k \leq m $ the relation $ \sim _{k} $ is an equivalence relation on $ \Omega ^{m} $.
\end{lema}
\begin{proof}
Omitted.
\end{proof}
Now we are ready to introduce the definition of relational complexity.
\begin{deft}[Relational complexity]\label{first} Let $G$ be a finite group acting non-trivially on a finite set $ \Omega $. The \textbf{relational complexity} of the action is the least $k \in \mathbb{N} $ such that $k \geq 2$ and whenever the following conditions are met,
\begin{itemize}
\item $ m \in \mathbb{N} $ where $k \leq m$,
\item $ I, J \in \Omega ^{m} $,
\item $ I \sim _{k} J $,
\end{itemize}
then $I \sim _{m} J $.
\end{deft}
The notation $ RC( \phi ) $ will be used to mean the relational complexity of an action $ \phi $. If it is clear what the action of a group $G$ on $ \Omega $ is, it will often be written as $ RC( G , \Omega  ) $.

\medskip
From the definition it is not clear that the relational complexity of the action of a finite group acting on a finite set exists. It will be shown that it does exist after some preliminary results. Each of these lemmas will be used repeatedly in this paper, so are worth remembering.
\begin{lema}\label{equale}
Let $m \in \mathbb{N} $ where $m \geq 2$. Let $I,J \in \Omega ^{m} $ where $I = (I_{1} , I_{2} , \dots , I_{m} )$ and $J = (J_{1} , J_{2} , \dots , J_{m} )$. Suppose $I \sim _{k} J $ for some $k \geq 2$. Let $i,j \in \{ 1, \dots , m \} $. Then $I_{i} = I_{j} $ if and only if $J_{i} = J_{j} $.
\end{lema}
\begin{proof}
If $ i = j $ then the result is true. So suppose $ i \neq j $. Let $ I ^{ \prime } $ be a $k$-subtuple of $I$ containing $I_{i} $ and $I_{j} $ as entries. Write $I ^{ \prime } = ( I_{p _{1} } , I_{p_{2}} , \dots , I_{p_{k}} ) $ where $1 \leq p_{1} < \dots <  p_{k} \leq m $. The corresponding $k$-subtuple $J^{ \prime } =  (J_{p_{1}} , J_{p_{2}} , \dots , J_{p_{k}} ) $ of $J$ contains $J_{i}$ and $J_{j} $ as entries. As $(I,J)$ is $k$-subtuple complete, there exists $g \in G $ such that $I_{p_{s} } ^{g}  = J _{p_{s} } $ for each $s \in \{ 1, \dots , k \} $. In particular $I_{i} ^{g} = J_{i} $ and $I_{j} ^{g} = J_{j} $. So if $I_{i} = I_{j} $ then $J_{i} = I_{i} ^{g} = I_{j} ^{g} = J_{j} $.

\medskip
Similarly $I_{p_{s} } ^{g ^{-1} }  = J _{p_{s} } $ for each $s \in \{ 1, \dots , k \} $. In particular $J_{i} ^{g^{-1} } = I_{i} $ and $J_{j} ^{g ^{-1}} = I_{j} $. Therefore if $J_{i} = J_{j} $ then $I_{i} = J_{i} ^{g ^{-1} } = J_{j} ^{g ^{-1} } = I_{j} $.
\end{proof}
\begin{lema}\label{existrc}
Suppose $G$ acts non-trivially on $ \Omega $. Put $n := | \Omega | $. Suppose the following conditions are met,
\begin{itemize}
\item $ k \in \mathbb{N} $ where $n \leq k$,
\item $ I, J \in \Omega ^{k} $,
\item $I \sim _{n} J $.
\end{itemize}
Then $ I \sim _{k} J $.
\end{lema}
\begin{proof}
Since the action is non-trivial, it must be that $n \geq 2$. Write $I = (I_{1}  , \dots , I_{k} )$ and $J = (J_{1} ,  \dots , J_{k} )$. Since $ | \Omega |  = n$, the entries of $I$  contain at most $n$ unique elements of $ \Omega $. Let $s$ be the number of unique entries in $I$ and let $ X = \{ I_{x_{1} } , \dots , I_{x_{s} }  \} $ be a set of unique entries of $I$, where $1 \leq x_{1} <  \dots <  x_{s} \leq k $.

\medskip
As $s \leq n \leq k $, there exists a $n$-subtuple of $I$, say $ I^{ \prime } =  (I_{p_{1}}  , \dots , I_{p_{n}} ) $ where $ 1 \leq p_{1} <  \dots <  p_{n} \leq k $, such that each of the elements of $X$ appear in the components of $I ^{ \prime } $. Let $J^{ \prime } =  (J_{p_{1}} , \dots , J_{p_{n}} ) $ be the corresponding $n$-subtuple of $J$. The supposition $ I \sim _{n} J $ implies there exists $g \in G $ such that $I_{p_{i} } ^{g}  = J _{p_{i} } $ for each $i \in \{ 1, \dots , n \} $. Since every element of $X$ appears in the entries of $I ^{ \prime } $ we have $ I_{x_{j}} ^{g}  = J_{x _{j} } $ for each $ j \in \{ 1, \dots , s \} $.

\medskip
Now let $ b \in \{ 1, \dots , k \} $. The entry $I_{b} $ of $I$ is equal to an element of $X$, say $I_{b} = I_{x _{t} } $ for some $ t \in \{ 1, \dots , s \} $. By Lemma \ref{equale} we then have $J_{b} = J_{x _{t} } $. Therefore $I_{b} ^{g} = I_{x _{t} } ^{g} = J_{x _{t} }  = J_{b}  $. It follows that $I ^{g} = J$, that is $I \sim _{k} J $.
\end{proof}
As there exists a number $ | \Omega | \geq 2$ satisfying the conditions of Lemma \ref{existrc}, there must exist a least natural number greater than $1$ that satisfies the same conditions. This number is the relational complexity of the action. Thus we have our next corollary.
\begin{corl}\label{rccor}
The relational complexity of a non-trivial action exists and $ RC(G, \Omega ) \leq  | \Omega | $.
\end{corl}
\begin{exmp}
Consider the regular action of a group $ G $ on itself. Let $ m \in \mathbb{N} $ and $ m \geq 2 $. Let $ I $ and $ J $ be $ m $-tuples whose entries are elements of $ G $. Write $ I = ( I_{1} , \dots , I_{m} ) $ and $ J = ( J_{1} , \dots , J_{m} ) $. For each $ i \in \{ 1, \dots , m \} $ there is only one $ g_{i} \in G $ such that $ I_{i} ^{ g_{i} } = J_{i} $. If $ I \sim _{2} J $, then for each $ j \in \{ 1 , \dots , m \} \setminus \{ i \} $ we have $ (I_{i} , I_{j} ) ^{g _{i} } = (I_{i} ^{g _{i} } , I_{j} ^{g _{i} } ) = (J_{i} , J_{j} ) $. Therefore $ I ^{g _{i} } = J $ and $ I \sim _{m} J $. Thus the relational complexity of this action is $ 2 $.
\end{exmp}
\section{Basic Results}
With relational complexity defined, we want to establish some basic results on it and look at how to determine the relational complexity of specific basic group actions.

\medskip
Assume from now on, unless said otherwise, that $ \Omega $ is acted on by $ G $. Also for the rest of this chapter take the action of $ G $ to be non-trivial. We will need to refer to the entries of some tuples quite often, as in the last section, so if $ m \in \mathbb{N} $ and $ I \in \Omega^{m} $ we will always take the entries of such an $ m $-tuple to be written as $ I = ( I_{1} , \dots , I _{m} ) $.

\medskip
Finding group actions that have the same relational complexity as each other reduces the number of actions to look at and equivalent group actions are an example of this.
\begin{deft}[Equivalent Group Actions]\label{quivdfn}
Suppose a group $G$ acts on sets $ \Omega $ and $ \Gamma $. The actions are \textbf{equivalent} if there exists a bijection $ \phi : \Omega \rightarrow \Gamma $ such that $ ( \omega ^{ g } ) \phi = (( \omega ) \phi )^{ g } $ for all $ \omega  \in \Omega $ and $ g \in G $.
\end{deft}
In the definition of an equivalent group action, the bijection between the sets is essentially swapping the symbols used for each element whilst leaving the group action unchanged otherwise. So the relational complexity of both actions is the same. Similar results are listed below. Formal proofs are long but not difficult, so have been left out.
\begin{lema}\label{equct}
The following statements are true:
\begin{itemize}
\item Suppose $G$ acts on $ \Omega $ and $ \Gamma $. If the actions are equivalent, then $RC( G, \Omega ) = RC (G, \Gamma ) $.
\item Denote the action of $G$ on $ \Omega $ by $ \phi _{1} $. Let $K = \ker ( \phi _{1} ) $. Let $ \phi _{2} $ be the action of $G / K $ on $ \Omega $ where $ \omega ^{K g } = \omega ^{g} $ for all $ \omega \in \Omega $ and $g \in G $. Then $ RC( \phi_{1} ) = RC( \phi_{2} ) $.
\end{itemize}
\end{lema}
\begin{proof}
Omitted.
\end{proof}
Now the relational complexity of specific actions can be looked at. For groups acting on very small sets, this can be found by looking at the size of the set, starting with actions on sets of size $2$.
\begin{lema}\label{transtwo}
If $|  \Omega | = 2$, then $ RC(G , \Omega ) = 2$.
\end{lema}
\begin{proof}
The relational complexity of the action must be at least $2$ from the definition of relational complexity. By Corollary \ref{rccor} the relational complexity of the action is at most $2$, so must be equal to $2$.
\end{proof}
Corollary \ref{rccor} shows that the size of the set being acted on is an upper bound for the relational complexity of an action. Lemma \ref{transtwo} shows that the relational complexity of an action on a set of size $2$ is equal to this upper bound. For most of this section results will be developed showing that for a set $ \Omega $ of size $ | \Omega | \geq 3 $ we can find a slightly better upper bound of $ | \Omega | - 1 $ and that there exists an action with relational complexity equal to this bound. First several basic results will be established. The following two lemmas will be useful many times throughout this text.
\begin{lema}\label{nimpk}
Suppose $G$ acts on $ \Omega $. Let $r \in \mathbb{N} $ and $I, J \in \Omega ^{r} $. Suppose that $I \sim _{m} J $ for some $m \leq r$. Then $I \sim _{k} J $ for all $1 \leq k \leq m $.
\end{lema}
\begin{proof}
Let $I ^{ \prime } = ( I_{p_{1} } , \dots , I_{p_{k} } )$ be a $k$-subtuple of $I$ where $1 \leq p_{1} < \dots < p_{k} \leq n $ and let $J ^{ \prime} =  ( J_{p_{1} } , \dots , J_{p_{k} } )$ be the corresponding $k$-subtuple of $J$.

\medskip
Since $k \leq m \leq n$, the $k$-tuple $I^{ \prime } $ is a $k$-subtuple of some $I ^{ \prime \prime } \in \Omega ^{m} $ that is also an $m$-subtuple of $I$. Similarly $J ^{ \prime } $ is a $k$-subtuple of some $m$-tuple $J^{ \prime \prime } $ which is an $m$-subtuple of $J$, where $J ^{ \prime \prime } $ contains the $i$-th entry of $J$ if and only if $I ^{ \prime \prime } $ contains the $i$-th entry of $I$.

\medskip
As $ I \sim _{m} J $ we have $(I ^{ \prime \prime } ) ^{g} = J ^{ \prime \prime }$ for some $g \in G $. In particular, for the entries of $I^{ \prime } $ we have $I_{p_{j} } ^{g} = J_{p_{j} } $ for all $j \in \{ 1, \dots , k \} $. Hence $ ( I ^{ \prime } )^{g} =  ( I_{p_{1} } ^{g} , \dots , I_{ p_{k} } ^{g} ) = ( J_{p_{1} } , \dots , J_{ p_{k} } ) = J ^{ \prime }$. Thus $ I \sim _{k} J $.
\end{proof}
\begin{lema}\label{uniqset}
Suppose $ G $ acts on $ \Omega $. Let $m \in \mathbb{N} $ where $m \geq 2$ and let $I,J \in \Omega ^{m}$. Suppose $I \sim _{k} J $ for some $k \geq 2$. Let $r$ be the number of distinct entries of $I$. Let $ \{ I_{p_{1}}  , \dots , I_{p_{r} } \} $ be the set of distinct entries of $I$, where $1 \leq p_{1} < \dots < p_{r} \leq n $. Then there are $r$ distinct entries of $J$ and $  \{ J_{p_{1} } , \dots , J_{p_{r} } \} $ is the set of such entries. Furthermore if there exists $g \in G$ such that $ ( I_{p_{1}} ^{g}  , \dots , I_{p_{r} } ^{g}  ) = ( J_{p_{1}}  , \dots , J_{p_{r} } ) $, then $ I^{g} = J $.
\end{lema}
\begin{proof}
As $I \sim _{k} J $ and $k \geq 2$, for $i,j \in \{1, \dots , m \} $ we have $I_{i} = I_{j} $ if and only if $J_{i} = J_{j} $ by Lemma \ref{equale}. Since each of the entries of $I$ are equal to one of $  \{ I_{p_{1}}  , \dots , I_{p_{r} } \} $, each of the entries in $J$ must be equal to one of $  \{ J_{p_{1} } , \dots , J_{p_{r} } \} $ and the $r$ elements of this set are distinct (if not then $J_{p_{s} } = J_{p_{t} } $ for some $s, t \in \{1, \dots , r \} $ with $ s \neq t$, which means by Lemma \ref{equale} that  $I_{p_{s} } = I_{p_{t} } $, giving a contradiction).

\medskip
Suppose there exists $g \in G$ such that $ ( I_{p_{1}} ^{g}  , \dots , I_{p_{r} } ^{g}  ) = ( J_{p_{1}}  , \dots , J_{p_{r} } ) $. For each $a \in \{ 1 , \dots , m \} $ we have $I_{a} = I_{p_{b} } $ for some $ b \in \{ 1 , \dots , r \} $. So by Lemma \ref{equale} we have $ J_{a} = J_{p_{b} } $. It follows that $ I_{a} ^{g} = I_{ p _{b} }  ^{g} = J_{p_{b} } = J_{a} $. Thus $I^{g} =  (I_{1} ^{g} , \dots , I_{m} ^{g} )  =  (J_{1} , \dots , J_{m} )  = J $.
\end{proof}
An important use of this lemma is that in order to check whether $I \sim _{m} J $, the calculation can be reduced to looking at whether an $r$-subtuple of the unique entries of $I$ can be sent to the corresponding $r$-subtuple of unique entries of $J$.

\medskip
The next lemma is obvious but useful.
\begin{lema}\label{obvi}
Let $n \in \mathbb{N} $ and let $I,J \in \Omega ^{n} $. Suppose $I ^{ \prime } $ is an $m$-subtuple of $I$ and $J ^{ \prime }  $ is the corresponding $m$-subtuple of $J$. If $I \sim _{k} J $ for some $k \leq m$, then $ I^{ \prime} \sim _{k} J ^{ \prime}  $.
\end{lema}
\begin{proof}
Suppose $I \sim _{k} J $ for some $k \leq m$. Any $k$-subtuple $I ^{ \prime \prime } $ of $I ^{ \prime } $ is also an $k$-subtuple of $I$ and the corresponding $k$-subtuple $J ^{ \prime  \prime } $ of $J ^{ \prime } $ is a $k$-subtuple of $J$ corresponding to $ I^{ \prime \prime } $ in $ I $. So there exists $g \in G$ such that $( I ^{ \prime \prime }  ) ^{g} = J ^{ \prime \prime } $. Hence $ I^{ \prime} \sim _{k} J ^{ \prime} $.
\end{proof}
These lemmas allow another definition of relational complexity to be formulated.
\begin{lema}\label{rcalt}
Let $ G $ be a group acting non-trivially on a finite set $ \Omega $ of size $| \Omega | = n \geq 2 $. Then $RC(G, \Omega ) $ is the least integer $k$ such that $2 \leq k \leq n $ and if
\begin{itemize}
\item $I,J \in \Omega ^{n} $,
\item $I \sim _{k} J $,
\end{itemize}
then $I \sim _{n} J $.
\end{lema}
\begin{proof}
Suppose $k \geq 2$ is the least integer such that whenever $I,J \in \Omega ^{n} $ and $I \sim _{k} J $, then $I \sim _{n} J$. Put $r = RC(G, \Omega ) $. Then $r \leq n $ by Corollary \ref{rccor}. Also $k \leq r $ because if $K,L \in \Omega ^{n} $ and $K \sim _{r} L $, then $K \sim _{n} L $ by the definition of relational complexity.

\medskip
Let $m \in \mathbb{N} $ with $m \geq k $. Suppose $M,N \in \Omega ^{m} $ and that $M \sim _{k} N $. Let $s$ be the number of unique entries in $M$ and let $A = \{ M_{p_{1} } , \dots , M_{p_{s} } \} $ be the set of unique entries of $ M $, where $p_{1} < \dots < p_{s} $. Lemma \ref{uniqset} shows that there are $ s $ unique entries of $ N $ and $  \{ N_{p_{1} } , \dots , N_{p_{s} } \} $ is the set of these unique entries.

\medskip
Let $M^{ \prime} =  ( M_{p_{1} } , \dots ,M_{p_{s} } ) $ and $N^{ \prime} =  ( N_{p_{1} } , \dots , N_{p_{s} } ) $. Note that these are $ s $-subtuples of $ M $ and $ N $ respectively. Also observe that $ s \leq n $ because $ A \subseteq \Omega $.

\medskip
If $ s \leq k  $ then from Lemma \ref{nimpk} we see that $ M \sim _{k} N $ implies $ M^{ \prime} \sim _{s} N^{ \prime} $. So in this case we have $ M \sim _{m} N $ by Lemma \ref{uniqset}.

\medskip
If $ k < s $, then the $ s $-th entries of $  M^{ \prime} $ and $ N^{ \prime} $ can be repeated $ n - s $ times to create $ n $-tuples
\newline $ M^{ \prime \prime } =  ( M_{p_{1} } , \dots , M_{p_{s} } , \dots , M_{p_{s} } ) $ and $ N^{ \prime \prime } =  ( N_{p_{1} } , \dots , N_{p_{s} } , \dots , N_{p_{s} } ) $. Observe that $ M \sim _{k} N $ implies $ M^{ \prime} \sim _{k} N^{ \prime} $ by Lemma \ref{obvi}. So clearly $ M^{ \prime \prime } \sim _{k} N^{ \prime \prime } $. Thus $ M^{ \prime \prime } \sim _{n} N^{ \prime \prime } $. It follows that $ (M^{ \prime }  )^{g} = N^{ \prime }  $ for some $ g \in G $. Again Lemma \ref{uniqset} can be used to show $ M \sim _{m} N $ in this case.

\medskip
From Definition \ref{first} of relational complexity we get $ r \leq k $. Thus $ r = k $.
\end{proof}
Lemma \ref{rcalt} is often a much easier definition of relational complexity to work with than the earlier original definition. From this theorem we can see that for an action on a set of size $n$, to show the relational complexity is $k$, it is sufficient to show that all $k$-subtuple complete pairs of $n$-tuples are also $n$-subtuple complete.

\medskip
This puts a limit on the amount of calculation required to find the relational complexity, which was not clear from earlier definitions. An application of this theorem can be seen in calculating the relational complexity of the symmetric group $S_{n} $.
\begin{exmp}\label{symrc}
For $n \geq	2$, the natural action of the symmetric group $S_{n} $ on $ \Omega =  \{ 1, \dots , n \} $ has relational complexity $2$. To see this, let $I,J \in \Omega ^{n} $. Suppose $I \sim _{2} J $. Let $s$ be the number of unique elements of $I$ and let $A =  \{ I_{p_{1} } , \dots , I_{p_{s} } \} $ be the set of unique entries of $I$ where $ 1 \leq p_{1} < \dots < p_{s} \leq n $. Then $B =  \{ J_{p_{1} } , \dots , J_{p_{s} } \} $ is the set of unique entries of $J$ by Lemma \ref{uniqset}. Since $A, B \subseteq \Omega $, a permutation of the following form, written in two line notation, exists in $S_{n} $;
\begin{align*}
g = 
\begin{pmatrix}
  I_{p_{1} } & I_{p_{2} } & \dots & I_{p_{s} } & \dots  \\
 J_{p_{1} } & J_{p_{2} } & \dots & J_{p_{s} } & \dots
\end{pmatrix}
\end{align*}
Now
\[
(I_{p_{1} } , \dots , I_{p_{s} }  ) ^{g} = (I_{p_{1} } ^{g} , \dots , I_{p_{s} } ^{g}  )  = (J_{p_{1} } , \dots , J_{p_{s} }  )  .
\]
So $I^{g} = J$ by Lemma \ref{uniqset}. Therefore $I \sim _{n} J$. Since $ | \Omega | = n $, this shows that $RC ( S_{n} , \Omega ) = 2$ by Lemma \ref{rcalt}.
\end{exmp}
We are now in a position to find an upper bound for the relational complexity of actions on sets of size $3$ or more.
\begin{thrm}\label{kminuso}
If $G$ acts on a set $ \Omega $ of size $|  \Omega  |  \geq 3$, then $ RC (G, \Omega ) \leq | \Omega | -1$.
\end{thrm}
\begin{proof}
Put $ n:= | \Omega | $. By Lemma \ref{rcalt}, it is sufficient to show that for all $I,J \in \Omega ^{n} $, if $I \sim _{ (n-1) } J $ then $ I \sim _{n} J $. So assume that $I \sim _{ (n-1) } J $.

\medskip
There are two cases to look at. First suppose that the components of $I$ are all unique. Since $ n-1 \geq 2 $ and $I \sim _{ (n-1) } J $, Lemma \ref{equale} shows that the components of $J$ are all unique.

\medskip
The entries of $(n-1)$-tuple $( I_{1} , \dots I_{n-1} ) $ contain all except one of the elements of $ \Omega $, namely $I_{n} $. Similarly the entries of $( J_{1} , \dots J_{n-1} ) $ contains all elements of $ \Omega $ except $ J_{n} $. Now there exists $g \in G $ such that
\[
( I_{1} , \dots I_{n-1} ) ^{g} = ( J_{1} , \dots J_{n-1} ) .
\]
So $I_{i} ^{g} = J_{i} $ for each $i \in \{ 1, \dots n-1 \} $. Also it must be that $I_{n} ^{g} = J_{n}  $. If not then $I_{n} ^{g} = J_{j} $ for some $j \in \{ 1, \dots , n -1 \} $, which implies that $I_{n} = I_{j} $, contradicting the assumption that each of the entries of $I$ are unique. Thus $I^{g} = J $ and $I \sim _{n} J $ in this case.

\medskip
Next suppose the entries of $I$ are not all unique. Then there exists $r,s \in \{1, \dots n \} $ with $r \neq s$ such that $I_{r} = I_{s} $. Also $J_{r} = J_{s} $ by Lemma \ref{equale}.

\medskip
Let $I^{ \prime } = ( I_{p_{1} } , \dots , I_{p_{n-1} } ) $ be the $(n-1) $-subtuple of $I$ that does not contain $I_{r} $. Then $I ^{ \prime } $ contains all of the remaining $n-1$ entries of $I$. Let $ J^{ \prime } = ( J_{p_{1} } , \dots , J_{p_{n-1} } ) $ be the corresponding $(n-1) $-subtuple of $J$ that contains all entries of $J$ except $J_{r} $.

\medskip
As $I \sim _{ (n-1) } J $, there exists $h \in G $ such that $ ( I ^{ \prime } ) ^{h} = J ^{ \prime } $. Therefore $I_{k} ^{h} = J_{k} $ for all $ k \in \{ 1 , \dots , n \} \setminus \{ r \} $. Also $ I_{r} ^{h} = I_{s} ^{h} = J_{s} = J_{r} $ . Hence $ I ^{h} = J $ and $I \sim _{n} J $ in this case.
\end{proof}
An immediate use of this theorem is to calculate the relational complexity for all non-trivial actions on sets of size $3$.
\begin{exmp}\label{setthre}
Suppose $ | \Omega | = 3$. Then $2 \leq RC ( G , \Omega ) $ by the definition of relational complexity. Theorem \ref{kminuso} shows that $ RC ( G , \Omega ) \leq 2 $. Thus $ RC ( G , \Omega ) = 2 $.
\end{exmp}
The next example shows if we are looking for a bound that depends on the size of the set being acted on, then the bound given by Theorem \ref{kminuso} cannot be improved upon in general.
\begin{exmp}\label{rcgrpaltg}
It can be shown that for $n \geq	3$, the natural action of the alternating group $A_{n} $ on $\Omega =  \{ 1, \dots , n \} $ has relational complexity $n-1$. First suppose that $n = 3$. Then Example \ref{setthre} shows that the relational complexity of the action is $2$.

\medskip
Next suppose that $n \geq 4$. By Lemma \ref{rcalt}, it is sufficient to find $ I , J \in \Omega ^{n} $ such that $I \sim _{k} J $ for all $2 \leq k \leq n - 2$, but $I \not\sim _{n} J $. This will mean that $  RC ( A_{n} , \Omega ) > n - 2 $ and so $  RC ( A_{n} , \Omega )  = n-1 $ by Theorem \ref{kminuso}.

\medskip
Let $I = (1, 2, 3, 4, \dots , n-1 , n ) $ and $J =  (2, 1, 3, 4 , \dots  , n-1 , n )  $. Since $I$ and $J$ both contain every element of $ \Omega $ in their entries, there is a unique permutation $g := (1  \ 2 ) \in S_{n} $ such that $I ^{g} = J$. As $g$ is an odd permutation, $g \notin A_{n} $, so $ I \not\sim _{n} J $.

\medskip
Rather than checking $ I \sim _{k} J $ for all $ 2 \leq k \leq n -2$, it is enough to check $ I \sim _{ (n-2) } J $ since that implies $ I \sim _{k} J $ by Lemma \ref{nimpk}. Let $I ^{ \prime } = ( I_{1} ^{ \prime } , \dots , I_{n-2} ^{ \prime } ) $ be an $ (n-2) $-subtuple of $I$ and $ J ^{ \prime } = ( J_{1} ^{ \prime }  , \dots , J_{n-2} ^{ \prime }  ) $ be the corresponding $ (n-2) $-subtuple of $J$. If $I ^{ \prime } $ does not contain $1$ or $ 2$ as one of its entries, then neither does $J ^{ \prime } $ and the identity sends $I ^{ \prime } $ to $J ^{ \prime } $.

\medskip
If $I^{ \prime } $ contains both $1$ and $2 $ in its entries, then so does $J ^{ \prime } $. Therefore $I_{1} ^{ \prime } = 1 = J _{2} ^{ \prime }  $ and $I_{2} ^{ \prime } = 2 = J _{1} ^{ \prime } $. Also $I_{j} ^{ \prime } = J_{j}  ^{ \prime } $ for all $j \in \{ 3, \dots , n-2 \} $. There exists two numbers $ a , b \in \Omega \setminus \{ I_{1} ^{ \prime } , \dots , I_{n-2} ^{ \prime } \} $. So the permutation $ (1 \ 2 ) ( a \ b )  \in A _{n} $ sends $I ^{ \prime } $ to $J ^{ \prime } $ in this case.

\medskip
Next suppose that $I^{ \prime} $ contains $1$ as an entry, but not $2$. Then $J^{ \prime}$ contains $2$ as an entry but not $1$. Hence $I_{1} = 1 $ and $J _{1} = 2 $ and for all $i \in \{ 2, \dots , n-2  \} $ it must be that $I_{i} ^{ \prime } = J_{i} ^{ \prime }  \notin \{ 1, 2 \}$. There are two entries of $I$ that do not appear in the entries of $I ^{ \prime } $. One of the missing entries, say $c$, must not be equal to $1$ or $2$. But then $c$ is also missing from the entries of $J ^{ \prime} $. The element $g = (1 \ 2 \ c ) \in A_{n} $. Observe that $( I_{1} ^{ \prime }  ) ^{g} = 1^{g} = 2 = J_{1} ^{ \prime }  $ and that $( I_{i} ^{ \prime }  )  ^{g} = J_{i} ^{ \prime }  $ for all $i  \in \{ 2, \dots , n-2  \} $.  Hence $ (I ^{ \prime } ) ^{g} = J ^{ \prime } $.

\medskip
Similar reasoning shows that if $I^{ \prime} $ contains $2$ as an entry, but not $1$ then there is some element of $A_{n} $ that sends $I ^{ \prime } $ to $J ^{ \prime } $. Thus $ I \sim _{(n-2)} J $ and $ I \not\sim _{n} J $ as required.
\end{exmp}
When dealing with primitive actions of $ PSL_{2} (q) $ and $ PGL_{2} (q) $ later, we will be calculating the relational complexity of some $2$-transitive and $ 3 $-transitive actions. For any $ m \in \mathbb{N} $ a lower bound can be found for $m$-transitive group actions.
\begin{lema}\label{trnprm}
Suppose $ | \Omega | \geq 2 $ and $ K $ is the kernel of the action of $ G $ on $ \Omega $. Suppose the action is $m$-transitive for some $2 \leq m \leq | \Omega | $. Then exactly one of the following is true;
\begin{itemize}
\item $ G / K  \cong Sym ( \Omega ) $ as permutation groups and $ RC ( G , \Omega )  = 2$,
\item $ G / K  \not \cong Sym ( \Omega ) $ as permutation groups and $ RC ( G, \Omega  )  > m $,
\end{itemize}
where $ Sym ( \Omega ) $ acts naturally on $ \Omega $.
\end{lema}
\begin{proof}
Suppose $ k \in \{ 2, \dots , | \Omega | \} $ is the the greatest integer such that the action of $ G / K $ is $ k $-transitive.

\medskip
If $ k = | \Omega | $ then for every permutation on $ \Omega $ there exists a corresponding element in $ G / K $. So $ G / K \cong Sym ( \Omega ) $.

\medskip
Now suppose $ k < | \Omega | $. Then $ G / K \not \cong Sym ( \Omega ) $. There exists $ I , J \in \Omega ^{k+1} $, both with $ k $ distinct entries, where $ I \not \sim _{(k+1) } J $.  For all $ r \leq k $ the action is $ r $-transitive. Therefore $ I \sim _{r} J $. It follows from Definition \ref{first} that $ RC ( G, \Omega  )  > k $.
\end{proof}
Returning to Example \ref{rcgrpaltg}, it was shown that the relational complexity of the natural action of $ A_{n} $ for $ n \geq 3 $ is $ n -1 $. An alternative way of calculating this is to use the well known fact that $ A_{n} $ is $ (n-2) $-transitive, which means the relational complexity is at least $ n-1 $ by the above lemma. Using Lemma \ref{kminuso} we see that it must be exactly $ n - 1 $.

\medskip
This section is finished off with a simple but extremely useful lemma.
\begin{lema}\label{frsnrtytql}
Let $ 1 \leq r \leq m $. Let $I,J \in \Omega ^{m} $. Let $ k_{1} , \dots , k_{r} $ be distinct elements of $ \{ 1 , \dots , m \} $. Then $ I \sim _{r} J $ for some $r \leq m $ if and only if $ I \sim _{r} J ^{a} $ for all $ a \in G $. Additionally, $ I \sim _{r} J $ if and only if there exists $g \in G $ such that $ I \sim _{r} J ^{g} $ and $ I _{k_{i}} = J _{k_{i}} ^{g} $ for all $ i \in \{ 1, \dots , r \} $.
\end{lema}
\begin{proof}
The statement $ I \sim _{r} J $ for some $r \leq m $ if and only if $ I \sim _{r} J ^{a} $ for all $ a \in G $ is clearly true.

\medskip
If $ I \sim _{r} J $, there exists $ h \in G $ such that $ I_{k_{i} } ^{h} = J_{k_{i}} $ for all $ i \in \{ 1 , \dots , r \} $, and so $ J_{k_{i}} ^{h^{-1} } = I_{k_{i} } $. Since $ I \sim _{r} J^{h^{-1} } $, this proves one direction.

\medskip
Conversely, suppose there exists $ g \in G $ such that $ I \sim _{r} J ^{g} $ and $ I _{k_{i}} = J _{k_{i}} ^{g} $ for all $ i \in \{ 1, \dots , r \} $. Then $ I \sim _{r} J^{gg^{-1}} = J $.
\end{proof}
The above lemma tells us that if $m$-tuples $ I $ and $ J $ are $ r $-subtuple complete for some $ r \leq m $ and we want to check if they are $ m $-subtuple complete, then we can always carefully choose $ J $ so that $ r $ of the entries of $ I $ are equal to the corresponding $ r $ entries of $ J $. This means that any element of $ G $ that sends $ I $ to $ J $ must lie inside the intersection of the point stabilizers of those $ r $ entries, a fact that will be used often.
\section{Height}
So far all of the bounds on relational complexity have depended on the size of the set. A different bound can be determined for each group action that depends on the height of the action. To define what this is, the notion of an independent set is needed. First some notation. If $ \Delta = \{ \delta _{1} , \dots , \delta _{m} \} \subseteq \Omega $ is a set of $m \in \mathbb{N} $ elements, then we will write either $ G _{ \delta _{1} , \dots , \delta _{m} } $ or $ G _{ ( \Delta ) } $ for the point stabilizer of $ \Delta $. 
\begin{deft}[Independent Set]\label{stdfnind}
A non-empty set $ \Delta \subseteq \Omega $ is an \textbf{independent set} if for any proper non-empty subset $ \Gamma \subset \Delta $ we have $ G_{ ( \Delta ) } \neq G_{ ( \Gamma ) } $. Any set of size $1$ is always defined to be independent.
\end{deft}
\begin{exmp}\label{xmpind}
Consider the natural action of $G := S_{4} $ on $ \Omega =  \{ 1 , 2, 3, 4 \} $. Let $ \Delta = \{ 1, 2 \} $. Then $ G _{ ( \Delta ) } = \{ e, (3 \  4 ) \}  $, where $  e $ is the identity element. The only subsets of $ \Delta $ are $ \{ 1 \} $ and $ \{ 2 \} $, whose stabilizers contain $ ( 2 \ 3 \ 4 )  $ and $ ( 1 \  3 \ 4 )  $ respectively. Neither element is in $ G _{ ( \Delta ) } $, so $ \Delta $ is an independent set. This can be generalised to show that any two element subset of $ \Omega $ is independent and must have a point stabilizer of order $2$.
\end{exmp}
\begin{deft}[Height of a Group Action]
The \textbf{height} of a group action is the maximum size of an independent set. This will be denoted as $Ht( G, \Omega  ) $ if it is clear what the action of a group $G$ on $ \Omega $ is, or $Ht( \phi ) $ if we want to refer to a specific action $ \phi  $.
\end{deft}
Note that for any group action on a finite set, the height of the action must exist because a one element subset is independent and the largest independent set is at most the size of the set being acted on.
\begin{exmp}\label{xmtwo}
Again look at the natural action of $G := S_{4} $ on $\Omega =  \{ 1 , 2, 3, 4 \} $. The point stabilizer of $ \Omega $ is $ G_{ ( \Omega ) } = \{ e \} $, the trivial group. The subset $ \Gamma = \{ 1, 2, 3 \} $ also has point stabilizer $ G_{ ( \Gamma ) } = \{ e \} $ because any element of $ G $ that fixes three points of $ \Omega $ must also fix the fourth point. As $  G_{ ( \Gamma ) } =  G_{ ( \Omega ) } $, the set $ \Omega $ is not independent. Hence $Ht(G, \Omega ) < 4 $. However Example \ref{xmpind} shows that all one and two element subsets of $ \Gamma $ must have non-trivial point stabilizers. Therefore $ \Gamma $ is independent. Thus $Ht(G, \Omega )  = 3 $. 
\end{exmp}
There is a result for height that is analogous to Lemma \ref{equct} for group actions that are equivalent. As with that lemma, the proof essentially involves swapping the symbols used for the set being acted on, so is omitted.
\begin{lema}\label{qvlnthght}
Suppose a group $G$ acts on finite sets $ \Omega $ and $ \Gamma $. If the actions are equivalent, then $Ht( G, \Omega ) = Ht (G, \Gamma ) $.
\end{lema}
\begin{proof}
Omitted.
\end{proof}
We also have a similar result to Lemma \ref{equct} for the height of the action of the quotient of a group by its kernel.
\begin{lema}\label{fthflhght}
Let $K$ be the kernel of the action of $ G $ on $ \Omega $. Let $G / K $ act on $ \Omega $ where $ \omega ^{K g } = \omega ^{g} $ for all $ \omega \in \Omega $ and $g \in G $. Then a subset of $ \Omega $ is independent under the action of $G$ if and only if it is independent under the action of $G / K$. Furthermore $ Ht( G , \Omega ) = Ht( G/K , \Omega ) $.
\end{lema}
\begin{proof}
Let $ \Delta \subseteq \Omega $ be an independent set under the action of $G$. If $| \Delta | = 1$ then $ \Delta $ is independent under the action of $G/K $. So suppose $ | \Delta | > 1 $. Let $ \Delta ^{ \prime } \subset \Delta $ be a non-empty subset. The set $ \Delta $ being independent means that $ G _{ ( \Delta ^{ \prime } ) } \neq G _{ ( \Delta  ) } $. So $K \leq G _{ ( \Delta  ) } < G _{ ( \Delta ^{ \prime } ) } $. Hence, $ G _{ ( \Delta  ) } / K < G _{ ( \Delta ^{ \prime } ) } / K $. From the way the action of $ G/K$ is defined, $  G _{ ( \Delta  ) } / K  = (G/K ) _{ ( \Delta ) } $ and $  G _{ ( \Delta ^{ \prime }  ) } / K  = (G/K ) _{ ( \Delta ^{ \prime } ) } $. Therefore $(G/K ) _{ ( \Delta ) }  \neq  (G/K ) _{ ( \Delta ^{ \prime } ) } $, which shows that $ \Delta $ is an independent set under the action of $ G/ K $. It follows that $ Ht ( G, \Omega ) \leq Ht ( G/K , \Omega ) $.

\medskip
Similar reasoning shows that independent sets under the action of $G / K$ are independent under the action of $G$. So $ Ht ( G / K , \Omega ) \leq Ht ( G , \Omega ) $.
\end{proof}
In Examples \ref{xmpind} and \ref{xmtwo} we saw the set being acted on, $ \Omega $, is not independent. This is no coincidence, as we will now see.
\begin{lema}\label{htbnd}
If $ | \Omega | \geq 2 $, the height of an action on $ \Omega $ has the bounds $ 1 \leq Ht( G, \Omega ) \leq | \Omega | - 1 $.
\end{lema}
\begin{proof}
Put $ n:= | \Omega | $. Pick an element $ \omega _{1} \in \Omega $. Label the remaining elements of $ \Omega $ as $ \omega _{2} , \dots , \omega _{n} $. Put $ \Gamma := \{  \omega _{2} , \dots , \omega _{n} \} $. Let $g \in G_{ ( \Gamma ) }$. Then $\omega _{1} ^{g} = \omega _{1} $ since $ \omega $ cannot be sent to an element of $ \Gamma $. Hence $     G_{  ( \Gamma ) }  \leq G_{  ( \Omega )  } $. But $ G_{  ( \Omega )  } \leq  G_{  ( \Gamma ) }  $. Therefore $ G_{  ( \Omega )  } =  G_{  ( \Gamma ) }  $. Since $ \Gamma \subset \Omega $, this shows that $ \Omega $ is not an independent set, so $ Ht( G, \Omega ) \leq | \Omega | - 1  $. The bound $ 1 \leq Ht( G, \Omega )  $ follows from the definition of height.
\end{proof}
Using the original definition of an independent set to find the height of an action would take a lot of calculation. So next up is an alternative definition of an independent set.
\begin{lema}\label{indsbstsntqual}
A non-empty set $ \Delta \subseteq \Omega $ is independent if and only if $ G_{ ( \Gamma _{1} ) } \neq G_{ ( \Gamma _{2} ) } $ for all non-empty subsets $ \Gamma _{1} , \Gamma _{2} \subseteq \Delta $ with $ \Gamma_{1} \neq \Gamma_{2} $.
\end{lema}
\begin{proof}
If $ G_{ ( \Gamma _{1} ) } \neq G_{ ( \Gamma _{2} ) } $ for all non-empty subsets $ \Gamma _{1} , \Gamma _{2} \subseteq \Delta $ with $ \Gamma _{1} \neq \Gamma _{2} $ then $ \Delta $ is independent by definition.

\medskip
Next suppose there exists $ \Lambda _{1} , \Lambda _{2} \subseteq \Delta $ with $ \Lambda _{1} \neq \Lambda _{2} $ such that $ G_{ ( \Lambda _{1} ) } = G_{ ( \Lambda _{2} ) } $. Then we can assume without loss of generality that there exists $ \lambda _{1} \in \Lambda _{1} \setminus ( \Lambda _{1} \cap \Lambda _{2} ) $. Since $ \Lambda _{2} \subseteq \Delta \setminus \{ \lambda _{1} \} $, it must be that $ G_{ ( \Delta \setminus \{ \lambda _{1} \} ) } \leq G_{ ( \Lambda _{2} ) } = G_{ ( \Lambda _{1} ) } \leq G_{ \lambda _{1} } $. Thus $ G_{ ( \Delta ) } = G_{ ( \Delta \setminus \{ \lambda _{1} \} ) } $, which shows that $ \Delta $ is not an independent set.
\end{proof}
From Lemma \ref{indsbstsntqual} we immediately get the following corollary.
\begin{corl}\label{indsub}
Every non-empty subset of an independent set is independent.
\end{corl}
\begin{lema}\label{ltdfndpndnt}
A non-empty set $ \Delta \subseteq \Omega $ is not independent if and only if $ | \Delta | \neq 1 $ and there exists $ \delta \in \Delta $ such that $ G _{ ( \Delta ) }  = G _{ ( \Delta \setminus \{ \delta \}  ) }$.
\end{lema}
\begin{proof}
If $ | \Delta | \neq 1 $ and there exists $ \delta \in \Delta $ such that $ G _{ ( \Delta ) }  = G _{ ( \Delta \setminus \{ \delta \}  ) }$ then $ \Delta $ is not independent by definition.

\bigskip
For the opposite implication, suppose that $ \Delta $ is not independent. Since sets of size $1$ are independent, it must be that $ | \Delta | \neq 1 $. So there exists $ \Gamma \subset \Delta $ such that $  G _{ ( \Delta )} = G_{ ( \Gamma ) }$. Now $ \Gamma \subseteq \Gamma ^{ \prime }  $ for some $ \Gamma ^{ \prime } \subset \Delta $ where $ \Gamma ^{ \prime } = \Delta \setminus \{ \delta ^{ \prime } \} $ for some $ \delta ^{ \prime } \in \Delta $. Observe that $ G _{ ( \Delta ) } \leq G_{ ( \Gamma ^ { \prime } ) } \leq G_{ ( \Gamma  ) } = G _{ ( \Delta ) } $. Hence $ G _{ ( \Delta ) }  = G_{ ( \Gamma ^ { \prime } ) } = G _{ ( \Delta \setminus \{ \delta ^{ \prime }  \}  ) }$.
\end{proof}
The contrapositive of the above lemma gives another definition of an independent set, which is stated below for reference later.
\begin{lema}\label{scndfnt}
Suppose $G$ acts on $ \Omega $. A non-empty set $ \Delta \subseteq \Omega $ is independent if and only if one of the following is true;
\begin{itemize}
\item $ | \Delta | = 1 $ or
\item $ | \Delta | > 1 $ and $ G _{ ( \Delta ) }  \neq G _{ ( \Delta \setminus \{ \delta \}  ) }$ for all $ \delta \in \Delta $.
\end{itemize}
\end{lema}
Example \ref{xmtwo} illustrated that the height of the natural action of $ S_{4} $ is $ 3 $. In general the height of the natural action of the symmetric group on a finite set of size $ n $ is $n - 1 $, which can now be proved.
\begin{exmp}\label{hghtsym}
Let $ \Omega $ be a finite set of size $n \geq 2 $. Put $G := Sym ( \Omega ) $. By Lemma \ref{htbnd}, $ Ht ( Sym ( \Omega ) , \Omega ) \leq n -1 $, so it is sufficient to show that there exist independent sets of size $n-1$. Let $ \Delta $ be a set of size $ | \Delta | = n-1 $. If $n = 2 $ then $ | \Delta | = 1 $ and $ \Delta $ is independent by definition. So suppose $ n \geq 3 $. Then $ | \Delta | > 1 $.

\medskip
Label the elements as $ \Delta =  \{ \delta _{1} , \dots , \delta _{n-1} \} $. Let $ g \in G _{ ( \Delta ) } $. There is only one element $ \omega \in \Omega \setminus \Delta $ and since $g$ fixes all elements of $ \Delta $ it must be that $ \omega ^{g} = \omega $. Therefore $g$ stabilizes every element of $ \Omega $ and so $ g = 1_{G} $. Thus $ G _{ ( \Delta ) } = \{ 1_{G} \} $.

\medskip
Now for each $ i \in \{ 1, \dots , n-1 \} $ put $ \Delta _{i} := \Delta \setminus \{ \delta _{i} \} $. Then $ \delta _{i} , \omega \notin \Delta _{i} $ and so the permutation $ ( \delta _{i} \ \omega ) \in G _{ ( \Delta _{i} ) } $. Hence $ G _{ ( \Delta _{i} ) }  \neq G _{ ( \Delta  ) }  $. As It follows from Lemma \ref{scndfnt} that $ \Delta $ is an independent set. Thus $  Ht ( Sym ( \Omega ) , \Omega ) = n -1 $.
\end{exmp}
We can get the height of the natural action of $ A_{n} $ while we are at it.
\begin{exmp}\label{hghtalt}
Let $ n \in \mathbb{N} $ with $ n \geq 3 $. Consider the action of $ A_{n} $ on $ \Omega := \{ 1 , \dots , n \} $. Let $  \omega _{1} , \omega _{2} , \omega _{3} \in \Omega $ be distinct points. There is no even permutation that stabilizes the points of $ \Omega \setminus \{ \omega _{1} , \omega _{2} \} $, so it has trivial point stabilizer. Same goes for $ \Omega \setminus \{ \omega _{1} \} $. Therefore $ \Omega \setminus \{ \omega _{1} \} $ is not an independent set and $ Ht( A_{n} , \Omega ) \leq n - 2 $.

\medskip
If $ n = 3 $ then $ | \Omega \setminus \{ \omega _{1} , \omega _{2} \} | = 1$ and is independent by definition, implying $ Ht( A_{3} , \Omega ) = n - 2 = 1 $. Now suppose $ n \geq 4 $. The permutation $ ( \omega _{1} \  \omega _{2}  \ \omega _{3} ) $ is in the point stabilizer of $ \Omega \setminus \{ \omega _{1} , \omega _{2} , \omega _{3} \} $. Hence $ \Omega \setminus \{ \omega _{1} , \omega _{2} \} $ is an independent set by Lemma \ref{scndfnt}. This shows that $ Ht( A_{n} , \Omega ) = n - 2 $.
\end{exmp}
In Examples \ref{symrc} and \ref{rcgrpaltg} we saw that the relational complexity for the natural actions of $ S_{n} $ and $ A_{n} $ is $ 2 $ and $ n - 1 $ respectively, the least and greatest values possible on a set of size $ n $. However the above two examples show that the height is $ n - 1 $ and $ S_{n} $ and $ n -2 $ for $ A_{n} $. These two actions exhibit very different behaviour when looking at these statistics even though the natural action of $ A_{n} $ is the restriction of the action of $ S_{n} $.

\medskip
Some remarks on point stabilizers are now made that will be of use shortly. If $ \Gamma \subseteq \Delta \subseteq \Omega $, then the point stabilizer of $ \Delta $ must stabilize all points of $ \Gamma $. Hence $ G _{ ( \Delta ) } \leq G _{ ( \Gamma ) } $. This means that if we have a chain of subsets, there also exists a chain of stabilizers, which leads to the next definition and lemma.
\begin{deft}[Stabilizer Chain]
For an ordered collection of $m \in \mathbb{N}$ elements $ \omega _{1} , \omega _{2} , \dots , \omega _{m} \in \Omega $, a \textbf{stabilizer chain} is a chain of subgroups $G_{ \omega _{1}  } \geq G_{ \omega_{1} , \omega _{2}  } \geq   \dots  \geq  G_{ \omega_{1} , \omega _{2}  \dots , \omega _{m} } $. If the inclusions are all strict then the stabilizer chain is said to be \textbf{irredundant}.
\end{deft}
It is worth noting that whilst the order of elements of independent sets does not matter (because they are sets), the order of elements does make a difference for stabilizer chains. It is possible for a stabilizer chain can go from being irredundant to redundant by reordering.
\begin{lema}\label{stach}
Let $m \in \mathbb{N} $ and $ \omega _{1} , \omega _{2} , \dots , \omega _{m} \in \Omega $. If $ \{  \omega _{1} , \omega _{2} ,  \dots , \omega _{m}  \} $ is an independent set then $G_{ \omega _{1} }  > G_{ \omega _{1} , \omega _{2} } > \dots > G_{ \omega _{1} , \omega _{2} , \dots , \omega _{m} } $ is an irredundant stabilizer chain.
\end{lema}
\begin{proof}
If $m = 1$ then there is only one stabilizer in the stabilizer chain and the result holds. So suppose $m \geq 1 $. For each $ i \in \{ 1, \dots , m \} $ put $ \Delta _{i} := \{ \omega _{1} , \dots , \omega _{i } \} $. So $G _{ ( \Delta _{i } ) } =  G_{  \omega _{1} , \dots , \omega _{i } } $. The set $ \Delta _{m} $ being independent implies that each of the subsets $ \Delta _{i} $ are independent by Corollary \ref{indsub}. For all $j \in \{ 1, \dots , m - 1 \} $ we have $ G _{ ( \Delta _{ j } ) } \geq G _{ ( \Delta _{j + 1}  ) } $ and by the definition of an independent set this inclusion must be strict. Therefore $G_{ \omega _{1} }  > G_{ \omega _{1} , \omega _{2} } > \dots > G_{ \omega _{1} , \omega _{2} , \dots , \omega _{m} } $.
\end{proof}
Three more lemmas are needed before establishing the main result of this section, a bound on the relational complexity of the action that depends on height.
\begin{lema}\label{rodr}
Let $m \in \mathbb{N} $ and $I,J \in \Omega ^{m} $. Suppose $I \sim _{k} J $ for some $k \in \mathbb{N} $. If $K $ is an $m$-tuple obtained by permuting the entries of $I$ and $L $ is the corresponding $n$-tuple obtained by permuting the entries of $J$ in the same way, then $K \sim _{k} L $.
\end{lema}
\begin{proof}
Let $K^{ \prime } $ be a $k$-subtuple of $K$ and $L ^{ \prime } $ the corresponding $k$-subtuple of $L$. The entries of $K ^{ \prime } $ can be reordered to give a $k$-subtuple $I ^{ \prime } $ of $I$ and similarly the entries of $L^{ \prime } $ can be reordered in the same way to give the corresponding $k$-subtuple $J ^{ \prime } $ of $J$. Since $I \sim _{k} J $, there exists $g \in G $ such that $ (I ^{ \prime } ) ^{g} = J ^{ \prime } $. So $g$ sends the $i$-th entry of $I ^{ \prime } $ to the $i$-th entry of $J ^{ \prime } $ for each $i \in \{1, \dots , k \} $. Therefore $g$ sends the $j$-th entry of $ K ^{ \prime } $ to the $j$-entry of $L ^{ \prime } $ for each $j \in \{ 1, \dots , k \} $. Thus $K \sim _{k} L $.
\end{proof}
\begin{lema}\label{horrid}
Let $ G$ be a group acting on a finite set $ \Omega $. Given a collection of $m \in \mathbb{N}$ (not necessarily distinct) elements of $ \Omega$, they can be labeled as $\omega _{1} , \dots , \omega _{m} \in \Omega $ so that there exists some $t \leq Ht( G, \Omega )  $ such that $ t \leq m $ and both of the following hold;
\begin{enumerate}
\item $G_{ \omega _{1} } > G_{ \omega _{1} , \omega _{2} } > \dots > G_{ \omega _{1} , \omega _{2} , \dots , \omega _{t} } $ is an irredundant stabilizer chain.
\item $G_{ \omega _{1} , \omega _{2} , \dots , \omega _{t} } =  G_{ \omega _{1} , \omega _{2} , \dots , \omega _{s} }$ for all $ s \in \{ t, t+1,  \dots , m \} $.
\end{enumerate}
\end{lema}
\begin{proof}
Denote the collection of $m $ elements as $ \gamma _{1} , \dots , \gamma_{m} $ so that the first $r$ elements are distinct and the remaining $m - r $ are repeated elements. Put $ \Gamma := \{ \gamma _{1} , \dots , \gamma _{r} \} $. Let $ \Gamma ^{ \prime } \subseteq \Gamma $ be a non-empty set such that $ G _{ ( \Gamma ^{ \prime } ) } = G _{ ( \Gamma ) } $ and $ | \Gamma ^{ \prime } | $ is as small as possible, which must exist because $ G _{ ( \Gamma ) } = G _{ ( \Gamma ) } $. Put $ t ^{ \prime } := | \Gamma ^{ \prime } | $ and arbitrarily label the elements of $ \Gamma ^{ \prime } $ as $ \omega _{1} , \dots , \omega _{t ^{ \prime } } $. Since $ G _{ ( \Gamma ^{ \prime \prime }  ) } \neq G _{ ( \Gamma ) }  = G _{ ( \Gamma ^{ \prime } ) } $ for every non-empty proper subset $ \Gamma ^{ \prime \prime } \subset \Gamma ^{ \prime } $, the set $ \Gamma ^{ \prime } $ is an independent. Therefore $ t ^{ \prime } \leq Ht( G, \Omega )  $ and the chain $ G_{ \omega _{1} } > G_{ \omega _{1} , \omega _{2} } > \dots > G_{ \omega _{1} , \omega _{2} , \dots , \omega _{t ^{ \prime } } } $ is irredundant by Lemma \ref{stach}.

\medskip
Next arbitrarily label the $r - t ^{ \prime } $ elements of $ \Gamma \setminus \Gamma ^{ \prime } $ as $ \omega _{ t^{ \prime } + 1 } ,   \ \omega _{ t^{ \prime } + 2 } , \dots ,  \  \omega _{ r } $. Then $ \Gamma = \{ \omega _{1} , \omega _{2} , \dots , \omega _{r}  \} $. For each $ a \in \{ t  ^{ \prime } ,  t  ^{ \prime } + 1 , \dots , r \} $ we have $ \Gamma ^{ \prime } \subseteq \{ \omega _{1} , \dots , \omega _{a} \} \subseteq \Gamma $ and so
\[
G _{ ( \Gamma ) }  \leq G_{ \omega _{1} , \dots , \omega _{a} }  \leq  G _{ ( \Gamma ^{ \prime } ) } =  G _{ ( \Gamma ) } .
\]
Thus $ G_{ \omega _{1} , \dots , \omega _{a} }  =   G _{ ( \Gamma ^{ \prime } ) } =  G_{ \omega _{1} , \omega _{2} , \dots , \omega _{t ^{ \prime } } }  $. If $r = m$ then we can put $t = t^{ \prime } $ and we are done. So suppose that $ r < m $. Set $ \omega _{i} := \gamma _{i}  $ for each $i \in \{ r +1 , \dots , m \} $. Observe that for all $k \in \{ r+1 , \dots , m \} $,
\[
G _{  \omega _{1} , \dots , \omega_{k} } = \bigcap _{i = 1} ^{k} G_{ \omega _{i} } = \bigcap _{i = 1} ^{r} G_{ \omega _{i} } = G _{  \omega _{1} , \dots , \omega_{r} } = G _{ ( \Gamma ) }
= G _{ ( \Gamma ^{ \prime } ) } 
=  G_{ \omega _{1} , \omega _{2} , \dots , \omega _{t ^{ \prime } } } 
\]
because for each $i \in \{ r +1 , \dots , m \} $ we have $ G_{ \gamma _{i} } = G _{ \gamma _{j} } $ for some $ j \in \{ 1 , \dots , r \} $. Again putting $ t ^{ \prime } = t $, when $ r < m $ it has been shown that $G_{ \omega _{1} , \omega _{2} , \dots , \omega _{t} } =  G_{ \omega _{1} , \omega _{2} , \dots , \omega _{s} }$ for all $ s \in \{ t, t+1,  \dots , m \} $.
\end{proof}
Earlier it was shown the relational complexity of the natural actions of $ S_{n} $ and $ A_{n} $ is at most $ 1 $ more than the height. We now prove the main result of this section; that height always bounds relational complexity in this way.
\begin{thrm}\label{rchgt}
Suppose $ n:= | \Omega | \geq 2 $. Let $h :=  Ht( G , \Omega ) $. If $ I, J \in \Omega ^{n} $ and $I \sim _{(h+1) } J $ then $ I \sim _{n} J $. In particular $ RC(G, \Omega ) \leq Ht( G , \Omega ) + 1 $.
\end{thrm}
\begin{proof}
Note that $2 \leq  h + 1  \leq n$ by Lemma \ref{htbnd}. Lemma \ref{rodr} shows that the entries of $I$ can be reordered and the entries of $J$ reordered correspondingly and the resulting pair of $n$-tuples will still be $(h+1) $-subtuple complete. Therefore by Lemma \ref{horrid} we can assume that the entries of $I$ and $J$ are ordered so that 
\[
G_{I_{1} } > G_{ I_{1} , I_{2} } > \dots > G_{ I_{1} , I_{2} , \dots , I_{t} }
\]
for some $t \leq h $ and that
\[
 G_{ I_{1} , I_{2} , \dots , I_{t} } =  G_{ I_{1} , I_{2} , \dots , I_{s} }    \tag{*}
\]
for all $s \in \{ t, t+1,  \dots , n \} $. We may also assume $ I_{j} = J_{j} $ for all $ j \in \{ 1, \dots , h+1 \} $ by Lemma \ref{frsnrtytql}.

\medskip
For each $ k \in \{t + 1 , t+ 2 , \dots , n \} $ the $ (t+1) $-tuples $ I^{ \prime } =  ( I_{1} , \dots , I_{t} , I_{k} ) $ and $ J^{ \prime } = ( J_{1} , \dots , J_{t} , J_{k} ) = ( I_{1} , \dots , I_{t} , J_{k} ) $ are $ (t+1) $-subtuples of $ I $ and $ J $ respectively. The supposition $ I \sim _{ (h+1) } J $ implies $ I \sim _{ (t+1) } J $ by Lemma \ref{nimpk}. Hence $ (I^{ \prime } ) ^{g} = J^{ \prime }  $ for some $ g \in G_{ I_{1} , I_{2} , \dots , I_{t} } =  G_{ I_{1} , I_{2} , \dots , I_{k} } $. Therefore $ J_{k} = I^{g} _{k} = I_{k} $. From this we get $ I = J $ and $ I \sim _{n} J $. Lemma \ref{rcalt} then shows that $ RC(G, \Omega ) \leq h+1 = Ht( G , \Omega ) + 1 $.
\end{proof}
This lemma can be used to give another definition of relational complexity, which can cut down the calculation needed if we already know the height of an action.
\begin{corl}\label{rcshsbtplcmp}
Let $ n:= | \Omega |  \geq 2 $. Let $h =  Ht (G , \Omega ) $. Suppose $ k $ is the least element in $  \{ 2, \dots , h +1 \}  $ such that if $ I, J \in  \Omega ^{h+1} $ and $ I \sim _{k} J $ then $ I \sim _{ (h+1) } J $. Then $ k = RC ( G , \Omega )  $.
\end{corl}
\begin{proof}
Put $r := RC (G , \Omega ) $. Let $K,L \in \Omega ^{n} $ and suppose $K \sim _{k} L $. Note that $2 \leq  h + 1  \leq n$ by Lemma \ref{htbnd}. Let $K ^{*} $ be a $(h+1)$-subtuple of $K $ and $L  ^{*} $ the corresponding $(h+1)$-subtuple of $L $. Lemma \ref{obvi} shows that $K ^{*} \sim _{ k } L ^{*} $. Hence $K ^{*} \sim _{ (h+1) } L ^{*} $ by the definition of $k$. Therefore $ K \sim _{ (h+1) } L $ and it follows from Theorem \ref{rchgt} that $K \sim _{n} L $. The definition of relational complexity provided by Lemma \ref{rcalt} gives $r \leq k $.

\medskip
Next let $M, N \in \Omega ^{h+1} $. Note that $ r \leq h+1 $ by Theorem \ref{rchgt}. Suppose $M \sim _{r} N  $. Then $ M \sim _{(h+1)} N $ by Definition \ref{first}. Therefore $ k \leq r $.
\end{proof}
\section{Almost Independent Sets}
The tables in \cite{WISCONS1} list the primitive actions of degree at most $ 100 $, along with their height and relational complexity, calculated using the GAP code provided at \cite{WISCONS2}. For most of these actions, either $ RC(G, \Omega ) = Ht(G, \Omega ) $ or $ RC(G, \Omega ) = Ht(G, \Omega ) + 1$. So when calculating relational complexity, if we have narrowed it down to $ Ht(G, \Omega ) $ or $ Ht(G, \Omega ) + 1 $ it would be helpful to have some techniques to hand to rule one or the other out. We begin with a couple of lemmas.
\begin{lema}\label{rptdntrs}
Let $m \geq 2 $ and $I , J \in \Omega ^{m+1} $. If $I \sim_{m} J$ and any of the entries of $I$ are repeated then $I \sim _{(m+1)} J $.
\end{lema}
\begin{proof}
Suppose $I \sim _{ m } J $. If $I$ has any repeated entries then $I_{s} = I _{t} $ for some $s, t \in \{ 1 , \dots , m+1 \} $ with $ s \neq t $. Since $m \geq 2$, Lemma \ref{equale} shows that $ J_{t} = J_{s} $. As $ I \sim _{m} J$, there exists $g \in G$ such that $ I_{i} ^{g} = J_{i} $ for all $ i \in \{ 1 , \dots , m + 1 \} \setminus \{ s \}  $. Therefore $ I_{s} ^{g} = I_{t} ^{g} = J_{t} = J_{s} $. Hence $ I^{g} = J $ and $I \sim _{(m + 1)} J $.
\end{proof}
\begin{lema}\label{ntndpndntntrs}
Let $m \geq 2 $ and $I , J \in \Omega ^{m+1} $. Let $ \Delta $ be the set of entries of $I$. If $I \sim _{ m } J $ and if there exists $ \Gamma \subset \Delta   $ such that $ \Gamma $ is not an independent set, then $I \sim _{(m+1)} J $.
\end{lema}
\begin{proof}
Suppose $I \sim _{ m } J $. If $I$ has any repeated entries then $I \sim _{(m+1)} J $ by Lemma \ref{rptdntrs} and we are done. So suppose that $I$ has no repeated entries. Then $ \Delta = \{ I_{1} , \dots , I _{m+1} \} $. Suppose there exists $ \Gamma \subset \Delta $ where $ \Gamma $ is not an independent set. Then there exists $ \Delta ^{ \prime } \subset \Delta $  such that $ \Gamma \subseteq \Delta ^{ \prime } $ and $ \Delta ^{ \prime } = \Delta \setminus \{ \delta \} $ for some $ \delta \in \Delta $. Observe that $ \Delta ^{ \prime } $ is not independent by Corollary \ref{indsub}. So there exists $ \delta ^{ \prime } \in \Delta ^{ \prime } $ such that $ G_{ ( \Delta ^{ \prime } ) } = G _{ ( \Delta ^{ \prime } \setminus \{ \delta ^{ \prime } \} ) } $ by Lemma \ref{ltdfndpndnt}.

\medskip
Using Lemma \ref{rodr}, we can assume that the entries of $I$ are ordered so that $ \delta = I _{m+1} $. Hence $ \Delta = \{ I_{1} , \dots , I_{m} \} $. Using the same lemma again it can be assumed that $ \delta ^{ \prime } = I_{m} $. Therefore $ G_{ I_{1} , \dots , I_{m-1} } = G_{ I_{1} , \dots , I_{m} } $.

\medskip
From Lemma \ref{frsnrtytql}, we can assume $ J_{i} = I_{i} $ for all $ i \in \{ 1, \dots , m \} $. Since $ I \sim _{m} J $, there exists $ g \in G $ that sends
\[( I_{1} , \dots , I_{m-1} , I_{m+1} ) 
\]
to
\[
 ( J_{1} , \dots , J_{m-1} , J_{m+1} ) =  ( I_{1} , \dots , I_{m-1} , J_{m+1} ) .
\] 
It must be that $ g \in G_{ I_{1} , \dots , I_{m-1} } $, so $ I_{m} ^{g} = I_{m} = J_{m} $. Thus $ I^{g} = J $ and $ I \sim _{(m+1)} J $.
\end{proof}
Put $ h:= Ht(G, \Omega ) $. If $ RC(G, \Omega ) = h + 1$ then there exists $ I , J \in \Omega ^{h+1} $ such that $ I \sim _{h} J $ and $ I \not\sim _{h+1} J $ (otherwise for all $ n \geq h +1 $, each pair of $ n $-tuples that are $ h $-subtuple complete would be $ (h+1) $-subtuple complete, hence $ n $-subtuple complete, by Theorem \ref{rchgt}, implying $ RC(G,  \Omega ) = Ht(G , \Omega ) $).

\medskip
The first of the above lemmas tells us that the entries of $ I $ contain no repeats. If $ \Delta $ is the set of entries of $ I $, then all proper subsets of $ \Delta $ are independent by the second lemma. However $ | \Delta | = h +1 $, so $ \Delta $ is not independent.

\medskip
We see from this that sets with the same properties as $ \Delta $ play an important role in determining if $ RC(G , \Omega ) = h + 1 $ or not, so give them a name.
\begin{deft}\label{almstndstdfn}
A set $ \Delta \subseteq \Omega $ is \textbf{almost independent} if $ \Delta $ is not independent but every proper subset of $ \Delta $ is independent.
\end{deft}
From the above discussion, if $ \Omega $ does not contain any almost independent sets of size $ h + 1 $ then $ RC( G , \Omega ) \leq h $. So we have proved our next lemma.
\begin{lema}
Suppose $ RC(G, \Omega ) = Ht( G , \Omega ) + 1 $. Then there exist almost independent sets of size $ Ht( G , \Omega ) + 1 $ in $ \Omega $.
\end{lema}
Although it is necessary for there to be almost independent subsets of size $ h + 1 $ in $ \Omega $ to have $ RC( G , \Omega ) = h + 1 $, there are examples of group actions where such almost independent sets exist and $ RC( G , \Omega ) \leq h $. We see this with the natural action of $ S_{n} $ when $ n \geq 3 $. In this case the set $ \{ 1, \dots , n \} $ is almost independent, but $ RC( S_{n} , \{ 1, \dots , n \} ) = 2 $, as was seen in Example \ref{symrc}.
\section{Independent Sets of Multiply Transitive Actions}
The main aim of this paper is to calculate the height and relational complexity of primitive actions of  $ PSL_{2} (q) $ and $ PGL_{2} (q) $, which include some $ 2 $ and $ 3 $-transitive actions. In this final section of the chapter we first find a lower bound on height when an action is $m$-transitive, then an upper bound and finally develop a method to decide whether an $m$-transitive action has height exactly $m$ or not.
\begin{lema}\label{krnstvndpndnt}
Suppose $G$ acts $m$-transitively on $ \Omega $ and that $ | \Omega | \neq m $. Let $ \Delta \subseteq \Omega $. If $1 \leq | \Delta | \leq m $, then $ \Delta $ is independent.
\end{lema}
\begin{proof}
Put $s:= | \Delta | $ and $ \Delta : = \{ \delta _{1} , \dots , \delta _{s} \} $. Suppose $ \Delta $ is not independent. Then by Lemma \ref{ltdfndpndnt}, $ s \neq 1 $ and the elements can be labelled so that $ G _{ ( \Delta ) }  = G _{ ( \Delta \setminus \{ \delta _{s} \}  ) }$. Assume $ s \leq m $. Then, since $ m < | \Omega | $, there exists $ \omega \in \Omega \setminus \Delta $. As the action is $m$-transitive, it is also $s$-transitive. So there exists $g \in G $ such that
\[
( \delta _{ 1}  ^{g} , \dots , \delta _{s - 1}  ^{g} , \delta _{s} ^{g} )
= ( \delta _{ 1}   , \dots , \delta _{s - 1}  , \omega ) .
\]
Therefore $ g \in G _{ ( \Delta \setminus \{ \delta _{s} \}  ) } =  G _{ ( \Delta ) } $. It follows that $  \delta _{s} = \delta _{s} ^{g} =  \omega \in \Omega \setminus \Delta $, which is a contradiction. Thus the assumption that $ s \leq m $ must be wrong. Hence $  | \Delta | > m $.
\end{proof}
\begin{corl}\label{crltrnstvhght}
For an $ m $-transitive action, if $ m \neq | \Omega | $ then $Ht ( G , \Omega )  \geq m $.
\end{corl}
\begin{proof}
Let $ \Delta \subseteq \Omega $ and suppose $ | \Delta  | = m $. Then $ \Delta $ is independent by Lemma \ref{krnstvndpndnt} and so $Ht ( G , \Omega )  \geq m $.
\end{proof}
Next we will find an upper bound for the height of an $ m $-transitive action under particular conditions. Some general results on stabilizers and intersections of subgroups are required, the first of which is not difficult to prove. More notation is needed for this. For $ g \in G $ and $ \Delta \subseteq \Omega $ let $ \Delta ^{g} = \{ \delta ^{g} : \delta \in \Delta \} $.

\medskip
\begin{lema}\label{conjeq}
Let $n \in \mathbb{N} $. Let $ G _{1} , \dots , G _{n}  $ be subgroups of $G$. Let $g \in G$. Then
\[
\bigcap _{i = 1 } ^{n} g ^{-1}  G_{ i } g = g ^{-1} \bigg( \bigcap _{i = 1 } ^{n}   G_{i }  \bigg) g
\]
\end{lema}
\begin{proof}
Omitted.
\end{proof}
\begin{corl}\label{crlcnj}
Let $ \Delta _{1} , \Delta _{2} \subseteq \Omega $ where $ \Delta _{2} = \Delta _{1} ^{g} $ for some $g \in G$. Then $G _{ ( \Delta _{2} ) } = g^{-1} G _{ ( \Delta _{1} ) } g $.
\end{corl}
\begin{proof}
It follows from Lemma \ref{conjeq} that
\begin{align*}
g^{-1} G _{ ( \Delta _{1} ) } g = g^{-1} \Bigg( \bigcap _{ \delta \in \Delta _{1} } G_{ \delta }  \Bigg) g
= \bigcap _{ \delta \in \Delta _{1} } g^{-1} G_{ \delta } g
= \bigcap _{ \delta \in \Delta _{1} } G_{ \delta ^{g} }
= \bigcap _{ \delta ^{ \prime } \in \Delta _{2} } G_{ \delta ^{ \prime } }
= G _{ ( \Delta _{2} ) }
.
\end{align*}
\end{proof}
\begin{lema}\label{hghtkrnlllsts}
Suppose $G$ acts on $ \Omega $ with kernel $K$. If there exists $m \in \mathbb{N} $ such that $ G _{ ( \Delta ) } = K$ for all subsets $ \Delta \subseteq \Omega $ of size $ | \Delta | = m $, then $ Ht( G, \Omega ) \leq m $.
\end{lema}
\begin{proof}
Suppose $ \Lambda \subseteq \Omega $ with $ | \Lambda | > m $. Then there exists $ \Lambda ^{ \prime } \subset \Lambda $ of size $ | \Lambda ^{ \prime } | = m $. Since $K$ stabilizes all points of $ \Lambda $, we have
\[
K \leq G_{ ( \Lambda ) } \leq G_{ ( \Lambda  ^{ \prime } ) } = K .
\]
Thus $ G_{ ( \Lambda ) } = K = G_{ ( \Lambda  ^{ \prime } ) } $. This implies that $ \Lambda $ is not an independent set. Hence $ Ht( G, \Omega ) \leq m $.
\end{proof}
If $ n \geq 2 $, any subset of $ \{ 1 , \dots , n \} $ of size $ n -1 $ has trivial point stabilizer under the natural action of $ S_{n} $. Also the action is $ (n-1) $-transitive. So, using the above lemma, this action has height at most $ n -1 $. In fact, it was proved to be exactly $ n - 1 $ in Example \ref{hghtsym}. The final result of this section gives a way of finding out whether an $m$-transitive action has height exactly $m$ or not: check to see if a set of of size $m$ has the kernel as a stabilizer.
\begin{lema}\label{trnsqvhght}
Suppose $G$ is a finite group acting sharply $m$-transitively on a finite set $ \Omega $. Let $K$ be the kernel of the action. Then the following are equivalent;
\begin{itemize}
  \item[(i)] $  Ht ( G , \Omega )  = m $,
  \item[(ii)] $K = G_{ ( \Lambda ) } $ for all subsets $ \Lambda \subset \Omega $ of size $ | \Lambda | = m $,
    \item[(iii)] $K = G_{ ( \Delta ) } $ for some subset $ \Delta \subset \Omega $ of size $ | \Delta | = m $.
\end{itemize}
\end{lema}
\begin{proof}
$\textbf{(i)} \Rightarrow \textbf{(ii)}:$ Suppose $K \neq G_{ ( \Gamma _{1} ) } $ for some subset $ \Gamma _{1} \subset \Omega $ of size $ | \Gamma _{1} | = m $. Then $ G_{ ( \Gamma _{1} ) }  $ does not stabilize all elements of $ \Omega $. So there exists some $ \omega _{1} \in \Omega \setminus \Gamma _{1} $ and $g \in G_{ ( \Gamma _{1} ) } $ such that $ \omega _{1} ^{g} \neq \omega _{1} $. Put $ \Gamma ^{ \prime }  = \Gamma _{1} \cup \{ \omega _{1} \} $. Then
\[
| G_{ ( \Gamma ^{ \prime }  ) } | 
= | G_{ ( \Gamma _{1} ) } \cap G _{ \omega _{1} } |
< | G_{ ( \Gamma _{1} ) } | .
\]
Let $ \Gamma _{2} \subset \Gamma ^{ \prime }  $ be a subset of size $ |  \Gamma _{2}  |   = | \Gamma ^{ \prime }  | - 1 = m $. Then $ \Gamma _{2} = \Gamma ^{ \prime }  \setminus \{ \omega _{2}  \} $ for some $ \omega _{2} \in \Gamma ^{ \prime }  $. As the action is $m$-transitive, there exists some $ h \in G $ such that $ \Gamma _{2} = \Gamma _{1} ^{h} $. Hence $ G_{ ( \Gamma _{2} ) } = h ^{ -1} G_{ ( \Gamma _{1} ) } h $ by Corollary \ref{crlcnj}. It follows that 
\[
 | G_{ ( \Gamma ^{ \prime }  ) } | < | G_{ ( \Gamma _{1} ) } | 
 = | h^{-1} G_{ ( \Gamma _{1} ) } h | 
 =  | G_{ ( \Gamma _{2} ) } | .
 \]
So $ G_{ ( \Gamma ^{ \prime }  ) } \neq G_{ ( \Gamma _{2} ) } $. Therefore $ \Gamma ^{ \prime }  $ is independent by Lemma \ref{scndfnt} and so $ Ht ( G , \Omega ) \geq | \Gamma ^{ \prime }  | =  m + 1 $. Thus $ Ht ( G , \Omega ) \neq m $.

\medskip
$\textbf{(ii)} \Rightarrow \textbf{(i)} $: Suppose $K = G_{ ( \Lambda ) } $ for all subsets $ \Lambda \subset \Omega $ of size $ | \Lambda | = m $. Then $  Ht ( G , \Omega )  \leq m $ by Lemma \ref{hghtkrnlllsts}. Corollary \ref{crltrnstvhght} shows that $  Ht ( G , \Omega )  \geq m $ because the action is $m$-transitive. Hence $  Ht ( G , \Omega )  = m $.

\medskip
$\textbf{(ii)} \Rightarrow \textbf{(iii)} $: If $K = G_{ ( \Lambda ) } $ for all subsets $ \Lambda \subset \Omega $ of size $ | \Lambda | = m $ then clearly  $K = G_{ ( \Delta ) } $ for some subset $ \Delta \subset \Omega $ of size $ | \Delta | = m $.

\medskip
$\textbf{(iii)} \Rightarrow \textbf{(ii)} $: Suppose $K = G_{ ( \Delta ) } $ for some subset $ \Delta \subset \Omega $ of size $ | \Delta | = m $. Let $ \Sigma \subseteq \Omega $ where $ |  \Sigma  | = m $. As the action is $m$-transitive, there exists $a \in G$ such that $ \Delta ^{a} = \Sigma $. Using Corollary \ref{crlcnj} and the fact that $K \trianglelefteq G$, it follows that $G _{ (  \Sigma ) } = a^{-1} G _{ ( \Delta ) } a = a ^{-1} K a = K$.
\end{proof}

\newpage
\chapter{Primitive Actions of General and Special Linear Groups}\label{chapfour}
For the remainder of this text, the height and relational complexity of the primitive actions of $ PSL_{2} (q) $ and $ PGL_{2} (q) $ will be calculated. This chapter will provide the foundations, giving a description of primitive actions, $ PGL_{2} (q) $ and $ PSL_{2} (q) $. Some general results will be given to support the later chapters.
\section{Primitive Actions}
The following theorems help describe primitive group actions.
\begin{thrm}\label{transq}
Let $G$ be a group acting transitively on a set $ \Omega $. Let $ \omega \in \Omega $. The action of $G$ on $ \Omega $ is equivalent to the action by right multiplication of $G$ on the set of right cosets of the stabilizer $ G_{ \omega } $ in $G$.
\end{thrm}
\begin{proof}
See \cite{ROSE}, Chapter 4, page 76, Theorem 4.20.
\end{proof}
\begin{thrm}\label{maxg}
Let $G$ be a group acting on a set $ \Omega $ where $ | \Omega | \geq 2 $. Then $G$ is primitive if and only if the action is transitive and for each $ \omega \in \Omega $ the point stabilizer $G_{ \omega } $ is a maximal subgroup of $G$. 
\end{thrm}
\begin{proof}
See \cite{DANDM}, Chapter 1, page 14, Corollary 1.5A.
\end{proof}
Theorems \ref{transq} and \ref{maxg} together show that, up to equivalence, the primitive actions of a group are precisely the actions of right multiplication on the right cosets of maximal subgroups. This means that to look at the primitive actions of groups, it will be useful to understand the conjugacy classes of maximal subgroups.

\medskip
Most of the primitive actions that will be looked at will be those on the right cosets of a non-normal maximal subgroup. The following lemmas show that we can consider these the same as the conjugation action on the same maximal subgroup.
\begin{lema}\label{mxmlnrmlizr}
Let $ G$ be a group and $ H < G $ a maximal subgroup that is not normal in $ G $. Then $ N_{G} (H) = H $.
\end{lema}
\begin{proof}
The subgroup $ H $ normalizes itself, so $ H \leq N_{G} (H) \leq G $. Since $ H $ is not normal, $  N_{G} (H) \neq G $. it follows from the maximality of $ H $ that $ H = N_{G} (H) $.
\end{proof}
\begin{lema}\label{cstquivtcnj}
Let $ G $ be a group. Suppose $ H < G$ is maximal and $ H \ntriangleleft G $. Then the action by right multiplication on $ \Omega = \{ Hg : g \in G \} $ is equivalent to the action by conjugation on $ \Gamma = \{ g^{-1} H g : g \in G \} $.
\end{lema}
\begin{proof}
Define a map
\begin{align*}
\phi : \Omega & \rightarrow \Gamma \\
Hg & \mapsto g^{-1} H g .
\end{align*}
Suppose $g_{1} , g_{2} \in G $ and $ g^{-1} _{1} H g_{1}  = g^{-1} _{2} H g _{2}  $. Then $ H = g_{1} g^{-1} _{2} H g_{2} g^{-1} _{1} $. Hence $ g_{2} g^{-1} _{1} \in H $ by Lemma \ref{mxmlnrmlizr} and so $ g_{2} \in H g_{1} $. Therefore $ H g_{1} = H g_{2} $. It follows that $ \phi $ is injective.

\medskip
Clearly $  \phi $ is surjective, so the map is a bijection.

\medskip
Let $ g_{3} , g_{4} \in G $. Observe that $ ( H g_{3} g_{4} ) \phi = g_{4} ^{-1} g_{3} ^{-1} H g_{3} g_{4} =  g_{4} ^{-1} ( H g_{3} ) \phi g_{4}$.

\medskip
Following Definition \ref{quivdfn}, this shows that the actions are equivalent.
\end{proof}
This section is finished off with a few of results on primitive actions for use later.
\begin{lema}\label{nrmlprstmtsk}
Let $ G $ be group such that it is either simple or all proper non-trivial normal subgroups are maximal. Let $ M < G $ be a maximal subgroup. If $  H \triangleleft M $ and $ H \neq \{ 1_{G} \} $ then $ N_{G} (H) = M $.
\end{lema}
\begin{proof}
Since $ M \leq N_{G} (H) \leq G $, either $ N_{G} (H) = M  $ or $ N_{G} (H) = G  $. If $ G $ is simple then $ N_{G} (H) = M  $. So suppose $ G $ is non-simple and all proper non-trivial normal subgroups are maximal. If $ N_{G} (H) = G $ then $ H $ is maximal and $ H = M \neq H $, a contradiction.
\end{proof}
\begin{lema}\label{ntnrmlmxml}
Let $ G$ be group such that it is either simple or all proper non-trivial normal subgroups are maximal. Let $ H_{1} $ and $ H_{2} $ be distinct maximal subgroups. Suppose $ K \leq H_{1} \cap H_{2} $ and that $ K \neq \{ 1_{G} \} $. Then at least one of the following is true;
\begin{itemize}
\item $ K \ntrianglelefteq H_{1} $
\item $ K \ntrianglelefteq H_{2} $.
\end{itemize}
\end{lema}
\begin{proof}
Since $ H_{1} \neq H_{2} $ it follows that $ H_{1} \neq H_{1} \cap H_{2} \neq H_{2} $ and $ K $ is a proper subgroup of both $ H_{1} $ and $ H_{2} $. If $ K \trianglelefteq H_{1} $ and $ K \trianglelefteq H_{2} $ then $ H_{1} = N_{G} (K) = H_{2} $ by Lemma \ref{nrmlprstmtsk}, which is a contradiction. 
\end{proof}
\begin{corl}\label{cntrcnjgt}
Let $ G$ be group such that it is either simple or all proper non-trivial normal subgroups are maximal. Suppose $ H_{1} $ and $ H_{2} $ are maximal subgroups of $ G $. If $  Z(H_{1} ) \cap Z(H_{2} )  \neq \{  1_{G} \} $ then $ H_{1} = H_{2} $.
\end{corl}
\begin{proof}
We have $ Z(H_{1} ) \cap Z(H_{2} )  \trianglelefteq H_{1} $ and $ Z(H_{1} ) \cap Z(H_{2} )  \trianglelefteq H_{2} $. Also $ Z(H_{1} ) \cap Z(H_{2} )  \neq \{ 1 _{G} \} $. Thus $ H_{1} = H_{2} $ by Lemma \ref{ntnrmlmxml}.
\end{proof}
\section{General and Special Linear Groups}
The general linear group of invertible $ 2 \times 2 $ matrices over a finite field $ \mathbb{F} _{q} $ will be written as $ GL_{2} (q) $. The special linear subgroup of matrices with determinant $ 1 $ is denoted as $ SL_{2} (q) $.
\begin{lema}\label{cntrpglndpsltwq}
The centre of $ GL_{2} (q) $ is
\[
Z( GL_{2} (q) ) = \{ \lambda I : \lambda \in \mathbb{F} _{q} ^{*} \}
\]
and the centre of $ SL_{2} (q) $ is
\[
Z( SL_{2} (q) ) =  SL _{2} (q) \cap Z( GL_{2} (q) )  = \{ I, -I \}
.
\]
where $ I $ is the identity matrix.
\end{lema}

\newpage
\begin{proof}
Omitted.
\end{proof}
The quotient of $GL_{2} (q) $ by its centre is called the projective linear group and will be written as $ PGL_{2} (q)  = GL_{2} (q) / Z( GL_{2} (q) ) $.

\medskip
The projective special linear group is the quotient of $ SL_{2} (q) $ by its centre and will be written as $ PSL_{2} (q)  = SL_{2} (q) / Z( SL_{2} (q) ) $.
\begin{lema}\label{lnsz}
Let $d = gcd ( 2, q-1 ) $, where gcd is the greatest common divisor. The orders of the groups $ GL _{2} (q) $, $SL _{2} (q) $, $PGL _{2} (q) $ and $PSL _{2} (q) $ are
\begin{align*}
& | GL _{2} (q) | = q (q ^{2} - 1) ( q - 1) , \\
& | PGL _{2} (q) | = | SL _{2} (q) | = q (q ^{2} - 1)  , \ \ \ \  \text{ and} \\
& | PSL _{2} (q) |  = d ^{-1} q (q ^{2} - 1)  .
\end{align*}
\end{lema}
\begin{proof}
See \cite{TAYLOR}, Chapter 4, page 19.
\end{proof}
The next three facts are well known or easy to prove, so proofs have been left out.
\begin{lema}\label{sqrsltwq}
Suppose $q$ is odd. Let $I \in GL _{2} (q) $ be the identity matrix. If $ g \in GL _{2} (q) $ and $ g ^{2} = \lambda I $ for some $ \lambda \in \mathbb{F} _{q} ^{*} $, then either $ g = \mu I $ for some $ \mu  \in \mathbb{F} _{q} ^{*} $ or $ g$ has the form
\begin{align*}
g =
\begin{pmatrix*}[r]
w & x \\
y & -  w
\end{pmatrix*}
, \ \ \ \ \ w, x, y \in \mathbb{F} _{q} .
\end{align*}
\end{lema}
\begin{proof}
Omitted.
\end{proof}
\begin{lema}\label{pslsnpgl}
For any $ q $, the group $ PGL _{2} (q) $ contains a normal subgroup isomorphic to $ PSL _{2} (q) $.
\end{lema}
\begin{proof}
Omitted.
\end{proof}
Note that if $q$ is even then the above lemmas show that $PSL_{2} (q)$ has the same number of elements as $PGL_{2} (q)$ and $PSL_{2} (q) $ is a subgroup of $PGL_{2} (q)$. So the two groups are isomorphic and do not need to be considered separately when examining their primitive actions.
\begin{lema}\label{nrmlsbgrpspgl}
Suppose $ q \geq 4 $. The only normal subgroups of $ PGL _{2} (q) $ are itself, the trivial subgroup and one subgroup isomorphic to $ PSL _{2} (q) $. 	
\end{lema}
\begin{proof}
Omitted.
\end{proof}
\begin{lema}
Suppose $ q $ is odd. The action of $ PGL_{2} (q) $ on the right cosets of its $ PSL_{2} (q) $ subgroup has height $ 1 $ and relational complexity $ 2 $.
\end{lema}
\begin{proof}
The index of $ PSL_{2} (q) $ in $ PGL_{2} (q) $ is $ 2 $. This is the size of the set being acted on, so the height of the action is $ 1 $ by Lemma \ref{htbnd}. Lemma \ref{transtwo} tells us the relational complexity is $ 2 $.
\end{proof}
With this example out of the way, all other primitive actions of $ PSL_{2} (q) $ and $ PGL_{2} (q) $ will be on subgroups that are not normal.

\medskip
As the height and relational complexity of all primitive actions of $ PGL_{2} (q) $ will be calculated, we want to know its maximal subgroups. For $ q \leq 3 $, the height and relational complexity have been computed using GAP in \cite{WISCONS1} for both $ PSL_{2} (q) $ and $ PGL_{2} (q) $. So we will only consider $ q \geq 4 $. When $ q $ is odd, the subgroup of $ PGL_{2} (q) $ isomorphic to $ PSL_{2} (q) $ has index $ 2 $, so is maximal. The remaining maximal subgroups are given in \cite{GUIDICI1}, Theorem 3.5 and listed in tables below. The notation for subgroups used in the tables, and for the rest of this text, is $ C_{m} , D_{m} $ and $ E_{m} $ for cyclic, dihedral and elementary abelian groups respectively, each having $ m $ elements.

\medskip
The names that these subgroups are commonly called are provided in the table as well.
\begin{tble}\label{tableone}
Let $ q = p^{m} $ where $ p $ is an odd prime and $ m \in \mathbb{N} $. Suppose $ q \geq 5 $. The maximal subgroups of $ PGL_{2} (q) $ are;

\medskip
\begin{tabular}{|l|l|c|c|c|}
\hline
 Subgroup Name & \begin{tabular}{@{}c@{}} Subgroup \\ Structure \end{tabular} & Order & \begin{tabular}{@{}c@{}} No. of \\ Conjugacy Classes \end{tabular} & Notes \\
\hline
 $ PSL_{2} (q) $ & $ PSL_{2} (q) $ & $ \tfrac{1}{2} q (q+1)(q-1) $  & 1 &  \\
 Borel & $ E_{q} \rtimes C_{q-1} $ & $ q (q-1) $  & 1 &  \\
 Split Torus &  $ D_{2(q-1)} $ & $ 2(q-1) $  & 1 & $ q \neq 5 $ \\
 Non-Split Torus & $ D_{2(q+1)} $ & $ 2(q+1) $  & 1 & \\
 $ S_{4} $ & $ S_{4} $ & $ 24 $ & 1 & $ q=p \equiv \pm 3 \pmod{8} $ \\
 Subfield & $ PGL_{2} (q_{0} ) $ & $ q_{0} (q_{0} +1)(q_{0} -1) $ & 1 & \begin{tabular}{@{}c@{}} $ q= q_{0} ^{r} $, for \\ some odd prime $ r $ \end{tabular} \\
\hline
\end{tabular} 
\end{tble}
For $ q $ even, the maximal subgroups of $ PSL_{2} (q) $ are provided in \cite{BRAYHOLTRONEYDOUGAL}, page 377.
\begin{tble}\label{tabletwo}
Suppose $q $ is even and $ q \geq 4 $. The maximal subgroups of $ PSL_{2} (q) \cong PGL_{2} (q) $ are;

\medskip
\begin{tabular}{|l|l|c|c|c|}
\hline
 Subgroup Name & \begin{tabular}{@{}c@{}} Subgroup \\ Structure \end{tabular} & Order & \begin{tabular}{@{}c@{}} No. of \\ Conjugacy Classes \end{tabular} & Notes \\
\hline
 Borel & $ E_{q} \rtimes C_{q-1} $ & $ q (q-1) $  & 1 &  \\
 Split Torus &  $ D_{2(q-1)} $ & $ 2(q-1) $  & 1 &  \\
 Non-Split Torus & $ D_{2(q+1)} $ & $ 2(q+1) $  & 1 & \\
 Subfield & $ PSL_{2} (q_{0} ) $ & $ q_{0} (q_{0} +1)(q_{0} -1) $ & 1 & \begin{tabular}{@{}c@{}} $ q= q_{0} ^{r} $, for some $ q_{0} \neq 2 $ \\  and some prime $ r $ \end{tabular} \\
\hline
\end{tabular} 
\end{tble}
For $ q $ odd, the maximal subgroups of $ PSL_{2} (q) $ are provided in \cite{BRAYHOLTRONEYDOUGAL}, page 380.
\begin{tble}\label{tablethree}
Let $ q = p^{m} $ where $ p $ is an odd prime and $ m \in \mathbb{N} $. Suppose $ q \geq 5 $. The maximal subgroups of $ PSL_{2} (q) $ are;

\medskip
\begin{tabular}{|l|l|c|c|c|}
\hline
 Subgroup Name & \begin{tabular}{@{}c@{}} Subgroup \\ Structure \end{tabular} & Order & \begin{tabular}{@{}c@{}} No. of \\ Conjugacy Classes \end{tabular} & Notes \\
\hline
Borel & $ E_{q} \rtimes C_{(q-1)/2} $ & $ \tfrac{1}{2} q (q-1) $  & 1 &  \\
Split Torus &  $ D_{(q-1)} $ & $ (q-1) $  & 1 & $ q \neq 5,7,9,11 $ \\
Non-Split Torus & $ D_{(q+1)} $ & $ (q+1) $  & 1 & $ q \neq 7 , 9 $ \\
$ A_{4} $ & $ A_{4} $ & $ 12 $  & 1 & $ q = p \equiv \pm 3, 5, \pm 13 \pmod{40} $ \\
$ S_{4} $ & $ S_{4} $ & $ 24 $  & 2 & $ q = p \equiv \pm 1 \pmod{8} $ \\
Subfield & $ PSL_{2} (q_{0} ) $ & $ \tfrac{1}{2} q_{0} (q_{0} +1)(q_{0} -1) $ & 1 & \begin{tabular}{@{}c@{}} $ q= q_{0} ^{r} $, for some \\ prime $ r $ \end{tabular} \\
Subfield & $ PGL_{2} (q_{0} ) $ & $ q_{0} (q_{0} +1)(q_{0} -1) $ & 2 & \begin{tabular}{@{}c@{}} $ q= q_{0} ^{2} $ \end{tabular} \\
$ A_{5} $ & $ A_{5} $ & $ 60 $  & 2 & \begin{tabular}{@{}c@{}} $ q = p \equiv \pm 1 \pmod{10} $ or \\
$ q = p^{2} $ and $ p \equiv \pm 3 \pmod{10} $  \end{tabular} \\
\hline
\end{tabular} 
\end{tble}

\newpage
\section{General Results for Linear Groups}
Various results on finite fields and linear groups are collected in this section, to be used in later chapters. It is likely the reader is familiar with most (if not all) of these lemmas. Proofs are provided for completeness.
\begin{lema}\label{smrphsmspsltrn}
We have the following isomorphisms
\begin{itemize}
\item $ PSL_{2} (2) \cong S_{3} $,
\item $ PSL_{2} (3) \cong A_{4} $,
\item $ PSL_{2} (4) \cong PSL_{2} (5) \cong A_{5} $,
\item $ PSL_{2} (9) \cong A_{6} $,
\item $ PGL_{2} (3) \cong S_{4} $,
\item $ PGL_{2} (5) \cong S_{5} $.
\end{itemize}
\end{lema}
\begin{proof}
See \cite{HUPPERT1}, Theorem 6.14, page 183 and \cite{COXETERMOSER1}, page 96.
\end{proof}
\begin{lema}\label{sqnmbrf}
If $ q $ is odd, the number of elements that are squares in $ \mathbb{F} _{q} ^{*} $ is $ \tfrac{1}{2} (q-1) $. If $ q $ is even there are $ q-1 $ squares in $ \mathbb{F} _{q} ^{*} $.
\end{lema}
\begin{proof}
The map $  (x) \psi = x^{2} $ is a group homomorphism from $  \mathbb{F} _{q} ^{*}  $ to $ \mathbb{F} _{q} ^{*} $. Since $ \ker ( \psi ) = \{ 1 , -1 \} $, the image of $ \psi $ has $ \tfrac{1}{2} (q-1) $ elements if $ q $ is odd and $ q - 1 $ if $ q $ is even.
\end{proof}
\begin{lema}\label{mxmlsbpslntr}
Put $ G := PGL_{2} (q) $ where $ q $ is odd and $ q \geq 13 $. Let $ M < G $ be a maximal subgroup that is not normal in $ G $. Also suppose that if $ q $ is prime then $ M \not\cong S_{4} $. Let $ S \leq G $ where $ S \cong PSL_{2} (q) $. Then $ M \cap S $ is maximal in $ S $ and $ | M \cap S | = \tfrac{1}{2} | M | $.
\end{lema}
\begin{proof}
Observe that $ S \trianglelefteq G $ by Lemma \ref{pslsnpgl}. Hence $ MS/S \cong M / (M \cap S ) $ by the isomorphism theorems for groups.

\medskip
As $ M $ is not normal in $ G $ we have $ M \neq S $. Using Lemma \ref{lnsz}, we see that $ | G : S | = 2 $. Hence $ S $ is maximal in $ G $. Therefore $ M \setminus ( M \cap S ) \neq \emptyset $. It follows that $ MS = G $. We then have
\[
| M / ( M \cap S ) | = | MS/S | = | G / S | = 2 .
\]
So $ | M : M \cap S | = 2 $. Theorem 1.1 of \cite{GUIDICI1} shows that $ M \cap S $ is maximal in $ S $.
\end{proof}
\begin{lema}\label{pslnvltnscnjg}
All elements of order $2$ lie in a single conjugacy class in $ PSL_{2} (q) $. The number of involutions in $ PSL_{2} ( q) $ is
\begin{itemize}
\item $ (q + 1)(q-1) $ if $ q $ is even,
\item $ \tfrac{1}{2} q (q-1) $ if $ q \equiv 3 \pmod{4} $,
\item $ \tfrac{1}{2} q (q+1) $ if $ q \equiv 1 \pmod{4} $.
\end{itemize}
\end{lema}
\begin{proof}
See \cite{GURALNICK1}, Lemma A.3.
\end{proof}
\begin{lema}\label{pgltqnvcnjc}
Suppose $ q $ is odd with $ q \geq 5 $. There are two conjugacy classes of elements of order $2$ in $ PGL_{2} (q) $. One conjugacy class contains $ \tfrac{1}{2} q (q-1) $ elements and the other contains $ \tfrac{1}{2} q (q+1) $ elements.
\end{lema}
\begin{proof}
It is shown in \cite{GURALNICK1}, Lemma A.3 that there exist $ q ^{2} $ elements of order $ 2 $ in $ PGL_{2} (q) $ and that there are two conjugacy classes containing these involutions.

\medskip
In Table \ref{tableone} we see that there exists a maximal $ D < PGL_{2} (q) $ with $ D \cong D_{2 (q+1) } $ and $ D $ does not contain $ PSL_{2} (q) $ as a subgroup. Also $ D \ntriangleleft PGL_{2} (q) $ by Lemma \ref{nrmlsbgrpspgl}.

\medskip
Put $ \Gamma := \{ g^{-1} D g : g \in PGL_{2} (q) \} $. By Lemma \ref{cstquivtcnj} the action on the right cosets of $ D $ is equivalent to the action by conjugation on $ \Gamma $. Therefore $ | \Gamma | = | PGL_{2} (q) | / |D | = \tfrac{1}{2} q (q-1) $.

\medskip
Let $ D^{ \prime } , D^{ \prime \prime } \in \Gamma $ and suppose $ D^{ \prime } \neq D^{ \prime \prime } $. Since these subgroups are isomorphic to $  D_{2 (q+1) } $ and $ q + 1 $ is even with $ q + 1 > 2 $, it follows that $ |Z(D) | = 2 $. Therefore there exist elements $ a^{ \prime } \in Z(D^{ \prime } ) $ and $ a^{ \prime  \prime } \in Z(D^{ \prime \prime } ) $ of order $ 2 $. Also $ D^{ \prime \prime } = h^{-1} D^{ \prime } h $ for some $ h \in PGL_{2} (q) $, so $a^{ \prime \prime } =  h^{-1} a^{ \prime } h $.

\medskip
As $ Z(D^{ \prime } ) \triangleleft D^{ \prime } $ and $ Z(D^{ \prime \prime  } ) \triangleleft D^{ \prime \prime  } $, it follows from Lemma \ref{ntnrmlmxml} that $ Z(D^{ \prime } ) \neq Z(D^{ \prime \prime } ) $. Hence $ a^{ \prime } \neq a^{ \prime \prime }  $. This shows that the non-identity elements in the centres of the subgroups in $ \Gamma $ are distinct and lie in some conjugacy class $ C_{1} \subset PGL_{2} (q) $. Also $   \tfrac{1}{2} q (q-1) = | \Gamma |  \leq | C_{1} | $.

\medskip
If $ b \in C_{1} $ then $ b = f^{-1} a ^{ \prime } f $ for some $ f \in PGL_{2} (q) $. So $ b \in Z(f^{-1} D ^{ \prime } f ) $. Hence $ | C_{1} | \leq  | \Gamma | = \tfrac{1}{2} q (q-1) $. Thus $ | C_{1} | = \tfrac{1}{2} q (q-1) $.

\medskip
All other involutions must lie in the second conjugacy class, $ C_{2} $, which must have size $ | C_{2} | = q^{2} - | C_{1} | =  
\tfrac{1}{2} q (q+1) $.
\end{proof}
\begin{lema}\label{refconjug}
Let $ D $ be a dihedral group and suppose $ | D | $ is divisible by $ 4 $. Let $ K \leq  D $ be a Klein four-subgroup. Let $ r_{1} , r_{2} \in K $ be involutions that are also reflections in $ D $. Then $ r_{1} $ and $ r_{2} $ are conjugate to each other if and only if $ |D| $ is divisible by $ 8 $.
\end{lema}
\begin{proof}
Put $ n := | D | / 2 $. We can write $ D = \langle s , r_{1} \ : \ s^{n} = r_{1}^{2} = 1_{D} , \ r_{1}sr_{1} = s^{-1} \rangle $, where $ s $ generates the subgroup of rotations. We can then express $ K $ as $ K = \{ 1_{D} , s^{n/2} , r_{1} , r_{2} \} $.

\medskip
As $ n $ is even the reflections split into two conjugacy classes; $ \{ r_{1}, \ r_{1} s^{2} , \dots , \ r_{1} s^{n-2} \} $ and $ \{ r_{1} s, \ r_{1} s^{3} , \dots , \ r_{1} s^{n-1} \} $.

\medskip
If $ | D | $ is divisible by $ 8 $, then $ n/2 $ is even and $ r_{2} = r_{1} s^{n/2} $ must lie in the same conjugacy class as $ r_{1} $. If $ | D | $ is not divisible by $ 8 $, then $ n/2 $ is odd and $ r_{2} = r_{1} s^{n/2} $ must lie in opposite conjugacy classes.
\end{proof}
The notation $ K_{4} $ will always be used for a Klein-four group.
\begin{lema}\label{allkfrcnjdm}
Let $ m \geq 8 $ with $ m \equiv 0 \pmod{4} $. All $ K_{4} $ subgroups of $ D_{m} $ are conjugate if and only if $ m \equiv 4 \pmod{8} $.
\end{lema}
\begin{proof}
Let $ D = D_{m} $ and $ K \leq D $ with $ K \cong K_{4} $.

\medskip
If $ m \equiv 0 \pmod{8} $ then the reflections of $ D$ that lie in $ K$ are conjugate to each other by Lemma \ref{refconjug}. Same goes for all other Klein four subgroups. The Klein four subgroups split into two conjugacy classes, corresponding to the two conjugacy classes of reflections in $ D$.

\medskip
If $ m \equiv 4 \pmod{8} $ then the reflections of $ D$ that lie in $ K$ are not conjugate to each other by Lemma \ref{refconjug}. Same goes for all other Klein four subgroups. Each reflection of $ D $ lies in exactly one $ K_{4} $. Therefore all $ K_{4} $ are conjugate.
\end{proof}

\newpage
\begin{lema}\label{sylwlmma}
Let $ G $ be either $ PGL_{2} (q) $ or $ PSL_{2} (q) $ where $ q \geq 4 $. Let $ r $ be an odd prime that divides $ | G | $ but does not divide $ q $. Then the Sylow $ r $-subgroups of $ G $ are cyclic and are subgroups of some $ C_{ (q \pm 1) / \delta } $ in $ G $, where $ \delta = 2 $ if $ G = PSL_{2} (q) $ with $ q $ odd and $ \delta = 1 $ if $ G = PGL_{2} (q) $ with $ q $ odd or even.
\end{lema}
\begin{proof}
Let $ S \leq G $ be a Sylow $r$-subgroup. The order of $ G $ is $ q (q+1) (q-1) / \delta $ and $ r $ does not divide $ q $, so $ r $ must divide $ (q+1) (q-1) $. As $ r $ is an odd prime it cannot divide both $ q+1 $ and $ q-1 $, so divides exactly one of these numbers. Let $ m \in \{ q+1 , q-1 \} $ where $ r | m $. Observe that as $ | S | $ is a power of $ r $ it cannot divide $ q $ and it must divide $ m $. Since $ \delta \in \{ 1 , 2 \} $ and $ r $ is an odd prime, we see that both $ r $ and $ |S | $ divide $ m / \delta $.

\medskip
Now $ |G| $ is divisible by at least two primes. Therefore $ S \neq G $ and $ S $ lies in some maximal subgroup of $ G $.

\medskip
There exists a maximal Borel subgroup $ B < G $ and $ C < B $ with $ C \cong C_{ (q-1) / \delta } $. So if $ m = q - 1 $ then $ C $ has a cyclic subgroup with $ | S | $ elements, which is a Sylow $r$-subgroup of $ G $. Every Sylow $ r $-subgroup then lies in some conjugate of $ C $.

\medskip
If $ m = q+1 $ and $ G $ contains a maximal $ D_{2(q+1) / \delta } $, then its rotational subgroup is a $ C_{ (q+1) / \delta } $, which in turn contains a cyclic subgroup with $ | S | $ elements. Thus all Sylow $ r $-subgroups lie in some conjugate of this $ C_{ (q+1) / \delta } $.

\medskip
If $ m = q+1 $ and $ G $ does not contain a maximal $ D_{2(q+1) / \delta } $ then it can be seen from Tables \ref{tableone}, \ref{tabletwo} and \ref{tablethree} that $ G = PSL_{2} (q) $ with $ q $ odd and $ q \in \{ 7,9 \} $.

\medskip
We can rule out $ PSL_{2} (7) $ and $ m= 7+1 = 8 $ because in that case $ m $ has no odd prime divisors.

\medskip
The remaining case is $ G = PSL_{2} (9) $ and $ m = 10 $. The only odd divisor of $ m $ that is a power of a prime is $ 5 $, so we are looking for a $ C_{5} $ in $ G $. This must exist because $ 5 $ divides $ | G | $.
\end{proof}
\begin{lema}\label{dcvlkjhsdp}
Let $ G $ be $ PSL_{2} (q) $ or $ PGL_{2} (q) $, where $ q \geq 4 $. There exist dihedral subgroups of $ G $ of order $ 2(q+1) / \delta $ and $ 2(q-1) / \delta $, where $ \delta = 2 $ if $ G = PSL_{2} (q) $ with $ q $ odd and $ \delta = 1 $ if $ G = PGL_{2} (q) $ with $ q $ odd or even. If $ D \leq G $ is such a subgroup, with $ | D | \geq 4 $, and $ R < D $ is a non-trivial subgroup of the rotational subgroup (or $ R $ is any $ C_{2} $ if $ D \cong D_{4} $), then $ D = N_{G} ( R ) $.
\end{lema}
\begin{proof}
The only time we can have $ D_{ 2(q \pm 1) / \delta  } \cong D_{4} $ is when $ G = PSL_{2} (5) \cong A_{5} $. For this choice of $ G $, the elements whose order does not divide $ q $ have order $ 2 $ or $ 3 $ and it can be quickly checked the subgroups generated by these elements are normalized by a $ D_{4} $ or $ D_{6} $ as required.

\medskip
Suppose for the rest of the proof that $ G \neq PSL_{2}(5) $, so $ | D_{ 2(q \pm 1) / \delta  } | \geq 6 $. For almost all $ q \geq 4 $ there exist $ D_{2 (q \pm 1 ) / \delta } $ that are maximal. Any non-trivial subgroup of the rotational subgroup of a maximal $ D_{2 (q \pm 1 ) / \delta } $ must be normalized by the dihedral containing it because $ PSL_{2} (q) $ and $ PGL_{2} (q) $ do not contain non-trivial cyclic normal subgroups.

\medskip
To ensure that all dihedral subgroups of order $ 2 (q \pm 1 ) / \delta $ are the normalizers of subgroups in their rotational subgroup, it needs to be checked whether any non-maximal dihedrals could exist alongside maximal ones of the same order. Suppose $ H < G $ and $ H $ is isomorphic to a maximal $ D_{2 (q \pm 1 ) / \delta } $. The rotational subgroup is a $ C_{ (q \pm 1 ) / \delta } $. Let $ C$ be a subgroup of the rotational subgroup and $ r $ be a prime dividing $ | C | $. Then there exists $ c \in C $ with $ o (c) = r $. Also $ r $ divides $ q- 1 $ or $ q + 1 $ and so does not divide $ q $.

\medskip
If $ r $ is odd, the Sylow $ r $-subgroups lie in a $ C_{2 (q - 1 ) / \delta } $ or $ C_{2 (q + 1 ) / \delta } $ in $ G $ by Lemma \ref{sylwlmma}, but not both since $ r $ cannot divide both $ q-1 $ and $ q+1 $. Therefore there must be Sylow $ r $-subgroups in maximal dihedrals isomorphic to $ H $. One of these maximal dihedral contains $ \langle c \rangle $ and normalizes it. This maximal dihedral then also contains $ H$ and therefore must be $ H $.

\medskip
If $ r = 2 $ then $ c $ is the unique rotation of order $ 2 $ in $ H $. The maximal dihedral isomorphic to $ H $ also have a rotation of order $ 2 $. If $ G = PSL_{2} (q) $ then all involutions in $ G $ are conjugate by Lemma \ref{pslnvltnscnjg}. Hence $ \langle c \rangle $ is normalized by some maximal dihedral containing $ H $, which must be $ H $ itself because $ | H | $ cannot divide both $ 2(q-1) / \delta $ and $ 2(q+1) / \delta $. 

\medskip
If $ G = PGL_{2} (q) $ then both the $ D_{2(q-1) } $ and $ D_{ 2(q+1) } $ in $ G $ have involutions in their centres. Lemma \ref{pgltqnvcnjc} shows that there are two conjugacy class of involutions. So by the same reasoning as above, the involutions in one conjugacy class must be normalized by $ D_{2(q-1) } $ and the involutions in the other conjugacy class by $ D_{2(q+1) } $. Whichever type of dihedral normalizes $ \langle c \rangle $ also $ H $. Only one of these types of dihedral can contain a subgroup of order $ |H| $ and therefore $ H $. So $ H $ must be isomorphic to the same subgroups in that conjugacy class and be maximal.

\medskip
It can be seen in Tables \ref{tableone}, \ref{tabletwo} and \ref{tablethree} that the cases where maximal dihedrals of order $ 2 (q - 1 ) / \delta $ or $ 2 (q + 1 ) / \delta $ do not exist are $ PGL_{2} (5) $ and $ PSL_{2} (q) $ where $ q \in \{ 7,9,11 \} $.

\medskip
For $ PGL_{2} (5) $, the subgroup structure of $ S_{5} \cong PGL_{2} (5) $ can be examined and it is easily seen that there exist $ D_{8} $ and $ D_{10} $ subgroups, which are also the normalizers in $ PGL_{2} (5) $ of their rotational subgroups.

\medskip
For the cases $ PSL_{2} (q) $ where $ q \in \{ 7,9,11 \} $, the non-maximal dihedrals of order $ 2(q-1) / \delta $ or $ 2(q + 1) / \delta $ are contained in some maximal $ S_{4} $ or $ A_{5} $ (for $ PSL_{2} (9) $ the $ S_{4} $ containing $ D_{8} $ are subfield subgroups; $ PGL_{2} (3) \cong S_{4} $). If $ C^{ \prime } $ is in the rotational subgroup of one of these dihedrals $ H^{ \prime } < G $ and $ C^{ \prime } \neq \{ 1_{G} \} $ then $ H^{ \prime } \leq N_{G} ( C^{ \prime } ) $.

\medskip
Since $ H^{ \prime } $ is not maximal, there exists some maximal subgroup $ M < G $ such that $ C^{ \prime } < H^{ \prime } \leq N_{G} (C^{ \prime } ) \leq M < G $. It can be checked that if $ M $ is an $ S_{4} $ or $ A_{5} $ that $ C^{ \prime } $ would not be normal in any subgroup of $ M $ except $ H^{ \prime } $. Clearly we could not have $ M \cong H^{ \prime } $. Similarly we cannot have $ H^{ \prime } \cong D_{2 (q - 1 ) / \delta } $ and $ M \cong  D_{2 (q + 1 ) / \delta } $, or vice versa, because $  | H^{ \prime }  | $ would not divide $ | M | $. The only other possibility for $ M $ is that it is a Borel subgroup or in the case $ q = 9 $ a subfield subgroup isomorphic to $ PGL_{2} (3) $.

\medskip
If $ G = PSL_{2} (9) $ and $ M \cong PGL_{2} (3) \cong S_{4} $ then $ H^{ \prime } \cong D_{8} $ and $ C^{ \prime } $ is either a $ C_{2} $ or $ C_{4} $. In either case $ N_{M} (C^{ \prime } ) = H^{ \prime } $.

\medskip
The final possibility for $ M $ is that it is a Borel subgroup of $ PSL_{2} (q) $. The only way this is possible is if there exist involutions in $ M $. Looking at the order of Borel subgroups, we see the only outstanding case where there exist Borel subgroups with involutions is when $ G = PSL_{2} (9) $. Borel subgroups have order $ 36 $ in $ PSL_{2} (9) $, so do not contain a $ D_{8} $ or $ D_{10} $.
\end{proof}
\begin{lema}\label{nrmlzr}
Let $ G $ be either $ PGL_{2} (q) $ or $ PSL_{2} (q) $ where $ q \geq 3 $. Let $ g \in G $. Suppose $ o(g) $ is coprime to $ q $. Then $ N_{G} ( \langle g \rangle ) \cong D_{ 2( q \pm 1 ) / \delta } $, where $ \delta = 2 $ if $ G = PSL_{2} (q) $ or $ \delta = 1 $ if $ G = PGL_{2} (q) $. In addition, all $ D_{ 2( q \pm 1 ) / \delta } $ in $ G $ are conjugate to $ N_{G} ( \langle g \rangle ) $ and if $ G $ has maximal subgroups isomorphic to $ N_{G} ( \langle g \rangle ) $ then $ N_{G} ( \langle g \rangle ) $ is maximal.
\end{lema}
\begin{proof}
If $ G = PSL_{2} (3) \cong A_{4} $ or $ G = PGL_{2} (3) \cong S_{4} $ it is straight forward to check that $ o(g) \in \{ 2 , 4 \} $ and $ \langle g \rangle $ is normalized by a $ D_{ 4 / \delta } $ or a $ D_{ 8 / \delta } $. So suppose $ q \geq 4 $.

\medskip
First suppose $ o(g) = 2^{m} $ for some $ m  \in \mathbb{N} $. Then $ q $ must be odd. Also there exists an involution $ g_{1} \in \langle g \rangle $. Note that $ \langle g \rangle \leq C_{G} (g_{1} ) \leq N_{G} ( \langle g_{1} \rangle ) $. If $ G = PSL_{2} (q) $ then Lemma \ref{dcvlkjhsdp} shows there exists $ D < G $ with $ D \cong D_{q \pm 1} $ that has an involution $ d $ at its centre and $ D = N_{G} ( \langle d \rangle ) $. All involutions in $ G $ are conjugate by Lemma \ref{pslnvltnscnjg}. Hence $ \langle g \rangle $ is normalized by some conjugate of $ D $. 

\medskip
If $ G = PGL_{2} (q) $ then both the $ D_{2(q-1) } $ and $ D_{ 2(q+1) } $ in $ G $ have involutions in their centres. Lemma \ref{pgltqnvcnjc} shows that there are two conjugacy class of involutions. So by the same reasoning as above, the involutions in one conjugacy class must be normalized by $ D_{2(q-1) } $ and the involutions in the other conjugacy class by $ D_{2(q+1) } $.

\medskip
In both the above cases, if $ D ^{ \prime } < G $ is maximal and $ D^{ \prime } \cong D $ then it has an involution in its centre, is the normalizer of the $ C_{2} $ the involution generates and so must be conjugate to $ D $, meaning $ D $ is maximal.

\medskip
Next suppose $ o(g) $ is divisible by some odd prime $ r $. Then there exists an element of $ g_{2} \in \langle g \rangle $ with $ o(g_{2} ) = r $. As before, $ \langle g \rangle \leq C_{G} (g_{2} ) \leq N_{G} ( \langle g_{2} \rangle ) $. By Lemma \ref{sylwlmma}, a Sylow $ r $-subgroup is contained in some $ C_{(q \pm 1) /\delta } $. So Sylow $ r $-subgroups are in the rotational subgroups of either the $ D_{2(q-1) / \delta } $ or $ D_{2(q+1) / \delta } $ in $ G$, which exist by Lemma \ref{dcvlkjhsdp}. The same lemma shows that such a dihedral is the normalizer of the Sylow $ r $-subgroup it contains. As all Sylow $ r $-subgroups are conjugate, $ N_{G} ( \langle g_{2} \rangle ) \cong D_{2(q \pm 1) / \delta } $. It must be that $ \langle g \rangle $ lies in the rotational subgroup of $  N_{G} ( \langle g_{2} \rangle ) $, implying $ N_{G} ( \langle g \rangle ) = N_{G} ( \langle g_{2} \rangle )  $ by Lemma \ref{dcvlkjhsdp}.

\medskip
All subgroups isomorphic to $ N_{G} ( \langle g \rangle ) $ normalize a subgroup conjugate to $ \langle g_{2} \rangle $, thus all $ D_{ 2( q \pm 1 ) / \delta }  $ are conjugate. So if $ D^{ \prime \prime } < G $ is maximal and $ D^{ \prime \prime } \cong D $ then $ D $ is maximal.
\end{proof}
\begin{lema}\label{kfrsbgrpllcn}
Suppose $ q \equiv \pm 3 \pmod{8} $. Any Klein-four subgroups of $ PSL_{2} (q) $ are conjugate to each other.
\end{lema}
\begin{proof}
For these choices of $ q $ the order of $ PSL_{2} (q) $ is not divisible by $ 8 $. So any Klein four-subgroups of $ PSL_{2} (q) $ are Sylow $ 2 $-subgroups, which lie in a single conjugacy class by Sylow's theorems.
\end{proof}

\newpage
\chapter{The Borel Action}\label{chapfive}
The remaining chapters will each be structured in the following way. First a preliminary results section will give any lemmas needed that are not specific to the action in question. Next will be a section describing the action. That will be followed by a section for calculating the height of the action and finally a section for finding the relational complexity. Where appropriate, the sections for height and relational complexity might be combined into one. In the chapter for the $ A_{5} $ action there is no section for preliminary results.

\medskip
In this chapter the Borel actions of $ PSL_{2} (q) $ and $ PGL_{2} (q) $ will be looked at. Exceptionally, the case $ q \geq 3 $ is considered here (unlike other actions where $ q \geq 4 $). The case $ q = 3 $ is required for use later in this paper. Suppose $ G $ is $ PSL_{2} (q) $ or $ PGL_{2} (q) $ acting on $ \Omega $, the right cosets of a maximal Borel subgroup. The main results of this chapter for height are:

\medskip
If $ G = PSL_{2} (3) $ then $ Ht(G, \Omega ) = 2 $. See Theorem \ref{slbrlhght}.

\medskip
If $ G = PSL_{2} (q) $ with $ q \geq 4 $ or $ G = PGL_{2} (q) $ with $ q \geq 3 $ then $ Ht(G, \Omega ) = 3 $. See Theorems \ref{glbrlhght} and \ref{slbrlhght}.

\medskip
The main results of this chapter for relational complexity are:

\medskip
If $ G = PGL_{2} (3) $ then $ RC(G, \Omega ) = 2 $. See Theorem \ref{brlglwhtever}.

\medskip
If $ G = PGL_{2} (q) $ with $ q \geq 4 $ then $ RC(G, \Omega ) = 4 $. See Theorem \ref{brlglwhtever}.

\medskip
If $ G = PSL_{2} (3) $ or $ G = PSL_{2} (5) $ then $ RC(G, \Omega ) = 3 $. See Theorem \ref{xckjdsqwmmmm}.
 
\medskip
If $ G = PSL_{2} (q) $ with $ q $ odd and $ q \geq 7 $ then $ RC(G, \Omega ) = 4 $. See Theorem \ref{xckjdsqwmmmm}.
\section{Preliminary Results}
There is just one preliminary result for this chapter.
\begin{lema}\label{ktrnstiveprmt}
A $2 $-transitive action is primitive.
\end{lema}
\begin{proof}
Let $ G $ act $2$-transitively on $ \Omega $. Suppose the action is not primitive. Then there exists a non-trivial block of imprimitivity $ B \subset \Omega  $. Let $ \beta _{1} , \beta _{2} \in B $ with $ \beta _{1} \neq \beta _{2} $. Let $ \alpha \in \Omega \setminus B $. As the action is $2$-transitive, $ \beta _{1} ^{g} = \beta _{1} $ and $ \beta _{2} ^{g} = \alpha $ for some $ g \in G $. However $ \beta _{1} \in B \cap B^{g} $ and $ B \neq B^{g} $. We conclude $ B$ is not a block, which is a contradiction.
\end{proof}

\newpage
\section{Action Description}
The first action to be looked at is the action on the so called Borel subgroups of $ PGL_{2} (q) $ and $ PSL_{2} (q) $. Although in most chapters we will usually take $ q $ to be at least $ 4 $, for this chapter we will assume $ q \geq 3 $. The reason for including $ q = 3 $ this time is that some later chapters will rely on looking at Borel subgroups, including those in $ PGL_{2} (3) $. Lemmas will be developed in this chapter to support those later results.

\medskip
For $ q \geq 4 $, these are the subgroups of the form $ E_{q} \rtimes C_{ q-1} $ and $ E_{q} \rtimes C_{ (q-1)/2} $ in Tables \ref{tableone}, \ref{tabletwo} and \ref{tablethree}. These tables point out the Borel subgroups lie in a single conjugacy class.  For $ PGL_{2} (3) $, observe that $ PGL_{2} (3) \cong S_{4} $, which also has a conjugacy class of maximal $ E_{3} \rtimes C_{2} \cong D_{6} $. For $ PSL_{2} (3) $ we have $ PSL_{2} (3)  \cong A_{4} $, which has maximal $ E_{3} \rtimes C_{1} \cong C_{3} $.

\medskip
It is going to be easier to work with the corresponding maximal subgroups of $ GL_{2} (q) $ and $ SL_{2} (q) $.  First some maximal subgroups will be defined in these groups then shown to correspond to those in the projective groups.

\medskip
Let $ GL_{2} (q) $ or $ SL_{2} (q) $. Let $ \mathbb{F}_{q} $ be a field with $ q $ elements and let $ \Omega $ be the set of $ 1$-dimensional subspaces of $\mathbb{F} _{q} ^{2} $. It is not difficult to show $ G $ acts on $ \Omega $ with the mapping $ W ^{g} =  \langle w \rangle ^{g} = \{ w g : w \in W \}  $ for each $ W \in \Omega $ and $ g \in G $. For the rest of this text, the notation $ \mathbb{F}_{q} ^{*} $ is used for $ \mathbb{F}_{q} \setminus \{ 0 \} $.
\begin{lema}\label{szbrlst}
The set $ \Omega  $ has size $ | \Omega  | = q + 1 $.
\end{lema}
\begin{proof}
The elements of $ \Omega  $ are the $1$-dimensional subspaces of $ \mathbb{F} _{q} ^{2} $ and each of these can be written uniquely as $ \langle ( 1, 0 ) \rangle $, $ \langle ( 0, 1 ) \rangle $ or $ \langle ( 1, x ) \rangle $ where $ x \in \mathbb{F}_{q} ^{*} $. So $ | \Omega  | = q + 1  $.
\end{proof}
It will be useful to know what the kernels of these actions are.
\begin{lema}\label{krnlslgl}
The kernel of the action of $ GL_{2} (q) $ is
\[
Z( GL_{2} (q) ) = \Bigg\{ 
\begin{pmatrix}
a & 0 \\
0 & a
\end{pmatrix}
:
a \in \mathbb{F}_{q} ^{*}
\Bigg\}
\]
and the kernel of the action of $ SL_{2} (q) $ is
\[
Z( SL_{2} (q) ) =  SL _{2} (q) \cap Z( GL_{2} (q) )  = \Bigg\{ 
\begin{pmatrix}
 1 & 0 \\
0 & 1
\end{pmatrix}
,
\begin{pmatrix*}[r]
 -1 & 0 \\
0 & -1
\end{pmatrix*}
\Bigg\}
.
\]
\end{lema}
\begin{proof}
See \cite{TAYLOR}, Chapter 4, page 19.
\end{proof}
As the kernel of the action is the centre of $ G $, a faithful action of $ PGL_{2} (q) $ or $ PSL_{2} (q) $ on $ \Omega $ can be defined where $ \omega ^{g} = \omega ^{ Z(G) g } $ for all $ \omega \in \Omega $ and $ g \in G $.
\begin{lema}\label{psl2tr}
The action of $  PSL_{2} (q) $ is $2$-transitive.
\end{lema}
\begin{proof}
See \cite{TAYLOR}, Chapter 4, page 20, Theorem 4.1.
\end{proof}
\begin{corl}\label{cor2t}
The actions of $ SL_{2} (q) $, $ GL_{2} (q) $ and $ PGL_{2} (q) $ are $2$-transitive.
\end{corl}
\begin{proof}
Since the action of $ PSL_{2} (q) $ is $ 2 $-transitive, so is the action of $ SL_{2} (q) $. The fact $ SL_{2} (q) \leq GL_{2} (q) $ means the action of $ GL_{2} (q) $ is $ 2 $-transitive. The action of $ PGL_{2} (q) $ can then be seen to be $ 2 $-transitive.
\end{proof}

\newpage
\begin{corl}\label{prmtvbrlll}
The actions of $GL_{2} (q) $, $PGL_{2} (q) $, $ SL_{2} (q) $ and $ PSL_{2} (q) $ are primitive.
\end{corl}
\begin{proof}
Combine Lemma \ref{psl2tr}, Corollary \ref{cor2t} and Lemma \ref{ktrnstiveprmt}.
\end{proof}
In $ GL_{2} (q) $ the stabilizer, of $ \langle (0,1) \rangle $ is
\[
\Bigg\{ 
\begin{pmatrix}
a & b \\
0 & c
\end{pmatrix}
:
a,c \in \mathbb{F} _{q} ^{*} , b \in \mathbb{F} _{q}
\Bigg\}
,
\]
which has order $ q (q-1)^{2} $. Since the action on $ \Omega $ is primitive, this stabilizer is maximal by Theorem \ref{maxg}. Taking the quotient of this subgroup by the kernel of the action gives a maximal subgroup of order $ q(q-1) $ in $ PGL_{2} (q) $. Looking at Tables \ref{tableone} and \ref{tabletwo} and comparing the orders of maximal subgroups, we see only the subgroups of the form $ E_{q} \rtimes C_{ q-1 } $ have order $ q(q-1) $. Hence the above stabilizer corresponds to a Borel subgroup of $ PGL_{2} (q) $.

\medskip
When $ q $ is odd, the stabilizer of the same point in $ SL_{2} (q) $ is
\[
\Bigg\{ 
\begin{pmatrix*}[l]
a & b \\
0 & a^{-1}
\end{pmatrix*}
:
a \in \mathbb{F} _{q} ^{*} , b \in \mathbb{F} _{q}
\Bigg\}
.
\]
Similar to above, this is maximal and has order $ q (q-1) $. The quotient of this by the kernel of the action of $ SL_{2} (q) $ gives a maximal subgroup of order $ \tfrac{1}{2} q (q-1) $ in $ PSL_{2} (q) $. For $ q \geq 4$, on comparing the orders of subgroups in Table \ref{tablethree} it must be that the subgroup has structure $ E_{q} \rtimes C_{ (q-1) /2 } $, a Borel subgroup. For $ q = 3 $ the only subgroups of the correct order in $ PGL_{2} (3) $ and $ PSL_{2} (3) $ are also the Borel subgroups.

\medskip
By Lemmas \ref{equct} and \ref{fthflhght}, the relational complexity and height of the actions of $ PSL_{2} (q) $ and $ PGL_{2} (q) $ are equal to the relational complexity and height of the respective actions of $ SL_{2} (q) $ and $ GL_{2} (q) $. So for the rest of the chapter we will work in $ SL_{2} (q) $ and $ GL_{2} (q) $ as it is more convenient.

\medskip
This section is finished with a lemma that seems appropriate to place here, although will not be used until later chapters.
\begin{lema}\label{brlklnfrsbgrp}
If $ q $ is odd, the Borel subgroups of $ PSL_{2} (q) $ and $ PGL_{2} (q) $ do not contain Klein four-subgroups.
\end{lema}
\begin{proof}
Since the Borel subgroups are conjugate to each other, we only need to check one. Let $ G $ be $ GL_{2} (q) $ or $ SL_{2} (q) $ and let $ G ^{ \prime}  $ be $ PGL_{2} (q) $ or $ PSL_{2} (q) $ respectively.

\medskip
In $ G $, the stabilizer, $ H $, of $ \langle (0,1) \rangle $ is
\[
\Bigg\{ 
\begin{pmatrix}
a & b \\
0 & c
\end{pmatrix}
:
a,c \in \mathbb{F} _{q} ^{*} , b \in \mathbb{F} _{q}
\Bigg\}
,
\ \ \ \ \ \text{if $ G = GL_{2} (q) $}
\]
or
\[
\Bigg\{ 
\begin{pmatrix*}[l]
a & b \\
0 & a^{-1}
\end{pmatrix*}
:
a \in \mathbb{F} _{q} ^{*} , b \in \mathbb{F} _{q}
\Bigg\}
,
\ \ \ \ \ \text{if $ G = SL_{2} (q) $.}
\]
Let $ H^{ \prime} $ be the corresponding Borel subgroup of $ G ^{ \prime} $. Let $ I \in GL_{2} (q) $ be the identity matrix and $ N $ be $ \{ \lambda I : \lambda \in \mathbb{F} _{q} ^{*} \} $ if $ G = GL_{2} (q) $ or $ \{ \pm I \} $ if $ G = SL_{2} (q) $.

\medskip
Suppose there exists a Klein four-group $ K \leq H^{ \prime } $. Let $ g_{1} ^{ \prime} , g_{2} ^{ \prime} \in K $ be involutions with $ g_{1} ^{ \prime} \neq g_{2} ^{ \prime} $. These correspond to elements $ g_{1} , g_{2} \in H \setminus N $ such that $ g_{1} ^{2} , g_{2} ^{2} \in N $. By Lemma \ref{sqrsltwq} these have the form
\begin{align*}
g_{1} =
\begin{pmatrix*}[r]
a_{1} & b_{1} \\
0 & -a_{1}
\end{pmatrix*}
\ \ \ \ \
\text{and}
\ \ \ \ \ 
g_{2} =
\begin{pmatrix*}[r]
a_{2} & b_{2} \\
0 & -a_{2}
\end{pmatrix*}
, \ \ \ \ \ a_{1} , a_{2} , b_{1}, b_{2} \in \mathbb{F} _{q} , \ a_{1} \neq 0 \neq a_{2} .
\end{align*}
Also $  g_{1} ^{ \prime} g_{2} ^{ \prime} $ is an involution in $ K $ and the corresponding element in $ H $ is
\begin{align*}
g_{1} g_{2} =
\begin{pmatrix*}[c]
a_{1} a_{2} & a_{1} b_{2} - a_{2} b_{1} \\
0 & a_{1} a_{2}
\end{pmatrix*} .
\end{align*}
As before, $ g_{1} g_{2} \in H \setminus N $ and $ (g_{1} g_{2} )^{2} \in N $. So from Lemma \ref{sqrsltwq} we have $ a_{1} a_{2} = - a_{1} a_{2} $, which implies $ 1 = - 1 $. This is a contradiction because $ q $ is odd. Thus the assumption that there exists a Klein four-subgroup of $ H ^{ \prime } $ is wrong.
\end{proof}
\section{Height of Borel Action}
The height of the Borel actions for $ GL_{2} (q) $ and $ SL _{2} (q) $ are now calculated. Occasionally, related results will be proved along the way for use in later chapters. For the remainder of the chapter, let $ \Omega $ be the set of $1$-dimensional subspaces of $ \mathbb{F}_{q} ^{2} $. We begin with an important fact on the action of $ GL_{2} (q) $.
\begin{lema}\label{pgthtrns}
The Borel action of $GL_{2} (q) $ is $3$-transitive.
\end{lema}
\begin{proof}
See \cite{ROTMAN1}, Theorem 9.48, Chapter 9, page 283. 
\end{proof}
To find the height, we will take a set of points and keep adding to it until the point stabilizer is eventually the kernel of the action. Note that, except for $ \langle ( 0 , 1 ) \rangle $, every $1$-dimensional subspace of $ \mathbb{F} _{q} ^{2} $ contains a vector $ (1, x ) $, where $ x \in \mathbb{F} _{q}  $. Furthermore each of these subspaces contains only one vector of the form $ (1, x) $, so they can be written uniquely as $ \langle ( 1, x) \rangle $.

\medskip
Writing the elements of $ \Omega $ in the form $  \langle ( a, b )  \rangle $ can be messy, so for the rest of the chapter they will be denoted as $ \omega _{ (0,1) } := \langle ( 0, 1) \rangle $ and $ \omega _{ (1,x) } := \langle ( 1, x ) \rangle $ for each $ x \in \mathbb{F} _{q} $.
\begin{lema}\label{stblzrtwpnts}
Consider the Borel action of  $G := GL_{2} (q) $. Let $ \Delta  = \{  \omega _{ (1,0) } ,  \omega _{ (0,1) } \} $. Then
\[
G _{ ( \Delta ) } 
=
\Bigg\{ 
\begin{pmatrix}
a & 0 \\
0 & b
\end{pmatrix}
:
a, b \in \mathbb{F} _{q}  ^{*}
\Bigg\} .
\]
\end{lema}
\begin{proof}
The stabilizer $ G _{ ( \Delta ) } $ is the intersection of $ G_{\omega _{ (1,0) } } $ and $ G_{ \omega _{ (0,1) } } $, each being easy to calculate.
\end{proof}
\begin{corl}\label{stbtwpntssl}
In $H:= SL _{2} (q) $, the stabilizer of $ \Delta  = \{  \omega _{ (1,0) } ,  \omega _{ (0,1) } \} $ is
\[
H_{ ( \Delta ) } : =
\Bigg\{ 
\begin{pmatrix*}[l]
a & 0 \\
0 & a^{-1}
\end{pmatrix*}
:
a \in \mathbb{F} _{q}  ^{*}
\Bigg\} .
\]
\end{corl}
\begin{proof}
Put $ G := GL_{2} (q) $. Now $ H_{ ( \Delta ) } = H \cap G_{ ( \Delta ) } $ because $ H \leq G$ and by Lemma \ref{stblzrtwpnts},
\[
H \cap G_{ ( \Delta ) } = 
\Bigg\{ 
\begin{pmatrix*}[l]
a & 0 \\
0 & a^{-1}
\end{pmatrix*}
:
a \in \mathbb{F} _{q}  ^{*}
\Bigg\}
.
\]
\end{proof}
The next lemma will not be used until the final chapter of this paper, but seems more appropriate to prove here whilst we are considering how Borel subgroups intersect.
\begin{lema}\label{brlntrsccyclc}
Let $ G $ be $ PSL_{2} (q) $ or $ PGL_{2} (q) $, where $ q \geq 3 $. Let $ B_{1} , B_{2} \leq G $ be Borel subgroups. Either $ B_{1} = B_{2} $ or $ B_{1} \cap B_{2} \cong C_{ (q-1) / \delta } $, where $ \delta = 2 $ if $ G = PSL_{2} (q) $ and $ \delta = 1 $ if $ G = PGL_{2} (q) $.
\end{lema}

\newpage
\begin{proof}
Let $ G^{*} $ be $ SL_{2} (q) $ or $ GL_{2} (q) $. Let $ B_{1} ^{*} $ and $ B_{2} ^{*} $ be Borel subgroups of $ G^{*} $ corresponding to $ B_{1} $ and $ B_{2} $ respectively.

\medskip
Suppose $ B_{1} \neq B_{2} $. Then $ B_{1} ^{*} \neq B_{2} ^{*} $. Let $ \lambda \in \mathbb{F} _{q} $ be a generator of the multiplicative group $ \mathbb{F} _{q} ^{*} $. If $ G^{*} = SL_{2} (q) $ then we see from Corollary \ref{stbtwpntssl} that $ G^{*}_{ \omega _{ (1,0) } ,  \omega _{ (0,1) } } $ is cyclic and generated by
\[
\begin{pmatrix*}[l]
\lambda & 0 \\
0 & \lambda^{-1}
\end{pmatrix*} .
\]
Hence $ |G^{*}_{ \omega _{ (1,0) } ,  \omega _{ (0,1) } } | = o( \lambda ) = q - 1 $ and the corresponding subgroup of $ G $ is cyclic of order $ (q-1) / | Z(G) | = (q-1) / 2 $. The subgroups $ B_{1} ^{*} $ and $ B_{2} ^{*} $ stabilize some points $ \gamma _{1} , \gamma _{2} \in  \Omega $. The action of $ G $ on $  \Omega $ is $ 2 $-transitive by Corollary \ref{cor2t}. So $ \gamma _{1} = \omega _{ (1,0) } ^{g} $ and $ \gamma _{2} = \omega _{ (0,1) } ^{g} $ for some $ g \in G^{*} $. Hence
\[
B_{1} ^{*} \cap B_{2} ^{*} 
= (g^{-1} G^{*}_{ \omega _{ (1,0) } } g ) \cap (g^{-1} G^{*}_{ \omega _{ (0,1) } } g) 
= g^{-1} G^{*}_{ \omega _{ (1,0) } ,  \omega _{ (0,1) } } g .
\]
Thus $ B_{1} \cap B_{2} \cong C_{ (q-1) / 2 } $.

\medskip
If $ G^{*} = GL_{2} (q) $ then Lemma \ref{stblzrtwpnts} shows the an element in $ G^{*}_{  \omega _{ (1,0) } ,  \omega _{ (0,1) } } $ has the form is
\[ 
\begin{pmatrix}
a & 0 \\
0 & b
\end{pmatrix}
,
\ \ \ \ \
\text{where $ a, b \in \mathbb{F} _{q}  ^{*} $.}
\]
Such a matrix corresponds to the following element of $ G $;
\[
Z 
\begin{pmatrix}
1 & 0 \\
0 & a^{-1} b
\end{pmatrix}
,
\]
where $ Z = Z(G) $. So the subgroup $ G_{  \omega _{ (1,0) } ,  \omega _{ (0,1) } } \leq G $ corresponding to $ G^{*}_{  \omega _{ (1,0) } ,  \omega _{ (0,1) } } $ is generated by
\[
Z 
\begin{pmatrix}
1 & 0 \\
0 & \lambda
\end{pmatrix}
.
\] 
Hence $ G_{  \omega _{ (1,0) } ,  \omega _{ (0,1) } } $ is cyclic of order $ o( \lambda )  = q-1 $. As before, using the fact the action of $ G^{*} $ on $ \Omega $ is $ 2 $-transitive we have $ B_{1} \cap B_{2} \cong C_{ q-1 } $.
\end{proof}
Continuing with the task of calculating height, we look at the point stabilizer of sets of size $ 3 $.
\begin{lema}\label{thrkrnlind}
Let $G = GL_{2} (q) $. Let  $ \omega _{ (1,0) } $,  $ \omega _{ (0,1) }  $ and $ \delta $ be three distinct elements of $  \Omega  $ and let $ \Delta  = \{  \omega _{ (1,0) } ,  \omega _{ (0,1) } , \delta \} $. Then $ G _{ ( \Delta ) } = Z(G ) $, the kernel of the action.
\end{lema}
\begin{proof}
Let $ \Delta ^{ \prime } =  \{  \omega _{ (1,0) } ,  \omega _{ (0,1) }  \} $. The element $ \delta $ can be written as $ \delta = \langle ( d_{1} , d_{2} ) \rangle $ for some non-zero $ ( d_{1} , d_{2} ) \in \mathbb{F} _{q} ^{2} $. Now $G _{  \delta  } $ consists of all matrices of $ GL_{2} (q) $ that scale $  ( d_{1} , d_{2} ) $. So if $M \in G_{ ( \Delta ) } = G_{ \delta } \cap G _{ ( \Delta ^{ \prime } ) } $, then $ (d_{1} , d_{2} ) M = \lambda (d_{1} , d_{2} ) =  ( \lambda d_{1} , \lambda d_{2} ) $ for some non-zero $ \lambda \in \mathbb{F} _{q} $. Since $M \in G _{ ( \Delta ^{ \prime } ) }  $, Lemma \ref{stblzrtwpnts} shows that it must be of the form
\[
M = 
\begin{pmatrix}
a & 0 \\
0 & b
\end{pmatrix}
,
\ \ \ \ \ \ \ 
a , b \in \mathbb{F} _{q}  ^{*}
.
\]
So $ (d_{1} , d_{2} ) M = (a d_{1} ,   b d_{2} ) $. It follows that $a d_{1} = \lambda d_{1} $ and $b d_{2} = \lambda d_{2} $. Since $ \omega _{ (1,0 ) } \neq \delta \neq  \omega _{ (0,1 ) }  $, it must be that $  ( d_{1} , d_{2} ) $ is not a scalar multiple of $ (1,0) $ or $ (0,1 ) $. Therefore $d_{1} \neq 0 \neq d_{2} $. Hence $ a = \lambda = b $. Thus
\[
M = 
\begin{pmatrix}
a & 0 \\
0 & a
\end{pmatrix}
\in 
Z(G) 
\]
and $ G_{ ( \Delta ) } \subseteq Z(G) $. By Lemma \ref{krnlslgl}, $ Z(G ) $ is the kernel of the action, so $ Z(G) \subseteq G_{ ( \Delta ) } $. Hence $ G_{ ( \Delta ) }  =  Z(G) $.
\end{proof}
\begin{corl}\label{dstnctslhght}
Consider the Borel action of  $H := SL_{2} (q) $. Let  $ \omega _{ (1,0) } $,  $ \omega _{ (0,1) }  $ and $ \delta $ be three distinct elements of $  \Omega  $ and let $ \Delta  = \{  \omega _{ (1,0) } ,  \omega _{ (0,1) } , \delta \} $. Then $ H _{ ( \Delta ) } = Z(H ) $, the kernel of the action.
\end{corl}
\begin{proof}
Put $ G: = GL_{2} (q) $. Since $ H \leq G$, we have $ H_{ ( \Delta ) } = H \cap G_{ ( \Delta ) } = H \cap Z(G )  =  Z(H ) $ by Lemmas \ref{krnlslgl} and \ref{thrkrnlind}.
\end{proof}
This leads us to computing the height of these actions.
\begin{thrm}\label{glbrlhght}
For $q \geq 3$, the height of the Borel action of $ GL_{2} (q) $ is $ 3  $.
\end{thrm}
\begin{proof}
Lemma \ref{pgthtrns} shows that the action is $3$-transitive and Lemma \ref{thrkrnlind} shows that the stabilizer of 
\newline $ \{  \omega _{ (1,0) } ,  \omega _{ (0,1) } , \omega _{(1,1) } \} $ is the kernel of the action. Also there exists $x \in \mathbb{F} _{q} $ where $ 0 \neq x \neq 1 $ because $ |  \mathbb{F} _{q} | \geq 3 $. So $ \langle ( 1, x ) \rangle \in  \Omega $ and $ \langle ( 1, x ) \rangle \notin  \{  \omega _{ (1,0) } ,  \omega _{ (0,1) } , \omega _{(1,1) } \} $. Hence $ |  \Omega | \geq 4 > 3 $. Thus $ Ht (G, \Omega ) = 3  $ by Lemma \ref{trnsqvhght}.
\end{proof}
\begin{thrm}\label{slbrlhght}
The height of the Borel action of $ SL_{2} (q) $ is
\[
 Ht (SL_{2} (q), \Omega ) = 
\begin{cases}
2
\ \ \ \ \ \text{if $ q = 3 $,}
\\
3
\ \ \ \ \ \text{if $ q \geq 4 $.}
\end{cases}
\]
\end{thrm}
\begin{proof}
Put $ G:= SL_{2} (q) $. The case $ q = 3 $ is dealt with in the tables in \cite{WISCONS1}. So suppose $ q \geq 4 $. Then $ | \mathbb{F} _{q}  | \geq 4 $. Therefore there exists $ \lambda \in \mathbb{F} _{q}  ^{*} \setminus \{ 1 , -1 \} $. Let $ \Delta  = \{  \omega _{ (1,0) } ,  \omega _{ (0,1) } \} $. Corollary \ref{stbtwpntssl} shows that
\[
\begin{pmatrix*}[l]
\lambda & 0 \\
0 & \lambda ^{-1}
\end{pmatrix*}
\in 
G_{ ( \Delta ) } .
\]
By Lemma \ref{krnlslgl}, $ G_{ ( \Delta ) } $ is not equal to the kernel of the action, $ Z ( SL _{2} (q)) $. Corollary \ref{cor2t} shows the action is $2$-transitive. Also $ | \Omega | = q + 1 > 2 $ by Lemma \ref{szbrlst}. So it follows from Lemma \ref{trnsqvhght} that $  Ht (G, \Omega ) > 2 $.

\medskip
Let $ \Gamma = \{ \gamma _{1} , \gamma _{2} , \gamma _{3} \} \subseteq  \Omega $. Since the action is $2$-transitive,
\[
\Gamma ^{g} = \{  \omega _{ (1,0) } ,  \omega _{ (0,1) } , \gamma _{3} ^{g} \} 
\]
for some $ g \in G $. Now $ G _{ ( \Gamma ^{g} ) } =  Z ( SL _{2} (q)) $ by Corollary \ref{dstnctslhght}. Since $ Z ( SL _{2} (q)) \trianglelefteq G $, it follows from Corollary \ref{crlcnj} that $ G _{ ( \Gamma ) }  = g G _{ ( \Gamma ^{g} ) } g ^{-1} = Z ( SL _{2} (q)) $. So $ Ht (G, \Omega ) \leq 3 $ by Lemma \ref{hghtkrnlllsts}. Thus $ Ht (G, \Omega )  = 3 $.
\end{proof}
\section{Relational Complexity of the Borel Action}
From Lemma \ref{pslsnpgl} we see $ PGL_{2} (q) $ contains a subgroup isomorphic to $ PSL_{2} (q) $. When $q $ is even we have $ |  PGL_{2} (q) | = | PSL_{2} (q) | $ and therefore $ PGL_{2} (q) \cong PSL_{2} (q) $. Since Lemma \ref{equct} tells us the relational complexity of the action of $ GL_{2} (q) $ is equal to that of $ PGL_{2} (q) $, calculating relational complexity for each Borel action can be split into two cases; $ SL_{2} (q) $ when $ q $ is odd and $ GL_{2} (q) $ for any $ q $. First we deal with $ GL_{2} (q) $.

\newpage
\begin{thrm}\label{brlglwhtever}
The relational complexity of the Borel action of $ GL_{2} (q) $ is
\[
 RC (GL_{2} (q), \Omega ) = 
\begin{cases}
2
\ \ \ \ \ \text{if $ q = 3 $,}
\\
4
\ \ \ \ \ \text{if $ q \geq 4 $.}
\end{cases}
\]
\end{thrm}
\begin{proof}
For $ q = 3 $, use the fact $ PGL_{2} (3) \cong S_{4} $ and $ S_{4} $ has one transitive action of degree $ 4 $, the natural action. This action has relational complexity $ 2 $ by Example \ref{symrc}. Therefore $ RC( GL_{2} (3) , \Omega ) = 2 $.

\medskip
Suppose $ q \geq 4 $. Lemma \ref{pgthtrns} shows that the action is $3$-transitive. Also we see $   PGL _{2} (q) \not \cong  Sym (  \Omega  )  $ as permutation groups because the groups have different orders. So $ RC (G, \Omega ) > 3  $ by Lemma \ref{trnprm}. Theorem \ref{glbrlhght} and Theorem \ref{rchgt} together show that $ RC (G, \Omega ) \leq 4 $. Thus $ RC (G, \Omega ) = 4 $.
\end{proof}
When computing the relational complexity of the $SL_{2} (q) $ Borel Action when $q$ is odd, $ q = 3 $ and $ q = 5 $ will be dealt with as special cases at the end of the section. So until then, we will mostly deal with $ q $ odd and $ q \geq 7 $.

\medskip
To prepare for the general case, the point stabilizers of some sets of points will be determined, first for the action Borel action of $GL _{2} (q) $ and then for $SL _{2} (q) $.
\begin{lema}\label{stbglnn}
For any $q$ the point stabilizer of $  \{  \omega _{ (1,0) } ,  \omega _{ (1,1) } \} $ under the Borel action of  $G := GL_{2} (q) $ is
\[
G _{ \omega _{ (1,0) } ,  \omega _{ (1,1) } } 
=
\Bigg\{ 
\begin{pmatrix}
a & 0 \\
b-a & b
\end{pmatrix}
:
a, b \in \mathbb{F} _{q}  ^{*}
\Bigg\} 
\]
and the point stabilizer of $ \{  \omega _{ (0,1) } ,  \omega _{ (1,1) } \} $ is
\[
G _{  \omega _{ (0,1) } ,  \omega _{ (1,1) } } 
=
\Bigg\{ 
\begin{pmatrix}
a & a-b \\
0 & b
\end{pmatrix}
:
a, b \in \mathbb{F} _{q}  ^{*}
\Bigg\} 
.
\]
\end{lema}
\begin{proof}
The stabilizer of $ \omega _{ (1,0) } $ is the set of invertible matrices that scale $ (1,0 ) $ by some non-zero scalar, which are
\[
G_{ \omega _{ (1,0) }  } = 
\Bigg\{ 
\begin{pmatrix}
a & 0 \\
c & b
\end{pmatrix}
:
a, b \in \mathbb{F} _{q} ^{*} , c \in \mathbb{F} _{q}
\Bigg\}
.
\]
Now $ G _{ \omega _{ (1,0) } ,  \omega _{ (1,1) } }  $ is the set of elements of $ G_{ \omega _{ (1,0) }  } $ that stabilize $ \langle (1,1) \rangle $. If
\[
\begin{pmatrix}
a & 0 \\
c & b
\end{pmatrix}
\in
G_{ \omega _{ (1,0) }  }
\]
stabilizes $ \langle (1,1) \rangle $ then
\[
(1, 1)
\begin{pmatrix}
a & 0 \\
c & b
\end{pmatrix}
= 
(a + c , b )
\in 
\langle (1,1) \rangle 
.
\]
Therefore $ (a + c , b ) = ( \lambda , \lambda ) $ for some $ \lambda \in \mathbb{F} ^{*} $. Hence $ a + c = \lambda = b $ and so $ c = b -a $. Thus the matrix has the form
\[
\begin{pmatrix}
a & 0 \\
b - a & b
\end{pmatrix}
.
\]
Also every matrix of this form stabilizes $ \langle (1,1) \rangle  $. Hence
\[
G _{ \omega _{ (1,0) } ,  \omega _{ (1,1) } } 
=
\Bigg\{ 
\begin{pmatrix}
a & 0 \\
b-a & b
\end{pmatrix}
:
a, b \in \mathbb{F} _{q}  ^{*} 
\Bigg\} 
.
\]
Similar reasoning shows that
\[
G _{ \omega _{ (0,1) } ,  \omega _{ (1,1) } } 
=
\Bigg\{ 
\begin{pmatrix}
a & a-b \\
0 & b
\end{pmatrix}
:
a, b \in \mathbb{F} _{q}  ^{*}
\Bigg\}
. 
\]
\end{proof}
\begin{corl}\label{stbslnn}
For any $q$ the point stabilizer of $  \{  \omega _{ (1,0) } ,  \omega _{ (1,1) } \} $ under the Borel action of  $H := SL_{2} (q) $ is
\[
H _{ \omega _{ (1,0) } ,  \omega _{ (1,1) } } 
=
\Bigg\{ 
\begin{pmatrix}
a  & 0 \\
a ^{-1} -a & a ^{-1}
\end{pmatrix}
:
a \in \mathbb{F} _{q}  ^{*}
\Bigg\} 
\]
and the point stabilizer of $ \{  \omega _{ (0,1) } ,  \omega _{ (1,1) } \} $ is
\[
H _{  \omega _{ (0,1) } ,  \omega _{ (1,1) } } 
=
\Bigg\{ 
\begin{pmatrix}
a  & a  -a ^{-1} \\
0 & a ^{-1}
\end{pmatrix}
:
a \in \mathbb{F} _{q}  ^{*}
\Bigg\} 
.
\]
\end{corl}
\begin{proof}
Put $ G := GL_{2} (q) $. Now $ H_{ \omega _{ (1,0) } ,  \omega _{ (1,1) } }  = H \cap G_{ \omega _{ (1,0) } ,  \omega _{ (1,1) } } $ because $ H \leq G$. The result then follows from Lemma \ref{stbglnn}.
\end{proof}
When calculating the relational complexity for the action of $ SL _{2} (q) $ for $q \geq 7 $, a pair of elements in $ \mathbb{F} _{q} ^{*} $ will need to be picked that satisfy a particular set of conditions. This final lemma lists these conditions and shows that such elements can be chosen.
\begin{lema}\label{xlmbdcndtns}
If $ q $ is odd and $q \geq 7 $ then there exists $ x , \lambda \in \mathbb{F} _{q} ^{*} $ such that
\begin{itemize}
\item $ \lambda ^{2} \neq 1 $,
\item $  \lambda ^{-2} \neq  x \neq 1 $ and
\item $  ( 1 - \lambda ^{2} x )  ( 1- x) ^{-1} $ is a square in $ \mathbb{F} _{q} ^{*} $.
\end{itemize}
\end{lema}
\begin{proof}
Since $ q \geq 7 $, the number of squares in $ \mathbb{F} _{q} ^{*} $ is $ \tfrac{1}{2} (q-1) \geq 3 $ by Lemma \ref{sqnmbrf}. So there exists a square in $  \mathbb{F} _{q} ^{*} \setminus \{ 1 \} $. Fix $ \lambda \in  \mathbb{F} _{q} ^{*} $ so that $ \lambda ^{2} \neq   1 $. Now define a map
\begin{align*}
\phi :  \mathbb{F} _{q} ^{*} \setminus \{ 1  \} &  \rightarrow  \mathbb{F} _{q}  \\
y & \mapsto ( 1 - \lambda ^{2} y )  ( 1- y) ^{-1} .
\end{align*}
Suppose $ y_{1} , y_{2} \in \mathbb{F} _{q} ^{*} \setminus \{ 1  \} $ and $ ( y_{1} ) \phi = ( y_{2} ) \phi $. Then
\begin{align*}
( y_{1} ) \phi 
= ( y_{2} ) \phi 
& \Longleftrightarrow  ( 1 - \lambda ^{2} y _{1} )  ( 1- y_{1} ) ^{-1}  
=  ( 1 - \lambda ^{2} y _{2}  ) ( 1- y _{2} )  ^{-1} \\
& \Longleftrightarrow  ( 1 - \lambda ^{2} y _{1}  ) ( 1- y _{2} )
 =  ( 1 - \lambda ^{2} y _{2} ) ( 1- y_{1} )   \\
& \Longleftrightarrow   1  - y _{2} - \lambda ^{2} y _{1} + \lambda ^{2} y _{1} y _{2} =  1   - y_{1} - \lambda ^{2} y _{2} + \lambda ^{2} y_{1} y _{2} \\
& \Longleftrightarrow   - y _{2}  - \lambda ^{2} y _{1}  =  - y_{1}   - \lambda ^{2} y _{2}  \\
& \Longleftrightarrow   y _{1}  ( 1 - \lambda ^{2} ) =  y _{2} ( 1 - \lambda ^{2} )  \\
& \Longleftrightarrow   y _{1}   =  y _{2} . \tag{since $ \lambda ^{2} \neq 1 $}
\end{align*}
Hence $ \phi $ is one-one. Therefore $ | \text{im} ( \phi ) | = |  \mathbb{F} _{q} ^{*} \setminus \{ 1  \}  | = q - 2 $. Notice that $ ( \lambda ^{ -2 } ) \phi = 0 $, so the image of the remaining elements lie in $ \mathbb{F} _{q} ^{*} $. Therefore
\[
 | \text{im} ( \phi ) \cap \mathbb{F} _{q} ^{*}  | = | \text{im} ( \phi ) | - 1 = q - 3 = | \mathbb{F} _{q} ^{*}  | - 2 .
\]
As there are at least three squares in in $ \mathbb{F} _{q} ^{*}  $, at least one of these squares lies in $ \text{im} ( \phi ) \cap \mathbb{F} _{q} ^{*} $. Choose $ x \in \mathbb{F} _{q} ^{*} \setminus \{ 1  \} $ so that $ ( x) \phi =  ( 1 - \lambda ^{2} x ) ( 1- x) ^{-1} $ is a square in $ \mathbb{F} _{q} ^{*}  $. It must be that $ x \neq \lambda ^{ -2} $ because $ ( x) \phi \neq 0 $.
\end{proof}
And now we are in a position to finish finding the relational complexity for the Borel action of $ SL _{2} (q) $.
\begin{thrm}\label{xckjdsqwmmmm}
Suppose $q$ is odd. If $ q = 3 $ or $ q = 5 $ then $ RC( SL _{2} (q) ,  \Omega )  = 3 $. If $ q \geq 7 $ then $ RC( SL _{2} (q) , \Omega  )  = 4 $.
\end{thrm}
\begin{proof}
The cases $ q = 3 $ and $ q = 5 $ are dealt with in the tables in \cite{WISCONS1}. So suppose $ q \geq 7 $.

\medskip
Put $ G : = SL _{2} (q) $. Corollary \ref{cor2t} shows that the action is $2$-transitive. Also we see $   PSL _{2} (q) \not \cong  Sym (  \Omega  )  $ as permutation groups because the groups have different orders. So using Lemma \ref{trnprm}, we have $  RC(G,  \Omega ) \geq 3 $. Theorem \ref{rchgt} and Lemma \ref{slbrlhght} together show that $  RC(G,  \Omega  ) \leq 4 $. So either $ RC(G,  \Omega  ) = 3 $ or $ RC(G,  \Omega  ) = 4 $.

\medskip
It will be shown that $ RC(G,  \Omega  ) \neq 3 $ by constructing $I , J \in \Omega ^{4} $ so that $I \sim _{3} J $ but $ I \not \sim _{4} J $. Using Lemma \ref{xlmbdcndtns} pick $ x , \lambda \in \mathbb{F} _{q} ^{*} $ such that
\begin{itemize}
\item $ \lambda ^{2} \neq 1 $,
\item $  \lambda ^{-2}  \neq  x \neq 1 $ and
\item $  ( 1 - \lambda ^{2} x )  ( 1- x) ^{-1} $ is a square in $ \mathbb{F} _{q} ^{*} $.
\end{itemize}
Put $ y = \lambda ^{2} x $ and observe that $ x \neq y $ because $ \lambda ^{2} \neq 1 $. Let
\begin{align*}
& I = (  \omega _{ (1,0) } \ , \ \omega _{ (0,1) } \ , \  \omega _{ (1,1) } \ , \ \omega _{ (1,x) } ) \ \ \text{and} \\
& J = (  \omega _{ (1,0) } \ , \ \omega _{ (0,1) } \ , \  \omega _{ (1,1) } \ , \ \omega _{ (1,y) } ) .
\end{align*}
If $ I \sim _{4} J $ then $ I ^{g} = J $ for some $ g \in SL _{2} (q) $ and so $ g $ stabilizes the points $  \omega _{ (1,0) } $, $  \omega _{ (0,1) } $ and $  \omega _{ (1,1) } $. By Corollary \ref{dstnctslhght}, the stabilizer of these points is $ Z(SL_{2} (q) ) $, the kernel of the action. Therefore $g \in Z(SL_{2} (q) ) $ and $ g $ stabilizes $ \omega _{ (1,x  ) } $. Hence $  \omega _{ (1,x  ) } = \omega _{ (1,x  ) } ^{g} = \omega _{ (1,y  ) } $. But $ (1, y) \notin \omega _{ (1,x  ) } $ because $ x \neq y $, which means $ \omega _{ (1,x  ) } \neq \omega _{ (1,y  ) } $. This is a contradiction, so it must be that $ I \not\sim _{4} J $.

\medskip
Next it needs to be shown that $ I \sim _{3} J $, which requires checking each $3$-subtuple of $I$ gets sent to the corresponding $3$-subtuple of $J$. There are four such pairs of $3$-subtuples of $I$ and $J$ and each case will be looked at in turn. To make it easier to tell the difference between the $3$-subtuples of $I$ and $J$, write the entries as $ I = ( I_{1} , I_{2} , I _{3} , I _{4} ) $ and $ J = ( J_{1} , J_{2} , J _{3} , J _{4} ) $.

\medskip
\textbf{Case 1 : $ ( I_{1} , I_{2} , I _{3} ) $ and $ ( J_{1} , J_{2} , J _{3} ) $.} Since $ I _{i} = J _{i} $ for each $ i \in \{ 1, 2 , 3 \} $ the identity element sends the first $ 3 $-tuple to the other.

\bigskip
\textbf{Case 2 : $ ( I_{1} , I_{2} , I _{4} ) $ and $ ( J_{1} , J_{2} , J _{4} ) $.} Let
\[
g_{1} := 
\begin{pmatrix*}[l]
\lambda ^{-1} & 0 \\
0 & \lambda 
\end{pmatrix*}
.
\]
Then
\[
(1,x) ^{g _{1} } = 
(1,x)
\begin{pmatrix*}[l]
\lambda ^{-1} & 0 \\
0 & \lambda 
\end{pmatrix*}
=
( \lambda ^{-1} , \lambda x)
=
\lambda ^{-1} ( 1 , \lambda ^{2} x)
=
\lambda ^{-1} (1, y) .
\]
Hence
\[
 I _{4} ^{ g _{1} }
 = \omega _{ (1,x) } ^{ g _{1} } 
 = \langle (1 , x ) \rangle ^{ g _{1} } 
= \langle (1 , x ) ^{ g _{1} }  \rangle 
 = \langle \lambda ^{ -1 } (1 , y )   \rangle 
  = \langle  (1 , y )   \rangle 
 = \omega _{ (1,y) } 
 =  J _{4} .
\]
Also Corollary \ref{stbtwpntssl} shows that $ g _{1} \in G _{  \omega _{ (1,0) } ,  \omega _{ (0,1) } } $. Therefore $ ( I_{1} , I_{2} , I _{4} )  ^{g _{1} }  =  ( J_{1} , J_{2} , J _{4} ) $.

\bigskip
\textbf{Case 3 : $ ( I_{1} , I_{3} , I _{4} ) $ and $ ( J_{1} , J_{3} , J _{4} ) $.} Since $  ( 1 - \lambda ^{2} x ) ( 1- x)  ^{-1} $ is a square in $ \mathbb{F} _{q} ^{*} $, there exists $ \mu \in \mathbb{F} ^{*} $ such that $ \mu ^{2} = \lambda ^{-2} ( 1 - \lambda ^{2} x )  ( 1- x) ^{-1} $. Let 
\[
g_{2} := 
\begin{pmatrix*}[c]
\mu  & 0 \\
\mu ^{-1} - \mu  & \mu  ^{ -1}
\end{pmatrix*}
.
\]
Then
\begin{align*}
(1,x) ^{g _{2} } & = 
(1,x)
\begin{pmatrix*}[c]
\mu  & 0 \\
\mu ^{-1} - \mu  & \mu  ^{ -1}
\end{pmatrix*}
\\
& =
( \mu  + \mu ^{-1} x - \mu x , \mu ^{ -1}  x ) \\
& =
(  \mu ^{ -1}  x + \mu (1- x) , \mu ^{ -1}  x ) \\
& =
\mu ^{ -1} (   x + \mu ^{ 2} (1- x) ,  x ) \\
& =
\mu ^{ -1} (   x + \lambda ^{-2} ( 1 - \lambda ^{2} x ) ( 1- x) ^{-1}  (1- x) ,  x ) \\
& =
\mu  ^{ -1}(   x + \lambda ^{-2}  ( 1 - \lambda ^{2} x )   ,  x ) \\
& =
\mu ^{ -1} (   x + \lambda ^{-2} -  x ,  x ) \\
& =
\mu ^{ -1} (   \lambda ^{-2}  ,  x ) \\
& =
\mu ^{ -1} \lambda ^{-2} (   1  , \lambda ^{2} x ) \\
& =
\mu ^{ -1} \lambda ^{-2} (   1  , y ) .
\end{align*}
Hence
\[
 I _{4} ^{ g _{2} }
 = \omega _{ (1,x) } ^{ g _{2} } 
 = \langle (1 , x ) \rangle ^{ g _{2} } 
= \langle (1 , x ) ^{ g _{2} }  \rangle 
 = \langle \mu ^{ -1} \lambda ^{-2} (   1  , y )  \rangle 
  = \langle  (1 , y )   \rangle 
 = \omega _{ (1,y) } 
 =  J _{4} .
\]
Also Corollary \ref{stbslnn} shows that $ g _{2} \in G _{  \omega _{ (1,0) } ,  \omega _{ (1,1) } } $. Therefore $ I _{1} ^{g _{2} } = I _{1} = J _{1} $ and $ I _{3} ^{g _{2} } = I _{3} = J _{3} $. It follows that $ ( I_{1} , I_{3} , I _{4} )  ^{g _{2} } =  ( J_{1} , J_{3} , J _{4} ) $.

\bigskip
\textbf{Case 4 : $ ( I_{2} , I_{3} , I _{4} ) $ and $ ( J_{2} , J_{3} , J _{4} ) $.} Since $  ( 1 - \lambda ^{2} x )  ( 1- x) ^{-1} $ is a square in $ \mathbb{F} _{q} ^{*} $ it can be written as $ \gamma ^{2} =  ( 1 - \lambda ^{2} x )  ( 1- x) ^{-1} $ for some $ \gamma \in  \mathbb{F} _{q} ^{*} $. Let
\[
g _{3}
=
\begin{pmatrix}
\gamma  ^{-1}  & \gamma ^{-1}  - \gamma  \\
0 & \gamma
\end{pmatrix}
.
\]
Then

\newpage
\begin{align*}
(1,x) ^{g _{3} } & = 
(1,x)
\begin{pmatrix}
\gamma  ^{-1}  & \gamma ^{-1}  - \gamma  \\
0 & \gamma
\end{pmatrix}
\\
& = ( \gamma ^{-1} , \gamma ^{-1}  - \gamma  + \gamma  x ) 
\\
& = \gamma ^{-1} ( 1 , 1   - \gamma ^{2} + \gamma ^{2} x )
\\
& = \gamma ^{-1} ( 1 , 1   - \gamma ^{2} (1 - x ) )
\\
& = \gamma ^{-1} ( 1 , 1   -  ( 1 - \lambda ^{2} x ) ( 1- x) ^ {-1} (1 - x ) ) 
\\
& = \gamma ^{-1} ( 1 , 1   -  ( 1 - \lambda ^{2} x ) )
\\
& = \gamma ^{-1} ( 1 ,  \lambda ^{2} x )
\\
& = \gamma ^{-1} ( 1 ,  y )
.
\end{align*}
Hence
\[
 I _{4} ^{ g _{3} }
 = \omega _{ (1,x) } ^{ g _{3} } 
 = \langle (1 , x ) \rangle ^{ g _{3} } 
= \langle (1 , x ) ^{ g _{3} }  \rangle 
 = \langle \gamma ^{-1} (   1  , y )  \rangle 
  = \langle  (1 , y )   \rangle 
 = \omega _{ (1,y) } 
 =  J _{4} .
\]
Also Corollary \ref{stbslnn} shows that $ g _{3} \in G _{  \omega _{ (0,1) } ,  \omega _{ (1,1) } } $. Therefore $ ( I_{2} , I_{3} , I _{4} )  ^{g _{3} }  =  ( J_{2} , J_{3} , J _{4} ) $.

\bigskip
These four cases together show that $ I \sim _{3} J $ and it has been shown earlier that $ I \not\sim_{4} J $. Therefore $ RC(G, \Omega ) \neq 3 $ by Definition \ref{first}. Thus $ RC(G, \Omega ) = 4 $.
\end{proof}

\newpage
\chapter{The Dihedral Actions}\label{chapsix}
Here the actions of $ PSL_{2} (q) $ and $ PGL_{2} (q) $ on maximal dihedral subgroups will be looked at, where $ q \geq 4 $. Suppose $ G $ is $ PGL_{2} (q) $ or $ PSL_{2} (q) $ acting on $ \Omega $, the right cosets of a maximal $ D_{ 2(q \pm 1) / \delta } $, where $ \delta = 1 $ if $ G = PGL_{2} (q) $ and $ \delta = 2 $ if $ G \neq PGL_{2} (q) $.

\medskip
For height, the main result that will be proved is: If $ G $ is $ PSL_{2} (4) $ or $ PSL_{2} (5) $ acting on a maximal $ D_{6} $ then $ Ht (G , \Omega ) = 2 $. Otherwise $ Ht (G , \Omega ) = 3 $.

\medskip
The relational complexity of these actions will be shown to always be $ 3 $.

\medskip
A lot of results in this chapter are quite general and can be applied to groups other than $ PSL_{2} (q) $ and $ PGL_{2} (q) $. At the end of the chapter it will also be proved that the action of any Suzuki group on a maximal dihedral also has relational complexity $ 3 $.
\section{Preliminary Results}
Each of the preliminary results here are simple to prove, so proofs are omitted. For a group $ G $ acting on a set $ \Omega $, recall the notation $ \text{Send}_{G} ( \omega_{1} , \omega _{2} ) = \{ g \in G : \omega _{1} ^{g} = \omega _{2} \} $. The next lemma is simple to prove.
\begin{lema}\label{sndcst}
Let $ G $ be a group acting on a set $ \Omega $. Let $ \omega_{1} , \omega _{2} \in \Omega $, with both elements lying in the same orbit. Suppose $ h \in \text{Send}_{G} ( \omega_{1} , \omega _{2} ) $. Then $ \text{Send}_{G} ( \omega _{1} , \omega _{2} ) = G _{\omega_{1} }  h $.
\end{lema}
\begin{proof}
Omitted.
\end{proof}
The notation $ D_{m} $ will be used for the dihedral group of order $ m $, sometimes writing $ D_{2n} $ to put emphasis on the rotational subgroup having order $n$. The next fact is well known and used frequently.
\begin{lema}\label{dhdrlrottnnrml}
Let $ n \geq 2 $. Suppose $ K $ is a subgroup of the rotational subgroup of $ D_{2n} $. Then $ K \trianglelefteq D_{2n} $.
\end{lema}
\begin{proof}
Omitted.
\end{proof}
\section{Action Description}
In Tables \ref{tableone} and \ref{tabletwo}, for $ q \geq 4 $ and $ q \neq 5 $ we see in $ PGL_{2} (q) $ there exists a conjugacy class of maximal subgroups isomorphic to $ D_{2(q-1) } $. Also there is a conjugacy class of $ D_{ 2(q+1) } $ in $ PGL_{2} (q) $ for all $ q \geq 4 $.

\medskip
In Table \ref{tablethree}, for $ q $ odd and $ q \geq 13 $ we see in $ PSL_{2} (q) $ there exists a conjugacy class of maximal subgroups isomorphic to $ D_{(q-1) } $. For $ q $ odd and $ q \geq 5 $ there is a conjugacy class of $ D_{ (q+1) } $ in $ PSL_{2} (q) $, except when $ q = 7 $ or $ q = 9 $.

\medskip
Let $ G $ be either $ PSL_{2} (q) $ or $ PGL _{2} (q) $ and suppose $ H \leq G $ is one of the above maximal dihedral subgroups. Put $ \Omega := \{ g^{-1} H g : g \in G \} $. It follows from Lemma \ref{cstquivtcnj} that rather than  describing the action in terms of right cosets of $ H $, we can look at the action by conjugation on $ \Omega $ since the two are equivalent.

\medskip
For the rest of this chapter let $ q $, $ G $, $ H $ and $ \Omega $ be as defined above.
\section{Height of the Dihedral Actions}
Throughout this section let $ G $ be a finite group that is either simple or whose proper normal subgroups are all maximal. Also suppose $ G $ contains a maximal subgroup $ H $ that is a dihedral group of order at least $ 6 $. Also suppose $ H \ntriangleleft G $.

\medskip
Let $ \Omega $ be the set of right cosets of $ H$. Assume $ | \Omega | \geq 4 $. After all, if $ | \Omega | \in \{ 2, 3 \} $ we know the action has height at most $ 2 $ by Lemma \ref{htbnd} and relational complexity $ 2 $ by Lemma \ref{transtwo} and Example \ref{setthre}. Since $ H \ntriangleleft G $, there must exist $ \omega _{1} , \omega _{2} \in \Omega $ with stabilizers $ G_{ \omega_{1} } $ and $ G_{ \omega_{2} } $ conjugate to $ H $, where $ G_{ \omega_{1} } \neq G_{ \omega_{2} } $.

\medskip
We will work toward showing that $ Ht(G, \Omega ) \leq 3 $. First an upper bound on the height will be found, starting by looking at intersections of pairs of stabilizers.
\begin{lema}\label{klnfrntrsctn}
Let $ \omega _{1} , \omega _{2} \in \Omega $ with $ G_{ \omega_{1} } \neq G_{ \omega_{2} } $. Then $ | G_{ \omega_{1} , \omega_{2} } | \in \{ 1, 2, 4 \} $. Furthermore if $ | G_{ \omega_{1} , \omega_{2} } | = 4 $ then $  G_{ \omega_{1} , \omega_{2} }  $ is a Klein four-group whose non-identity elements consist of two reflections and a rotation in $ G_{ \omega_{1} } $, as well as two reflections and a rotation in $ G_{ \omega_{2} } $. 
\end{lema}
\begin{proof}
Both $ G_{ \omega_{1} } $ and $ G_{ \omega_{2} } $ are dihedral groups because they are conjugate to $H$. Let $ h \in  G_{ \omega_{1} , \omega_{2} } = G_{ \omega_{1} } \cap G_{ \omega_{2} }$.

\medskip
If $ o(h) > 2 $ then $ h $ must be a rotation in both $ G_{ \omega_{1} } $ and $ G_{ \omega_{2} } $ because reflections have order $ 2 $. But then $ \langle h \rangle $ is a subgroup of the rotational subgroups of $ G_{ \omega_{1} } $ and $ G_{ \omega_{2} } $. Hence $ \langle h \rangle $ is normal in both $ G_{ \omega_{1} } $ and $ G_{ \omega_{2} } $ by Lemma \ref{dhdrlrottnnrml}. But Lemma \ref{ntnrmlmxml} shows this is not possible. Thus $ o(h) \leq 2 $.

\medskip
If $ o(h) = 2 $ and $ h $ is a rotation in $ G_{ \omega_{1} } $ then any other non-identity elements of $ G_{ \omega_{1} , \omega_{2} } $ must be reflections in $ G_{ \omega_{1} } $ because there can only exist at most one rotation of order $ 2 $ in a dihedral group.

\medskip
Composing two reflections with each other gives a rotation and so there can only be at most two elements of $ G_{ \omega_{1} , \omega_{2} } $ that are reflections in $ G_{ \omega_{1} } $, otherwise there would be at least two non-identity rotations of $ G_{ \omega_{1} } $ in the intersection. Hence the intersection contains at most the identity, a rotation and two reflections in $G_{ \omega_{1} } $, implying $ |  G_{ \omega_{1} , \omega_{2} } | \leq 4 $.

\medskip
If $ |  G_{ \omega_{1}, \omega_{2} } | = 4 $ then $ G_{ \omega_{1} , \omega_{2} } $ must be a Klein four-group. Similar reasoning as above shows that $ G_{ \omega_{1} , \omega_{2} } $ must contain two reflections and a rotation in $ G_{ \omega_{2} } $, as with $ G_{ \omega_{1} } $.

\medskip
Finally $ | G_{ \omega_{1} , \omega_{2} } | \neq 3 $ because the $ G_{ \omega_{1} , \omega_{2} } $ contains no elements of order $ 3 $. Thus $ | G_{ \omega_{1} ,  \omega_{2} } | \in \{ 1, 2, 4 \} $.
\end{proof}
To aid finding the height of the action we can rule out sets being independent if too many pairs of stabilizers have intersection of order $ 4 $.
\begin{lema}\label{szfrntrsctn}
Let $ \Delta := \{ \delta _{1} , \delta _{2} , \delta _{3} \} \subseteq \Omega $. Suppose $ | G_{ \delta _{1} , \delta _{2} } | = | G_{ \delta _{1} , \delta _{3} } | = 4 $. Then $ \Delta $ is not an independent set.
\end{lema}
\begin{proof}
If $ G_{ \delta _{1} } = G_{ \delta _{i} } $ for some $ i \in \{  2 , 3 \}  $ then $ G_{ \delta _{1} , \delta _{i} } =  G_{ \delta _{1} } \cap G_{ \delta _{i} } = G_{ \delta _{1} } $ and $ \Delta $ is not independent by Lemma \ref{stach}. So suppose $ G_{ \delta _{1} } $, $ G_{ \delta _{2} } $ and $ G_{ \delta _{3} } $ are distinct subgroups of $ G $.

\medskip
It follows from Lemma \ref{klnfrntrsctn} that both $ G_{ \delta _{1} , \delta _{2} }  $ and $ G_{ \delta _{1} , \delta _{3} } $ contain the unique rotation $ r_{1} \in G_{ \delta _{1} } $ of order $ 2 $. Hence $ \langle r_{1} \rangle \leq  G_{ \delta _{2} , \delta _{3} }  $ and it follows that $ | G_{ \delta _{2} , \delta _{3} } | \geq 2 $.

\medskip
If $  | G_{ \delta _{2} , \delta _{3} } | = 2 $ then $ G_{ \delta _{2} , \delta _{3} } = \langle r_{1} \rangle $. Hence $ G_{ \delta _{1} , \delta _{2} , \delta _{3} } = \langle r_{1} \rangle = G_{ \delta _{2} , \delta _{3} } $, which shows that $ \Delta $ is not an independent set by Lemma \ref{indsbstsntqual}.

\medskip
If $  | G_{ \delta _{2} , \delta _{3} } | \neq 2 $ then $ | G_{ \delta _{2} , \delta _{3} } | = 4 $ by Lemma \ref{klnfrntrsctn}. Also by the same lemma, both $ G_{ \delta _{1} , \delta _{2} } $ and $ G_{ \delta _{2} , \delta _{3} } $ contain the unique rotation $ r_{2} \in G_{ \delta _{2} } $ of order $ 2 $. Similarly $ G_{ \delta _{1} , \delta _{3} } $ and $ G_{ \delta _{2} , \delta _{3} } $ contain the unique rotation $ r_{3} \in G_{ \delta _{3} } $ of order $ 2 $.

\medskip
It must be that $ r_{1} \neq r_{2} $ otherwise Lemma \ref{dhdrlrottnnrml} shows that $ \langle r_{1} \rangle $ is normal in $ G_{ \delta _{1} }  $ and $ G_{ \delta _{2} }  $, which contradicts Lemma \ref{ntnrmlmxml}. By the same reasoning $ r_{1} \neq r_{3} $ and $ r_{2} \neq r_{3} $. Lemma \ref{klnfrntrsctn} shows that $ G_{ \delta _{1} , \delta _{2} } $ and $ G_{ \delta _{2} , \delta _{3} } $ are Klein four-groups. So $ G_{ \delta _{2} , \delta _{3} } = \{ 1_{G} , r_{1} , r_{2} , r_{3} \} $. Then we see $ r_{3} = r_{1} r_{2} \in G_{ \delta _{1} , \delta _{2} } $. Hence $ G_{ \delta _{1} , \delta _{2} } = \{ 1_{G} , r_{1} , r_{2} , r_{3} \} =  G_{ \delta _{2} , \delta _{3} } $. Thus $ \Delta $ is not independent by Lemma \ref{indsbstsntqual}.
\end{proof}
From here we can quickly get an upper bound for the height of the action on $ \Omega $.
\begin{lema}\label{trsndstszthr}
Let $ \Delta := \{ \delta _{1} , \delta _{2} , \delta _{3} \} \subseteq \Omega $. If $ \Delta $ is an independent set then $ | G_{ \delta _{1} , \delta _{2} , \delta _{3} } | = 1 $.
\end{lema}
\begin{proof}
From Lemma \ref{stach}, it must be that
\[
| G_{ \delta _{1} , \delta _{2}   } | 
> | G_{ \delta _{1} , \delta _{2}  , \delta _{3}  } | 
\]
and
\[
| G_{ \delta _{1} , \delta _{3}   } | 
> | G_{ \delta _{1} , \delta _{2}  , \delta _{3}  } |  .
\]
Also Lemma \ref{indsbstsntqual} tells us that $ G_{ \delta _{1} } $, $ G_{ \delta _{2} } $ and $ G_{ \delta _{3} } $ are distinct. Hence $ | G _{ \delta _{1} , \delta _{2} } | , | G _{ \delta _{1} , \delta _{3} } |  \in \{ 1, 2 , 4 \} $ by Lemma \ref{klnfrntrsctn}. Therefore $ | G_{ \delta _{1} , \delta _{2}  , \delta _{3}  } | \in \{ 1, 2 \} $. If $ | G_{ \delta _{1} , \delta _{2}  , \delta _{3}  } | = 2 $ then $ | G_{ \delta _{1} , \delta _{2} }  | = | G_{ \delta _{1} , \delta _{3} } | = 4 $. But this contradicts Lemma \ref{szfrntrsctn}. Thus $ | G_{ \delta _{1} , \delta _{2}  , \delta _{3}  } | = 1 $.
\end{proof}
\begin{corl}\label{trshght}
$ Ht(G, \Omega) \leq 3 $.
\end{corl}
\begin{proof}
Suppose $ Ht(G, \Omega) \geq 4 $. Then there exists an independent set $ \Delta \subseteq \Omega $ of size $ 4 $. Let $ \Delta ^{ \prime } \subset \Delta $ with $ | \Delta | = 3 $. Then $ \Delta ^{ \prime } $ is independent by Corollary \ref{indsub}. It follows from Lemma \ref{trsndstszthr} that $ | G_{ ( \Delta ) } | = 1 $. But then we see from Lemma \ref{stach} that $ \Delta $ is not independent, a contradiction. Thus $ Ht(G, \Omega) \leq 3 $.
\end{proof}
Next some conditions are given that show when height can be exactly $ 3 $.
\begin{lema}\label{htthrcndtns}
Suppose $ H_{1} $ and $ H_{2} $ are subgroups of $ G$ that are conjugate to $ H $, with $H_{1} $, $H_{2} $ and $ H$ distinct. Suppose $ |  H_{1} \cap H_{2} | \geq | H_{1} \cap H | = | H_{2} \cap H | = 2 $ and $ H_{1} \cap H \neq H_{2} \cap H $. Then $ H $, $ H_{1} $ and $ H_{2} $ are the stabilizers of points of an independent set $ \{ \omega _{1} , \omega _{2} , \omega _{3} \} \subseteq \Omega $ and $ Ht(G, \Omega ) = 3 $.
\end{lema}
\begin{proof}
There exists $ g_{1} , g_{2} \in G $ such that $ H_{1} = g_{1} ^{-1} H g_{1} $ and $ H_{2} = g_{2} ^{-1} H g_{2} $. Since $ H $ is the stabilizer of itself when considered as a point in $ \Omega $, the subgroup $ H_{1} $ stabilizes $ Hg_{1} $ and $ H_{2} $ stabilizes $ Hg_{2} $.

\medskip
If $ | H_{1} \cap H_{2} \cap H | = | H_{1} \cap H  |  $ then $  H_{1} \cap H_{2} \cap H = H_{1}  \cap H $. Also $ | H_{1} \cap H_{2} \cap H | = | H_{2} \cap H  |  $, so $  H_{1} \cap H_{2} \cap H = H_{2}  \cap H $. Hence $ H_{1} \cap H  = H_{2}  \cap H $, which is a contradiction. Therefore $ | H_{1} \cap H_{2} \cap H | < | H_{1} \cap H  |  $. Hence $ | H_{1} \cap H_{2} \cap H | < | H_{2} \cap H  |  $ and $ | H_{1} \cap H_{2} \cap H | < | H_{1} \cap H_{2}  |  $.

\medskip
This shows that $ \{ H , Hg_{1} , Hg_{2} \} $ is an independent set by Lemma \ref{scndfnt}. Thus $ Ht(G, \Omega ) \geq 3 $. Finally, $Ht(G, \Omega ) = 3 $ by Corollary \ref{trshght}.
\end{proof}
To find the height of the actions of $ PGL_{2} (q) $ and $ PSL_{2} (q) $ on a maximal dihedral subgroup $ D $, various counting arguments will be employed that usually involve picking a reflection from a dihedral subgroup, looking at how many different conjugates of the subgroup contain the same reflection and then looking at intersections of these subgroups.

\newpage
The groups $ PGL_{2} (q) $ and $ PSL_{2} (q) $ have slightly different arguments when $ q $ is odd. Also different choices of $ q $ will lead to different methods of counting, largely because of the reflections in the dihedral subgroups being contained in one conjugacy class in some cases and two conjugacy classes in others. These cases will be given as examples through the section. I apologise ahead for the repetitiveness of these cases as they are similar but different enough that various closely related arguments are needed.

\medskip
The height of the dihedral actions are given in the following Examples as we go through, then collected together in a theorem at the end;
\begin{itemize}
\item Example \ref{xmpcrect} is $ PSL_{2} (q) $ with $ q $ even, $ q \geq 4 $ and $ D \cong D_{2(q+1) } $.
\item Example \ref{expslqevn} is $ PGL_{2} (q) $ with $ q $ odd, $ q \geq 5 $ and $ D \cong D_{2(q+1) } $.
\item Example \ref{psfljdshu} is $ PSL_{2} (q) $ with $ q $ odd, $ q \geq 11 $ and $ \tfrac{1}{2} |D | $ odd.
\item Example \ref{mssngxmp} is $ PSL_{2} (q) $ with $ q $ even, $ q \geq 8 $ and $ D \cong D_{2(q-1) } $.
\item Example \ref{pslqdddvfr} is $ PSL_{2} (q) $ with $ q $ odd, $ q \geq 11 $ and $ \tfrac{1}{2} |D | $ even.
\item Example \ref{fnlxmpatlst} is $ PGL_{2} (q) $ with $ q $ odd, $ q \geq 13 $ and $ D \cong D_{ 2(q-1) } $.
\end{itemize}
Comparing with Tables \ref{tableone}, \ref{tabletwo} and \ref{tablethree} we see all maximal dihedral subgroups of $ PGL_{2} (q) $ and $ PSL_{2} (q) $ are covered above, except $ PSL_{2} (5) $ acting on a $ D_{6} $, $ PSL_{2} (4) $ acting on $ D_{6} $ and $ PGL_{2} (q) $ acting on $ D_{ 2(q-1) } $ when $ q \in \{ 7 , 9 , 11 \} $. These special cases will be included in a theorem at the end of the section.

\medskip
We begin with some propositions to count the number of subgroups of $ G $ that contain a particular reflection in $ H $. For the rest of the section set $ \Gamma = \{ g^{-1} H g : g \in G \} $. Lemma \ref{cstquivtcnj} shows that $ | \Gamma | = | \Omega | = | G | / | H | $. For any group $ G ^{ \prime } $ and for $ g \in G ^{ \prime } $ the notation $ g^{G ^{ \prime }}  $ is used for the conjugacy class containing $ g $ in $ G ^{ \prime } $.
\begin{lema}\label{sbgrpscntngh}
Let $ R $ be the set of reflections of $H$. Suppose $r \in R $. The number of subgroups in $ \Gamma $ that contain $ r $ as a reflection is
\[
\frac{|G| | r^{G} \cap R | }{ |H| | r^{G} | } .
\]
Furthermore any element of $ r^{G} $ appears as a reflection in the same number of subgroups in $ \Gamma $ as $r$.
\end{lema}
\begin{proof}
Since $ H $ is a maximal subgroup of $ G $ and not normal, by Lemma \ref{cstquivtcnj} the action on the right cosets of $ H $ is equivalent to the action by conjugation on $ \Gamma $. Therefore $ | \Gamma | = |G| / |H| $.

\medskip
Let $ H ^{ \prime } \in \Gamma $ and let $ R^{ \prime } \subseteq H ^{ \prime } $ be its set of reflections. Then $ H^{ \prime} = g_{1} ^{-1} H g_{1} $ for some $ g_{1} \in G $. We are assuming $ | H | \geq 6 $, so there exists a unique cyclic subgroup of $H$ of order $ \tfrac{1}{2} |H| $, the subgroup of rotations. Hence the rotational subgroup of $ H $  must be sent to the subgroup of rotations of $ H ^{ \prime } $ under conjugation.

\medskip
This means that reflections are sent to reflections, that is $ R^{ \prime } = g_{1} ^{-1} R g_{1} $. Clearly $ a \in r^{G} \cap R  $ if and only if  $ g_{1}^{-1} a g_{1} \in r^{G} \cap R^{ \prime } $. Hence $ | r^{G} \cap R  | = |  r^{G} \cap R^{ \prime } | $. So the total number of reflections in subgroups in $ \Gamma $ that also lie in $ r^{G} $ including repeats is
\[
| \Gamma | | r^{G} \cap R | = \frac{|G| | r^{G} \cap R |  }{ |H| }  .
\]
Let $ r^{ \prime } \in r^{G} $. Then there exists some $ g_{2} \in G $ such that $ r^{ \prime }  = g_{2}^{-1} r g_{2} $. Let $ m $ be the number of subgroups in $ \Gamma $ containing $ r $ as a reflection and label them as $ H_{1} , \dots , H_{m} $. Let $ n $ be the number of subgroups containing $ r^{ \prime } $ as a reflection and label them as $ H^{ \prime } _{1} , \dots , H^{ \prime } _{n} $.

\medskip
For each $ i \in \{1, \dots , m \} $ the subgroups $ g_{2}^{-1} H_{i} g_{2} $ contain $ r ^{ \prime } $. So $ m \leq n $. Similarly for each $ j \in \{1, \dots , n \} $ the subgroups $ g_{2} H_{j}^{ \prime } g_{2}^{-1} $ contain $ r  $. So $ n \leq m $. Hence $ n = m $. Thus each element of $ r^{G} $ appears as a reflection in the same number of subgroups in $ \Gamma $. Therefore the number of subgroups in $ \Gamma $ that contain $ r $ as a reflection is
\[
\frac{|G| | r^{G} \cap R | }{ |H| | r^{G} | } .
\]
\end{proof}
\begin{corl}\label{crlrsrflctn}
Let $ R $ be the set of reflections of $H$. Suppose $r \in R $. The number of subgroups in $ \Gamma $ that contain $ r $ as a reflection is
\[
\begin{dcases}
 \frac{|G|  }{ 2 | r^{G} | }  , \ \ \ \ \  \text{if $ | r^{G} \cap R | = \tfrac{1}{2} |H| $,} \\
\frac{|G|  }{ 4 | r^{G} | } , \ \ \ \ \  \text{if $ | r^{G} \cap R | \neq \tfrac{1}{2} |H| $.}
\end{dcases}
\]
\end{corl}
\begin{proof}
The reflections in $ H $ either lie in a single conjugacy class in $ H$, with size $ \tfrac{1}{2} | H | $, or two conjugacy classes of size $ \tfrac{1}{4} | H | $. So either $ | r^{G} \cap R | = \tfrac{1}{2} |H| $ or $ | r^{G} \cap R | = \tfrac{1}{4} |H| $. Substituting these into the formula from Lemma \ref{sbgrpscntngh} gives the result.
\end{proof}
\begin{lema}\label{intrscttpssbl}
Suppose $ H_{1} , H_{2} \in \Gamma $. Let $ R_{1} $ and $ R_{2} $ be the set of reflections of these subgroups respectively. Then $ H_{1} \neq H_{2} $ if and only if $ | R_{1} \cap R_{2} | \leq 1 $.
\end{lema}
\begin{proof}
Since $ | H | \geq 6 $ there are at least $3 $ reflections in $ H $. So $ | R_{1} | = | R_{2} | \geq 3 $. If $ H_{1} = H_{2} $ then $ R_{1} = R_{2} $ and $ | R_{1} \cap R_{2} | = | R_{1} | > 1 $.

\medskip
Next suppose $ | R_{1} \cap R_{2} | > 1 $. Then there exists $ r_{1} , r_{2} \in R_{1} \cap R_{2} $ with $ r_{1} \neq r_{2} $. It must be that $ r_{1} r_{2} $ is a rotation in both $ H_{1} $ and $ H_{2} $ and is not the identity because $ r_{1} $ and $ r_{2} $ are self inverse. Hence $ \langle  r_{1} r_{2} \rangle $ is a non-trivial subgroup of the rotational subgroups of $ H_{1} $ and $ H_{2} $. Therefore $ \langle  r_{1} r_{2} \rangle \trianglelefteq H_{1} $ and $ \langle  r_{1} r_{2} \rangle \trianglelefteq H_{2} $ by Lemma \ref{dhdrlrottnnrml}. It follows from Lemma \ref{ntnrmlmxml} that $ H_{1} = H_{2} $.
\end{proof}
\begin{lema}\label{ttlcnjgtsbrflc}
Let $ R $ be the set of reflections of $ H $. Suppose $ r \in R $. Let $ \Delta $ be the set of subgroups of $ G $ conjugate to $ H $ such that if $ H^{ \prime } \in \Delta $ and $ R^{ \prime } $ is the set of reflections of $ H^{ \prime } $ then $ R \cap R^{ \prime } \cap r^{G} \neq \emptyset $. Then
\[
| \Delta | = 
\begin{dcases}
 \frac{|G| | H |  }{ 4 | r^{G} | } - \frac{ |H| }{2} + 1  , \ \ \ \ \  \text{if $ | r^{G} \cap R | = \tfrac{1}{2} |H| $,} \\
 \frac{|G| | H |  }{ 16 | r^{G} | } - \frac{ |H| }{4} + 1  , \ \ \ \ \  \text{if $ | r^{G} \cap R | \neq \tfrac{1}{2} |H| $.}
\end{dcases}
\]
\end{lema}
\begin{proof}
We will develop a general formula for $ | \Delta | $ before considering the size of $ r^{G} \cap R $. Let $ \Gamma = \{ g^{-1} H g : g \in G  \} $. First suppose that $ H $ is the only subgroup in $ \Gamma $ containing $ r $ as a reflection. It follows from Lemma \ref{sbgrpscntngh} that
\[
\frac{|G| | r^{G} \cap R | }{ |H| | r^{G} | } = 1.
\]
Lemma \ref{sbgrpscntngh} also shows that each element of $ r^{G} \cap R $ is a reflection in exactly one subgroup in $ \Gamma $, namely $ H $. Therefore $ | \Delta | = 1 $. Hence 
\[
\Bigg( \frac{|G| | r^{G} \cap R | }{ |H| | r^{G} | } - 1 \Bigg)| r^{G} \cap R | + 1  = (1-1) | r^{G} \cap R | + 1  = 1 = | \Delta |  \tag{1}.
\]
Next suppose $ r $ is a reflection in at least two subgroups in $ \Gamma $. Then each element of $ r^{G} \cap R $ is a reflection in at least two subgroups in $ \Gamma $ by Lemma \ref{sbgrpscntngh}, in particular each element appears as a reflection in at least one subgroup other than $ H $.

\medskip
Lemma \ref{sbgrpscntngh} shows that an element $ r_{1} \in r^{G} \cap R $ is a reflection in
\[
\frac{|G| | r^{G} \cap R | }{ |H| | r^{G} | } - 1
\]
subgroups in $ \Gamma \setminus \{ H \} $. Suppose $ H^{ \prime} \in \Gamma \setminus \{ H \} $ and $ R^{ \prime } $ is the set of reflections of $ H^{ \prime } $. If $ r_{2} \in r^{G} \cap R \setminus \{ r_{1} \} $ and $ r_{2} \in R ^{ \prime } $ then $r_{1} \notin R^{ \prime } $, otherwise $ r_{1} , r_{2} \in R \cap R^{ \prime} $, which is not possible by Lemma \ref{intrscttpssbl}. So the subgroups in $ \Gamma \setminus \{ H \} $ containing $ r_{2} $ as a reflection are distinct from those containing $ r_{1} $ as a reflection. Hence the total number of distinct subgroups of $ \Gamma \setminus \{ H \} $ containing an element of $ r^{G} \cap R $ as a reflection is
\[
\Bigg( \frac{|G| | r^{G} \cap R | }{ |H| | r^{G} | } - 1 \Bigg)| r^{G} \cap R | .
\]
The only other subgroup in $ \Gamma $ that contains a reflection from $ r^{G} \cap R $ is $ H $ itself. Counting this gives
\[
| \Delta | = 
\Bigg( \frac{|G| | r^{G} \cap R | }{ |H| | r^{G} | } - 1 \Bigg)| r^{G} \cap R | + 1 .
\]
So in this case we again have the same formula as $ (1) $.

\medskip
Now the reflections in $ H $ either lie in a single conjugacy class in $ H$ of size $ \tfrac{1}{2} | H | $ or two conjugacy classes that each have size $ \tfrac{1}{4} | H | $. So either $ | r^{G} \cap R | = \tfrac{1}{2} |H| $ or $ | r^{G} \cap R | = \tfrac{1}{4} |H| $. Substituting these into the above formula and expanding gives the result we are after.
\end{proof}
These lemmas can be applied to the first of our examples.
\begin{exmp}\label{xmpcrect}
Let $ G = PSL_{2} (q) $ where $ q $ is even and $ q \geq 4 $, let $ H $ be a maximal $ D_{2 (q+1) } $ and let $ \Omega $ be the set of right cosets of $ H$. Then $ | H | = 2(q+1) $. Lemma \ref{lnsz} shows $ | G | =  q(q+1) (q-1) $.

\medskip
Let $ R $ be the reflections of $H $ and suppose $ r \in R$. By Lemma \ref{pslnvltnscnjg} all involutions in $ G $ are conjugate to $ r $ and $ | r^{G} | = (q+1)(q-1) $. It must be that $   r^{G} \cap R = R $, therefore $ | r^{G} \cap R | = \tfrac{1}{2} | H | $.

\medskip
From Lemma \ref{ttlcnjgtsbrflc} we see that the number of subgroups conjugate to $ H $ that contain a reflection from $ R $ is
\begin{align*}
\frac{|G| | H |  }{ 4 | r^{G} | } - \frac{ |H| }{2} + 1 
&  = \frac{ 2q(q+1)^{2} (q-1)}{ 4(q+1)(q-1) } - (q+1) + 1 \\
& = \tfrac{1}{ 2}q(q+1) - q \\
& = \tfrac{1}{ 2} q(q-1) .
\end{align*}
Let $ \Gamma  = \{ g^{-1} H g : g \in G \} $. It was noted earlier that $ | \Gamma | = | G | / | H | = \tfrac{1}{ 2} q(q-1) $. So every subgroup in $ \Gamma $ intersects $ H $ non-trivially. As $ q+1 $ is odd, $ D_{2 (q+1) } $ contains no Klein four-subgroups. So using Lemma \ref{klnfrntrsctn} we have $ | H \cap H ^{ \prime } | = 2 $ for each $ H^{ \prime } \in \Gamma \setminus \{ H \} $. However $ H $ is an arbitrary maximal $ D_{2(q+1) } $, which means that every subgroup in $ \Gamma $ intersects each of the others in a subgroup of order $ 2 $. 

\medskip
Now $ q + 1 \geq 5 $, so there exists $ r_{1} , r_{2} \in R $ with $ r_{1} \neq r_{2} $. From Corollary \ref{crlrsrflctn} we see that $ r_{1} $ is a reflection in $ \tfrac{1}{2} q \geq 2 $ subgroups in $ \Gamma $ and same goes for $ r_{2} $. Let $ H_{1} , H_{2} \in \Gamma \setminus \{ H \} $ with $ r_{1} \in H_{1} $ and $ r_{2} \in H_{2} $. Then $ H_{1} \cap H = \langle r_{1} \rangle \neq \langle r_{2} \rangle =  H_{2} \cap H $. Using Lemma \ref{htthrcndtns} and the fact $ | H_{1} \cap H_{2} | =  | H_{1} \cap H | =  | H_{2} \cap H | = 2$, we see that $ Ht(G, \Omega ) = 3 $.
\end{exmp}

\newpage
For the next lemma, note that there is no maximal $ D_{2 (q-1 ) } $ in $ PGL _{2} (5) \cong S_{5} $. So if $ q \geq 5 $ and $ H \leq  PGL _{2} (q) $ is a maximal dihedral subgroup then $ | H | \geq 6 $ and we can distinguish the reflections in $ H $ from the rotation of order $ 2 $ and can make use of earlier lemmas in this subsection (unlike if $ H $ was a Klein four group).
\begin{lema}\label{pglcnjclssrfl}
Let $ q \geq 5 $ and $ q $ odd. Suppose $ G \cong PGL_{2} (q) $ and that $ H < G $ is a maximal $ D_{2(q-1) } $ or $ D_{2(q+1) } $. Let $ R $ be the reflections of $ H $ and let $ r \in R $. Then $ | r^{G} \cap R | = \tfrac{1}{4} | H | = \tfrac{1}{2} | R | $.
\end{lema}
\begin{proof}
For the case $q = 5 $ we have $ PGL_{2} (5) \cong S_{5} $ and it can be checked by looking at the subgroup structure of $ S_{5} $ that the $ D_{12} $ and $ D_{8} $ subgroups have reflections  from different conjugacy classes.

\medskip
Now suppose $ q > 5 $. By Lemma \ref{pslsnpgl} there exists $ S \leq G $ with $ S \cong PSL_{2} (q) $. By Lemma \ref{dcvlkjhsdp}, there exists a dihedral subgroup $ D \leq S $ of order $ \tfrac{1}{2} | H | $. Let $ A < D $ be the rotational subgroup. Observe $ | A | = \tfrac{1}{4} | H | > \tfrac{1}{2} (5-1) = 2 $. Hence $ | A | $ only divides one of $ q+1 $ or $ q-1 $. Also $ | A | $ is coprime to $ q $. So Lemma \ref{nrmlzr} tells us $ N_{S} (A) = D $. By the same lemma, $ N_{G} ( A ) \cong H $ and $ N_{G} ( A ) $ is conjugate to $ H $ in $ G$. 

\medskip
Half of the reflections of $ N_{G} ( A ) $ must lie in $ D $ and therefore $ S $ because $ | N_{G} ( A ) | = | H | = 2|D| $. Hence half of the reflections of $ N_{G} ( A ) $ do not lie in $ S $. By Lemma \ref{pslsnpgl}, the subgroup $ S $ is normal in $ G $. Thus half the reflections in $ N_{G} ( A ) $ lie in one conjugacy class in $ G $ and the other half of the reflections lie in a separate conjugacy class $ G $. It follows that the same applies to the reflections in $ H $.
\end{proof}
\begin{lema}\label{asoidspdoc}
Let $ R $ be the set of reflections of $ H $ and let $ r \in R $. Suppose $ | R | $ is even. Then there exists at most one subgroup $ H^{ \prime } \leq G $ such that;
\begin{itemize}
\item [$(1)$] $ H^{ \prime } \neq H $,
\item [$(2)$] $ H^{ \prime } \in \Gamma $,
\item [$(3)$] $ r $ is a reflection in $ H^{ \prime } $ and
\item [$(4)$] $ | H \cap H^{ \prime } | = 4 $. 
\end{itemize}
\end{lema}
\begin{proof}
Suppose there exists $  H_{1} $ and $ H_{2} $ that satisfy the same conditions at $ H ^{ \prime } $ in $ (1) $ to $ (4) $ above.

\medskip
Since $ | R | $ is even, $ |H | $ is divisible by $ 4 $ and there exists a unique rotation $ z \in H $ of order $ 2 $. Note that $ z \in Z(H) $. Let $ i \in \{ 1, 2 \} $. The fact $ H_{i} $ is conjugate to $ H $ means it is a stabilizer of some point in $ \Omega $. So it follows from Lemma \ref{klnfrntrsctn} that $ H \cap  H_{i} $ is a Klein four-subgroup of $ H $ and contains $ z $. As $ r $ is a reflection in both $ H $ and $ H_{i} $, we must have $ \langle r , z \rangle = H _{i} \cap H $.

\medskip
If $ z \in Z( H _{i} ) $ then $ H_{i} = H $ by Corollary \ref{cntrcnjgt}, which is a contradiction. So $ z \notin Z( H _{i} ) $ and $ z $ must be a reflection in $ H _{i} $. Hence $ zr $ is the unique rotation of order $ 2 $ in $ H_{i} $ and $ zr \in Z( H _{i} ) $. That means $ zr \in Z( H _{1} ) \cap Z( H _{2} ) $, which implies $ H_{1} = H_{2} $, using Corollary \ref{cntrcnjgt} again.
\end{proof}
\begin{exmp}\label{expslqevn}
Let $ G = PGL_{2} (q) $ where $ q $ is odd and $ q \geq 5 $, let $ H $ be a maximal $ D_{2 (q+1) } $ and let $ \Omega $ be the set of right cosets of $ H$. Then $ | H | = 2(q+1) $ and Lemma \ref{lnsz} shows $ | G | =  q(q+1) (q-1) $.

\medskip
Let $ R $ be the reflections of $H $. Using Lemma \ref{pglcnjclssrfl} and the fact $ | R | = \tfrac{1}{2} | H | $, there exists $ r_{1} , r_{2} \in R$ with $ r_{2} \notin  r_{1} ^{G}  $. Lemma \ref{pglcnjclssrfl} also shows that $ | r_{1} ^{G} \cap R | = | r_{2} ^{G} \cap R | = \tfrac{1}{4} | H | $. Hence
\[
( r_{1} ^{G} \cap R ) \cup ( r_{2} ^{G} \cap R ) = R .
\]
For $ i \in \{ 1, 2 \} $ let $ \Delta _{i} $ be the set of subgroups in $ \Gamma $ that contain a reflection that is also in $ r_{i} ^{G} \cap R $. We may suppose that $ | r_{1} ^{G} | = \tfrac{1}{2} q (q-1) $ and $ | r_{2} ^{G} | = \tfrac{1}{2} q (q+1) $ by Lemma \ref{pgltqnvcnjc}. So using Lemma \ref{ttlcnjgtsbrflc} we have
\begin{align*}
| \Delta _{1} \setminus \{ H \} |
 =  \frac{|G| | H |  }{ 16 | r_{1} ^{G} | } - \frac{ |H| }{4}  
 =  \frac{ 2q(q+1) ^{2} (q-1)  }{ 8 q (q-1) } - \frac{1}{2} (q+1)  
&  =  \frac{1}{4} (q+1) ^{2}  - \frac{1}{2} (q+1)  \\
&  = \frac{1}{4} (q+1)(q-1) .
\end{align*}
and
\begin{align*}
| \Delta _{2} \setminus \{ H \}  |
 =  \frac{|G| | H |  }{ 16 | r_{2} ^{G} | } - \frac{ |H| }{4}  
 =  \frac{ 2q(q+1) ^{2} (q-1)  }{ 8 q (q+1) } - \frac{1}{2} (q+1)   
& =  \frac{ 1}{4}(q+1) (q-1)  - \frac{1}{2} (q+1)   \\
& =  \frac{ 1}{4}(q+1) (q-3)    .
\end{align*}

\medskip
Suppose $ H^{ \prime } \in \Delta _{1} \setminus \{ H \} $ and $ R^{ \prime } $ is the set of reflections in $ H^{ \prime } $. If $ R^{ \prime } $ contained an element of $ r_{2} ^{G} \cap R $ then $ | R^{ \prime } \cap R | \geq 2 $, which is not possible by Lemma \ref{intrscttpssbl}. Therefore $ R^{ \prime } $ does not contain an element of $ r_{2} ^{G} \cap R $ and it follows that $ H^{ \prime } \notin \Delta _{2} \setminus \{ H \} $. Thus
\begin{align*}
| \ (  \Delta _{1} \setminus \{ H \}  ) \ \cup \ ( \Delta _{2} \setminus \{ H \} ) \ \cup \  \{ H \} \ | 
& = | \Delta _{1} \setminus \{ H \} | + | \Delta _{2} \setminus \{ H \} | +  | \{ H \}  | \\
& = \frac{1}{4} (q+1)(q-1) + \frac{ 1}{4}(q+1) (q-3) +  1 \\
& = \frac{1}{4} (q+1)(2q-4) +  1 \\
& = \frac{1}{2} q^{2} - \frac{1}{2}q - 1 +  1 \\
& = \tfrac{1}{2} q (q-1) .
\end{align*}
Now Lemma \ref{cstquivtcnj} shows $ \Gamma = | G | / | H | = \tfrac{1}{2} q (q-1) $. Since $ \ (  \Delta _{1} \setminus \{ H \}  ) \ \cup \ ( \Delta _{2} \setminus \{ H \} ) \ \cup \  \{ H \}  \subseteq \Gamma $, it must be that  $ \ (  \Delta _{1} \setminus \{ H \}  ) \ \cup \ ( \Delta _{2} \setminus \{ H \} ) \ \cup \  \{ H \}  = \Gamma $. So every subgroup in $ \Gamma $ intersects $ H $ non-trivially. However $ H $ is an arbitrary maximal $ D_{2(q+1) } $, which means that every subgroup in $ \Gamma $ intersects each of the others non-trivially. 

\medskip
The fact $ q  \geq 5 $ means $ | r_{1} ^{G} \cap R | = \tfrac{1}{2} (q+1) \geq 3 $. Let $ s_{1} , s_{2} \in r_{1} ^{G} \cap R $ with $ s_{1} \neq s_{2} $. From Corollary \ref{crlrsrflctn} we see that $ s_{1} $ is a reflection in $ \tfrac{1}{2} (q+1) \geq 3 $ subgroups in $ \Gamma $ and same goes for $ s_{2} $. So there exists $ H_{1} , H_{2} \in \Gamma \setminus \{ H \} $ such that $ s_{1} $ and $ s_{2} $ are reflections in $ H_{1} $ and $ H_{2} $ respectively.

\medskip
Using Lemma \ref{asoidspdoc} we can assume that $ | H_{1} \cap H | \neq 4 \neq | H_{2} \cap H | $. Both $ H_{1} $ and $ H_{2} $ are stabilizers of points of $ \Omega $, so it follows from Lemma \ref{klnfrntrsctn} that $ | H_{1} \cap H | = | H_{2} \cap H | = 2 $. Hence $ H_{1} \cap H = \langle s_{1} \rangle \neq \langle s_{2} \rangle = H_{2} \cap H $.

\medskip
Finally the fact that $ H_{1} \cap H_{2} $ is non-trivial means $ | H_{1} \cap H_{2} | \geq 2 $, so $ Ht(G, \Omega ) = 3 $ by Lemma \ref{htthrcndtns}.
\end{exmp}

\medskip
In the above examples the height of some dihedral actions of $ PSL_{2} (q) $ or $ PGL_{2} (q) $ was found by counting the number of subgroups conjugate to a particular dihedral group and showing they all intersected each other non-trivially. However this is not the case for maximal dihedral subgroups in all choices of $ PSL_{2} (q) $ or $ PGL_{2} (q) $. Instead a method of counting the number of reflections in particular subgroups will be used to find the height the remaining actions. 
\begin{lema}\label{zzkfziitrw}
Let $ R $ be the set of reflections of $ H $. Suppose $ | R | $ is odd. Let $ r \in R $. Let $ \Delta $ be the set of subgroups of $ G $ conjugate to $ H $ such that if $ H^{ \prime } \in \Delta $ and $ R^{ \prime } $ is the set of reflections of $ H^{ \prime } $ then $ R \cap R^{ \prime } \neq \emptyset $. Suppose $ | \Delta | > 1 $ and $ Ht(G, \Omega ) \neq 3 $. Then the number of distinct reflections in subgroups in $ \Delta $ is
\[
 \Bigg( \frac{|G| | H |  }{ 4 | r^{G} | } - \frac{ |H| }{2} \Bigg) \Bigg( \frac{| H |}{2} - 1 \Bigg)  + \frac{ | H | }{2}  .
\]
\end{lema}
\begin{proof}
Suppose $ H_{1} , H_{2} \in \Delta  \setminus \{ H \} $ with $ H_{1} \neq H_{2} $. Let $ R_{1} $ and $ R_{2} $ be the reflections of $ H_{1} $ and $H_{2} $ respectively.

\medskip
First note that $ | R_{1} \cap R |  = 1 $. This is because by definition of $ \Delta $ it must be that $ | R_{1} \cap R | \geq 1 $ and Lemma \ref{intrscttpssbl} shows that $ | R_{1} \cap R | \leq 1 $. Also $ | R_{2} \cap R | = 1 $ can be shown in the same way.

\medskip
Next we show that each of the $ | H | / 2 - 1 $ reflections in $ R_{1} \setminus R_{1} \cap R $ are distinct from each of the $ | H | / 2 -1 $ reflections in $ R_{2} \setminus R_{2} \cap R $. This is split into two cases.

\medskip
For the first case suppose $ R_{1} \cap R  =  R_{2} \cap R  $. Then $ R_{1} \cap R = R_{1} \cap R_{2} \cap R \subseteq R_{1} \cap R_{2} $ and $ | R_{1} \cap R_{2} | \geq 1 $. It follows from Lemma \ref{intrscttpssbl} that $ | R_{1} \cap R_{2} | = 1 $. So the elements of $ R_{1} \setminus  R_{1} \cap R $ are distinct from the elements of $ R_{2} \setminus  R_{2} \cap R $.

\medskip
For the second case suppose $ R_{1} \cap R  \neq  R_{2} \cap R  $. Both $ H_{1} $ and $ H_{2} $ are conjugate to $H $, so they stabilise some elements of $ \Omega $. Therefore Lemma \ref{klnfrntrsctn} can be used to show any non-identity elements in $ H_{1} \cap H_{2} $ have order $ 2 $.

\medskip
Since $ | R | $ is odd, the rotational subgroup of $ H$ has odd order. Hence $ H$ does not contain a rotation of order $ 2 $ nor a Klein four-subgroup. Hence neither do $H_{1} $ or $ H_{2} $. So the non-identity elements of $ H_{1} \cap H_{2} $ must be reflections in both $ H_{1} $ and $H_{2} $. Lemma \ref{intrscttpssbl} shows $ | R_{1} \cap R_{2} | \leq 1 $, implying $ | H_{1} \cap H_{2} | \leq 2 $.

\medskip
Similar reasoning shows that $ | H_{1} \cap H  | = | H_{2} \cap H  | = 2 $ since $ | R_{1} \cap R | = | R_{2} \cap R | = 1 $.

\medskip
If $ | R_{1} \cap R_{2} | = 1 $ then $ | H_{1} \cap H_{2} | = 2 $ and there exists a single reflection $ s \in R_{1} \cap R_{2} $. It cannot be that $ s \in R $ (otherwise the fact that $ | R_{1} \cap R | = | R_{2} \cap R | = 1 $ would imply $  R_{1} \cap R = \{ s \} =  R_{2} \cap R \neq R_{1} \cap R $). This means that $ s \notin R_{1} \cap R $, so $ R_{1} \cap R_{2} \neq R_{1} \cap R $ and in turn $ H_{1} \cap H_{2} \neq H_{1} \cap H $. It follows that $ Ht(G, \Omega ) =3 $ by Lemma \ref{htthrcndtns}, which is a contradiction. Thus $ | R_{1} \cap R_{2} | \neq 1 $.

\medskip
Therefore $ | R_{1} \cap R_{2} | = 0 $, showing all reflections in $ R_{1} $ are distinct from those in $ R_{2} $. In particular the elements of $ R_{1} \setminus  R_{1} \cap R $ are distinct from those in $ R_{2} \setminus  R_{2} \cap R $ in this case.

\medskip
Since $ | R | $ has odd order, all the elements of $ R $ are conjugate in $ H $. Therefore they are also conjugate in $G $. Hence
\[
| \Delta \setminus \{ H \} | = 
 \frac{|G| | H |  }{ 4 | r^{G} | } - \frac{ |H| }{2}
\]
by Lemma \ref{ttlcnjgtsbrflc}. So the total number of distinct elements that appear as reflections in subgroups in $ \Delta \setminus \{ H \} $ but are not reflections in $ H $ are
\[
 | \Delta \setminus \{ H \} | | R_{1} \setminus  R_{1} \cap R | = 
 \Bigg( \frac{|G| | H |  }{ 4 | r^{G} | } - \frac{ |H| }{2} \Bigg) \Bigg( \frac{ |H| }{2} - 1 \Bigg) .
\]
The only other reflections in subgroups in $ \Delta $ that have not been counted are the reflections of $ H $ itself. Including these gives
 \[
 \Bigg( \frac{|G| | H |  }{ 4 | r^{G} | } - \frac{ |H| }{2} \Bigg) \Bigg( \frac{ |H| }{2} - 1 \Bigg) + \frac{ |H| }{2} 
\]
 distinct reflections in total.
\end{proof}
With this lemma, two more examples from the earlier list can now be ticked off.
\begin{exmp}\label{psfljdshu}
Let $ G = PSL_{2} (q) $ where $ q $ is odd and $ q \geq 11 $. Then $ | G | = \tfrac{1}{2} q (q+1) (q-1) $.

\medskip
Let $ H $ be a maximal $ D_{ (q-1) } $ or $ D_{ (q+1) } $ and let $ \Omega $ be the set of right cosets of $ H$. Furthermore let $ R $ be the reflections of $ H $ and suppose $ | R | $ is odd. Let $ r \in R$. By Lemma \ref{pslnvltnscnjg} all involutions in $ G $ are conjugate to $ r $. Therefore $   r^{G} \cap R = R $ and so $ | r^{G} \cap R | = \tfrac{1}{2} | H | $.

\medskip
Let $ \Delta $ be the set of subgroups of $ G $ conjugate to $ H $ such that if $ H^{ \prime } \in \Delta $ and $ R^{ \prime } $ is the set of reflections of $ H^{ \prime } $ then $ R \cap R^{ \prime } \neq \emptyset $.

\medskip
From here two cases will be looked at; when $ q \equiv 1 \pmod{4} $ and when $ q \equiv 3 \pmod{4} $.

\medskip
First suppose $ q \equiv 1 \pmod{4} $. Then Lemma \ref{pslnvltnscnjg} shows that $ |  r^{G}|  = \tfrac{1}{2} q ( q+ 1) $. If $ H \cong D_{q-1 } $ then $ | R | = \tfrac{1}{2} | H |  $ is even, which contradicts our earlier assumption. So $ H \cong D_{q+1 } $ and $ | H | = q + 1 $.

\medskip
From Lemma \ref{ttlcnjgtsbrflc} we see $ | \Delta | \geq 25 $ since $ q \geq 11 $. Let $ n_{1} $ be the number of distinct reflections in subgroups in $ \Delta $. These reflections lie in $  r^{G} $ so $ n_{1} \leq |  r^{G}| $ and $ n_{1} -   |  r^{G}| \leq 0 $. If $ Ht(G, \Omega ) \neq 3 $ then Lemma \ref{zzkfziitrw} shows that
\begin{align*}
n_{1} -  |  r^{G}| & = \Bigg( \frac{|G| | H |  }{ 4 | r^{G} | } - \frac{ |H| }{2} \Bigg) \Bigg( \frac{| H |}{2} - 1 \Bigg)  + \frac{ | H | }{2} -  \frac{q ( q+ 1)}{2}  \\
& = \Bigg( \frac{ \tfrac{1}{2} q (q+1) ^{2} (q-1)    }{ 2 q ( q+ 1) } - \frac{ q+1 }{2} \Bigg) \Bigg( \frac{q+1}{2} - 1 \Bigg)  + \frac{ q+1 }{2} -  \frac{q ( q+ 1)}{2} \\
& = \Bigg( \frac{   (q+1)  (q-1)    }{ 4  } - \frac{ q+1 }{2} \Bigg) \frac{q-1}{2}   - \frac{(q+1)(q-1)}{2}  \\
& =  \frac{1}{8}(q+1) (    q-1    -2 ) (q-1)   - \frac{(q+1)(q-1)}{2}  \\
& =  \frac{1}{8}(q+1) (q-3 ) (q-1)   - \frac{(q+1)(q-1)}{2}  \\  
& =  \frac{1}{8}(q+1)(q-1) ( q-3     - 4) \\  
& =  \frac{1}{8}(q+1)(q-1) ( q- 7) \\
& > 0. \tag{Since $q \geq 11$}
\end{align*}
This is a contradiction. Thus $ Ht(G, \Omega ) = 3 $ in this case.

\medskip
For the second case suppose  $ q \equiv 3 \pmod{4} $. Then Lemma \ref{pslnvltnscnjg} shows that $ |  r^{G}|  = \tfrac{1}{2} q ( q- 1) $. If $ H \cong D_{q+1 } $ then $ | R | = \tfrac{1}{2} | H |  $ is even, which contradicts our earlier assumption. So $ H \cong D_{q-1 } $ and $ | H | = q - 1 $.

\medskip
From Lemma \ref{ttlcnjgtsbrflc} we see $ | \Delta | \geq 26 $ since $ q \geq 11 $. Let $ n_{2} $ be the number of distinct reflections in subgroups in $ \Delta $. As in the first case, $ n_{2} \leq |  r^{G}| $ and $ n_{2} -   |  r^{G}| \leq 0 $. If $ Ht(G, \Omega ) \neq 3 $ then Lemma \ref{zzkfziitrw} shows that
\begin{align*}
n_{2} -   |  r^{G}| & = \Bigg( \frac{|G| | H |  }{ 4 | r^{G} | } - \frac{ |H| }{2} \Bigg) \Bigg( \frac{| H |}{2} - 1 \Bigg)  + \frac{ | H | }{2}  -  \frac{q ( q- 1)}{2} \\
& =  \Bigg( \frac{ \tfrac{1}{2} q (q+1) (q-1)  ^{2}  }{ 2 q ( q-1) } - \frac{ q-1 }{2} \Bigg) \Bigg( \frac{q-1}{2} - 1 \Bigg)  + \frac{ q-1 }{2}  -  \frac{q ( q- 1)}{2}  \\
& =  \Bigg( \frac{  (q+1) (q-1)    }{ 4} - \frac{ q-1 }{2} \Bigg)  \frac{q-3}{2}   - \frac{ (q-1)^{2} }{2}  \\
& =  \frac{ 1  }{ 8} ( q-1 ) (   q+1  - 2 ) (q-3)    - \frac{ (q-1)^{2} }{2} \\
& =  \frac{ 1  }{ 8} ( q-1 ) ^{2} (q-3)    - \frac{ (q-1)^{2} }{2} \\
& =  \frac{ 1  }{ 8} ( q-1 ) ^{2} (q-3 - 4) \\
& =  \frac{ 1  }{ 8} ( q-1 ) ^{2} (q-7) \\
& > 0. \tag{Since $q \geq 11$}
\end{align*}
Again we have a contradiction, so $ Ht(G, \Omega ) = 3 $ in this case as well.
\end{exmp}
\begin{exmp}\label{mssngxmp}
Let $ G = PSL_{2} (q) $ where $ q $ is even and $ q \geq 8 $. Then $ | G | = q (q+1) (q-1) $.
\medskip
Let $ H $ be a maximal $ D_{ 2(q-1) } $ and let $ \Omega $ be the set of right cosets of $ H$. Furthermore let $ R $ be the reflections of $ H $. Note $ | H | = 2(q - 1) $ and that $ | R | = q - 1$ is odd. 

\medskip
Let $ r \in R$. By Lemma \ref{pslnvltnscnjg} all involutions in $ G $ are conjugate to $ r $ and $ |  r^{G}|  =  ( q+ 1)( q- 1) $. Therefore $   r^{G} \cap R = R $ and so $ | r^{G} \cap R | = \tfrac{1}{2} | H | $. 

\medskip
Let $ \Delta $ be the set of subgroups of $ G $ conjugate to $ H $ such that if $ H^{ \prime } \in \Delta $ and $ R^{ \prime } $ is the set of reflections of $ H^{ \prime } $ then $ R \cap R^{ \prime } \neq \emptyset $.

\medskip
From Lemma \ref{ttlcnjgtsbrflc} we see $ | \Delta | \geq 22 $ since $ q \geq 8 $. Let $ n $ be the number of distinct reflections in subgroups in $ \Delta $. These reflections lie in $  r^{G} $ so $ n \leq |  r^{G}| $ and $ n -   |  r^{G}| \leq 0 $. If $ Ht(G, \Omega ) \neq 3 $ then Lemma \ref{zzkfziitrw} shows that
\begin{align*}
n -  |  r^{G}| & = \Bigg( \frac{|G| | H |  }{ 4 | r^{G} | } - \frac{ |H| }{2} \Bigg) \Bigg( \frac{| H |}{2} - 1 \Bigg)  + \frac{ | H | }{2} -  ( q+ 1)( q- 1)  \\
& =  \Bigg( \frac{ 2q (q+1) (q-1)^{2}  }{ 4 (q+1)(q-1) } - (q-1) \Bigg) ( q-1 - 1 )  + q-1 -  ( q+ 1)( q- 1) \\
& =  \Bigg( \frac{ q  (q-1) }{ 2 }  - (q-1) \Bigg) ( q-2 )   -   q( q- 1) \\
& =  \frac{1}{2} (q-1)  ( q-2 )^{2}   -   q( q- 1) \\
& =  \frac{1}{2} (q-1)  (( q-2 )^{2}   -   2q) \\
& =  \frac{1}{2} (q-1)  (q^{2} - 6q + 4) \\
& =  \frac{1}{2} q (q-1) (q-6) + 2 (q-1) \\
& > 0. \tag{Since $ q \geq 8 $}  
\end{align*}
This is a contradiction. Thus $ Ht(G, \Omega ) = 3 $.
\end{exmp}
\begin{lema}\label{sctnszfrcntr}
Let $ R $ be the set of reflections of $ H $. Suppose $ | R | $ is even. Let $ r \in R $ and $ z \in Z(H) $. Suppose $ R \subseteq r^{G} $ and $ z \in r^{G} $. Then there exists a unique $ H^{ \prime } \leq G $ such that
\begin{itemize}
\item [$(1)$] $ H^{ \prime } \neq H $,
\item [$(2)$] $ H^{ \prime } \in \Gamma $,
\item [$(3)$] $ r $ is a reflection in $ H^{ \prime } $ and
\item [$(4)$] $ | H \cap H^{ \prime } | = 4 $.
\end{itemize}
\end{lema}
\begin{proof}
Observe that $ | Z(H) | = 2 $ because $ | R | $ is even. Since $ z \in r^{G}  $, it must be that $ z $ is not the identity. Hence $ o(z) = 2 $ and $ \langle r , z \rangle $ is a Klein four-subgroup of $ H$. Also $ zr \in R \subseteq r^{G} $ and so $ zr $ is conjugate to $ z $ in $ G $. Thus there exists $ g \in G $ such that $ zr = g^{-1} z g $. Hence $ zr \in Z( g^{-1} H g ) $.

\medskip
Note that $ r^{-1} (zr) r \in Z(r^{-1} ( g^{-1} H g) r ) $ and $ r^{-1} (zr) r = rzrr = rz = zr \in Z( g^{-1} H g ) $. Therefore $ r^{-1} ( g^{-1} H g) r = g^{-1} H g $ by Corollary \ref{cntrcnjgt}. We are assuming $ H $ not a normal subgroup of $ G $, which means that $ g^{-1} H g $ is not normal either. So it follows from Lemma \ref{mxmlnrmlizr} that $ r \in g^{-1} H g $.

\medskip
The rotational subgroup of $ g^{-1} H g $ is cyclic, which means it contains a unique subgroup $ \langle zr \rangle $ of order $ 2 $. Thus $ r $ is a reflection in $ g^{-1} H g $. Now $ \langle r , zr \rangle = \langle r , z \rangle \leq g^{-1} H g $. Hence $ \langle r , z \rangle \leq  H \cap g^{-1} H g $ and so $ | H \cap g^{-1} H g  | \geq 4 $. The subgroups $ H $ and $ g^{-1} H g $ stabilize the points $ H , Hg \in \Omega $, so $ | H \cap g^{-1} H g  | = 4 $ by Lemma \ref{klnfrntrsctn}.

\medskip
Finally $ g^{-1} H g  $ is the only subgroup of $ G$ that satisfies conditions $ (1) $ to $ (4 ) $ by Lemma \ref{asoidspdoc}.
\end{proof}
\begin{lema}\label{wqpodcisok}
Let $ R $ be the set of reflections of $ H $. Let $ r \in R $ and $ z \in Z(H) $. Let $ \Delta $ be the set of subgroups of $ G $ conjugate to $ H $ such that if $ H^{ \prime } \in \Delta $ and $ R^{ \prime } $ is the set of reflections of $ H^{ \prime } $ then $ R \cap R^{ \prime }  \neq \emptyset $. Suppose the followng;
\begin{itemize}
\item $ | R | $ is even,
\item $ R \subseteq r^{G} $,
\item $ z \in r^{G} $,
\item $ Ht(G, \Omega ) \neq 3 $.
\end{itemize}
Then the number of distinct reflections in subgroups in $ \Delta $ is at least
\[
 \Bigg( \frac{|G| | H |  }{ 4 | r^{G} | } - |H| \Bigg) \Bigg( \frac{| H |}{2}  - 1 \Bigg) + \frac{ | H | }{2} 
\]
\end{lema}
\begin{proof}
Let $ r_{1} \in R $. Lemma \ref{sctnszfrcntr} shows that there is a unique subgroup $ H_{1} \in \Delta \setminus \{ H \} $ such that $ r_{1} $ is a reflection in $ H_{1} $ and $ | H_{1} \cap H | = 4 $.

\medskip
Suppose $ r_{2} \in R $ and $ r_{2} \neq r_{1} $. Let $ H_{2} \in \Delta \setminus \{ H \} $ be the unique subgroup such that $ r_{2} $ is a reflection in $ H _{2} $ and $ | H_{2} \cap H | = 4 $. Observe $ r_{2} $ is not a reflection in $ H_{1} $, otherwise $ H_{1} $ would contain the non-trivial rotation $ r_{1} r_{2} $ which implies $ H_{1} = H $ by Lemma \ref{ntnrmlmxml}. Therefore $ H_{1} \neq H_{2} $.

\medskip
This means there are $ | R | $ subgroups in $ \Delta \setminus \{ H \}  $ that intersect $ H $ in a subgroup of order $ 4 $. Let $ \Delta ^{ \prime } $ be the set of such subgroups. The fact that $  R \subseteq r^{G} $ means $ |  r^{G} \cap R | = | R | = \tfrac{1}{2} | H | $. So using Lemma \ref{ttlcnjgtsbrflc} we have
\begin{align*}
| \Delta \setminus ( \Delta ^{ \prime } \cup \{ H \} ) | = | \Delta | - | \Delta ^{ \prime } | - 1 = | \Delta | - | R | - 1
= \frac{|G| | H |  }{ 4 | r^{G} | } - |H|
\end{align*}
The subgroups in $ \Delta $ are stabilizers of some points in $ \Omega $, so Lemma \ref{klnfrntrsctn} shows that the subgroups in $  \Delta \setminus ( \Delta ^{ \prime } \cup \{ H \} )  $ intersect $ H $ in a subgroup of order $ 2 $, generated by the reflection each subgroup shares with $ H $. The number of reflections in subgroups in $ | \Delta \setminus ( \Delta ^{ \prime } \cup \{ H \} ) | $ can now be counted.

\medskip
Suppose $ H_{3} , H_{4} \in  \Delta \setminus ( \Delta ^{ \prime } \cup \{ H \} ) $ with $ H_{3} \neq H_{4} $. Let $ R_{3} $ and $ R_{4} $ be the set of reflections of $ H_{3} $ and $H_{4} $ respectively. We have $ | R_{3} \cap R | = | R_{4} \cap R | = 1 $ by Lemma \ref{intrscttpssbl}.

\medskip
We show that each of the $ | H | / 2 - 1 $ reflections in $ R_{3} \setminus R_{3} \cap R $ are distinct from each of the $ | H | / 2 -1 $ reflections in $ R_{4} \setminus R_{4} \cap R $. This is split into two cases.

\medskip
For the first case suppose $ R_{3} \cap R  =  R_{4} \cap R  $. Then $ R_{3} \cap R \subseteq R_{3} \cap R_{4} $ and $ | R_{3} \cap R_{4} | \geq 1 $. It follows from Lemma \ref{intrscttpssbl} that $ | R_{3} \cap R_{4} | = 1 $. Hence $ R_{3} \cap R_{4} =  R_{3} \cap R $. So the $ \tfrac{1}{2} | H | - 1 $ elements of $ R_{3} \setminus  R_{3} \cap R $ are distinct from the $ \tfrac{1}{2} | H | - 1 $ elements of $ R_{4} \setminus  R_{4} \cap R $.

\medskip
For the second case suppose $ R_{3} \cap R  \neq  R_{4} \cap R  $. Lemma \ref{intrscttpssbl} shows $ | R_{3} \cap R_{4} | \leq 1 $.

\medskip
Lemma \ref{klnfrntrsctn} shows that $ | H_{3} \cap H  | = | H_{4} \cap H  | = 2 $ since $ | R_{3} \cap R | = | R_{4} \cap R | = 1 $ and $ | H_{3} \cap H  | \neq 4 \neq | H_{4} \cap H  | $. The fact $ R_{3} \cap R  \neq  R_{4} \cap R  $ implies $ H_{3} \cap H  \neq H_{4} \cap H $.

\medskip
If $ | R_{3} \cap R_{4} | = 1 $ then $ | H_{3} \cap H_{4} | \geq 2 $ and it follows that $ Ht(G, \Omega ) =3 $ by Lemma \ref{htthrcndtns}, which is a contradiction. Thus $ | R_{3} \cap R_{4} | \neq 1 $.

\medskip
It must be that $ | R_{3} \cap R_{4} | = 0 $, showing that the reflections in $ R_{3} $ are distinct from those in $ R_{4} $. In particular the $ \tfrac{1}{2} | H | - 1 $ reflections in $ R_{3} \setminus ( R_{3} \cap R ) $ are distinct from the $ \tfrac{1}{2} | H | - 1 $ reflections in $ R_{4} \setminus ( R_{4} \cap R ) $.

\medskip
The total number of reflections in subgroups in $ \Delta \setminus ( \Delta ^{ \prime } \cup \{ H \} )  $ excluding those in $ R $ are
\[
| \Delta \setminus ( \Delta ^{ \prime } \cup \{ H \} ) | \Bigg( \frac{| H |}{2}  - 1 \Bigg) 
= \Bigg( \frac{|G| | H |  }{ 4 | r^{G} | } - |H| \Bigg) \Bigg( \frac{| H |}{2}  - 1 \Bigg)  .
\]
The $ \tfrac{1}{2} | H | $ in $ R $ have not been counted, so including these gives at least
\[
 \Bigg( \frac{|G| | H |  }{ 4 | r^{G} | } - |H| \Bigg) \Bigg( \frac{| H |}{2}  - 1 \Bigg) + \frac{ | H | }{2}   
\]
reflections in subgroups in $ \Delta $.
\end{proof}
\begin{exmp}\label{pslqdddvfr}
Let $ G = PSL_{2} (q) $ where $ q $ is odd and $ q \geq 11 $. Then $ | G | = \tfrac{1}{2} q (q+1) (q-1) $.

\medskip
Let $ H $ be a maximal $ D_{ (q-1) } $ or $ D_{ (q+1) } $ and let $ \Omega $ be the set of right cosets of $ H$. Furthermore let $ R $ be the reflections of $ H $ and suppose $ | R | $ is even.

\medskip
Let $ r \in R$. By Lemma \ref{pslnvltnscnjg} all involutions in $ G $ are conjugate to $ r $. Therefore $   r^{G} \cap R = R $ and so $ | r^{G} \cap R | = \tfrac{1}{2} | H | $. The rotational subgroup of $ H$ must have even order, which means it contains a unique involution $ z \in Z(H) $ and this must lie in $  r^{G} $.

\medskip
Let $ \Delta $ be the set of subgroups of $ G $ conjugate to $ H $ such that if $ H^{ \prime } \in \Delta $ and $ R^{ \prime } $ is the set of reflections of $ H^{ \prime } $ then $ R \cap R^{ \prime } \neq \emptyset $.

\medskip
From here two cases will be looked at; when $ q \equiv 1 \pmod{4} $ and when $ q \equiv 3 \pmod{4} $.

\medskip
First suppose $ q \equiv 1 \pmod{4} $. Then $ q \geq 13 $. Lemma \ref{pslnvltnscnjg} shows that $ |  r^{G}|  = \tfrac{1}{2} q ( q+ 1) $. If $ H \cong D_{q+1 } $ then $ | R | = \tfrac{1}{2} | H |  $ is odd, which contradicts our earlier assumption. So $ H \cong D_{q-1 } $ and $ | H | = q - 1 $.

\medskip
Let $ n_{1} $ be the number of distinct reflections in subgroups in $ \Delta $. These reflections lie in $  r^{G} $ so $ n_{1} \leq |  r^{G}| $ and $ n_{1} -   |  r^{G}| \leq 0 $. If $ Ht(G, \Omega ) \neq 3 $ then Lemma \ref{wqpodcisok} shows that
\begin{align*}
n_{1} -   |  r^{G}|
& \geq 
 \Bigg( \frac{|G| | H |  }{ 4 | r^{G} | } - |H| \Bigg) \Bigg( \frac{| H |}{2}  - 1 \Bigg) + \frac{ | H | }{2} -  |  r^{G}| \\
& =  \Bigg( \frac{ \tfrac{1}{2} q (q+1) (q-1)^{2} }{ 2 q ( q+ 1)  } - (q-1) \Bigg) \Bigg( \frac{q-1}{2}   - 1 \Bigg) + \frac{ q-1 }{2}  -  \frac{q ( q+ 1) }{2} \\
& =  \Bigg( \frac{  (q-1)^{2} }{ 4  } - (q-1) \Bigg) \frac{q-3}{2}    -  \frac{q ^{2} + 1 }{2}  \\
& =  \frac{1}{8} (q-1) (   q-1  - 4 ) (q-3)    -  \frac{q ^{2} + 1 }{2} \\
& =  \frac{1}{8} (q-1) (   q-5 ) (q-3)    -  \frac{q ^{2} + 1 }{2} \\
& =  \frac{1}{8} (q^{3} - 9q^{2} + 23q - 15)    -  \frac{q ^{2} + 1 }{2}  \\
& =  \frac{1}{8} (q^{3} - 13q^{2} + 23q - 19)     \\
& =  \frac{1}{8} q^{2}(q - 13 ) + \frac{1}{8} (23q - 19)     \\
& > 0. \tag{Since $ q \geq 13 $}
\end{align*}
This is a contradiction. Thus $ Ht(G, \Omega ) = 3 $ in this case.

\medskip
For the second case suppose  $ q \equiv 3 \pmod{4} $. Then Lemma \ref{pslnvltnscnjg} shows that $ |  r^{G}|  = \tfrac{1}{2} q ( q- 1) $. If $ H \cong D_{q-1 } $ then $ | R | = \tfrac{1}{2} | H |  $ is odd, which contradicts our earlier assumption. So $ H \cong D_{q+1 } $ and $ | H | = q + 1 $.

\medskip
Let $ n_{2} $ be the number of distinct reflections in subgroups in $ \Delta $. These reflections lie in $  r^{G} $ so $ n_{2} \leq |  r^{G}| $ and $ n_{2} -   |  r^{G}| \leq 0 $. If $ Ht(G, \Omega ) \neq 3 $ then Lemma \ref{wqpodcisok} shows that
\begin{align*}
n_{2} -   |  r^{G}|
& \geq 
 \Bigg( \frac{|G| | H |  }{ 4 | r^{G} | } - |H| \Bigg) \Bigg( \frac{| H |}{2}  - 1 \Bigg) + \frac{ | H | }{2} -  |  r^{G}|  \\
 & =
 \Bigg( \frac{\tfrac{1}{2} q (q+1) ^{2} (q-1) }{ 2 q ( q- 1) } - ( q + 1 )  \Bigg) \Bigg( \frac{ q + 1 }{2}  - 1 \Bigg) +  \frac{  q + 1  }{2} -  \frac{ q ( q- 1) }{2}    \\
  & =
 \Bigg( \frac{  (q+1) ^{2}  }{4  } - ( q + 1 )  \Bigg)  \frac{ q - 1 }{2} + \frac{  1 }{2} (2q - q^{2} +1) \\
   & =
\frac{1}{8} ( q + 1 )  (   q+1   - 4  ) (q - 1) + \frac{  1 }{2} (2q - q^{2} +1)   \\
   & =
\frac{1}{8}  (   q-3  ) (q^{2} - 1) + \frac{  1 }{2} (2q - q^{2} +1)  \\
& =
\frac{1}{8}  (   q^{3}  - 3 q^{2} - q +3 ) + \frac{  1 }{2} (2q - q^{2} +1)  \\
& =
\frac{1}{8}  (   q^{3} -7 q^{2} +7 q +7 ) \\
& =
\frac{1}{8} q^{2} (   q - 7)  + \frac{7}{8} q + \frac{7}{8} \\
& > 0. \tag{Since $ q \geq 11 $}
\end{align*}
Again we have a contradiction, so $ Ht(G, \Omega ) = 3 $ in this case as well.
\end{exmp}
\begin{lema}\label{mxmpgldhdral}
Suppose $ q $ is odd with $ q \geq 13 $. Put $ G := PGL_{2} (q) $. Let $ H < G $ be a maximal $ D_{2 (q-1) } $ and let $ S \leq G $ where $ S \cong PSL_{2} (q) $. Then $ H \cap S $ is maximal in $ S $ and $ H \cap S \cong D_{ (q-1)} $.
\end{lema}
\begin{proof}
Lemma \ref{mxmlsbpslntr} tells us $ H \cap S $ is maximal in $ S $ and $ | H \cap S | = \tfrac{1}{2} | H | = q - 1 $. As $ H \cap S $ is a subgroup of a dihedral group $ H $, it is either dihedral or cyclic. We see from Table \ref{tablethree} that there are no maximal cyclic subgroups of order $ q-1 $ in $ PSL _{2} (q) $. Thus $ H \cap S \cong D_{ (q-1) } $.
\end{proof}
\begin{exmp}\label{fnlxmpatlst}
Let $ G = PGL_{2} (q) $ where $ q $ is odd and $ q \geq 13 $. Suppose $ H \leq G $ is a maximal $ D_{ 2(q-1) } $ and let $ \Omega $ be the set of right cosets of $ H$.

\medskip
By Lemma \ref{pslsnpgl} there exists $ S \leq G $ with $ S \cong PSL_{2} (q) $. Lemma \ref{mxmpgldhdral} shows that $ H{^ \prime} := H \cap S \cong D_{(q-1)}  $ and is maximal in $ S $. Let $ \Omega ^{ \prime } $ be the set of right cosets of $ H^{ \prime } $ in $ S $.

\medskip
We see from Examples \ref{psfljdshu} and \ref{pslqdddvfr} that $ Ht( S, \Omega ^{ \prime } ) = 3 $. Therefore there exists an independent set $ \Delta ^{ \prime } := \{ \delta_{1} ^{ \prime } , \delta_{2} ^{ \prime } , \delta_{3} ^{ \prime } \} \subseteq \Omega ^{ \prime }  $. We see from Lemma \ref{scndfnt} that for each $ i , j  \in \{ 1, 2,3 \} $ with $ i \neq j $, there exists some $ g_{i,j} \in S_{  \delta _{i} ^{ \prime}   } \cap S_{  \delta _{j} ^{ \prime}   } $ such that $  g_{i,j} \notin S_{ ( \Delta ^{ \prime} ) } $.

\medskip
For each $ k \in \{1, 2 , 3 \} $, the stabilizers $ S_{  \delta _{k} ^{ \prime}   } = h_{k} ^{-1} H ^{ \prime } h_{k} $ for some $ h_{k} \in S $. So $ S_{  \delta _{k} ^{ \prime }   } = h_{k} ^{-1} H h_{k} \cap S $.

\medskip
Observe that for each $ k \in \{ 1, 2 , 3 \} $ the subgroup $ h_{k} ^{-1} H h_{k} $ is a stabilizer of some point $ \delta _{k} \in \Omega $. For each $ i , j  \in \{ 1, 2,3 \} $ with $ i \neq j $ we have $ g_{i,j} \in H_{i} \cap H_{j} $ and $ g_{i,j} \notin H_{1} \cap H_{2} \cap H_{3} $. Hence $ H_{1} $, $ H_{2} $ and $ H_{3} $ are distinct from each other and it follows that $ \delta_{1} $, $ \delta_{2} $ and $ \delta_{3} $ are distinct as well.

\medskip
By Lemma \ref{scndfnt} the set  $ \{ \delta_{1} , \delta_{2} , \delta _{3} \} $ is independent. Therefore $ Ht(G, \Omega ) \geq 3 $. Using Corollary \ref{trshght} we get $ Ht(G, \Omega ) = 3 $.
\end{exmp}
\begin{thrm}
Let $ G $ be $ PGL_{2} (q) $ or $ PSL_{2} (q) $ acting on the cosets $ \Omega $ of a maximal $ D_{2(q \pm 1) / \delta } $ subgroup, where $ \delta = 1 $ if $ G = PGL_{2} (q) $ and $ \delta = 2 $ if $ G = PSL_{2} (q) $. If $ G $ is $ PSL_{2} (4) $ or $ PSL_{2} (5) $ acting on a maximal $ D_{6} $ then $ Ht (G , \Omega ) = 2 $. Otherwise $ Ht (G , \Omega ) = 3 $.
\end{thrm}
\begin{proof}
If $ G $ is $ PSL_{2} (4) $ or $ PSL_{2} (5) $ acting on a maximal $ D_{6} $, the tables in \cite{WISCONS1} show $ Ht(G, \Omega ) = 2 $. The same tables show $ Ht(G, \Omega ) = 3 $ for $ PGL_{2} (q) $ acting on $ D_{ 2(q-1) } $ when $ q \in \{ 7 , 9 , 11 \} $. 

\medskip
All other cases are covered in Examples \ref{xmpcrect}, \ref{expslqevn}, \ref{psfljdshu}, \ref{mssngxmp}, \ref{pslqdddvfr} and \ref{fnlxmpatlst}.
\end{proof}
\section{Relational Complexity of the Dihedral Actions}
Throughout this section let $ G $ be a finite group that is either simple or whose proper normal subgroups are all maximal. Also suppose $ G $ contains a maximal subgroup $ H $ that is a dihedral group of order at least $ 6 $. Also suppose $ H \ntriangleleft G $. Let $ \Omega $ be the set of right cosets of $ H$.

\medskip
Corollary \ref{trshght} shows $ Ht(G, \Omega) \leq 3 $, so it follows that $ RC(G, \Omega ) \leq 4 $ by Theorem \ref{rchgt}. The aim is to show $ RC(G, \Omega ) = 3 $.

\medskip
Let $ I , J \in \Omega^{4} $, where $ I = ( I_{1} , \dots , I_{4} ) $ and $ J = ( J_{1} , \dots , J_{4} ) $. By Corollary \ref{rcshsbtplcmp} is is sufficient to show that if $ I \sim _{3} J $ then $ I \sim _{4} J $. So suppose $ I \sim _{3} J $.

\medskip
If any of the entries of $ I $ are repeated then $ I \sim _{4} J $ by Lemma \ref{rptdntrs} and there is nothing more to show. So suppose none of the entries of $I$ are repeated. By Lemma \ref{frsnrtytql} we may assume that $ I _{i} = J _{i} $ for all $ i \in \{ 1, 2 , 3 \} $.

\medskip
Let $ X = \{ I_{1} , \dots , I _{4} \} $, the set of entries of $ I$. If any proper subset of $ X $ is not independent, then $ I \sim _{4} J $ by Lemma \ref{ntndpndntntrs}. First we deal with a case where this happens.
\begin{lema}\label{rttnordrtwndpdn}
Suppose $ H$ does not contain a rotation of order $ 2 $. Then there exists $ \Gamma \subset X $ such that $ \Gamma $ is not independent.
\end{lema}
\begin{proof}
Suppose for a contradiction that every proper subset of $ X $ is independent. By Lemma \ref{klnfrntrsctn} we have $ | G_{ I _{j} , I _{k} } | \neq 4 $ for all $ j , k \in \{ 1, \dots , 4 \} $ with $ j \neq k $. Lemmas \ref{stach}, \ref{klnfrntrsctn} and \ref{trsndstszthr} together show that $ | G_{ I _{j} , I _{k} } | = 2 $. So these stabilizers can be written as
\begin{align*}
& G_{ I _{1} , I _{2} } = \{ 1_{G}   , \ r_{1,2}  \} \\
& G_{ I _{1} , I _{3} } = \{ 1_{G}   , \ r_{1,3} \} , \\
& G_{ I _{1} , I _{4} } = \{ 1_{G}   , \ r_{1,4} \} .
\end{align*}
Since $ \{ I _{1} , I _{2} , I_{3}  \} $ is an independent subset of $ X$, we have $ | G_{ I _{1} , I _{2} } \cap G_{ I _{1} , I _{3} }  | = | G_{ I _{1} , I _{2} , I_{3} } | = 1 $ by Lemma \ref{trsndstszthr}. By the same reasoning, any pair of the above stabilizers have trivial intersection. So each of their non-identity elements are distinct in $ G $.

\medskip
Now consider the restriction of the action of $ G$ to the subgroup $ G _{I_{1} } $. Recall the notation $ \text{Send} _{ G_{I_{1} } } ( I_{4} , J_{4} ) := \{ g \in G_{I_{1}} : I_{4} ^{g} = J_{4} \} $. Since $ I \sim _{3} J $, there exists $ h_{1} \in G $ such that
\[
(I_{1} , I_{2} , I_{4} ) ^{h_{1} } = (J_{1} , J_{2} , J_{4} ) = (I_{1} , I_{2} , J_{4} ) .
\]
Hence $ h_{1} \in G_{ I _{1} , I _{2} } \cap \text{Send} _{ G_{I_{1} } } ( I_{4} , J_{4} ) $. Since $ J_{4} \neq I_{4} $, it must be that $ h_{1} =  r_{1,2} $. Similarly $ r_{1,3} \in G_{ I _{1} , I _{3} } \cap \text{Send} _{ G_{I_{1} } } ( I_{4} , J_{4} ) $. By Lemma \ref{sndcst} we have
\[
\text{Send} _{ G_{I_{1} } } ( I_{4} , J_{4} ) 
= (G_{ I _{1} , I _{4} } ) r_{1,3} 
= \{ r_{1,3} , \  r_{1,4} r_{1,3} \} .
\]
So $ r_{1,2} =  r_{1,4} r_{1,3} $. However $ r_{1,4} $ and $ r_{1,3} $ are reflections in $ G_{I_{1} } $. This means that $ r_{1,2} $ must be a rotation of order $ 2 $, which is a contradiction.
\end{proof}
The above lemma and Lemma \ref{ntndpndntntrs} show that $ I \sim _{4} J $ when $ H $ does not contain a rotation of order $ 2 $. So assume from now that $ H $ contains a rotation of order $ 2 $ (so that each  stabilizer of a point in $ \Omega $ also does) and that every proper subset of $ X $ is independent.

\medskip
The set $ \{ I_{1} , I_{2} , I_{3} \} $ is then independent. Note that $ | G_{ I _{1} , I _{2} , I _{3} } | = 1 $ by Lemma \ref{trsndstszthr}. So $ | G_{ I _{1} , I _{i} } | \in \{ 2, 4 \} $ for each $ i \in \{ 2,3,4 \} $ by Lemmas \ref{stach} and \ref{klnfrntrsctn}. Both possible values of $ | G_{ I _{1} , I _{i} } | $ are now examined.
\begin{lema}\label{szfrntrstsctn}
Suppose $  | G_{ I _{1} , I _{2} } |  = 4 $. Then $ J_{4} = I_{4} $.
\end{lema}
\begin{proof}
Suppose $ J_{4} \neq I_{4} $. Since $ \{I_{1} , I_{3} , I_{4} \} $ and $ \{I_{2} , I_{3} , I_{4} \} $ are independent subsets of $ X $, it follows from Lemma \ref{szfrntrsctn} that $ | G_{ I _{1} , I _{3} } | = | G_{ I _{1} , I _{4} } | = | G_{ I _{2} , I _{3} } | = | G_{ I _{2} , I _{4} } | = 2 $. So these subgroups can be written as
\begin{align*}
& G_{ I _{1} , I _{2} } = \{ 1_{G}   , \ r_{1,2}  , \ s_{1,2}  , \ t_{1,2} \} \\
& G_{ I _{1} , I _{3} } = \{ 1_{G}   , \ r_{1,3} \} , \\
& G_{ I _{1} , I _{4} } = \{ 1_{G}   , \ r_{1,4} \} , \\
& G_{ I _{2} , I _{3} } = \{ 1_{G}  , \ r_{2,3} \} , \\
& G_{ I _{2} , I _{4} } = \{ 1_{G}   , \ r_{2,4} \} .
\end{align*}
Lemma \ref{klnfrntrsctn} shows that $ G_{ I _{1} , I _{2} } $ contains the rotation of order $ 2 $ and two reflections from $ G_{ I _{1} } $. So suppose $ r_{1,2} $ is the rotation in $ G_{ I _{1} } $, with $ s_{1,2} $ and $ t_{1,2} $ being reflections.

\medskip
Similarly  $  G_{ I _{1} , I _{2} } $ contains the rotation of order $ 2 $ from $ G_{ I _{2} } $. If $  r_{1,2} $ is the rotation in $ G_{ I _{2} } $ then by Lemma \ref{dhdrlrottnnrml} the subgroup $ \langle r_{1,2} \rangle $  would be normal in $ G_{ I _{1} } $ and $ G_{ I _{2} } $, which are distinct subgroups of $ G$ because $ \{ I _{1} , I_{2} \} $ is an independent subset of $ X $. But this contradicts Lemma \ref{ntnrmlmxml}, so $ r_{1,2} $ is not a rotation in $ G_{ I _{2} } $. Therefore it can be assumed that $ s_{1,2} $ is the rotation of order $ 2 $ in $ G_{ I _{2} } $.

\medskip
Now consider the restriction of the action of $ G$ to the subgroup $ G _{I_{1} } $. Since $ I \sim _{3} J $, there exists $ h_{1} \in G $ such that
\[
(I_{1} , I_{3} , I_{4} ) ^{h_{1} } = (J_{1} , J_{3} , J_{4} ) = (I_{1} , I_{3} , J_{4} ) .
\]
Hence $ h_{1} \in G_{ I _{1} , I _{3} } \cap \text{Send} _{ G_{I_{1} } } ( I_{4} , J_{4} ) $. Since $ J_{4} \neq I_{4} $, it must be that $ h_{1} =  r_{1,3} $. By Lemma \ref{sndcst} we have
\[
\text{Send} _{ G_{I_{1} } } ( I_{4} , J_{4} ) = (G_{ I _{1} , I _{4} } ) r_{1,3} = \{ r_{1,3} , \  r_{1,4} r_{1,3} \} . \tag{*}
\]
Note that $ r_{1,4} \notin G_{ I _{1} , I _{2} , I_{4} } $ because $ \{ I _{1} , I _{2} , I_{4}  \} $ is an independent subset of $ X$ and $ | G_{ I _{1} , I _{2} , I_{4} } | = 1 $ by Lemma \ref{trsndstszthr}. In particular $ r_{1,4} \notin G_{ I _{1} , I _{2} } $ and so $ r_{1,4} \neq r_{1,2} $. Using similar reasoning, $ r_{1,3} \notin G_{ I _{1} , I _{2} } $ and $ r_{1,3} \neq r_{1,2} $. It can also be shown that $ r_{2,3} \neq r_{1,2} $. Therefore $ r_{1,3} $ and $ r_{1,4} $ are reflections in $ G_{ I_{1} } $. Hence $ r_{1,4} r_{1,3} $ is a rotation.

\medskip
Similarly there exists $ h_{2} \in G $ such that
\[
(I_{1} , I_{2} , I_{4} ) ^{h_{2} } = (J_{1} , J_{2} , J_{4} ) = (I_{1} , I_{2} , J_{4} ) .
\]
So $ h_{2} \in G_{ I _{1} , I _{2} } \cap \text{Send} _{ G_{I_{1} } } ( I_{4} , J_{4} ) $. Again from Lemma \ref{sndcst} we have
\[
\text{Send} _{ G_{I_{1} } } ( I_{4} , J_{4} ) = (G_{ I _{1} , I _{4} } ) h_{2} = \{ h_{2} , \  r_{1,4} h_{2} \} .
\]
Comparing with (*) and using the fact $ r_{1,3} \notin G_{ I _{1} , I _{2} } $, we have $  r_{1,4} r_{1,3} = h_{2} \in G_{ I _{1} , I _{2} } $. Since this is a rotation of order $2$ in $ G_{ I _{1} } $, it follows that $  r_{1,4} r_{1,3} = r_{1,2} $. Also $ r_{1,2} \in \text{Send} _{ G_{I_{2} } } ( I_{4} , J_{4} ) $ because $ r_{1,2} \in G_{I_{2} }  $.

\medskip
Following the same reasoning above replacing $ I_{1} $ with $ I_{2} $ where needed, it can be shown that
\[
\text{Send} _{ G_{I_{2} } } ( I_{4} , J_{4} ) = (G_{ I _{2} , I _{4} } ) r_{2,3} = \{ r_{2,3} , \  s_{1,2} \} .
\]
But then $ r_{1,2} \notin \text{Send} _{ G_{I_{2} } } ( I_{4} , J_{4} ) $, which is a contradiction. So the assumption that $ J_{4} \neq I_{4} $ must be wrong.
\end{proof}
\begin{corl}\label{crlstblzrszfr}
Suppose $  | G_{ I _{i} , I _{j} } |  = 4 $ for some $ i , j \in \{ 1 , \dots 4 \} $ with $ i \neq j $. Then $ J_{4} = I_{4} $.
\end{corl}
\begin{proof}
First suppose $ i \neq 4 \neq j $. The entries of $ I$ can be reordered to give a $4$-tuple $ I^{ \prime } = ( I _{i} , I _{j} , \dots , I_{4} ) $ and the entries of $ J$ reordered correspondingly to give $ J^{ \prime } = ( J _{i} , J _{j} , \dots , J_{4} ) $. Also $ I^{ \prime } \sim _{3} J^{ \prime }  $ by Lemma \ref{rodr} and $ X $ is the set of entries of $ I^{ \prime } $, same as $I$. So Lemma \ref{szfrntrstsctn} can be applied to $ I^{ \prime } $ and $ J^{ \prime } $ to show that $ I _{4} = J _{4} $.

\medskip
Next suppose $ i = 4 $ and suppose without loss of generality that $ j = 1 $. Again the entries of $ I$ can be reordered to give a $4$-tuple $ I^{ \prime \prime } = ( I _{1} , I _{4} , I_{3} , I_{2} ) $ and the entries of $ J$ reordered correspondingly to give $ J^{ \prime \prime } = ( J _{1} , J _{4} , J_{3} , J_{2} ) = ( I _{1} , J _{4} , I_{3} , I_{2} ) $. By Lemma \ref{frsnrtytql} there exists $ g \in G $ such that $ I^{ \prime \prime } \sim_{3}  ( J^{ \prime \prime } )^{g} $ and $ I_{t} = J_{t} ^{g} $ for each $ t \in \{ 1 , 2 , 3 \} $. Note that $ g \in G_{I_{1} , I_{3} } $. Again $ X $ is the set of entries of $ I^{ \prime \prime } $ Lemma \ref{szfrntrstsctn} can be applied to give $ I_{2} = J_{2} ^{g} = I_{2} ^{g} $. Hence $ g \in G_{I_{1} , I_{2} , I_{3} }  $. Since $ \{ I_{1} , I_{2} , I_{3} \} $ is an independent subset of $ X $, it follows from Lemma \ref{trsndstszthr} that $ g = 1_{G} $. Thus $ I_{4} = J_{4} ^{g} = J_{4} $.
\end{proof}
\begin{lema}\label{llhvsztw}
Suppose $  | G_{ I _{1} , I _{i} } |  = 2 $ for all $ i \in \{ 2,3,4 \} $. Then $ J_{4} = I_{4} $.
\end{lema}
\begin{proof}
Suppose $ I_{4} \neq J_{4} $. By Corollary \ref{crlstblzrszfr} we have $ | G_{ I _{j} , I _{k} } | \neq 4 $ for all $ j , k \in \{ 1, \dots , 4 \} $ with $ j \neq k $. Since proper subsets of $ X$ are independent, Lemmas \ref{stach}, \ref{klnfrntrsctn} and \ref{trsndstszthr} together show that $ | G_{ I _{j} , I _{k} } | = 2 $. So these stabilizers can be written as
\begin{align*}
& G_{ I _{1} , I _{2} } = \{ 1_{G}   , \ r_{1,2}  \} \\
& G_{ I _{1} , I _{3} } = \{ 1_{G}   , \ r_{1,3} \} , \\
& G_{ I _{1} , I _{4} } = \{ 1_{G}   , \ r_{1,4} \} , \\
& G_{ I _{2} , I _{3} } = \{ 1_{G}  , \ r_{2,3} \} , \\
& G_{ I _{2} , I _{4} } = \{ 1_{G}   , \ r_{2,4} \} \\
& G_{ I _{3} , I _{4} } = \{ 1_{G}  , \ r_{3,4} \} .
\end{align*}
Since $ \{ I _{1} , I _{2} , I_{3}  \} $ is an independent subset of $ X$, we have $ | G_{ I _{1} , I _{2} } \cap G_{ I _{1} , I _{3} }  | = | G_{ I _{1} , I _{2} , I_{3} } | = 1 $ by Lemma \ref{trsndstszthr}. By the same reasoning, any pair of the above stabilizers have trivial intersection. So each of their non-identity elements are distinct in $ G $.

\medskip
Now consider the restriction of the action of $ G$ to the subgroup $ G _{I_{1} } $. Since $ I \sim _{3} J $, there exists $ h_{1} \in G $ such that
\[
(I_{1} , I_{2} , I_{4} ) ^{h_{1} } = (J_{1} , J_{2} , J_{4} ) = (I_{1} , I_{2} , J_{4} ) .
\]
Hence $ h_{1} \in G_{ I _{1} , I _{2} } \cap \text{Send} _{ G_{I_{1} } } ( I_{4} , J_{4} ) $. Since $ J_{4} \neq I_{4} $, it must be that $ h_{1} =  r_{1,2} $. Similarly $ r_{1,3} \in G_{ I _{1} , I _{3} } \cap \text{Send} _{ G_{I_{1} } } ( I_{4} , J_{4} ) $. By Lemma \ref{sndcst} we have
\[
\text{Send} _{ G_{I_{1} } } ( I_{4} , J_{4} ) 
= (G_{ I _{1} , I _{4} } ) r_{1,3} 
= \{ r_{1,3} , \  r_{1,4} r_{1,3} \} .
\]
Hence $ r_{1,2} =  r_{1,4} r_{1,3} $. The elements $ r_{1,2} , \ r_{1,3} $ and $ r_{1,4} $ must either be reflections or the rotation of order $ 2 $ in $ G_{I_{1} } $. Since only at most one rotation of order $ 2 $ can exist, at least two of these elements are reflections. As the product of any two of these elements is equal to the third and composing two reflections with each other gives a rotation, one of the elements must be the rotation of order $2$. Hence these elements generate a Klein four-group;
\[
K_{1} := \{ 1_{G} , \ r_{1,2} , \ r_{1,3} , \ r_{1,4} \} \leq G_{ I_{1} } .
\]
Similar reasoning shows we can find two more Klein four-groups;
\begin{align*}
& K_{2} := \{ 1_{G} , \ r_{1,2} , \ r_{2,3} , \ r_{2,4} \} \leq G_{ I_{2} } , \\
& K_{3} := \{ 1_{G} , \ r_{1,3} , \ r_{2,3} , \ r_{3,4} \} \leq G_{ I_{3} } .
\end{align*}
The subgroup $ H $ is maximal and not normal in $ G $, so the same applies to $ G_{ I_{1} } $ because it is conjugate to $ H $ in $ G $. Therefore $ N_{G} ( G_{ I_{1} } ) = G_{ I_{1} } $ by Lemma \ref{mxmlnrmlizr}. As $ | G_{ I _{1} , I _{2} , I_{3} } | = 1 $ it must be that $ \ r_{2,3} \notin G_{ I_{1} } $. Hence $ r_{2,3}  G_{ I _{1} } r_{2,3}  \neq G_{ I _{1} } $.

\medskip
In $K_{2} $ and $ K_{3} $ it can be seen that $ r_{2,3} r_{1,2} r_{2,3} = r_{1,2} $ and $ r_{2,3} r_{1,3} r_{2,3} = r_{1,3} $. Therefore $ r_{2,3} K_{1} r_{2,3} = K_{1} $. It follows that $ r_{2,3} r_{1,4} r_{2,3} = r_{1,4}$.

\medskip
Now one of the non-identity elements $ k_{1} \in K_{1} $ is the rotation of order $ 2 $ in $ G_{I_{1} } $ and since $ r_{2,3} $ fixes this element under conjugation, it must be that $ k_{1} $ is the rotation of order $ 2 $ in $ r_{2,3}  G_{ I _{1} } r_{2,3} $. Hence $ \langle k_{1} \rangle $ is normal in both $  G_{ I _{1} }  $ and $ r_{2,3}  G_{ I _{1} } r_{2,3} $ by Lemma \ref{dhdrlrottnnrml}. This contradicts Lemma \ref{ntnrmlmxml}. Thus the assumption that $ I_{4} \neq J_{4} $ must be incorrect.
\end{proof}
Collecting together these results gives the main theorem of this subsection, an upper bound on the relational complexity of the action on $ \Omega $.
\begin{thrm}\label{mxmldhedrlrc}
Let $ G $ be a group that is either simple or all non-trivial normal subgroups are maximal. Let $ H $ be a maximal, non-normal subgroup that is dihedral. Let $ \Omega $ be the set of right cosets of $ H$. Then $ RC(G, \Omega ) \leq 3 $.
\end{thrm}
\begin{proof}
Corollary \ref{trshght} shows that $ Ht(G, \Omega ) \leq 3 $. If $ Ht(G, \Omega ) \leq 2 $ then $ RC(G, \Omega ) \leq 3 $ by Theorem \ref{rchgt} and there is nothing further to show. So suppose $ Ht(G, \Omega ) = 3 $.

\medskip
Let $ I , J \in \Omega^{4} $, where $ I = ( I_{1} , \dots , I_{4} ) $ and $ J = ( J_{1} , \dots , J_{4} ) $. By Corollary \ref{rcshsbtplcmp} is is sufficient to show that if $ I \sim _{3} J $ then $ I \sim _{4} J $. So suppose $ I \sim _{3} J $.

\medskip
If any of the entries of $ I $ are repeated then $ I \sim _{4} J $ by Lemma \ref{rptdntrs} and there is nothing more to show. So suppose none of the entries of $I$ are repeated.

\medskip
By Lemma \ref{frsnrtytql} we may assume that $ I _{i} = J _{i} $ for all $ i \in \{ 1, 2 , 3 \} $. By looking at what $ J_{4} $ could be, it will be shown that the only possibility is $ J_{4} = I _{4} $. 

\medskip
Let $ X = \{ I_{1} , \dots , I _{4} \} $, the set of entries of $ I$. If any proper subset of $ X $ is not independent, then $ I \sim _{4} J $ by Lemma \ref{ntndpndntntrs}. If $ H $ does not contain a rotation of order $ 2 $ then Lemma \ref{rttnordrtwndpdn} shows that there exists a proper subset of $ X $ that is not independent, and we would be done. So suppose  $ H $ contains a rotation of order $ 2 $ and that every proper subset of $ X $ is independent.

\medskip
Lemmas \ref{stach}, \ref{klnfrntrsctn} and \ref{trsndstszthr} show that $ | G_{I_{i} , I_{j} } | \in \{ 2 , 4 \} $ for each $ i , j \in \{ 1, \dots , 4 \} $ with $ i \neq j $. If some two-point stabilizer has order $ 4 $ then $ I_{4} = J_{4} $ by Corollary \ref{crlstblzrszfr}. The only case left is if $  | G_{I_{i} , I_{j} } | = 2 $ for each $ i , j \in \{ 1, \dots , 4 \} $, which again gives $ I_{4} = J_{4} $ by Lemma \ref{llhvsztw}.

\medskip
Thus $ I = J$ and $ I \sim_{4} J $. It follows that $ RC(G, \Omega ) \leq 3 $.
\end{proof}
With this theorem in place we can return to looking at the dihedral actions of $ PGL_{2} (q) $ and $ PSL _{2} (q) $.
\begin{thrm}
When $ q \geq 4 $, the action of $ PGL_{2} (q) $ on a maximal $ D_{ 2(q \pm  1) } $ or $ PSL_{2} (q) $ on a maximal $ D_{ q \pm  1 } $ has relational complexity $ 3 $.
\end{thrm}
\begin{proof}
For $PGL_{2} (q) $ with $ q $ odd the maximal dihedral subgroups are defined in Table \ref{tableone} only when $ q \geq 5 $. In this case Lemma \ref{nrmlsbgrpspgl} shows that the only non-trivial normal subgroup of $ PGL _{2} (q) $ other than itself is $ PSL_{2} (q) $. By Lemma \ref{lnsz} the index of $ PSL_{2} (q) $ is $ 2 $ in $ PGL_{2} (q) $. Hence all proper non-trivial normal subgroups are maximal.

\medskip
It is well known that $ PSL_{2} (q) $ is simple for all $ q \geq 4 $. So the maximal dihedral subgroups are obviously not normal in $PSL_{2} (q) $. In $ PGL_{2} (q) $ the maximal dihedral subgroups are not normal because they contain a rotational subgroup which is normal and therefore are not isomorphic to $ PSL_{2} (q) $. Hence the actions on the dihedral subgroups have relational complexity at most $ 3 $ by Theorem \ref{mxmldhedrlrc}.

\medskip
In \cite{GILL1}, Lemmas 3.3 and 3.4 show it is not $2$ when $ q > 9 $. Thus the relational complexity is equal to $ 3 $.

\medskip
For $ q \leq 9 $, it can be seen the relational complexity is $ 3 $ in the tables in \cite{WISCONS1}.
\end{proof}
Theorem \ref{mxmldhedrlrc} applies to other simple groups that have maximal dihedral subgroups. One such class of groups are the Suzuki groups, found by Suzuki in \cite{SUZUKI1} and explored in more detail in \cite{SUZUKI2}. Table 8.16 in \cite{BRAYHOLTRONEYDOUGAL}, page 385 shows there exist maximal dihedral subgroups in the Suzuki groups. Specifically if $ Sz(q) $ is a Suzuki group with $ q = 2^{e} $ where $ e $ is odd and $ e > 1 $ then there exist maximal $ D_{ 2(q-1) } $.
\begin{thrm}
Let $ Sz(q) $ is a Suzuki group with $ q = 2^{e} $ where $ e $ is odd and $ e > 1 $. The relational complexity of the action on the right cosets of a maximal $ D_{2(q-1) } $ is $ 3 $.
\end{thrm}
\begin{proof}
We know from Theorem \ref{mxmldhedrlrc} that the relational complexity is either $ 2 $ or $ 3 $. In \cite{GILL4}, Corollary 1.4, page 2 it is shown the relational complexity is not $ 2 $.
\end{proof}

\newpage
\chapter{The \texorpdfstring{$ A _{4} $}{A4} Action}\label{chapseven}
Here the $ A_{4} $ actions of $ PSL_{2} (q) $ when $ q $ is odd will be looked at. Although the action will initially be defined as that on the maximal $ A_{4} $ mentioned in Table \ref{tablethree}, the theorems developed will apply to $ PSL_{2} (q) $ acting on any maximal $ A_{4} $ (such as when the maximal $ A_{4} $ is a subfield subgroup). The main results of this chapter for height are:

\medskip
If $ G = PSL_{2} (5) $ then $ Ht(G, \Omega ) = 3 $. See Theorem \ref{pslfvnme}.

\medskip
If $ G = PSL_{2} (q) $ with $ q \neq 5 $ then $ Ht(G, \Omega ) = 2 $. See Theorem \ref{afrctnhght}.

\medskip
The main results of this chapter for relational complexity are:

\medskip
If $ G = PSL_{2} (5) $ then $ RC(G, \Omega ) = 4 $. See Theorem \ref{pslfvnme}.

\medskip
If $ G = PSL_{2} (q) $ with $ q \neq 5 $ then $ RC(G, \Omega ) = 3 $. See Theorem \ref{afrrcthrem}.
\section{Preliminary Results}
Group presentations will be used frequently throughout the rest of this text and the following theorems will be required.
\begin{thrm}\label{gnrtrhmmrphsm}
If $ G = \langle X : R \rangle $ and $ H = \langle X : S \rangle $, where $X$ is a set of generators and $ R $ and $ S$ are relations with $ R \subseteq S $, then there exists a surjective homomorphism $ \phi : G \rightarrow H $ fixing every $ x \in X $.
\end{thrm}
\begin{proof}
See \cite{JOHNSON1}, Chapter 4, Proposition 2, page 43.
\end{proof}
\begin{corl}\label{gnrtrcrlry}
Suppose $ G = \langle X : R \rangle $ is a finite group. Also suppose $ H = \langle X : S \rangle $,  with $ R \subseteq S $. Then $ | H | $ divides $ |G| $. If $ | G | = | H | $ then $ G = H $.
\end{corl}
\begin{proof}
By Theorem \ref{gnrtrhmmrphsm} there exists a surjective homomorphism $ \phi : G \rightarrow H $. Therefore $ | H | = | \text{im} ( \phi ) | $ divides $ | G | $. If $ |G | = | H | $ then $ \phi $ is injective.
\end{proof}

\newpage
\section{Action Description}
Throughout this subsection and the next let $p$ be an odd prime that satisfies the conditions below. In Table \ref{tablethree}, we see that when
\begin{itemize}
\item $p \geq 5 $,
\item $ p \equiv \pm 3 \pmod{8} $ \ \ \ \text{and}
\item $ p \not\equiv \pm 1 \pmod{10} $,
\end{itemize}
there exists a single conjugacy class of maximal subgroups isomorphic to $ A_{4} $ in $ PSL_{2} (p) $ (alternatively, to the above conditions we can say that $ p = 5 $ or $ p \equiv \pm 3 , \pm 13 \pmod{40} $).

\medskip
It follows from Lemma \ref{cstquivtcnj} that rather than  describing the action in terms of right cosets of a maximal $ A_{4} $, we can look at the action by conjugation on $ \Omega $ since the two are equivalent.
\section{Height and Relational Complexity of the \texorpdfstring{$A _{4} $}{} Action}
Although the action description talked about maximal $ A_{4} $ in $ PSL_{2} (p) $ where $ p $ is a prime, there also exist maximal $ A_{4} $ in $ PSL_{2} (q) $ where $ q $ is a power of $ 3 $. This will be discussed further in the chapter on subfield subgroups later. However the results on height and relational complexity in this chapter apply to those later cases if we look at what happens for the action of $ PSL_{2} (q) $ on a maximal $ A_{4} $ for any $ q $. So some quite general theorems will be developed here.

\medskip
The height of such an $ A_{4} $ action will be found and this will directly lead us to its relational complexity. A special case when $ p = 5 $ will be dealt with first.
\begin{thrm}\label{pslfvnme}
The $ A_{4} $ action of $ PSL _{2} (5) $ has height $ 3 $ and relational complexity $ 4 $.
\end{thrm}
\begin{proof}
Let $H \leq PSL _{2} (5) $ be a maximal $ A _{4} $ and let $ \Omega $ be its conjugacy class. Lemma \ref{lnsz} shows that $ | PSL _{2} (5) | = 60 $. Hence $ | \Omega | = | PSL _{2} (5) | / | A _{4} | = 60 / 12 = 5 $. Since the action of $ PSL _{2} (5) $ is faithful, the group can be embedded in $ S _{5} $ and the action considered as the action on $ \{ 1, \dots , 5 \} $. The only subgroup of order $ 60 $ in $ S _{5} $ is $ A _{5} $ and the natural action of $ A _{5} $ has height $ 3 $ and relational complexity $ 4 $ by Examples \ref{rcgrpaltg} and \ref{hghtalt}.
\end{proof}
For the rest of this subsection assume that $ p > 5 $, which in fact means $ p \geq 13 $. Later it will be shown that the height of the $ A_{4} $ action of $ PSL_{2} (q) $ is $2$. This will rely on first proving the action has no independent set of size $ 3$.

\medskip
It will be helpful to refer the Cayley table for $ A _{4} $ provided below. To be consistent with notation used later, the identity element is written as $ e $, the elements of order $ 2 $ will be labelled as $r , s $ and $ t $ and the elements of order $ 3 $ will be labelled as $ a, b , c $ and $ d $ (and their inverses written accordingly). Cycle notation is provided in the left column to check the multiplication of the elements is correct. Multiplication in cycle notation is read from left to right here, for example $ (1 \ 2)(3 \ 4)(1 \ 2 \ 3) = (1 \ 3 \ 4) $. The conjugacy classes in this group are $ \{ e \} $, $ \{ r,s,t \} $, $ \{ a,b,c,d \} $ and $ \{ a^{-1} , b^{-1} , c^{-1} , d^{-1} \} $.

\newpage
\begin{tble}\label{afrcylytbl} Cayley table for $ A_{4} $;

\medskip
\begin{tabular}{l | c c c c c c c c c c c c }
  & $ e $ & $ r$ & $s$ & $t$ & $a$ & $b$ & $c$ & $d$ & $a^{-1}$ & $b^{-1}$ & $c^{-1}$ & $d^{-1}$  \\
    \cline{1-13}
  $ e = () $ & $ e $ & $ r$ & $s$ & $t$ & $a$ & $b$ & $c$ & $d$ & $a^{-1}$ & $b^{-1}$ & $c^{-1}$ & $d^{-1}$ \\   
 $ r = (1 \ 2)(3 \ 4) $ & $ r $ & $ e $ & $ t $ & $ s $ & $ b $ & $ a $ & $d$ & $c$ & $ d^{-1} $ & $ c^{-1} $ & $ b^{-1} $ & $ a^{-1} $ \\
  $ s = (1 \ 4)(2 \ 3) $ & $ s $ & $ t $ & $ e $ & $ r $ & $ c$ & $d$ & $a$ & $b$ & $ b^{-1} $ & $ a^{-1} $ & $ d^{-1} $ & $ c^{-1} $ \\
  $ t = (1 \ 3)(2 \ 4) $ & $ t $ & $ s $ & $ r $ & $ e $ & $d$ & $c$ & $b$ & $a$ & $ c^{-1} $ & $ d^{-1} $ & $ a^{-1} $ & $ b^{-1} $ \\
 $ a = (1 \ 2 \ 3) $ & $ a $ & $ d $ & $ b $ & $ c $ & $a^{-1} $ & $ c^{-1} $ & $ d^{-1} $ & $ b^{-1} $ & $e$ & $r$ & $s$ & $t$ \\
 $ b = (1 \ 3 \ 4) $ & $ b $ & $ c $ & $ a $ & $ d $ & $ d^{-1} $ & $ b^{-1} $ & $ a^{-1} $ & $ c^{-1} $ & $r$ & $e$ & $t$ & $s$ \\
 $ c = (1 \ 4 \ 2) $ & $ c $ & $ b $ & $ d $ & $ a $ & $ b^{-1} $ & $ d^{-1} $ & $c^{-1} $ & $ a^{-1} $ & $s$ & $t$ & $e$ & $r$ \\
  $ d = (2 \ 4 \ 3) $ & $ d $ & $ a $ & $ c $ & $ b $ & $ c^{-1} $ & $ a^{-1} $ & $ b^{-1} $ & $d^{-1} $ & $t$ & $s$ & $r$ & $e$ \\
 $ a^{-1} = (1 \ 3 \ 2) $ & $ a^{-1} $ & $ b^{-1} $ & $ c^{-1} $ & $ d^{-1} $ & $e $ & $s$ & $t$ & $r$ & $a$ & $d$ & $b$ & $c$ \\
 $ b^{-1} = (1 \ 4 \ 3) $ & $ b^{-1} $ & $ a^{-1} $ & $ d^{-1} $ & $ c^{-1} $ & $ s $ & $ e $ & $r$ & $t$ & $c$ & $b$ & $d$ & $a$ \\
 $ c^{-1} = (1 \ 2 \ 4) $ & $ c^{-1} $ & $ d^{-1} $ & $ a^{-1} $ & $ b^{-1} $ & $ t $ & $ r $ & $e $ & $s$ & $d$ & $a$ & $c$ & $b$ \\
  $ d^{-1} = (2 \ 3 \ 4) $ & $ d^{-1} $ & $ c^{-1} $ & $ b^{-1} $ & $ a^{-1} $ & $ r $ & $ t $ & $ s $ & $e $ & $b$ & $c$ & $a$ & $d$
\end{tabular}
\end{tble}
First the height of the $ A_{4} $ action will be calculated, then used to find an upper bound for relational complexity. To narrow down the height it is going to be shown there exists no independent set of size $ 3 $. This is done by considering what the intersections of pairs of point stabilizers would be for an independent set of size $ 3 $ if such a set existed. This boils down to looking at four different types of  potential independent set. One of these cases is going to be useful in later chapters, so has been pulled out as a separate lemma before dealing with the rest.
\begin{lema}\label{csefrafr}
Let $ G ^{ * } $ be $ PSL_{2} (q) $ where $ q $ is odd and $ q \geq 11 $ ($ q $ is not necessarily prime here). Suppose $ \Gamma $ is a conjugacy class of $ A_{4} $ in $ G ^{ * } $ (not necessarily maximal). Let $ H_{1} , H_{2} , H_{3} \in \Gamma $, with no two of these equal to each other. It is not possible that $ H_{1} \cap H_{2} \cong H_{2} \cap H_{3} \cong C_{3} $ and $ H_{1} \cap H_{3} \cong C_{2} $.
\end{lema}
\begin{proof}
For each $ i \in \{ 1 , 2 , 3 \} $ write
\[
H _{i} 
= \{   e   ,   r_{i}   ,  s_{i}   , t_{i} ,  a_{i} ,  b_{i} ,  c_{i} ,  d_{i} ,  a^{-1}_{i}  ,  b^{-1}_{i}  ,  c^{-1}_{i} ,  d^{-1}_{i} \}
\]
with multiplication the same as in Table \ref{afrcylytbl} when indices are ignored. So first up we have;
\[
H_{1}
= \{   e   ,   r_{1}   ,  s_{1}   , t_{1} ,  a_{1} ,  b_{1} ,  c_{1} ,  d_{1} ,  a^{-1}_{1}  ,  b^{-1}_{1}  ,  c^{-1}_{1} ,  d^{-1}_{1} \} .
\]
Suppose $ | H_{1} \cap H_{2}  | = 3$. Then we can assume that $ a_{1} \in H_{2} $ so that $  H_{1} \cap H_{2} = \{ e , a_{1} , a_{1}^{-1} \} $. We can label the elements of $ H_{2} $ so that $ a_{1} = a_{2} $, giving
\[
H_{2}
= \{   e   ,   r_{2}   ,  s_{2}   , t_{2} ,  a_{1} ,  b_{2} ,  c_{2} ,  d_{2} ,  a^{-1}_{1}  ,  b^{-1}_{2}  ,  c^{-1}_{2} ,  d^{-1}_{2} \} .
\]
For the remaining intersections we have $ | H_{1}  \cap H_{3}  | = 2 $ and $ | H_{2}  \cap H_{3}  | = 3$.

\medskip
Note that it cannot be that $ a_{1} \in H_{3} $, otherwise we would have $ H_{1} = \langle x , a_{1} \rangle = H_{3} \neq H_{1} $. Whatever elements order $ 3 $ are in $ H_{2}  \cap H_{3}  $, we can conjugate $ H_{3} $ by $ a_{1} $ or $ a_{1} ^{-1}$ if required to get some subgroup $ H_{3} ^{*} $ such that $ | H_{1}  \cap H_{3} ^{*}  | = 2 $ and $ | H_{2}  \cap H_{3} ^{*}  | = 3$ and $ b_{2} , b_{2}^{-1} \in H_{3} ^{*} $. So it is safe to assume $ b_{2} , b_{2}^{-1} \in H_{3} $. Since we have not given $ b_{2} $ a label in $ H_{3} $, we can set $ b_{3} := b_{2} $ and label other elements of $ H_{3} $ as needed to give
\[
H_{3}
= \{   e   ,   r_{3}   ,  s_{3}   , t_{3} ,  a_{3} ,  b_{2} ,  c_{3} ,  d_{3} ,  a^{-1}_{3}  ,  b^{-1}_{2}  ,  c^{-1}_{3} ,  d^{-1}_{3} \} .
\]
Label the involution in $ H_{1}  \cap H_{3} $ as $ x $ for now. Let $ y_{1} = x a_{1} $ and $ y_{2} = xb_{2} $. Observe that $ y_{1} \in H_{1} $ and $ y_{2} \in H_{3} $ and that these elements have order $3 $. Also note $ y_{1} \notin H_{3} $ otherwise $ xy_{1} = a_{1} \in H_{3} $. Using Table \ref{afrcylytbl}, it can be seen from the elements of $ H_{2} $ we get $ s_{2} = a_{1}^{-1} b_{2} $. Also $ y_{1} ^{-1} y_{2} = a_{1} ^{-1} x x b_{2} = a_{1}^{-1} b_{2} = s_{2} $. Thus $ s_{2} \in \langle y_{1} , y_{2} \rangle $.

\medskip
Now consider which group $ \langle y_{1} , y_{2} \rangle $ is isomorphic to. So far we have the following relations
\begin{itemize}
\item $ y_{1} ^{3} = y_{2} ^{3} = e $
\item $ (y_{1} ^{2} y_{2} )^{2} = (y_{1} ^{-1} y_{2} )^{2} = e $.
\end{itemize}
Use GAP code
\begin{lstlisting}
f := FreeGroup("a", "b");;
g := f/[f.1^3, f.2^3, (f.2*f.2*f.1)^2];
StructureDescription(g); 
\end{lstlisting}
to show that a group $ H $ generated by three elements satisfying only the above relations is isomorphic to $ A_{4} $. However it could be that there are additional relations that have not been considered.

\medskip
Lemma \ref{gnrtrhmmrphsm} shows that there exists a surjective homomorphism $ \phi : H \rightarrow \langle y_{1} , y_{2} \rangle $. There are at least six elements in $ \langle y_{1} , y_{2} \rangle $, namely $ e , y_{1} , y_{1} ^{-1} , y_{2} , y_{2} ^{-1} $ and $ s_{2} $. Hence $ | \text{im} ( \phi ) | \geq 6 $. Since $ | \ker ( \phi ) | | \text{im} ( \phi ) | = | H | = 12 $, it must be that either $ | \ker ( \phi ) | = 2 $ or $ | \ker ( \phi ) | = 1 $.  It cannot be that $ | \ker ( \phi ) | = 2 $ because $ A_{4} $ does not have a normal subgroup of order $ 2 $. Thus $ | \ker ( \phi ) | = 1 $, which shows $  \langle y_{1} , y_{2} \rangle \cong H \cong A_{4} $. Observe that
\begin{align*}
x r_{2} x = x a_{1} b_{2}^{-1} x = y_{1}  y_{2} ^{-1} \in \langle y_{1} , y_{2}  \rangle .
\end{align*}
If $ x r_{2} x = s_{2} $ then $y_{1}  y_{2} ^{-1} = y_{1} ^{-1} y_{2}  $ and it follows that $ y_{1} ^{2} = y_{2} ^{2} \in H_{3} $, which is a contradiction. Therefore $ x r_{2} x \neq s_{2} $.

\medskip
Both $ x r_{2} x $ and $ s_{2} $ commute with each other because they are involutions in an $ A_{4} $. Also $ r_{2} $ commutes with $ s_{2} $ because $ r_{2} , s_{2} \in H_{2} $. Hence $ r_{2} \ , \ x r_{2} x \in N_{G^{*}} ( \{ e , s_{2} \} ) $.

\medskip
Let $ m \in \{ q-1 , q+1 \} $ with $ m \equiv 0 \pmod{4} $. Since $ q \geq 11 $, Table \ref{tablethree} and Lemma \ref{nrmlzr} together show there exists a maximal $ D < G^{*} $ with $ D \cong D_{m} $ and $  D =  N_{G^{*}} ( \{ e , s_{2} \} ) $. The rotational subgroup of $ D $ has only one involution, $ s_{2} $, so $  r_{2} $ and $ x r_{2} x $ are reflections. It follows that $ r_{2} x r_{2} x $ is in the rotational subgroup and $ \langle r_{2} x r_{2} x \rangle \triangleleft D $ by Lemma \ref{dhdrlrottnnrml}.

\medskip
If $ r_{2} x r_{2} x = e  $ then  $  x r_{2} x = r_{2} $ and by looking at how elements multiply in the earlier table we have
\begin{align*}
x b_{2} y^{-1} _{1} = x b_{2} a^{-1} _{1} x = x r_{2} x = r_{2} = b_{2} a^{-1} _{1} .
\end{align*}
Rearranging gives $  a^{-1} _{1} y _{1} = b_{2} ^{-1} x b_{2}  \in H_{3}   $. But $  a^{-1} _{1} y _{1} = a_{1} ^{-1} x a_{1} $ is some involution in $ H_{1} $ and $ a_{1} ^{-1} x a_{1} \neq x $ since no involution is normalized by an element of order $ 3 $ in $ A_{4} $. Hence $ a_{1} ^{-1} x a_{1} \notin H_{3} $, a contradiction. Therefore  $ r_{2} x r_{2} x \neq e  $.

\medskip
Notice that $ r_{2} \ , \ x r_{2} x \in  x D x  $ and both of these elements must be reflections again. By the same reasoning as above $ \langle r_{2} x r_{2} x \rangle \triangleleft x D x $. Thus $ D = x D x $ by Lemma \ref{ntnrmlmxml}. Hence $ x $ normalizes $  \{ e , s_{2} \} $. In particular $ x s_{2} x = s_{2} $. Since $ s_{2} = a_{1} ^{-1} b_{2} $, we have $ x a_{1} ^{-1} b_{2} x = a_{1} ^{-1} b_{2} $, which can be rearranged to give
\begin{align*}
 a_{1} x a_{1} ^{-1} = b_{2} x b_{2} ^{-1} \in H_{3} .
\end{align*}
Also $ a_{1} x a_{1} ^{-1} \in H_{1} $. Hence $ a_{1} x a_{1} ^{-1} = x $. By the same reasoning used two paragraphs up, we have $  a_{1} x a_{1} ^{-1} \neq x $, a contradiction. Thus it is not possible that $ H_{1} \cap H_{2} \cong H_{2} \cap H_{3} \cong C_{3} $ and $ H_{1} \cap H_{3} \cong C_{2} $.
\end{proof}

\newpage
\begin{lema}\label{frctnndpndn}
Suppose $ q $ is odd and $ q \geq 11 $. Let $ \Omega $ be a conjugacy class of maximal $ A _{4} $ in $  G:= PSL _{2} (q) $. Then under the conjugation action of $ G $ on $ \Omega $, there does not exist an independent set of size $ 3 $.
\end{lema}
\begin{proof}
Suppose $ \Delta := \{ H _{1} , H _{2} , H _{3} \} \subseteq \Omega $ is an independent set. Using Lemma \ref{indsbstsntqual} and looking at the subgroup structure of $ A_{4} $ it must be that
\[
\{ e \} = \bigcap _{i=1} ^{3} H _{i} \ < H _{j}  \cap H _{k} \ < \  H _{j}  \tag{*}
\]
where $j, k \in \{ 1, 2 , 3 \} $ and $ j \neq k $. Also $ H _{j}  \cap H_{k} $ is either a cyclic group of order $2$ or a cyclic group of order $ 3 $ (it cannot be a Klein four group because it would be normal in both $ H_{j} $ and $ H _{k} $, which is not possible by Lemma \ref{ntnrmlmxml}). Now consider the $2$-element subsets of $ \Delta $,
\begin{align*}
& \{ H _{1} , H _{2} \} , \\
& \{ H _{1} , H _{3} \} \text{ and} \\
& \{ H _{2} , H _{3} \} .
\end{align*}
The collection of orders of the intersections of each of these pairs of stabilizers must be one the following;
\begin{itemize}
\item All three of the intersections have order $ 2 $.
\item All three of the intersections have order $3$.
\item Two of the intersections have order $2$ and one of them has order $3$.
\item Two of intersections have order $3$ and one of them has order $2$. 
\end{itemize}
These cases will be looked at in turn and shown to not be possible. I will use the following notation for the stabilizers in $ \Delta $. For each $ i \in \{ 1 , 2 , 3 \} $ write
\[
H _{i} 
= \{   e   ,   r_{i}   ,  s_{i}   , t_{i} ,  a_{i} ,  b_{i} ,  c_{i} ,  d_{i} ,  a^{-1}_{i}  ,  b^{-1}_{i}  ,  c^{-1}_{i} ,  d^{-1}_{i} \}
\]
with multiplication the same as in Table \ref{afrcylytbl} when indices are ignored, unless said otherwise. 

\bigskip
\textbf{Case 1:} Suppose all three intersections have order $ 2 $. We can start by writing 
\[
H _{1} 
= \{   e   ,   r_{1}   ,  s_{1}   , t_{1} ,  a_{1} ,  b_{1} ,  c_{1} ,  d_{1} ,  a^{-1}_{1}  ,  b^{-1}_{1}  ,  c^{-1}_{1} ,  d^{-1}_{1} \}
\]
as above. As $ | H_{1} \cap H_{2}   | = 2$ we can suppose that $ r _{1} \in H _{2} $ and that $ r _{1} = r_{2} $. So
\[
H _{2}
= \{   e   ,   r_{1}   ,  s_{2}   , t_{2} ,  a_{2} ,  b_{2} ,  c_{2} ,  d_{2} ,  a^{-1}_{2}  ,  b^{-1}_{2}  ,  c^{-1}_{2} ,  d^{-1}_{2} \} .
\]
Similarly $ | H_{1} \cap H_{3} | = 2$ so we can assume $ s_{1} \in H _{3} $ and $ s_{1} = s_{3} $. Note that it cannot be that $ r_{1} \in H _{3} $ otherwise we would have $ H _{1}  \cap H _{2} \cap H _{3} = \{ e , r_{1} \} $, contradicting $ (*) $. Again $ | H_{2} \cap H _{3}  | = 2$ so we can assume $ t_{2} \in H_{3}  $ and $ t_{2} = t_{3} $ (relabelling the elements of $ H_{2} $, except $r _{1} $, if needed). Hence
\[
H_{3} 
= \{   e   ,   r_{3}   ,  s_{1}   , t_{2} ,  a_{3} ,  b_{3} ,  c_{3} ,  d_{3} ,  a^{-1}_{3}  ,  b^{-1}_{3}  ,  c^{-1}_{3} ,  d^{-1}_{3} \} .
\]
Note that $ K:= \{   e   ,   r_{3}   ,  s_{1}   , t_{2} \} $ is a Klein four subgroup of $ H_{3} $ and $ K \triangleleft H_{3} $. Lemma \ref{nrmlprstmtsk} shows that $  H_{3} = N_{G} (K)  $. However from the multiplication in Table \ref{afrcylytbl} we have $ r_{1} ^{-1} s_{1} r_{1} = r_{1} s_{1} r_{1} = s_{1} $ and $ r_{1} ^{-1} t_{2} r_{1} = r_{1} t_{2} r_{1} = t_{2} $. Also
\begin{align*}
r_{1} ^{-1} r_{3} r_{1} 
& = r_{1} r_{3} r_{1} \\
& = r_{1} s_{1} t_{2} r_{1} \\
& = r_{1} s_{1} r_{1} r_{1} t_{2} r_{1} \\
& = s_{1} t_{2} \\
& = r_{3} .
\end{align*}
Hence $ r_{1} ^{-1} K r_{1} = K $ and it follows that $ r_{1} \in H_{3}  $, which is a contradiction. Thus the assumption that all three intersections of the form $ H_{j} \cap H_{k} $ have order $ 2 $ must be wrong.

\medskip
\textbf{Case 2:} Suppose all three intersections of the form $ H_{j} \cap H_{k}   $ have order $ 3 $. As before, we can start by writing 
\[
H_{1}  
= \{   e   ,   r_{1}   ,  s_{1}   , t_{1} ,  a_{1} ,  b_{1} ,  c_{1} ,  d_{1} ,  a^{-1}_{1}  ,  b^{-1}_{1}  ,  c^{-1}_{1} ,  d^{-1}_{1} \} .
\]
Since $ | H_{1} \cap H_{2}  | = 3$ we can assume that the elements of $ H_{1} $ are labelled so that $ a_{1} , a_{1}^{-1} \in  H_{2} $. At this point we can label the elements of $ H_{2} $ so that $ a_{1} = a_{2} $ (and $ a_{1}^{-1} = a_{2}^{-1} $) and then construct the rest of $ H_{2} $ and label the elements accordingly so that
\[
H_{2}
= \{   e   ,   r_{2}   ,  s_{2}   , t_{2} ,  a_{1} ,  b_{2} ,  c_{2} ,  d_{2} ,  a^{-1}_{1}  ,  b^{-1}_{2}  ,  c^{-1}_{2} ,  d^{-1}_{2} \} .
\]
Following the same procedure again $ | H_{1}  \cap H_{3}  | = | H_{2} \cap H_{3}  | = 3 $ so we can assume the elements of $ H_{1}  $ are labelled such that $ b_{1} , b_{1}^{-1} \in H_{3}  $ and the elements of $ H_{2} $ are labelled such that $ c_{2} , c_{2}^{-1} \in H_{3} $. However there are two conjugacy classes containing elements of order $3$ in $ A_{4} $ and it may be that $ b_{1} $ and $ c_{2} $ lie in the same conjugacy class or it could be that they lie in opposite conjugacy classes, so it is difficult to label the elements of $ H_{3} $ using the usual convention. Both of these cases are considered separately.

\medskip
First suppose  $ b_{1} $ and $ c_{2} $ lie in the same conjugacy class in $ H_{3} $. Then we can label the elements of $ H_{3} $ so that $ b_{1} = b_{3} $ and $ c_{2} = c_{3} $, giving
\[
 H_{3} = \{   e   ,   r_{3}   ,  s_{3}   , t_{3} ,  a_{3} ,  b_{1} ,  c_{2} ,  d_{3} ,  a^{-1}_{3}  ,  b^{-1}_{1}  ,  c^{-1}_{2} ,  d^{-1}_{3} \} .
\]
Using Table \ref{afrcylytbl}, it can be shown that $ b_{1} $ and $ c_{2} $ generate $ H_{3}  $. Since $ H_{3} $ is a maximal subgroup and does not contain $ a_{1} $, it must be that $ \langle a_{1} , b_{1} , c_{2} \rangle = G $. Looking at how these three generators multiply, we get the following relations:
\begin{itemize}
\item $ a_{1}^{3} = b_{1}^{3} = c_{2}^{3} = e $
\item $ (a_{1} b_{1})^{3} = (a_{1} c_{2})^{3} = (b_{1} c_{2})^{3} = e  $
\item $ (a_{1} b_{1}^{2})^{2} = (a_{1} c_{2}^{2})^{2} = (b_{1} c_{2}^{2})^{2} = e  $.
\end{itemize}
Use GAP code
\begin{lstlisting}
f := FreeGroup("a", "b", "c");;
g := f/[f.1^3, f.2^3, f.3^3, (f.1*f.2)^3, (f.1*f.3)^3, (f.2*f.3)^3,
(f.1*f.3*f.3)^2, (f.1*f.2*f.2)^2, (f.2*f.3*f.3)^2];
StructureDescription(g);
Size(g);
\end{lstlisting}
to show that a group $H$ generated by three elements satisfying only the above relations has order $96 $ and $ H \cong ((C_{2} \times C_{2} \times C_{2}) : (C_{2} \times C_{2})) : C_{3} $. However it could be that there are additional relations that have not been considered. Lemma \ref{gnrtrhmmrphsm} shows that there exists a surjective homomorphism $ \phi : H \rightarrow \langle a_{1} , b_{1} , c_{2} \rangle $. Thus $ | \text{im} ( \phi ) | = | \langle a_{1} , b_{1} , c_{2} \rangle | $. Since $ | H | = | \ker ( \phi ) | | \text{im} ( \phi ) | $, it follows that $ | PSL_{2} (q) | = | \langle a_{1} , b_{1} , c_{2} \rangle | \leq 96 $. Since $ q \geq 11 $, Lemma \ref{lnsz} shows that $ | PSL_{2} (q) | \geq \tfrac{1}{2} \cdot 11 \cdot (11^{2} - 1) = 660 $. So we have a contradiction, showing that it is not possible for $ b_{1} $ and $ c_{2} $ to lie in the same conjugacy class.

\medskip
Next suppose that $ b_{1} $ and $ c_{2} $ lie in opposite conjugacy classes in $ H _{3} $. Then we can label the elements of $ H_{3} $ so that $ b_{1} = b_{3} $ and $ c_{2} = c_{3}^{-1} $, giving
\[
H_{3} = \{   e   ,   r_{3}   ,  s_{3}   , t_{3} ,  a_{3} ,  b_{1} ,  c_{2}^{-1} ,  d_{3} ,  a^{-1}_{3}  ,  b^{-1}_{1}  ,  c_{2} ,  d^{-1}_{3} \} .
\]
As above it can be shown that $ b_{1} $ and $ c_{2} $ generate $ H_{3} $ and since $ H_{3}  $ is a maximal subgroup that does not contain $ a_{1} $, it must be that $ \langle a_{1} , b_{1} , c_{2} \rangle = G $. Looking at how these three generators multiply, we get the following relations:
\begin{itemize}
\item $ a_{1}^{3} = b_{1}^{3} = c_{2}^{3} = e $
\item $ (a_{1} b_{1})^{3} = (a_{1} c_{2})^{3} = (b_{1} c_{2})^{2} = e  $
\item $ (a_{1} b_{1}^{2})^{2} = (a_{1} c_{2}^{2})^{2} = (b_{1} c_{2}^{2})^{3} = e  $.
\end{itemize}
Use GAP code
\begin{lstlisting}
f := FreeGroup("a", "b", "c");;
g := f/[f.1^3, f.2^3, f.3^3, (f.1*f.2)^3, (f.1*f.3)^3, (f.2*f.3)^2, (f.1*f.3*f.3)^2,
(f.1*f.2*f.2)^2, (f.2*f.3*f.3)^3]; 
StructureDescription(g);
Size(g); 
\end{lstlisting}
to show that a group $H ^{\prime } $ generated by three elements satisfying only the above relations has order $60 $ and $ H ^{\prime }  \cong A_{5} $. Following the same reason as above, it must be that $ | PSL_{2} (q) | = | \langle a_{1} , b_{1} , c_{2} \rangle | \leq 60 < 660 \leq | PSL_{2} (q) | $, so we have a contradiction again. Thus it is not possible that all three intersections of the form $ H_{j} \cap H_{k}  $ have order $ 3 $.

\medskip
Note that the last part of this case gives relations that generate $ A_{5} $, which is exactly what happens when we allow $ q=5 $ and why $ PSL_{2} (5) $ had to be dealt with as a special case earlier.

\bigskip
\textbf{Case 3:} Suppose two of the intersections of the form $ H_{j} \cap H_{k}  $ have order $2$ and one of them has order $3$. Again, start by writing
\[
H_{1} 
= \{   e   ,   r_{1}   ,  s_{1}   , t_{1} ,  a_{1} ,  b_{1} ,  c_{1} ,  d_{1} ,  a^{-1}_{1}  ,  b^{-1}_{1}  ,  c^{-1}_{1} ,  d^{-1}_{1} \} .
\]
Suppose $ | H_{1}  \cap H_{2}  | = 3$. Then we can assume that $ a_{1} , a_{1}^{-1} \in H_{2}  $. Take note for later that $ H_{1}  \cap H_{2}  $ contains no element of order $ 2$, in particular does not contain $ s_{2} $ from $  H_{2} $. We can label the elements of $ H_{2}  $ so that $ a_{1} = a_{2} $ (and $ a_{1}^{-1} = a_{2}^{-1} $), giving
\[
H_{2}
= \{   e   ,   r_{2}   ,  s_{2}   , t_{2} ,  a_{1} ,  b_{2} ,  c_{2} ,  d_{2} ,  a^{-1}_{1}  ,  b^{-1}_{2}  ,  c^{-1}_{2} ,  d^{-1}_{2} \} .
\]
Now it must be that $ | H_{1} \cap H_{3}  | = | H_{2} \cap H_{3} | = 2$. We can assume that the elements of each of these stabilizers are labelled so that $ r _{1} \in H_{1} \cap H_{3} $ with $ r_{1} = r_{3} $ and that $ s_{2} \in H_{2} \cap H_{3} $ with $ s_{2} = s_{3} $. This gives
\[
H_{3} 
= \{   e   ,   r_{1}   ,  s_{2}   , t_{3} ,  a_{3} ,  b_{3} ,  c_{3} ,  d_{3} ,  a^{-1}_{3}  ,  b^{-1}_{3}  ,  c^{-1}_{3} ,  d^{-1}_{3} \} .
\]
Using Table \ref{afrcylytbl}, it can be shown that $ r_{1} $ and $ a_{1} $ generate $ H_{1} $. Since $ H_{1} $ is maximal and $ s_{2} \notin H_{1} $ it must be that $ \langle a_{1} , r_{1} , s_{2} \rangle = G $. Looking at how these three generators multiply, we get the following relations:
\begin{itemize}
\item $ a_{1}^{3} = r_{1}^{2} = s_{2}^{2} = e $
\item $ (a_{1} r_{1})^{3} = (a_{1} s_{2})^{3} = (r_{1} s_{2})^{2} = e  $.
\end{itemize}
Use GAP code
\begin{lstlisting}
f := FreeGroup("a", "b", "c");;
g := f/[f.1^3, f.2^2, f.3^2, (f.1*f.2)^3, (f.1*f.3)^3, (f.2*f.3)^2];
StructureDescription(g);
Size(g);
\end{lstlisting}
to show that a group $ H$ generated by three elements satisfying only the above relations has order $96 $ and $ H \cong ((C_{2} \times C_{2} \times C_{2}) : (C_{2} \times C_{2})) : C_{3} $. This situation appeared in the previous case and was shown to cause a contradiction with the order of $PSL _{2} (q) $. Thus it is not possible that two of the intersections of the form $ H _{j} \cap H_{k}  $ have order $2$ and one of them has order $3$.

\bigskip
\textbf{Case 4:} For the final case suppose two of the intersections of the form $ H _{j}  \cap H_{k}  $ have order $3$ and one of them has order $2$. Since $ q \geq 11 $, this case was shown to not be possible in Lemma \ref{csefrafr}.

\bigskip
Cases 1 to 4 together show that it is not possible for three maximal $ A_{4} $ subgroups to intersect each other pairwise in subgroups of order $ 2 $ and $ 3 $, contradicting $ (*) $. Thus $ \Delta $ cannot be an independent set. 
\end{proof}
With that short proof done, the height of the $ A_{4} $ action can immediately be found.
\begin{thrm}\label{afrctnhght}
Suppose $ q $ is odd and $ q \geq 11 $. Then the height of the action of $ PSL_{2} (q) $ by conjugation on a conjugacy class of maximal $ A _{4} $ is $ 2 $.
\end{thrm}
\begin{proof}
Let $ \Omega $ be a conjugacy class of maximal $ A_{4} $ in $ PSL_{2} (q) $. Since $ PSL_{2} (q) $ is simple, $ | \Omega | \geq 2 $. So there exist $ H_{1} , H_{2} \in \Omega $ with $ H_{1} \neq H_{2} $. Hence $ H_{1} \neq H_{1} \cap H_{2} \neq H_{2} $, which shows $ \{ H_{1} , H_{2} \} $ is an independent set by Lemma \ref{indsbstsntqual}. Therefore the height of the $ A _{4} $ action is at least $2$. If the height is $ 3 $ or more then there must exist an independent set of size $3$ by Corollary \ref{indsub}. Lemma \ref{frctnndpndn} shows that no such set exists and so the height is $2$.
\end{proof}
The relational complexity of the action can quickly be found as well.
\begin{thrm}\label{afrrcthrem}
Suppose $ q $ is odd and $ q \geq 11 $. Then the relational complexity of the action of $ PSL_{2} (q) $ by conjugation on a conjugacy class of maximal $ A _{4} $ is $ 3 $.
\end{thrm}
\begin{proof}
Theorem \ref{afrctnhght} and Theorem \ref{rchgt} together show that the relational complexity of the action is at most $3$. In \cite{GILL1}, Lemmas 2.3 and 3.3 show it is not $2$. Thus it is equal to $ 3 $.
\end{proof}

\newpage
\chapter{The \texorpdfstring{$ S_{4} $}{S4} Action}\label{chapeight}
Let $ p $ be a prime with $ p \geq 5 $. Suppose $ G $ is $ PSL_{2} (p) $ or $ PGL_{2} (p) $ and $ G $ contains a conjugacy class $ \Omega $ of maximal $ S_{4} $. The main results of this chapter are:

\medskip
If $ G = PGL_{2} (5) $ then $ Ht(G, \Omega ) = 4 $. If $ G \neq PGL_{2} (5) $ then $ Ht(G, \Omega ) = 3 $.

\medskip
If $ G = PGL_{2} (5) $ then $ RC(G, \Omega ) = 2 $. Otherwise $ RC(G, \Omega ) \in \{ 3,4 \} $.

\medskip
Also the following conjecture is made:

\medskip
Suppose $ p > 7 $. If $ G= PGL_{2} (p) $ then $ RC(G, \Omega ) = 4 $.
\section{Preliminary Results}
\begin{lema}
Let $ H_{1} , H_{2} \leq S_{4} $ with $ H_{1} \neq H_{2} $ and $ H_{1} \cong H_{2} \cong S_{3} $. Then $ | H_{1} \cap H_{2} | = 2 $.
\end{lema}
\begin{proof}
Each $ S_{3} $ subgroup of $ S_{4} $ is the stabilizer of a point of $ \{ 1 , 2 , 3 , 4 \} $ under the natural action of $ S_{4} $. The only non-identity element that stabilizes two points $ i , j \in \{ 1 , 2 , 3 , 4 \} $ is the transposition $ ( i , j ) $.
\end{proof}
\section{Action Description}
Throughout this section and the next let $p$ be an odd prime with $ p \geq 5 $. In Table \ref{tablethree}, we see that there exists maximal subgroups isomorphic to $ S_{4} $ in $ PSL_{2} (p) $ precisely when $ p \equiv \pm 1 \pmod{8} $ and they lie in two conjugacy classes.

\medskip
Also $ PGL _{2} (p) $ has maximal $ S_{4} $ subgroups if and only if $ p \equiv \pm 3 \pmod{8}  $ (see Table \ref{tableone}) and these subgroups lie in a single conjugacy class (see \cite{GUIDICI1}, Lemma 2.3).

\medskip
Let $ G $ be either $ PSL_{2} (p) $ or $ PGL _{2} (p) $ and suppose $ H \leq G $ is a maximal $ S_{4} $ subgroup. Put $ \Omega := \{ g^{-1} H g : g \in G \} $. It follows from Lemma \ref{cstquivtcnj} that rather than  describing the action in terms of right cosets of a maximal $ S_{4} $, we can look at the action by conjugation on $ \Omega $ since the two are equivalent.

\medskip
For the rest of this chapter let $ p $, $ G $, $ H $ and $ \Omega $ be as defined above. 

\medskip
Part of the work done in this chapter will be using GAP to check the height and relational complexity of these actions. The calculations can take some time so only one conjugacy class per group is going to be calculated in general. For $ PSL_{2} (q) $ there are two conjugacy classes, so it needs to be shown that the relational complexity and height does not depend on which class is chosen.

\medskip
Let $ \Omega _{1} $ and $ \Omega _{2} $ be two conjugacy classes of maximal $ S_{4} $ in $ PSL_{2} (p) $. Lemma \ref{pslsnpgl} tells is that $ PGL_{2} (p) $ contains a normal subgroup isomorphic to $ PSL_{2} (p) $, so we may as well consider our original $ PSL_{2} (p) $ to be this subgroup. The two conjugacy classes of $ A_{4} $ in the $ PSL_{2} (p) $ subgroup are fused in $ PGL_{2} (p) $ - see \cite{GUIDICI1}, Lemma 2.3. Since $ PSL_{2} (p) $ is normal, there exists $ g \in PGL_{2} (q) $ such that $ g^{-1} \Omega_{1} g = \Omega _{2} $ and $ g^{-1} \Omega_{2} g = \Omega _{1} $. Since we are looking at the action by conjugation of $ PSL_{2} (p) $ on these two conjugacy classes, it is not difficult to show from here the height of both actions must be same. Similarly the relational complexity is equal for both actions.
\section{Height of the \texorpdfstring{$S _{4} $}{} Action}
To calculate the height of the $ S_{4} $ action of $ G $ we start with some sufficient conditions for a lower bound.
\begin{lema}\label{sfrlwrbnd}
Let $ H^{*} $ be a maximal $ S_{4} $ in $ G $ with $ H^{*} \neq H $ ($ H^{*} $ not necessarily in the same conjugacy class as $ H $ if $ G = PSL_{2} (p) $). Also suppose $ H \cap H^{*} \cong S_{3} $. Then $  Ht(G, \Omega ) \geq 3 $.
\end{lema}
\begin{proof}
Since $ H \cong S_{4} $, it has four subgroups isomorphic to $ S_{3} $. Denote these as $ V_{1} , \dots ,  V_{4} $. Let $ i , j , k \in \{ 1 , \dots , 4 \} $ with $ i \neq j \neq k \neq i $. Then, by looking at the subgroup structure of $ S_{4} $, we have $ V_{i} \cap V_{j} \cong V_{i} \cap V_{k} \cong V_{j} \cap V_{k} \cong C_{2} $. Also none of $ V_{i} \cap V_{j} $, $ V_{i} \cap V_{k} $ and $ V_{j} \cap V_{k} $ are equal to each other.

\medskip
From the statement of the lemma, we may assume in $ G $ there exists a maximal $ S_{4} $ containing $ V_{i} $, which will be denoted $ H_{i} $, and further assume $ H_{i} \neq H $. The $ S_{3} $ subgroups in $ S_{4} $ are maximal. So $ V_{i} $ is maximal in $ H_{i} $. If $ V_{j} \leq H_{i} $ then $  H_{i} = \langle V_{i} , V_{j} \rangle = H $, a contradiction. Hence $ V_{j} \not\leq H_{i} $. Now $ V_{j} = h^{-1} V_{i} h $ for some $ h \in H \setminus V_{i} $. By the maximality of $ V_{i} $, we have $ \langle V_{i} , h \rangle = H $. Therefore $ h \notin H_{i} $.

\medskip
Either $ G $ is simple or Lemma \ref{nrmlsbgrpspgl} shows $ G $ has only one non-trivial normal subgroup of index $ 2 $, so $ H_{i} $ is not normal in $ G $. Therefore $ h ^{-1} H_{i} h \neq H_{i} $ by Lemma \ref{mxmlnrmlizr}. Put $ H_{j} := h ^{-1} H_{i} h $ and observe $ V_{j} \leq H_{j} $. We have $ H_{j} \neq H $, otherwise $ H = h H h^{-1} = h H_{j} h^{-1} = H_{i} $.

\medskip
The above can be repeated to find $ H_{k} \leq G $ such that $ V_{k} \leq H_{k} $ and $ H_{k} $ is conjugate to $ H_{i} $, with $ H_{k} \neq H_{i} $. If $ H_{k} = H_{j} $ then $ V_{k} \leq H_{j} $ and $ H_{j} = \langle V_{j} , V_{k} \rangle = H $, a contradiction. Hence $ H_{k} \neq H_{j} $.

\medskip
Let $ v_{jk} $ be the involution in $ V_{j} \cap V_{k} $. Then $ v_{jk} \notin V_{i} $. Since $ V_{i} $ is maximal in $ H $, we have $ \langle V_{i} ,  v_{jk} \rangle = H $. Observe $ v_{jk} \in H_{j} \cap H_{k} $. So $ v_{jk} \notin H_{i} $, otherwise $ H_{i} = H $. Hence $ H_{i} \cap H_{j} \cap H_{k} \neq H_{j} \cap H_{k} $. Similar reasoning shows $ H_{i} \cap H_{j} \cap H_{k} \neq H_{i} \cap H_{j} $ and $ H_{i} \cap H_{j} \cap H_{k} \neq H_{i} \cap H_{k} $. Thus $ \{ H_{i} , H_{j} , H_{k} \} $ is an independent subset of the conjugacy class that contains it by Lemma \ref{scndfnt}.

\medskip
If there are two conjugacy classes of maximal $ S_{4} $, say $ \Omega $ and $ \Omega ^{*} $, and $ H_{i} \in \Omega^{*} $ then $ \{ H_{i} , H_{j} , H_{k} \} \subseteq \Omega ^{*} $. So we need to make sure that an independent set of size $ 3 $ can be found in $ \Omega $. Notice from the initial assumptions earlier in the proof, all we need is a maximal $ S_{4} $ in $ \Omega $ that contains $ V_{i} $ and is not equal to $ H_{i} $. For this we can use $ H $.
\end{proof}
The above lower bound in fact always holds.
\begin{lema}\label{sfrhghtlwrbnd}
$ Ht(G, \Omega ) \geq 3 $.
\end{lema}
\begin{proof}
Let $ V \leq H $ with $ V \cong S_{3} $. There exists $ W \leq V $ with $ W \cong C_{3} $. Since $ G = PSL_{2} (p) $ or $ G = PGL_{2} (p) $ where $ p \neq 3 $, Lemma \ref{nrmlzr} tells us $ N_{G} (W) \cong D_{2(p \pm 1) / \delta } $, where $ \delta = 2 $ if $ G = PSL_{2} (p) $ or $ \delta = 1 $ if $ G = PGL_{2} (p) $. Observe that $ V \leq N_{G} (W) $.

\medskip
If $ 4 $ divides $ | N_{G} (W) | $ then there exists an involution $ a \in Z (N_{G} (W) ) $. None of the involutions in $ V $ commute with each other. Thus $ \langle V , a \rangle $ is a dihedral subgroup of $ N_{G} (W) $ with $ |  \langle V , a \rangle | > 6 $. The $ S_{3} $ subgroups of $ S_{4} $ are maximal and $ S_{4} $ is not dihedral, therefore $ \langle V , a \rangle \not\leq H $. In particular, $ a \notin H $. Either $ G $ is simple or Lemma \ref{nrmlsbgrpspgl} shows $ G $ has only one non-trivial normal subgroup of index $ 2 $, so $ H $ is not normal in $ G $. Therefore $ a ^{-1} H a \neq H $ by Lemma \ref{mxmlnrmlizr}. The maximality of $ V $ in $ H $ implies $ H \cap a ^{-1} H a = V $. Hence $ Ht(G, \Omega ) \geq 3 $ by Lemma \ref{sfrlwrbnd}.

\medskip
Next suppose $ | N_{G} (W) | $ is not divisible by $ 4 $. By Lemmas \ref{dcvlkjhsdp} and \ref{nrmlzr}, every normalizer of a $ C_{3} $ in $ G $ is isomorphic to $ N_{G} (W ) $ and the normalizers are conjugate to each other. Hence every $ S_{3} $ in $ G $ lies in a conjugate of $ N_{G} (W ) $.

\medskip
Let $ D $ be a conjugate of $ N_{G} (W) $. The rotational subgroup of $ D $ has odd order and so all involutions in $ D $ are conjugate. If $ R_{1} , R_{2} \leq D $ and $ R_{1} \cong R_{2} \cong S_{3} $, then $ R_{1} = \langle C , r_{1} \rangle $ and $ R_{2} = \langle C , r_{2} \rangle $ where $ C $ is generated by a rotation of order $ 3 $ in $ D $ and $ r_{1} $ and $ r_{2} $ are involutions. Now $ r_{2} = d^{-1} r_{1} d $ for some $ d \in D $. So $ R_{2} = d^{-1} R_{1} d $. Thus all $ S_{3} $ are conjugate in $ D $ and it follows that all $ S_{3} $ in $ G $ are conjugate to each other.

\medskip
The only way $ 4 $ does not divide $ | N_{G} (W) | = 2(p \pm 1) / \delta $ is if $ \delta = 2 $, that is if $ G = PSL_{2} (p) $. In this case $ G $ has two conjugacy classes of maximal $ S_{4} $, say $ \Omega $ and $ \Omega ^{*} $. Since $ H \in \Omega $ we have $ H \notin \Omega ^{*} $. The fact all $ S_{3} $ are conjugate in $ G $ means we can pick any subgroup $ H^{*} \in \Omega ^{*} $ and if it does not contain $ V $ then some conjugate of $ H^{*} $ will. So $ Ht(G, \Omega ) \geq 3 $ by Lemma \ref{sfrlwrbnd}.
\end{proof}
The next step in determining the height of the actions is proving there cannot exist an independent set of size $ 4 $. If such a set did exist then there are some restrictions what groups the intersection of point stabilizers can be, as we now see.
\begin{lema}\label{afrntrsctsfr}
Suppose $ G^{ \prime } $ is a group properly containing an $ S_{4} $ subgroup. Let $ \Omega ^{ \prime } $ be the set of right cosets of such an $ S_{4} $. If $ \Delta $ is an independent set of size $ 4 $ then $ G^{ \prime } _{ \delta _{1} , \delta _{2} } \not\cong A_{4} $ for $ \delta _{1} , \delta _{2} \in \Delta $.
\end{lema}
\begin{proof}
Write $ \Delta = \{ \delta _{1} , \delta _{2} , \delta _{3} , \delta _{4} \} $. Suppose $ G^{ \prime } _{ \delta _{1} , \delta _{2} } \cong A_{4} $.

\medskip
It follows from Lemma \ref{stach} that $ G^{ \prime } _{ \delta _{i} , \delta _{j} } > G^{ \prime } _{ \delta _{1} , \delta _{2} , \delta _{3} }  > G_{ ( \Delta ) } $ for $ i , j \in \{ 1,2,3 \} $, therefore $ G^{ \prime } _{ \delta _{1} , \delta _{2} , \delta _{3} } $ is either a $ C_{3} $, $ C_{2} $ or $ K_{4} $.

\medskip
First suppose $ G^{ \prime } _{ \delta _{1} , \delta _{2} , \delta _{3} } \cong C_{3} $. Then $ G^{ \prime } _{ \delta _{1} , \delta _{3} } $ and $ G^{ \prime } _{ \delta _{2} , \delta _{3} } $ are either $ S_{3} $ or $ A_{4} $ subgroups of $ G^{ \prime } _{ \delta _{3} } $. If either was an $ A_{4} $ then they would be equal to $ G^{ \prime } _{ \delta _{1} , \delta _{2} } $, which is ruled out by Lemma \ref{indsbstsntqual}. So they must both be isomorphic to $ S_{3} $. Each $ C_{3} $ in an $ S_{4} $ is contained in exactly one $ S_{3} $ subgroup. Hence $ G^{ \prime } _{ \delta _{1} , \delta _{3} } = G^{ \prime } _{ \delta _{2} , \delta _{3} } $. Lemma \ref{indsbstsntqual} again shows this is a contradiction. So $ G^{ \prime } _{ \delta _{1} , \delta _{2} , \delta _{3} } \not\cong C_{3} $.

\medskip
Next suppose $ G^{ \prime } _{ \delta _{1} , \delta _{2} , \delta _{3} } \cong K_{4} $. By Lemma \ref{stach} we have $ G^{ \prime } _{ \delta _{1} , \delta _{2}  }  > G^{ \prime } _{ \delta _{1} , \delta _{2} , \delta _{4} } >  G_{ ( \Delta ) }  $. Hence $ G^{ \prime } _{ \delta _{1} , \delta _{2} , \delta _{4} } $ is either a $ K_{4} $ or $ C_{2} $ (with $ C_{3} $ ruled out the same way as the above case). Lemma \ref{indsbstsntqual} prevents the possibility of $ G^{ \prime } _{ \delta _{1} , \delta _{2} , \delta _{4} } $ being a $ K_{4} $. Also if $ G^{ \prime } _{ \delta _{1} , \delta _{2} , \delta _{4} } \cong C_{2} $ then $ G^{ \prime } _{ \delta _{1} , \delta _{2} , \delta _{4} } < G^{ \prime } _{ \delta _{1} , \delta _{2} , \delta _{3} } $. This means $ G^{ \prime } _{ \delta _{1} , \delta _{2} , \delta _{4} } = G^{ \prime } _{ \delta _{1} , \delta _{2} , \delta _{3} , \delta _{4} } $, again giving a contradiction by the same lemma.

\medskip
The remaining case is $ G^{ \prime } _{ \delta _{1} , \delta _{2} , \delta _{3} } \cong C_{2} $. As in the previous case we cannot have $ G^{ \prime } _{ \delta _{1} , \delta _{3} } \leq G^{ \prime } _{ \delta _{1} , \delta _{2} } $, otherwise $ G^{ \prime } _{ \delta _{1} , \delta _{2} } = G^{ \prime } _{ \delta _{1} , \delta _{2} , \delta _{3} } $. Also $ G^{ \prime } _{ \delta _{1} , \delta _{3} } > G^{ \prime } _{ \delta _{1} , \delta _{2} , \delta _{3} } $. In $ S_{4} $, the $ C_{2} $ inside the $ A_{4} $ subgroup do not lie in any $ S_{3} $ subgroups. Therefore $ G^{ \prime } _{ \delta _{1} , \delta _{3} } $ is either a $ D_{8} $ or a non-normal $ K_{4} $ in $ G^{ \prime } _{ \delta _{1} } $. If $ G^{ \prime } _{ \delta _{1} , \delta _{3} } \cong D_{8} $ then it contains the normal $ K_{4} $ in $ G^{ \prime } _{ \delta _{1} , \delta _{2} } $ and it follows that $ G^{ \prime } _{ \delta _{1} , \delta _{2} , \delta _{3} } \cong K_{4} $, a contradiction. Hence $ G^{ \prime } _{ \delta _{1} , \delta _{3} } \not\cong D_{8} $. Thus $ G^{ \prime } _{ \delta _{1} , \delta _{3} } $ is a non-normal $ K_{4} $ in $ G^{ \prime } _{ \delta _{1} } $.

\medskip
Following the reasoning from the above cases we also have $ G^{ \prime } _{ \delta _{1} , \delta _{2} , \delta _{4} } \cong C_{2} $. Also $ G^{ \prime } _{ \delta _{1} , \delta _{4} } $ is a non-normal $ K_{4} $ in $ G^{ \prime } _{ \delta _{1} } $.

\medskip
Let $ N $ be the normal $ K_{4} $ in $ G^{ \prime } _{ \delta _{1} } $. Then looking at the subgroup structure of $ S_{4} $ we have $ G^{ \prime } _{ \delta _{1} , \delta _{3} } \cap N = G^{ \prime } _{ \delta _{1} , \delta _{2} , \delta _{3} } $ and $ G^{ \prime } _{ \delta _{1} , \delta _{4} } \cap N = G^{ \prime } _{ \delta _{1} , \delta _{2} , \delta _{4} } $.

\medskip
Since $ G^{ \prime } _{ \delta _{1} , \delta _{3}  }  > G^{ \prime } _{ \delta _{1} , \delta _{3} , \delta _{4} } > G^{ \prime } _{ \delta _{1} , \delta _{2} , \delta _{3} , \delta _{4} } \geq \{ 1_{G^{ \prime } } \} $ we have $ G^{ \prime } _{ \delta _{1} , \delta _{3} , \delta _{4} } \cong C_{2} $. Using Lemma \ref{indsbstsntqual} once again we see that $ G^{ \prime } _{ \delta _{1} , \delta _{2} , \delta _{3} } \neq G^{ \prime } _{ \delta _{1} , \delta _{3} , \delta _{4} } \neq G^{ \prime } _{ \delta _{1} , \delta _{2} , \delta _{3} } $. In particular, $ G^{ \prime } _{ \delta _{1} , \delta _{3} , \delta _{4} } $ is not a subgroup of $ N $.

\medskip
In $ S_{4} $, if a $ C_{2} $ subgroup is not contained in the normal $ K_{4} $, then it is contained in exactly one non-normal $ K_{4} $ subgroup. Therefore $ G^{ \prime } _{ \delta _{1} , \delta _{3} } =  G^{ \prime } _{ \delta _{1} , \delta _{4} } $ because $ G^{ \prime } _{ \delta _{1} , \delta _{3} , \delta _{4} } $ is a subgroup of both of these $ K_{4} $. This is a contradiction by Lemma \ref{indsbstsntqual}. Hence the original assumption $ G^{ \prime } _{ \delta _{1} , \delta _{2} } \cong A_{4} $ is wrong.
\end{proof}

\newpage
\begin{lema}\label{cntthnkfnm}
Suppose $ G^{ \prime } $ is a group properly containing an $ S_{4} $ subgroup. Let $ \Omega ^{ \prime } $ be the set of right cosets of such an $ S_{4} $. Suppose for each $ \omega \in \Omega ^{ \prime } $ that if $ K_{ \omega } $ is the normal $ K_{4} $ in $ G^{ \prime } _{ \omega } $ then $ N_{ G^{ \prime } } ( K_{ \omega } ) = G^{ \prime } _{ \omega } $. If there exists an independent set $ \Delta \subseteq \Omega ^{ \prime } $ of size $ 4 $ then $  G^{ \prime } _{( \Delta ) } = \{ 1_{G^{ \prime } } \} $ and $ | G^{ \prime } _{ \delta _{1} , \delta _{2} , \delta _{3} } | = 2 $ for $ \delta _{1} , \delta _{2} , \delta _{3}  \in \Delta $.
\end{lema}
\begin{proof}
Write $ \Delta ^{*} := \{ \delta _{1} , \delta _{2} , \delta _{3} \}  \subset \Delta $. Looking at the subgroup structure of $ S_{4} $ and using Lemma \ref{stach}, we have $ 1 \leq | G^{ \prime } _{( \Delta  ) } | < | G^{ \prime } _{( \Delta ^{*} ) } | \leq 4 $.

\medskip
If $  G^{ \prime } _{( \Delta ^{*} ) } \cong C_{4} $ then it follows from Lemma \ref{stach} that $ G^{ \prime } _{ \delta _{1} , \delta _{2} } \cong G^{ \prime } _{ \delta _{1} , \delta _{3} } \cong D_{8} $. For each $ C_{4} $ in an $ S_{4} $, there is only one $ D_{8} $ subgroup that contains it. Hence $  G^{ \prime } _{ \delta _{1} , \delta _{2} } = G^{ \prime } _{ \delta _{1} , \delta _{3} } $. This contradicts Lemma \ref{indsbstsntqual}, therefore $  G^{ \prime } _{( \Delta^{*}  ) } \not\cong C_{4} $.

\medskip
Suppose $  G^{ \prime } _{( \Delta ^{*} ) } \cong C_{3} $. Notice $ G^{ \prime } _{ \delta _{1} , \delta _{2} } \not\cong A_{4} \not\cong G^{ \prime } _{ \delta _{1} , \delta _{3} } $ by Lemma \ref{afrntrsctsfr}. Hence $ G^{ \prime } _{ \delta _{1} , \delta _{2} } \cong G^{ \prime } _{ \delta_{1} , \delta_{3} } \cong S_{3} $. As above, there is only one $ S_{3} $ subgroup that contains a particular $ C_{3} $ in an $ S_{4} $. So $  G^{ \prime } _{ \delta _{1} , \delta _{2} } = G^{ \prime } _{ \delta _{1} , \delta _{3} } $, again giving a contradiction.

\medskip
Suppose $ G^{ \prime } _{( \Delta ^{*} ) } \cong K_{4} $. If $ G^{ \prime } _{( \Delta ^{*} ) } $ is normal in both $ G_{ \delta _{1} }  $ and $ G_{ \delta _{2} }  $ then $ G_{ \delta _{1} }  = N_{G^{ \prime } } ( G^{ \prime } _{( \Delta ^{*} ) }  ) =   G_{ \delta _{2} }  $. Lemma \ref{indsbstsntqual} shows this is not possible. So we can assume without loss of generality that $ G^{ \prime } _{( \Delta ^{*} ) } $ is not normal in $ G_{ \delta _{1} } $. Looking at the subgroup structure of $ S_{4} $, for a non-normal $ K_{4} $ there is only one proper subgroup properly containing it, which is a $ D_{8} $. So by Lemma \ref{stach} we have $ G_{ \delta _{1} , \delta _{2} } \cong G_{ \delta _{1} , \delta _{3} } \cong D_{8} $ and it follows $ G_{ \delta _{1} , \delta _{2} } = G_{ \delta _{1} , \delta _{3} } $. Once again this is a contradiction.

\medskip
The only remaining possibility is $ G^{ \prime } _{( \Delta ^{*} ) } \cong C_{2} $, implying $ G^{ \prime } _{( \Delta  ) } $ is trivial.
\end{proof}
\begin{lema}\label{sdhssdfone}
Suppose $ G^{ \prime } $ is a group properly containing an $ S_{4} $ subgroup. Let $ \Omega ^{ \prime } $ be the set of right cosets of such an $ S_{4} $. Suppose for each $ \omega \in \Omega ^{ \prime } $ that if $ K_{ \omega } $ is the normal $ K_{4} $ in $ G^{ \prime } _{ \omega } $ then $ N_{ G^{ \prime } } ( K_{ \omega } ) = G^{ \prime } _{ \omega } $. If there exists an independent set $ \{ \omega _{1} , \omega _{2} , \omega _{3} , \omega _{4} \} \subseteq \Omega ^{ \prime } $ then one of the following is true;
\begin{itemize}
\item $ \langle G^{ \prime }_{\omega _{1}} , G^{ \prime }_{\omega _{2}} , G^{ \prime }_{\omega _{3}} , G^{ \prime }_{\omega _{4}} \rangle \cong PGL_{2} (5) $,
\item $ \langle G^{ \prime }_{\omega _{1}} , G^{ \prime }_{\omega _{2}} , G^{ \prime }_{\omega _{3}} , G^{ \prime }_{\omega _{4}} \rangle \cong PSL_{2} (7) $,
\item $ \langle G^{ \prime }_{\omega _{1}} , G^{ \prime }_{\omega _{2}} , G^{ \prime }_{\omega _{3}} , G^{ \prime }_{\omega _{4}} \rangle $ has order dividing $ 192 $ or
\item $  G^{ \prime }_{\omega _{1} , \omega _{2}} \cong K_{4} $, as well as $ G^{ \prime }_{\omega _{1} , \omega _{3}} \cong G^{ \prime }_{\omega _{1} , \omega _{4}} \cong G^{ \prime }_{\omega _{2} , \omega _{3}} \cong G^{ \prime }_{\omega _{2} , \omega _{4}} \cong S_{3} $ and $ G^{ \prime }_{\omega _{3} , \omega _{4}} $ is isomorphic to either $ K_{4} $ or $ D_{8} $ (or an equivalent statement where the points have been permuted).
\end{itemize}
\end{lema}
\begin{proof}
See Appendix \ref{appa}.
\end{proof}
Now we return to looking at the $ S_{4} $ actions of $ PSL_{2} (p) $ or $ PGL_{2} (p) $.
\begin{lema}\label{sdhssdftwo}
Let $ \{ H_{1} , H_{2} , H_{3} , H_{4} \} \subseteq \Omega $. Suppose $ H_{1} \cap H_{2} \cong K_{4} $ and $ H_{1} \cap H_{3}  \cong H_{1} \cap H_{4} \cong H_{2} \cap H_{3} \cong H_{2} \cap H_{4} \cong S_{3} $. Then $ \{ H_{1} , H_{2} , H_{3} , H_{4} \} $ is not independent.
\end{lema}
\begin{proof}
Suppose $ \{ H_{1} , H_{2} , H_{3} , H_{4} \} $ is independent.

\medskip
Let $ a \in H_{1} \cap H_{2} $ be an involution. Let $ N_{1} $ and $ N_{2} $ be the normal $ K_{4} $ subgroups of $ H_{1}  $ and $ H_{2} $ respectively. If $ a \in N_{1} \cap N_{2} $ then we see from looking at the structure of $ S_{4} $ that $ a $ is centralized by a cyclic subgroup $ F_{1} $ of order $ 4 $ in $ H_{1} $ and similarly a cyclic subgroup $ F_{2} $ of order $ 4 $ in $ H_{2} $. Lemma \ref{nrmlzr} tells us the centralizer of an involution in $ G $ is dihedral, so only contains one cyclic subgroup of order $ 4 $. Hence $ F_{1} = F_{2} $ and $ F_{1} \leq H_{1} \cap H_{2} $, a contradiction. Thus $ a \notin N_{1} \cap N_{2} $.

\medskip
Every $ K_{4} $ in an $ S_{4} $ has a an involution that lies in the normal Klein four-subgroup. So it follows that $ H_{1} \cap H_{2} $ has one involution in $ N_{1} $ and another involution in $ N_{2} $.

\medskip
Let $ j \in \{ 3,4 \} $. By Lemma \ref{stach} we have $ H_{1} \cap H_{2} > H_{1} \cap H_{2} \cap H_{j} >  H_{1} \cap H_{2} \cap H_{3} \cap H_{4} $. So it must be that $ H_{1} \cap H_{2} \cap H_{j} \cong C_{2} $.

\medskip 
Let $ h_{123} \in H_{1} \cap H_{2} \cap H_{3} $ and $ h_{124} \in H_{1} \cap H_{2} \cap H_{4} $ be involutions. Lemma \ref{indsbstsntqual} shows $ H_{1} \cap H_{2} \cap H_{3} \neq H_{1} \cap H_{2} \cap H_{4} $. Thus $ h_{123} \neq h_{124} $.

\medskip
So at least one of $  h_{123} $ or $  h_{124} $ is in $ N_{1} $ or $ N_{2} $. This is a contradiction because these involutions also lie in $ S_{3} $ subgroups of $ H_{1} $ and $ H_{2} $ and the $ S_{3} $ subgroups of an $ S_{4} $ intersect its normal $ K_{4} $ subgroup trivially. Therefore $ \{ H_{1} , H_{2} , H_{3} , H_{4} \} $ is not independent.
\end{proof}
\begin{lema}\label{sfrhghtw}
If $ G = PGL_{2} (5) $ then $ Ht(G, \Omega ) = 4 $. If $ G \neq PGL_{2} (5) $ then $ Ht(G, \Omega ) = 3 $.
\end{lema}
\begin{proof}
If $ G = PGL_{2} (5) $, notice that $ G \cong S_{5} $ and the $ S_{4} $ subgroups are point stabilizers of the natural action on $ \{ 1, 2, 3, 4, 5 \} $. Therefore $ Ht(G, \Omega ) = 4 $ by Example \ref{hghtsym}.

\medskip
Now suppose $ G \neq PGL_{2} (5) $. If $ Ht(G, \Omega ) \geq 4 $ then there exists an independent set $ \Delta \subseteq \Omega $ of size $ 4 $ by Corollary \ref{indsub}, in which case the fact the $ S_{4} $ are maximal together with Lemmas \ref{sdhssdfone} and \ref{sdhssdftwo} show $ G $ is one of $ S_{5} $, $ PSL_{2} (7) $ or has order dividing $ 192 $.

\medskip
We have already dealt with the $ S_{5} $ case above. The tables in \cite{WISCONS1} show the $ S_{4} $ action of $ PSL_{2} (7) $ has height $ 3 $ and therefore no independent set of size $ 4 $ (the reason the anomaly $ PSL_{2} (7) $ shows up in Appendix \ref{appa} is that it has two conjugacy classes of $ S_{4} $ and the GAP code did not take into consideration whether the $ S_{4} $ were conjugate or not). Finally $ | PSL_{2} (p) | $ and $ | PGL_{2} (p) | $ do not divide $ 192 $ for any $ p \geq 5 $.

\medskip
Thus $ Ht(G, \Omega ) < 4 $. Combining this with Lemma \ref{sfrhghtlwrbnd} gives $ Ht(G, \Omega ) = 3 $.
\end{proof}
\section{Relational Complexity of the \texorpdfstring{$S _{4} $}{} Action}
\begin{thrm}\label{eopwcnhg}
If $ G = PGL_{2} (5) $ then $ RC(G, \Omega ) = 2 $. Otherwise $ RC(G, \Omega ) \in \{ 3,4 \} $.
\end{thrm}
\begin{proof}
If $ G = PGL_{2} (5) \cong S_{5} $ then the action on maximal $ S_{4} $ is the natural action of $ S_{5} $ on $ \{ 1, 2, 3, 4, 5 \} $. In this case $ RC(G, \Omega ) = 2 $, as we saw in Example \ref{symrc}.

\medskip
Now suppose $ G \neq PGL_{2} (5) $. Theorem \ref{rchgt} and Lemma \ref{sfrhghtw} together tell us that $ RC(G, \Omega ) \leq 4 $. We see from the tables in \cite{WISCONS1}, if $ G = PSL_{2} (7) $ then $ RC(G, \Omega ) = 3 $. If $ p > 7 $ then \cite{GILL1}, Lemma 3.3 shows $ RC(G, \Omega )  \neq 2$.
\end{proof}
For the remainder of this chapter, we will see that for these $ S_{4} $ actions it is not possible to pin down the relational complexity to only one value in general. This is unlike most of the other primitive actions we look at for $ PSL_{2} (q) $ and $ PGL_{2} (q) $. This is investigated using a GAP program in Appendix \ref{appb}.

\medskip
The GAP program can be used to check the actions of $ PSL_{2} (7) $, $ PSL_{2} (17) $ and $ PSL_{2} (47) $ all have relational complexity $ 3 $. A look at $ PSL_{2} (79) $ gives output showing this also has relational complexity $ 4 $ (no need to run the whole calculation until the end for this group as it takes a very long time).

\medskip
Other arbitrary choices of $ p $ that I have checked with $ p \leq 79 $ give relational complexity $ 4 $. It does take GAP a while to process some of these larger values of $ p $. Note for later that, excluding $ PSL_{2} (7) $ (which I think is a special case), the groups here with relational complexity $ 3 $ correspond to $ p \equiv \pm 17 \pmod{64} $. The reason $ PSL_{2} (79) $ is mentioned above is to rule out the the possibility of the relational complexity being $3$ whenever $ p \equiv \pm 15 \pmod{32} $.

\medskip
For $ PGL_{2} (p) $, all cases I have checked so far have relational complexity $ 4 $.

\medskip
Recall from the comments directly before Definition \ref{almstndstdfn}, if $ RC( G , \Omega ) = 4 $ then there exists $ I, J \in \Omega ^{4} $ such that $ I \sim _{3} J $ and $ I \not\sim _{4} J $. There are conditions $ I $ and $ J $ must satisfy for the relational complexity to be $ 4 $ and these are laid out at the start of Appendix \ref{appb}, one of them being the entries of $ I $ form an almost independent set.

\medskip
If we pair the entries of $ I $ up, there are six pairs. The program in Appendix \ref{appb} does show the structure description of the intersections of each of these pairs if the relational complexity is $ 4 $.

\medskip
For $ PSL_{2} (23) $ we find the intersections always have a $ C_{2} $, a $ C_{3} $ and a $ C_{4} $. For $ PSL_{2} (31) $ we only have $ C_{2} $ and $ C_{3} $ in the intersections (sometimes the $ C_{3} $ appear in four intersections, sometimes only in three intersections and some times all six intersections are $ C_{3} $). 

\medskip
When looking at $ PGL_{2} (19) $ all intersections are $ C_{3} $, whereas $ PGL_{2} (13) $ has intersections consisting of $ C_{2} $, $C_{3} $ and $ C_{4} $.

\medskip
From analysing these results, even when the relational complexity is $ 4 $ it is impossible to find a consistent type of almost independent set for the entries of $ I $ that works for all choices of $ p $.

\medskip
What I suspect is happening is that when $ p $ is congruent to some number modulo another we get behaviour that allows or prevents certain types of almost independent set.

\medskip
A loosely explained example of this is now given when $ G = PSL_{2} (p) $ and $ p \equiv \pm 1 \pmod{10} $ (as well as the required $ p \equiv \pm 1 \pmod{8} $) - many of the facts in this argument will end up being used again in the next chapter where the details are fleshed out more. In this case $ G $ has a maximal $ A_{5} $ subgroup, say $ A \leq G $. In $ A $ we have five maximal $ A_{4} $ subgroups. Label them $ F_{1} , \dots , F_{5} $. Each of those $ A_{4} $ are normalized by an $ S_{4} $ in $ G $. For $ i \in \{ 1, \dots , 5 \} $, let $ H_{i} $ be the $ S_{4} $ that normalizes $ F_{i} $.

\medskip
Each of the $ A_{4} $ intersect each other pairwise in a $ C_{3} $. Hence the $ S_{4} $ containing them intersect pairwise in a $ C_{3} $ or $ S_{3} $. If $ H_{j} \cap H_{k} \cong S_{3} $ for some $ j,k \in \{ 1,2,3,4 \} $ then there exists an involution $ h \in H_{j} \cap H_{k} $ that normalizes $ F_{j} $ and $ F_{k} $. Now $ \langle F_{j} , F_{k} \rangle = A $, so it is straightforward to show $ s \in A $. But $ H_{j} = \langle F_{j} , h \rangle $ and so $ H_{j} \leq A $. There is no $ S_{4} $ in $ A $, so we have a contradiction. It follows $ H_{j} \cap H_{k} \not\cong S_{3} $. Hence $ H_{j} \cap H_{k} = F_{j} \cap F_{k} $. From here it can be checked that $ \{ H_{1} , H_{2} , H_{3} , H_{4} \} $ is an almost independent set.

\medskip
By looking at elements of $ A_{5} $, for $ r, s \in \{ 1, 2,3 \} $ with $ r \neq s $ there exists an element of order $ 3 $ in $ F_{r} \cap F_{s} $ (and thus in $ H_{r} \cap H_{s} $) that sends $ F_{4} $ to $ F_{5} $ by conjugation (and sends $ H_{4} $ to $ H_{5} $). It is not too difficult to then show that the tuples $ (H_{1} , H_{2} , H_{3} , H_{4} ) $ and $ (H_{1} , H_{2} , H_{3} , H_{5} ) $ are $ 3 $-subtuple complete, but not $ 4 $-subtuple complete. From this we infer the relational complexity of the action is $ 4 $.

\medskip
The above example was not too complex, but it only dealt with certain values of $ p $. Each of the double intersection of $ S_{4} $ in the example gave us a $ C_{3} $. However we have seen that there are almost independent sets with different pairwise intersections. To deal with all values of $ p $, we would need to determine all almost independent sets that can potentially exist, find which of these sets can suitably be used for the for the entries of a $4$-tuple $ I $ and then see which values of $ p $ each type of almost independent set matches to. Only once this is completely done for all possible configurations of almost independent set would we be able to rule out the values of $ p $ where no such set exist, then narrow down which values of $ p $ the relational complexity is $ 3 $ and when it is $ 4 $.

\medskip
Unfortunately, using existing methods this is going to descend into a lot of long case work. The amount of time and page space required is unsuitable for this paper, so based on the evidence we have, a conjecture is made on the $ S_{4} $ actions of $PGL _{2} (p) $. The group $ PGL_{2} (5) $ is missing from the statement below as we have already seen it has relational complexity $ 2 $ and I think it is likely a special case that does not fit the general pattern, as is often the case for small values of $ p $.
\begin{cjct}
Suppose $ p > 7 $. If $ G= PGL_{2} (p) $ then $ RC(G, \Omega ) = 4 $.
\end{cjct}
For $PSL_{2} (p)$ the situation is less clear. From the cases looked at with GAP so far, it might be tempting to conjecture that if $ p > 7 $ and $ p \equiv \pm 17 \pmod{64} $ then the $S_{4}$ actions of $PSL_{2} (p)$ have relational complexity $3$, with the relational complexity being $ 4$ in other cases. However, the number $ 239 $ is prime. Also $ 239 \equiv - 1 \pmod{8} $ and $ 239 \equiv -1 \pmod{10} $. So $ PSL_{2} (239) $ has maximal $ S_{4} $ and maximal $A_{5} $ subgroups, which means the relational complexity of the $ S_{4} $ actions is $ 4 $ by above discussions. Since $ 239 = 4 \times 64 - 17 $, this shows that there are cases where the relational complexity is not $ 3 $ when $ p \equiv \pm 17 \pmod{64} $.

\newpage
\chapter{The \texorpdfstring{$ A_{5} $}{A5} Action}\label{chapnine}
Here the $ A_{5} $ actions of $ PSL_{2} (q) $ will be looked at, defined in the section below. Suppose $ G $ is $ PSL_{2} (q) $ acting on $ \Omega $, the right cosets of a maximal $ A_{5} $. The main results of this chapter are:

\medskip
If $ G = PSL_{2} (9) $ then $ Ht(G, \Omega ) = 4 $. If $ G \neq PSL_{2} (9) $ then $ Ht(G, \Omega ) = 3 $.

\medskip
If $ G = PSL_{2} (9) $ then $ RC(G, \Omega ) = 5 $. Otherwise $ RC(G, \Omega ) = 4 $.

\medskip
These are proved in Theorems \ref{lblfrstthrem} and \ref{dfapoivqp}. There is no preliminary results section for this chapter.
\section{Action Description}\label{sctionafv}
Throughout this section and the next let $p$ be an odd prime with $ p \geq 5 $. In \cite{BRAYHOLTRONEYDOUGAL}, Table 8.7, page 380, we see that there exists maximal subgroups isomorphic to $ A_{5} $ in $ PSL_{2} (q) $ when either
\begin{itemize}
\item $ q = p $ and $ p \equiv \pm 1 \pmod{10} $ or
\item $ q = p^{2} $ and $ p \equiv \pm 3 \pmod{10} $.
\end{itemize}
In each case there exists two conjugacy classes of maximal $ A_{5} $.

\medskip
Let $ G $ be $ PSL_{2} (q) $ with $ q $ satisfying the above conditions. Suppose $ H \leq G $ is a maximal $ A_{5} $ subgroup. Put $ \Omega := \{ g^{-1} H g : g \in G \} $. It follows from Lemma \ref{cstquivtcnj} that rather than  describing the action in terms of right cosets of a maximal $ A_{5} $, we can look at the action by conjugation on $ \Omega $ since the two are equivalent.

\medskip
For the rest of this chapter let $ p $, $ q $, $ G $, $ H $ and $ \Omega $ be as defined above.
\section{Height and Relational Complexity of the \texorpdfstring{$A _{5} $}{} Action}
We begin the section by dealing with the $ A_{5} $ action of $ PSL_{2} (9) $ because this is a special case that does not fit the general pattern.
\begin{thrm}\label{lblfrstthrem}
For $ G := PSL_{2} (9) $ we have $ Ht(G, \Omega) = 4 $ and $ RC(G , \Omega ) = 5 $.
\end{thrm}
\begin{proof}
Since $ PSL_{2} (9) \cong A_{6} $, the action on $ A_{5} $ is the natural action of $ A_{6} $ on $ \{ 1, 2, \dots , 6 \} $. Examples \ref{hghtalt} and \ref{rcgrpaltg} show $ Ht(G, \Omega) = 4 $ and $ RC(G , \Omega ) = 5 $.
\end{proof}
From now on assume $ q \geq 11 $. The general plan is to show the height is $ 3 $. This will be done by finding an upper bound of $ 4 $ for the height, then ruling out $ 4 $ as a possibility, followed by an example of an independent set of size $ 3 $. That same independent set will be key to finding the relational complexity of the action.

\newpage
GAP will play a large part in determining the upper bound for height, similar to the approach used in the $ S_{4} $ chapter. Various assumptions are needed to prevent the code taking a long time to run. To put those assumptions on solid ground, several lemmas on the intersection of stabilizers are now proved.
\begin{lema}\label{afvactntrvl}
Suppose $ G^{ \prime } $ is a group properly containing an $ A_{5} $ subgroup. Let $ \Omega ^{ \prime } $ be the set of right cosets of such an $ A_{5} $. If there exists an independent set $ \Delta \subseteq \Omega ^{ \prime } $ of size $ 4 $ then $ G^{ \prime } _{ ( \Delta ) } = \{ 1_{  G^{ \prime } } \} $.
\end{lema}
\begin{proof}
Let $ \delta _{1} , \delta _{2} , \delta _{3} \in \Delta $, with no two of these elements equal to each other. Using Lemma \ref{stach} and looking at the subgroup structure of $ A_{5} $, the only way that $ G^{ \prime } _{ ( \Delta ) } \neq \{ 1_{  G^{ \prime } } \} $ is if $ G^{ \prime } _{ \delta _{1} , \delta _{2} , \delta _{3} } \cong K_{4} $ and $ G^{ \prime } _{ \delta _{1} , \delta _{2}  } \cong G^{ \prime } _{ \delta _{1} , \delta _{3}  } \cong G^{ \prime } _{ \delta _{2} , \delta _{3}  } \cong A_{4} $.

\medskip
For each $ K_{4} $ subgroup of $ A_{5} $ there is exactly one proper subgroup of $ A_{5} $ that properly contains a particular $ K_{4} $, which is an $ A_{4} $. This means that  $ G^{ \prime } _{ \delta _{1} , \delta _{2}  } = G^{ \prime } _{ \delta _{1} , \delta _{3}  } $. This is a contradiction by Lemma \ref{indsbstsntqual}. Hence $ G^{ \prime } _{ ( \Delta ) } = \{ 1_{  G^{ \prime } } \} $.
\end{proof}
The above lemma tells us the height of any group acting on an $ A_{5} $ is at most $ 4 $. Most of the rest of this section will be devoted to showing there does not exist an independent set of size $ 4 $.
\begin{lema}\label{afrnrmlzbysfr}
Suppose $ A \leq G $ and $ A \cong A_{4} $. If $ q \equiv \pm 3 \pmod{8} $ then $ N_{G} (A) \cong A_{4} $. If $ q \equiv \pm 1 \pmod{8} $ then $ N_{G} (A) \cong S_{4} $. Also, if $ K \leq A $ and $ K \cong K_{4} $ then $ N_{G} (K) = N_{G} (A) $.
\end{lema}
\begin{proof}
Suppose $ q \equiv \pm 3 \pmod{8} $. If $ q \neq p $ then $ q = p^{2} $. But $ p $ is odd and so $ p^{2} \equiv 1 \pmod{8} $, a contradiction. Thus $ q = p $.

\medskip
As $ G $ is simple, $ N_{G} (A) $ lies inside a maximal subgroup. Table \ref{tablethree} tells us the maximal subgroups of $ PSL_{2} (p) $ are either Borel, $ D_{p \pm 1} $ or $ A_{5} $. Borel subgroups cannot contain a $ K_{4} $ by Lemma \ref{brlklnfrsbgrp}, so cannot contain an $ A_{4} $. Neither can dihedral groups. So $ N_{G} (A) $ is in an $ A_{5} $ and it must be $ N_{G} (A) \cong A_{4} $. If $ K \leq A $ and $ K \cong K_{4} $ then $ K $ is normal in $ N_{G} (A) $. Since $ G $ is simple, $ N_{G} (K) $ is in a maximal subgroup of $ G $. By the above reasoning, that maximal subgroup is an $ A_{5} $. In $ A_{5} $ all $ K_{4} $ are normalized by $ A_{4} $. Thus $ N_{G} (K) = N_{G} (A) $.

\medskip
Next suppose $ q \equiv \pm 1 \pmod{8} $. The Klein four-subgroups of $ G $ normalize the involutions they contain, so by Lemma \ref{nrmlzr} are subgroups of some $ D _{m} $, where $ m = q \pm 1 $. Observe $ m $ must be divisible by $ 8 $. Lemma \ref{refconjug} tells us that each $ K_{4} $ in $ D _{m} $ has two reflections that are conjugate. Since the reflections in $ D_{m} $ are split into two conjugacy classes, the $ K_{4} $ are split into two conjugacy classes in $ D_{m} $. So the $ K_{4} $ in $ G $ are split into at most two conjugacy classes.

\medskip
Now consider what $ N_{G} ( A ) $ might be. This will be covered in two cases.

\medskip
First suppose $ q = p $. Then looking at the maximal subgroups of $ PSL_{2} (p) $ in Table \ref{tablethree} we see there exist two conjugacy classes of maximal $ S_{4} $. A maximal $ S_{4} $ is the normalizer in $ G $ of the normal $ K_{4} $ it contains, because $ PSL_{2} (p) $ is simple. So there are two conjugacy classes of $ K_{4} $ in $ G $, corresponding to the conjugacy classes of $ S_{4} $ that normalize them. Therefore there is exactly one $ A_{4} $ that normalizes each $ K_{4} $. Any subgroup that normalizes an $ A_{4} $ must also normalize the $ K_{4} $ inside it. Thus $ N_{G} ( A ) \cong S_{4} $.

\medskip
Next suppose $ q \neq p $. Then $ q = p^{2} $. Table \ref{tablethree} shows $ G $ has two conjugacy classes of maximal subgroups isomorphic to $ PGL_{2} (p) $. Note that $ p \geq 7 $ because we are assuming $ q \geq 11 $ and $ q \equiv \pm 1 \pmod{10} $.

\medskip
If $ p \equiv \pm 3 \pmod{8} $ then, by Table \ref{tableone}, $ PGL_{2} (p) $ has a conjugacy class of maximal $ S_{4} $. If $ p \equiv \pm 1 \pmod{8} $ then $ PSL_{2} (p) $ has two conjugacy classes of maximal $ S_{4} $ and $ PGL_{2} (p) $ has a subgroup isomorphic to $ PSL_{2} (p) $. Since $ 8 $ divides $ m $, the $ K_{4} $ in $ A $ is a normal subgroup of some $ D_{8} $ in a $ D_{m} $. Looking through the maximal subgroups of $ G $ and $ PGL_{2} (p) $ in Tables \ref{tableone} and \ref{tablethree}, we see the only subgroup that normalizes a $ K_{4} $ and contains a $ D_{8} $ and $ A_{4} $ is one of these $ S_{4} $ in a maximal $ PGL_{2} (p) $. Hence $ N_{G} (A ) \cong S_{4} $ as well in this case.
\end{proof}

\newpage
\begin{corl}\label{crlkfrnrmlz}
If $ H_{1} , H_{2} $ are distinct subgroups in $ \Omega $ and $ H_{1} \cap H_{2} $ contains a $ K_{4} $ then $ H_{1} \cap H_{2} \cong A_{4} $.
\end{corl}
\begin{proof}
In $ A_{5} $, each $ K_{4} $ is normalized by a maximal $ A_{4} $. By Lemma \ref{afrnrmlzbysfr}, only one $ A_{4} $ in $ G $ normalizes a particular $ K_{4} $. Hence $ H_{1} \cap H_{2} \cong A_{4} $.
\end{proof}
If an independent set of size $ 4 $ exists, then Lemma \ref{stach} tells us the point stabilizer of any subset of size $ 3 $ must be non-trivial. The GAP code used later will look at which groups could possibly contain an independent set of size $ 4 $ when acting on an $ A_{5} $ subgroup. The code became too complex if subsets of size $ 3 $ have point stabilizers of order more than $ 2 $. So it will be convenient to show this is not possible for the $ A_{5} $ action of $ PSL_{2} (q) $, which the next few lemmas work toward.
\begin{lema}\label{rdrelmnts}
Suppose $ s , a \in A_{5} $ with $ o(s) = 2 $ and $ o(a) = 3 $. Then one of the following holds;
\begin{itemize}
\item $ o(sa) = 3 $ and $ \langle s , a \rangle \cong A_{4} $,
\item $ o(sa) = 2 $ and $ \langle s , a \rangle \cong S_{3} $ or
\item $ o(sa) = 5 $ otherwise.
\end{itemize}
\end{lema}
\begin{proof}
If $ s $ and $ a $ lie in some $ A_{4} $ subgroup of $ A_{5} $, then, since $ \langle a \rangle \cong C_{3} $ and each $ C_{3} $ is maximal in $ A_{4} $, we have $ \langle s , a \rangle \cong A_{4} $.

\medskip
Now suppose $ s $ and $ a $ are not in an $ A_{4} $ subgroup of $ A_{5} $. Since each $ A_{4} $ subgroup is the stabilizer some point in $ \{ 1 , 2 , 3 , 4 , 5 \} $, it must be that $ s $ and $ a $ do not stabilize any point in common. So the elements have the form $ s = ( x_{1} \ x_{2} )( x_{3} \ x_{4} ) $ and either $ a = ( x_{1} \ x_{2} \ x_{5} ) $ or $ a = ( x_{1} \ x_{3} \ x_{5} ) $ where $ \{ x_{1} , x_{2} , x_{3} , x_{4} , x_{5} \} = \{ 1 , 2, 3, 4 , 5 \}  $.

\medskip
If $ a = ( x_{1} \ x_{2} \ x_{5} ) $ then $ sa = ( x_{2} \ x_{5} )( x_{3} \ x_{4} ) $ and it is not difficult to show $ \langle s , a \rangle \cong S_{3} $.

\medskip
If $ a = ( x_{1} \ x_{3} \ x_{5} ) $ then $ sa = ( x_{1} \ x_{2} \ x_{3} \ x_{4} \ x_{5} ) $ and $ o(sa) = 5 $.
\end{proof}
\begin{lema}\label{afrsmthdsk}
Suppose $ \{ H_{1} , H_{2} , H_{3} \} \subseteq \Omega $ is an independent set and $ H_{1} \cap H_{2} \cap H_{3} \cong C_{3} $. Then $ H_{i} \cap H_{j} \cong A_{4} $ for distinct $ i,j \in \{ 1,2,3 \} $.
\end{lema}
\begin{proof}
Put $ \Delta := \{ H_{1} , H_{2} , H_{3} \} $. Write $ H_{ij} $ for $ H_{i} \cap H_{j} $. Similarly $ H_{123} $ for $ H_{1} \cap H_{2} \cap H_{3} $. Looking at the subgroup structure of $ A_{5} $ and using Lemma \ref{stach}, either $ H_{ij} \cong S_{3} $ or $ H_{ij} \cong A_{4} $.

\medskip
In $ A_{5} $ there is only one $ S_{3} $ that contains a particular subgroup isomorphic to $ C_{3} $. So if two or more of $ H_{12} $, $ H_{13} $ or $ H_{23} $ are isomorphic to $ S_{3} $, at least two of these double intersections are equal. Lemma \ref{indsbstsntqual} then tells us that $ \Delta $ is not independent, a contradiction. So at most one double intersection is an $ S_{3} $.

\medskip
Suppose one of the double intersections is an $ S_{3} $. Without loss of generality suppose $ H_{12} \cong S_{3} $. Then $ H_{13} \cong H_{23} \cong A_{4} $. Note $ H_{13} \neq H_{23} $ by Lemma \ref{indsbstsntqual}.

\medskip
Let $ B \leq H_{3} $ with $ B \cong A_{4} $ and $ H_{13} \neq B \neq H_{23} $.

\medskip
Put $ C:= H_{123} = ( H_{13} ) \cap ( H_{23} ) $. In $ A_{5} $ each pair of distinct $ A_{4} $ subgroups intersect in a $ C_{3} $. Also each $ C_{3} $ subgroup lies in exactly two $ A_{4} $. Hence $ C \not\leq B $ and $ B \cap H_{13} \neq B \cap H_{23} $.

\medskip
For $ k \in \{ 1,2 \} $ there exists $ F_{k} \leq H_{k} $ with $ F_{k} \cong A_{4} $ and $ F_{k} \cap H_{k3} = B \cap H_{k3} \cong C_{3} $. Note for later this implies $ F_{k} \neq H_{k3} $. Since $ C_{3} $ are maximal in $ A_{4} $, either $ B \cap F_{k} \cong A_{4} $ or $ B \cap F_{k} \cong C_{3} $. It cannot be that $ B \cap F_{k} \cong A_{4} $ otherwise $ B = F_{k} $, implying $ B $ is a subgroup of $ H_{k} $ and $ H_{3} $, so $ H_{k3} = B \neq H_{k3} $. It follows that $ B \cap F_{k} \cong C_{3} $ and $ B \cap F_{k} = B \cap H_{k3} $.

\medskip
Let $ s \in H_{12} $ be an involution. Let $ h_{13} \in B \cap H_{13} $ and $ h_{23} \in B \cap H_{23} $ be elements of order $ 3 $. Let $ c \in C $ with $ o(c) = 3 $. We will look at what group is generated by $c $, $ s $, $ h_{13} $ and $ h_{23} $. First note that
\begin{itemize}
\item $ o(cs ) = 2 $.
\end{itemize}
Now consider how $ c$, $ h_{13} $ and $ h_{23} $ interact with each other. In $ A_{4} $ the elements of order three split into two conjugacy classes, one containing the inverses of the other. So, by relabelling if needed, we can assume $ c $ and $ h_{13} $ lie in the same conjugacy class in $ H_{13} $. Also we can assume $ h_{13} $ and $ h_{23} $ lie in the same conjugacy class in $ B $. Looking at how elements combine in Table \ref{afrcylytbl} we see that
\begin{itemize}
\item $ o(ch_{13} ) = 3 $,
\item $ o(c h_{13} h_{13} ) = 2 $,
\item $ o(h_{13} h_{23} ) = 3 $ and
\item $ o(h_{13} h_{23} h_{23} ) = 2 $.
\end{itemize}
In $ H_{3} $ the subgroups $ H_{13} $ and $ B $ are conjugate. In $ A_{4} $ any two $ C_{3} $ are conjugate. So we can assume $ B = g_{13} ^{-1} H_{13} g_{13} $ and $ g_{13} ^{-1}  \langle h_{13} \rangle g_{13} = \langle h_{13} \rangle $ for some $ g_{13} \in H_{3} $. Further conjugating $ B $ by $ h_{13} $ if needed, we can also assume $ g_{13} ^{-1}  C g_{13} \neq \langle h_{23} \rangle $. Now $ g_{13} \in N_{H_{3} } ( \langle h_{13} \rangle ) \setminus \langle h_{13} \rangle $. The normalizer of a $ C_{3} $ in $ A_{5} $ is isomorphic to $ S_{3} $. So $ g_{13} $ is an involution that sends $ h_{13} $ to $ h_{13} ^{-1} $. Let $ B_{1} $ and $ B_{2} $ be the conjugacy classes of elements of order $ 3 $ in $ B $. Let $ B_{1} $ be the set containing $ h_{13} = g_{13} ^{-1} h_{13} ^{-1} g_{13} $ (and therefore $ h_{23} $), as well as $ g_{13} ^{-1}  c^{-1} g_{13} $. Then $ B_{2} $ contains the inverses of these elements.

\medskip
Following the same method we have an involution $ g_{23} \in N_{H_{3} } ( \langle h_{23} \rangle ) \setminus \langle h_{23} \rangle $ such that $ g_{23} ^{-1} B g_{23} = H_{23} $ and $ g_{23} ^{-1} g_{13} ^{-1}  C g_{13} g_{23} = C $.

\medskip
Observe that $ g_{13} g_{23} $ sends $ H_{13} $ to $ H_{23} $, because these are the only $ A_{4} $ in $ H_{3} $ containing $ C $. So $ g_{13} g_{23} \notin C $. The normalizer of $ C $ in $ H_{3} $ is an $ S_{3} $. Hence $ g_{13} g_{23} $ is an involution. In $ S_{3} $, involutions send elements of order $ 3 $ to their inverses when conjugating. So $ g_{13} g_{23} $ sends $ c $ to $ c^{-1} $. However $ h_{23} ^{-1} $ and $  g_{13} ^{-1}  c g_{13} $ are in $ B _{2} $ and $ g_{23} ^{-1} h_{23} ^{-1} g_{23} = h_{23} $, so $ c^{-1} = g_{23} ^{-1} g_{13} ^{-1}  c g_{13}  g_{23} $ is in the same conjugacy class as $ h_{23} $ in $ H_{23} $. Looking at how elements in the same conjugacy class combine in Table \ref{afrcylytbl} we get 
\begin{itemize}
\item $ o(ch_{23} ) = 2 $,
\item $ o(c h_{23} h_{23} ) = 3 $.
\end{itemize}
In $ A_{5} $, any $ S_{3} $ subgroup intersects an $ A_{4} $ subgroup non-trivially. If $ c \in F_{1} $ then $ H_{13}  = \langle c , h_{13} \rangle =  F_{1} \neq H_{13} $. Hence $ c \notin F_{1} $. Therefore $ H_{12} \cap F_{1} $ contains an involution. We can assume $ s $ is this involution. As $ F_{1} $ is an $ A_{4} $ containing $ h_{13} $, we have
\begin{itemize}
\item $ o(sh_{13} ) = 3 $.
\end{itemize}
If $ s \in F_{2} $ then $ F_{1} \cap F_{2} \cong C_{2} $ (cannot be any larger group or it would contain more non-trivial elements of $ H_{12} $ and therefore $ H_{12} \leq F_{1} \cap F_{2}$, though $ F_{1} $ has no $ S_{3} $ inside it). Then we have three subgroups $ B $, $ F_{1} $ and $ F_{2} $ that are isomorphic to $ A_{4} $ and intersect in a way that is prohibited by Lemma \ref{csefrafr}. Hence $ s \not\in F_{2} $.

\medskip
If $ s $ was in an $ S_{3} $ containing $ \langle h_{23} \rangle $ then $ s $ would normalize this subgroup and therefore normalize $  \langle h_{23} , C \rangle = H_{23} $. This is a maximal $ A_{4} $ in $ H_{2} $, so $ s \in H_{23}  $ by Lemma \ref{mxmlnrmlizr}. But then this $ A_{4} $ contains $ H_{12} = \langle s, C \rangle $, an $ S_{3} $ subgroup, which is not possible. Hence $ s $ is not in an $ S_{3} $ with $ \langle h_{23} \rangle $. It follows from Lemma \ref{rdrelmnts} that $ o(sh_{23} ) \neq 2 $.

\medskip
The only two $ A_{4} $ in $ H_{2} $ that contain $ h_{23} $ are $ F_{2} $ and $ H_{23} $, with $ s $ in neither. So $ o(sh_{23} ) \neq 3 $ by Lemma \ref{rdrelmnts}. So by the same lemma it must be that
\begin{itemize}
\item $ o(sh_{23} ) =5 $.
\end{itemize}
Finally we want a relation for $ csh_{23} $. This element is in $ H_{2} $ and $ h_{23} \notin \langle c, s \rangle = H_{12} $. So $ cs \neq h_{23} ^{-1} $ and $ csh_{23} $ can only have order $ 2 $, $ 3 $ or $ 5 $. Each possibility will be tested. The GAP code below looks at what group is generated using the relations we have.
\begin{lstlisting}
N:=[2,3,5];

for n in N do

f:=FreeGroup("h13", "h23", "c", "s");; 
g:=f/[f.1^3, f.2^3, f.3^3, f.4^2,

(f.3*f.4)^2,

(f.3*f.1)^3,
(f.3*f.1*f.1)^2,

(f.1*f.2)^3,
(f.1*f.2*f.2)^2,

(f.3*f.2)^2,
(f.3*f.2*f.2)^3,

(f.4*f.1)^3,

(f.4*f.2)^5,

(f.3*f.4*f.2)^n,
];;
Print("\n", StructureDescription(g));

od;
\end{lstlisting}
GAP shows the group generated is either the trivial group or $ PSL_{2} (11) $. The trivial group is of no interest as it contains no $ A_{5} $. Since $ H_{1} $ and $ H_{3} $ are conjugate and all $ A_{4} $ in $ A_{5} $ are conjugate, there exists $ g \in G $ such that $ g^{-1} H_{1} g = H_{3} $ and $ g^{-1} H_{13} g = H_{13} $. Lemma \ref{afrnrmlzbysfr} shows the $ A_{4} $ subgroups of $ PSL_{2} (11) $ are self normalizing. Hence $ g \in H_{13} \leq H_{1} $ and $ g $ fixes $ H_{1} $. A contradiction. (GAP generates $ PSL_{2} (11) $ because there are two conjugacy classes of $ A_{5} $ and the code does not take into account the $ A_{5} $ subgroups being conjugate.)

\medskip
So the assumption $ H_{12} \cong S_{3} $ is wrong.
\end{proof}
\begin{lema}\label{asxpsltwnn}
There is no independent set $ \{ H_{1} , H_{2} , H_{3} \} \subseteq \Omega $ such that $ H_{1} \cap H_{2} \cap H_{3} \cong C_{3} $.
\end{lema}
\begin{proof}
By Lemma \ref{afrsmthdsk}, $ H_{i} \cap H_{j} \cong A_{4} $ for distinct $ i, j \in \{ 1,2,3 \} $. In Appendix \ref{appc} it is shown the only way $ G $ could contain an independent set of the above form is if $ | G | $ divides $ | A_{6} | = | PSL_{2} (9) | $. We are assuming $ q \geq 11 $, so $ | G | $ does not divide $ | PSL_{2} (9) | $.
\end{proof}
\begin{corl}\label{oiweuvw}
If $ \{ H_{1} , H_{2} , H_{3} \} \subseteq \Omega $ is an independent set then $ | H_{1} \cap H_{2} \cap H_{3} | \leq 2 $.
\end{corl}
\begin{proof}
Looking at the subgroup structure of $ A_{5} $ and using Lemma \ref{stach} the intersection $  H_{1} \cap H_{2} \cap H_{3} $ is either a $ K_{4} $, $ C_{5} $, $ C_{3} $, $ C_{2} $ or trivial. If $ H_{1} \cap H_{2} \cap H_{3} \cong K_{4} $ then Lemma \ref{afrnrmlzbysfr} tells us $ H_{1} $, $ H_{2} $ and $ H_{3} $ each contain the same $ A_{4} $ that normalizes the $ K_{4} $. So $ | H_{1} \cap H_{2} \cap H_{3}  | > 4 $, a contradiction.

\medskip
Using Lemma \ref{stach} again, if $ H_{1} \cap H_{2} \cap H_{3} \cong C_{5} $ then $ H_{1} \cap H_{2} \cong H_{1} \cap H_{3} \cong D_{10} $. In $ A_{5} $ there is only one $ D_{10} $ containing each $ C_{5} $. So $ H_{1} \cap H_{2} = H_{1} \cap H_{3} $, a contradiction by Lemma \ref{indsbstsntqual}. 

\medskip
Finally, as it is assumed $ q \geq 11 $, Lemma \ref{asxpsltwnn} rules out $ C_{3} $ as a possibility.
\end{proof}
Before the height and relational complexity are finally computed, one more configuration of stabilizers must be ruled out.
\begin{lema}\label{afrmdntrsctnx}
There is no independent set $ \{ H_{1} , H_{2} , H_{3} , H_{4} \} \subseteq \Omega $ such that $ H_{1} \cap H_{2} \cong H_{3} \cap H_{4} \cong A_{4} $.
\end{lema}
\begin{proof}
Suppose there exists an independent set $ \{ H_{1} , H_{2} , H_{3} , H_{4} \} \subseteq \Omega $ such that $ H_{1} \cap H_{2} \cong H_{3} \cap H_{4} \cong A_{4} $.

\medskip
For distinct $ i, j ,k \in \{ 1,2,3,4 \} $ write $ H_{ijk} = H_{i} \cap H_{j} \cap H_{k} $. Since $ q \geq 11 $, Corollary \ref{indsub}, Lemma \ref{stach} and Corollary \ref{oiweuvw} together show $ H_{ijk} \cong C_{2} $. Let $ h_{ijk} $ be the involution in $ H_{ijk} $.

\medskip
Lemma \ref{indsbstsntqual} shows $ h_{123} \neq h_{124} $. Therefore $ h_{124} \notin H_{3} $ and $ h_{123} \notin H_{4} $. Put $ H_{3} ^{*} = h_{124} ^{-1} H_{3} h_{124} $. Lemma \ref{mxmlnrmlizr} shows $ H_{3} \neq H_{3} ^{*} $. Also $ H_{3} ^{*} \neq H_{4} $, otherwise $ H_{4} = h_{124}  H_{4} h_{124} ^{-1} = h_{124}  H_{3} ^{*} h_{124} ^{-1} = H_{3} $.

\medskip
Now $ h_{124} \notin H_{3} \cap H_{4} $. Using Lemma \ref{mxmlnrmlizr} and the fact $ H_{3} \cap H_{4} $ is a maximal $ A_{4} $ in $ H_{3} $, we have  $ H_{3}^{*} \cap H_{4} = h_{124} ^{-1} H_{3} h_{124} \cap h_{124} ^{-1} H_{4} h_{124} = h_{124} ^{-1} (H_{3}  \cap  H_{4} ) h_{124} \neq H_{3} \cap H_{4} $.

\medskip
From this we also see $ H_{3}^{*} \cap H_{4} \cong H_{3} \cap H_{4} \cong A_{4} $. If $ H_{3} \cap H_{3}^{*} \cap H_{4} = H_{3}^{*} \cap H_{4} $ then $ H_{3}^{*} \cap H_{4} \leq H_{3} $, implying $ H_{3}^{*} \cap H_{4} = H_{3} \cap H_{4} $, a contradiction. Hence $ H_{3} \cap H_{3}^{*} \cap H_{4} \neq H_{3}^{*} \cap H_{4} $.

\medskip
Similar reasoning shows $ H_{3} \cap H_{3}^{*} \cap H_{4} \neq H_{3} \cap H_{4} $.

\medskip
Observe $ h_{123} $ and $ h_{124} $ are in a $ K_{4} $ together in $ H_{1} \cap H_{2} $. Thus $ h_{123} = h_{124} ^{-1} h_{123} h_{124} $ and $ h_{123} \in H_{3} \cap H_{3} ^{*} $. The fact $ h_{123} \notin H_{4} $ means $ H_{3} \cap H_{3}^{*} \cap H_{4} \neq H_{3} \cap H_{3}^{*} $.

\medskip
Therefore $ \{ H_{3} , H_{3} ^{*} , H_{4} \} $ is an independent set by Lemma \ref{scndfnt}.

\medskip
Any two $ A_{4} $ subgroups of $ A_{5} $ intersect in a $ C_{3} $. In particular $ H_{3} \cap H_{4} $ and $ H_{3} ^{*} \cap H_{4} $ intersect in a $ C_{3} $. That same $ C_{3} $ is in $ H_{3} \cap H_{3} ^{*} \cap H_{4} $. But Corollary \ref{oiweuvw} shows $ | H_{3} \cap H_{3} ^{*} \cap H_{4} | \leq 2 $, a contradiction. Therefore $ \{ H_{1} , H_{2} , H_{3} , H_{4} \}  $ is not independent.
\end{proof}
\begin{lema}\label{lwrbndafv}
The height of the action of $ G $ on $ \Omega $ is at most $ 3 $ and the relational complexity is at most $ 4 $.
\end{lema}
\begin{proof}
If $ Ht(G, \Omega) \geq 4 $ then, by Corollary \ref{indsub}, there exists an independent set of size $ 4 $, say $ \{ H_{1} , H_{2} , H_{3} , H_{4} \} \subseteq \Omega $. In Appendix \ref{appd} it is shown that the elements can be labelled so $ H_{1} \cap H_{2} \cong H_{3} \cap H_{4} \cong A_{4} $. This is not possible for $ PSL_{2} (q) $ by Lemma \ref{afrmdntrsctnx}.
\end{proof}
\begin{thrm}\label{dfapoivqp}
$ Ht(G, \Omega ) = 3 $ and $ RC(G, \Omega ) = 4 $.
\end{thrm}
\begin{proof}
Let $ \Omega $ and $ \Omega ^{*} $ be the two conjugacy classes of $ A_{5} $ in $ G $ and let $ H \in \Omega $. Let $ F_{1} , \dots , F_{5} $ be the five distinct $ A_{4} $ subgroups of $ H $. Suppose, for now, for each $ k \in \{ 1, \dots , 4 \} $ there exists $ H_{k} \in \Omega $ or $ H_{k} \in \Omega ^{*} $ with $F_{k} \leq H_{k} $ and  $ H_{k} \neq H $. Also suppose $ H_{l} $ is conjugate to $ H_{k} $ and $ H_{k} \cap H_{l} \not\cong A_{4} $ for $ l \in \{ 1, \dots , 4 \} \setminus \{ k \} $. These conditions will be shown to be true later.

\medskip
Note for later, $ H_{k} \neq H_{l} $, otherwise $ F_{l} \leq H_{k} $ and $ H_{k} = \langle F_{k} , F_{l} \rangle = H $.

\medskip
Let $ r, s \in \{ 1,2,3 \} $ with $ r \neq s $. Let $ c_{rs} \in F_{r} \cap F_{s} $ with $ o( c_{rs} ) = 3 $. In $ A_{5} $ each $ C_{3} $ is a subgroup of exactly two $ A_{4} $. Since $ A_{5} $ is simple and $ A_{4} $ are maximal, we have $ c_{rs} ^{-1} F_{4} c_{rs} \neq F_{4} $ by Lemma \ref{mxmlnrmlizr}. Also $  c_{rs} F_{4} c_{rs}  ^{-1} \neq c_{rs} ^{-1} F_{4} c_{rs}  $, otherwise $ c_{rs} ^{-1} F_{4} c_{rs} = c_{rs} ^{2} F_{4} c_{rs}  ^{-2} = F_{4} $. So one of $ c_{rs} $ or $ c_{rs} ^{-1} $ sends $ F_{4} $ to $ F_{5} $ by conjugation. Suppose without loss of generality that $ c_{rs} ^{-1} F_{4} c_{rs} = F_{5} $.

\medskip
Let $ t \in \{ 1,2,3 \} \setminus \{ r,s \} $. By the above reasoning we can assume $ c_{st} ^{-1} F_{4} c_{st} = F_{5} $. Hence $ c_{st}  c_{rs} ^{-1} F_{4} c_{rs} c_{st} ^{-1} = F_{4}  $. By Lemma \ref{mxmlnrmlizr}, we have $ c_{rs} c_{st} ^{-1}  \in F_{4} $ because $ F_{4} $ is maximal in the simple group $ H $.

\medskip
Set $ H_{5} := c_{rs} ^{-1} H_{4} c_{rs} $. Since $ F_{4} \leq H_{4} $ we have $ c_{st}  c_{rs} ^{-1} H_{4} c_{rs} c_{st} ^{-1} = H_{4}  $ and so $ c_{st} ^{-1} H_{4} c_{st} = c_{rs} ^{-1} H_{4} c_{rs} = H_{5} $. So for each double intersection of $ H_{1} $, $ H_{2} $ and $ H_{3} $ there exists an element that sends $ H_{4} $ to $ H_{5} $.

\medskip
Also $ F_{5} \leq H_{5} $. Thus $ H_{4} \neq H_{5} $, otherwise $ F_{5} \leq H_{4} $ and $ H_{4} = \langle F_{4} , F_{5} \rangle = H $.

\medskip
Consider the tuples $ I := (H_{1} , H_{2} , H_{3} , H_{4} ) $ and $ J := (H_{1} , H_{2} , H_{3} , H_{5} ) $. We have $ ( H_{r} , H_{s} , H_{4} ) ^{ c_{rs} } = ( H_{r} , H_{s} , H_{5} ) $. The identity sends $ (H_{1} , H_{2} , H_{3} ) $ to itself. Thus $ I \sim _{3} J $.

\medskip
Next we check if $ I \sim _{4} J $. Let us first look at what group $ H_{1} \cap H_{2} \cap H_{3} $ is. Since each double intersection of $ H_{1} $, $ H_{2} $ or $ H_{3} $ contains a $ C_{3} $, the double intersections must be isomorphic to one of $ A_{4} $, $ S_{3} $ or $ C_{3} $. We are assuming they are not $ A_{4} $ though.

\medskip
Suppose $ H_{1} \cap H_{2} \cap H_{3} $ is non-trivial. Observe $ \langle c_{12} \rangle $ is the only $ C_{3} $ in $ H_{1} \cap H_{2} $ and $ \langle c_{13} \rangle $ is the only $ C_{3} $ in $ H_{1} \cap H_{3}  $. Also $ \langle c_{12} \rangle \neq  \langle c_{13} \rangle $. Therefore $ H_{1} \cap H_{2} \cap H_{3} $ does not contain a $ C_{3} $. So the only possibility is $ H_{1} \cap H_{2} \cap H_{3} \cong C_{2} $, which implies $ H_{1} \cap H_{2} \cong H_{1} \cap H_{3} \cong S_{3} $.

\medskip
Let $ h \in H_{1} \cap H_{2} \cap H_{3} $ be an involution. Then $ h $ normalizes $ \langle c_{12} \rangle $ and $  \langle c_{13} \rangle $. Since $ F_{1} = \langle c_{12} , c_{13} \rangle $ we have $ h \in N_{G} (F_{1} ) = F_{1} $. But then $ H_{1} \cap H_{2} =  \langle c_{12} , h \rangle  \leq F_{1} $, which is not possible because $ F_{1} $ does not contain an $ S_{3} $. Hence $ H_{1} \cap H_{2} \cap H_{3} $ is trivial.

\medskip
Now we see that only the identity sends $ (H_{1} , H_{2} , H_{3} ) $ to itself, but does not send $ H_{4} $ to $ H_{5} $, so $ I \not\sim _{4} J $. If $ RC(G, \Omega ) < 4 $ then, by Definition \ref{first}, the fact $ I \sim _{3} J $ (and therefore also $ I \sim _{2} J $) would imply $ I \sim _{4} J $, which is not true. Thus $ RC(G, \Omega ) \geq 4 $. It follows from Theorem \ref{rchgt} that $ Ht(G, \Omega ) \geq 3 $. By Lemma \ref{lwrbndafv} we have $ RC(G, \Omega ) = 4 $ and $ Ht(G, \Omega ) = 3 $.

\medskip
The outstanding task now is to show we can find $ H_{1} , \dots , H_{4} $ that satisfy the conditions at the start of the proof. This is split into two cases.

\medskip
\textbf{Case 1:} Suppose $ q \equiv \pm 3 \pmod{8} $. The involutions in $ G$ have centralizers that are conjugate to each other and isomorphic to $ D_{q \pm 1} $ by Lemma \ref{nrmlzr}. So every Klein four-subgroup of $ G $ is in one of the centralizers. Since $ q \equiv \pm 3 \pmod{8} $, the centralizers have order not divisible by $ 8 $. They do have even order though, so the reflections split into two conjugacy classes in each centralizer.

\medskip
By Lemma \ref{refconjug}, each Klein-four subgroup of one of these dihedral subgroups contains one reflection from each conjugacy class. This means in each centralizer the Klein four-subgroups are all conjugate and therefore all Klein four-subgroups of $ G $ are conjugate.

\medskip
Now every Klein four-subgroup of $ G $ must be in some maximal $ A_{5} $ and is normalized by an $ A_{4} $ there. The normalizer of an $ A_{4} $ must also normalize the $ K_{4} $ it contains. Since $ G $ is simple, the normalizer of the $ K_{4} $ in $ F_{k} $ is contained in a maximal subgroup of $ G $.

\medskip
If $ q \neq p $ then $ q = p^{2} $. But $ p $ is odd and so $ p^{2} \equiv 1 \pmod{8} $, a contradiction. Thus $ q = p $. The maximal subgroups of $ PSL_{2} (p) $ are either Borel, $ D_{p \pm 1} $ or $ A_{5} $. Borel subgroups cannot contain a $ K_{4} $ by Lemma \ref{brlklnfrsbgrp}. The only type of the remaining subgroups that contains $ A_{4} $ are $ A_{5} $ and the $ K_{4} $ in $ A_{5} $ are normalized by $ A_{4} $. Therefore each $ K_{4} $ is normalized by exactly one $ A_{4} $ in $ G $. As all $ K_{4} $ are conjugate, the $ A_{4} $ in $ G $ are conjugate too.

\medskip
From this we infer there exists $ H_{k} , H_{l} \in \Omega ^{*} $ containing $ F_{k} $ and $ F_{l} $ respectively. We have $ H_{k} \neq H $ and $ H_{l} \neq H $; otherwise $ H \in \Omega ^{*} $ and $ \Omega = \Omega ^{*} $, a contradiction. Also $ H_{k} \neq H_{l} $, otherwise $ F_{l} \leq H_{k} $ and $ H_{k} = \langle F_{k} , F_{l} \rangle = H $.

\medskip
If a double intersection $ H_{k} \cap H_{l} \cong A_{4} $, there exists an element that sends $ H_{k} $ to $ H_{l} $ (because they are conjugate in $ G $) and fixes $ H_{k} \cap H_{l} $ (all $ A_{4} $ subgroups are conjugate in $ H_{l} $ and conjugating $ H_{l} $ by one of its own elements fixes $ H_{l} $). The $ A_{4} $ subgroups of $ G $ are self normalizing, so we would have $ H_{k} = H_{l} $, a contradiction. Hence $ H_{k} \cap H_{l} \not\cong A_{4} $. 

\medskip
\textbf{Case 2:} Suppose $ q \equiv \pm 1 \pmod{8} $. Recall $ N_{G} (F_{k} ) \cong S_{4} $ by Lemma \ref{afrnrmlzbysfr}. There is no $ S_{4} $ in $ A_{5} $. So there exists $ f \in N_{G} (F_{k} ) \setminus F_{k} $ and it must be that $ f \notin H $. Set $ H_{k} := f^{-1} H f $. Then $ F_{k} \leq H_{k} $ and by Lemma \ref{mxmlnrmlizr} we have $ H_{k} \neq H $.

\medskip
If $ H_{k} \cap H_{l} \cap H = H_{k} \cap H $ then $ F_{k} = H_{k} \cap H \leq H_{l} $ and so $ H_{l} = \langle F_{k} , F_{l} \rangle = H $, a contradiction. Hence $ H_{k} \cap H_{l} \cap H \neq H_{k} \cap H $.

\medskip
The same reasoning shows $ H_{k} \cap H_{l} \cap H \neq H_{l} \cap H $.

\medskip
Suppose $ H_{k} \cap H_{l} \cong A_{4} $. We cannot have $ H_{k} \cap H_{l}  = F_{k} $ otherwise $ H_{l} = \langle F_{k} , F_{l} \rangle = H $.

\medskip
If $ H_{k} \cap H_{l} \cap H = H_{k} \cap H_{l} $ then $  H_{k} \cap H_{l} \leq H $ and $ H = \langle F_{k} , H_{k} \cap H_{l} \rangle = H_{k} $, another contradiction. Thus $ H_{k} \cap H_{l} \cap H \neq H_{k} \cap H_{l} $.

\medskip
The requirements of Lemma \ref{scndfnt} have been met to show $ \{ H , H_{k} , H_{l} \} $ is an independent subset of $ \Omega $. In $ A_{5} $ each $ A_{4} $ intersections in a $ C_{3} $. Therefore $ H_{k} \cap H_{l} \cap H  $ contains the $ C_{3} $ subgroup $ F_{k} \cap F_{l} $. But $ | H_{k} \cap H_{l} \cap H | \leq 2 $ by Corollary \ref{oiweuvw}. So we once again have a contradiction. Hence $ H_{k} \cap H_{l} \not\cong A_{4} $.
\end{proof}

\newpage
\chapter{The Subfield Actions}\label{chapten}
Here the actions of $ PSL_{2} (q) $ or $ PGL_{2} (q) $ on the cosets of a subfield subgroup will be looked at, defined in the section below. Suppose $ G $ is $ PSL_{2} (q) $ or $ PGL_{2} (q) $ acting on $ \Omega $, the right cosets of a maximal subfield subgroup $ H $. The main results of this chapter are:

\medskip
If $ H \cong PSL_{2} (3) $ then $ Ht(G, \Omega ) = 2 $. Otherwise $ Ht(G, \Omega ) = 3 $.

\medskip
If $ H \cong PSL_{2} (3) $ then $ RC(G, \Omega ) = 3 $. Otherwise $ RC(G, \Omega ) = 4 $.

\medskip
This is proved in Theorem \ref{fnalthrmdndit}.
\section{Preliminary Results}
Often in this chapter, there will be reference to $ D_{ 2(q-1) / \delta } $ and $ C_{ (q-1) / \delta } $ subgroups of some $ G \in \{  PSL_{2} (q) , PGL_{2} (q) \} $, where $ \delta \in \{ 1 , 2 \} $. Throughout this chapter, let $ \delta = 2 $ if $ q $ is odd and $ G = PSL_{2} (q) $ and let $ \delta = 1 $ otherwise.

\medskip
Sometimes in $ PSL_{2} (q) $ or $ PGL_{2} (q) $ there there will be reference to multiple subgroups $ PSL_{2} (q_{m} ) $ or $ PGL_{2} (q_{m} ) $ where $ m \in \mathbb{N} $ and $ q_{m} $ is a divisor of $ q $. In these subgroups we will want to refer to similar sort of dihedral or cyclic subgroups as above where $ \delta $ does not apply (for example when referring to a cyclic subgroup inside $ PGL_{2} (q_{m} ) $, which in turn lies inside $ PSL_{2} (q) $). So if $ H $ is $ PSL_{2} (q_{m} ) $ or $ PGL_{2} (q_{m} ) $ we write $ D_{ 2(q_{m} -1) / \delta _{m} } $ and $ C_{ (q_{m} -1) / \delta _{m} } $, where $ \delta _{m} = 2 $ if $ q_{m} $ is odd and $ H = PSL_{2} (q_{m} ) $ and let $ \delta _{m} = 1 $ otherwise.

\medskip
The Borel subgroups of $ PSL_{2} (q) $ and $ PGL_{2} (q) $ play an important part in finding the relational complexity of the subfield actions. Almost all preliminary results here are revisiting the Borel subgroups, looking at the structure of them and how they interact with the rest of the larger group. Most of these lemmas are listed without much more commentary.

\medskip
Recall from the proof of Corollary \ref{prmtvbrlll} that there exist Borel subgroups in $ SL_{2} (q) $ and $ GL_{2} (q) $ of the form in the statement of the lemma below.
\begin{lema}\label{brlndcyclc}
Let $ G $ be $ GL_{2} (q) $ or $ SL_{2} (q) $ where $ q \geq 3 $. Let
\[
B = 
\begin{cases}
\
\Bigg\{ 
\begin{pmatrix*}[l]
a & b \\
0 & c
\end{pmatrix*}
:
a,c \in \mathbb{F} _{q} ^{*} , b \in \mathbb{F} _{q}
\Bigg\}
\ \ \ \ \ \text{if $ G = GL_{2} (q) $,}
\\
\\
\
\Bigg\{ 
\begin{pmatrix*}[l]
a & b \\
0 & a^{-1}
\end{pmatrix*}
:
a \in \mathbb{F} _{q} ^{*} , b \in \mathbb{F} _{q}
\Bigg\}
\ \ \ \ \ \text{if $ G = SL_{2} (q) $,}
\end{cases}
\]
a maximal Borel subgroup of $ G $. Let
\[
C = 
\begin{cases}
\
\Bigg\{ 
\begin{pmatrix*}[l]
a & 0 \\
0 & c
\end{pmatrix*}
:
a , c \in \mathbb{F} _{q} ^{*} 
\Bigg\}
\ \ \ \ \ \text{if $ G = GL_{2} (q) $,}
\\
\\
\
\Bigg\{ 
\begin{pmatrix*}[l]
a & 0 \\
0 & a^{-1}
\end{pmatrix*}
:
a \in \mathbb{F} _{q} ^{*} 
\Bigg\}
\ \ \ \ \ \text{if $ G = SL_{2} (q) $.}
\end{cases}
\]
Suppose $ x \in B \setminus C $. Then $ C \cap  x^{-1} C x  =  Z(G) $ and $ N_{B} (C) = C $.
\end{lema}
\begin{proof}
Suppose first that $ G = GL_{2}(q) $. Note from Lemma \ref{cntrpglndpsltwq} that $ Z(G) = \{ \lambda I : \lambda \in \mathbb{F} _{q} ^{*} \} $. So $ Z(G) \leq C $ and $ Z(G) \leq C \cap  x^{-1} C x $. Let $ y \in C \setminus Z(G) $. We can write
\[
y =
\begin{pmatrix*}[c]
y_{1} & 0 \\
0 & y_{2}
\end{pmatrix*}
,
\ \ \ \ \ 
y_{1} , y_{2} \in \mathbb{F} _{q} ^{*} , \ y_{1} \neq y_{2}
\]
and
\[
x =
\begin{pmatrix*}[c]
x_{1} & x_{2} \\
0 & x_{3}
\end{pmatrix*}
,
\ \ \ \ \ 
x_{1} , x_{3} \in \mathbb{F} _{q} ^{*} 
,
\ x_{2} \in \mathbb{F} _{q} 
.
\]
Let $ \mu = \det (x) $. Then
\begin{align*}
x^{-1} y x
=
\mu ^{-1}
\begin{pmatrix*}[c]
x_{3} & -x_{2} \\
0 & x_{1}
\end{pmatrix*}
\begin{pmatrix*}[c]
y_{1} & 0 \\
0 & y_{2}
\end{pmatrix*}
\begin{pmatrix*}[c]
x_{1} & x_{2} \\
0 & x_{3}
\end{pmatrix*}
& =
\mu ^{-1}
\begin{pmatrix*}[c]
x_{1} x_{3} y_{1} & x_{2} x_{3} y_{1}-x_{2} x_{3} y_{2} \\
0 & x_{1} x_{3} y_{2}
\end{pmatrix*}
\\
& = 
\begin{pmatrix*}[c]
\mu ^{-1} x_{1} x_{3} y_{1} & \mu ^{-1} x_{2} x_{3} (y_{1}- y_{2} ) \\
0 & \mu ^{-1} x_{1} x_{3} y_{2}
\end{pmatrix*}
.
\end{align*}
If $ x^{-1} y x \in C $, then $ x_{2} = 0 $. Therefore $ x \in C $, a contradiction. Thus $ x^{-1} y x \not\in C $ and $ C \cap  x^{-1} C x = Z(G) $. This implies $ N_{B} (C) = C $ because $ C \neq Z(G) $ when $ q \geq 3 $.

\medskip
The case $ G = SL_{2} (q) $ is proved in a similar way.
\end{proof}
\begin{lema}\label{brlcyclcprts}
Let $ G $ be either $ PGL_{2} (q) $ or $ PSL_{2} (q) $ where $ q \geq 3 $. Let $ p $ be the prime dividing $ q $. Let $ B < G $ be a maximal Borel subgroup. Let $ U < B $ be an elementary abelian group with $ q $ elements. Suppose $ A_{1} , A_{2} < B $ and $ A_{1} \cong A_{2} \cong C_{ (q-1) / \delta } $. Then 
\begin{itemize}
\item there are $ q $ subgroups of $ B $ isomorphic to $ C_{(q-1)/ \delta } $,
\item there is exactly one element $ y \in U $ such that $ y^{-1} A_{1} y = A_{2} $,
\item either $ A_{1} \cap A_{2} = \{ 1_{G} \} $ or $ A_{1} = A_{2} $,
\item $ N_{B} (A_{1} ) = A_{1} $,
\item $ U $ is the only $ p $-group of order $ q $ in $ B $,
\item $ U \trianglelefteq B $.
\item every non-identity element of $ B$ lies either in $ U $ (if its order divides $ q $) or exactly one conjugate of $ A_{1} $ (if its order does not divide $ q $).
\end{itemize}
\end{lema}
\begin{proof}
Let $ G^{ * } = GL_{2} (q) $ if $ G = PGL_{2} (q) $ or let $ G^{ * } = SL_{2} (q) $ if $ G = PSL_{2} (q) $. Let $ B^{ *} < G^{ * } $ be the Borel subgroup corresponding to $ B $. Since all Borel subgroups in $ G $ are conjugate, we may assume that
\[
B^{*} = 
\begin{cases}
\
\Bigg\{ 
\begin{pmatrix*}[l]
a & b \\
0 & c
\end{pmatrix*}
:
a,c \in \mathbb{F} _{q} ^{*} , b \in \mathbb{F} _{q}
\Bigg\}
\ \ \ \ \ \text{if $ G = GL_{2} (q) $,}
\\
\\
\
\Bigg\{ 
\begin{pmatrix*}[l]
a & b \\
0 & a^{-1}
\end{pmatrix*}
:
a \in \mathbb{F} _{q} ^{*} , b \in \mathbb{F} _{q}
\Bigg\}
\ \ \ \ \ \text{if $ G = SL_{2} (q) $.}
\end{cases}
\]
Let
\[
C^{*} = 
\begin{cases}
\
\Bigg\{ 
\begin{pmatrix*}[l]
a & 0 \\
0 & c
\end{pmatrix*}
:
a , c \in \mathbb{F} _{q} ^{*} 
\Bigg\}
\ \ \ \ \ \text{if $ G = GL_{2} (q) $,}
\\
\\
\
\Bigg\{ 
\begin{pmatrix*}[l]
a & 0 \\
0 & a^{-1}
\end{pmatrix*}
:
a \in \mathbb{F} _{q} ^{*} 
\Bigg\}
\ \ \ \ \ \text{if $ G = SL_{2} (q) $.}
\end{cases}
\]
Put $ C :=  C^{*} / Z(G) $. It is straight forward to check that $ C \cong C_{ (q-1) / \delta } $. Let $ x \in U \setminus \{ 1 _{G} \} $. Then $ o(x) =  p $. Now $ p $ does not divide $ q - 1 $, and so does not divide $ (q-1) / \delta $. Hence $ x \not\in C $. So $ x $ corresponds to some $ x^{*} \in B^{*} \setminus C^{*} $. By Lemma \ref{brlndcyclc} we get $ C^{*} \cap (x^{*})^{-1} C^{*} x^{*} = Z(G^{*} ) $. Therefore $ C \cap x^{-1} C x = Z(G^{*} ) / Z(G^{*} ) = \{ 1 _{G} \} $.

\medskip
Suppose $ x_{1} \in U $ and $ x_{1} ^{-1} C x_{1} = x^{-1} C x $. Then $ C = x_{1} x^{-1} C x x_{1} ^{-1} = (x x_{1} ^{-1}) ^{-1} C (x x_{1} ^{-1} ) $. By the above discussion $ x x_{1} ^{-1} \not\in  U \setminus \{ 1_{G} \} $. So $ x x_{1} ^{-1} = 1_{G} $ and $ x = x_{1} $. Therefore there are at least $ q $ distinct conjugates of $ C $ in $ B $; one for each element of $ U $. Also if $ x_{2} \in U $ and $ x_{2} \neq x_{1} $ then $ x_{1} ^{-1} C  x_{1} \cap x_{2} ^{-1} C  x_{2} = \{ 1_{G} \} $, otherwise we would have $ C = (x_{1} x_{2} ^{-1} )^{-1}  C  (x_{1} x_{2} ^{-1} ) $ with $ x_{1} x_{2} ^{-1} \in U \setminus \{ 1_{G} \} $.

\medskip
Now we show that $ U $ and the subgroups conjugate to $ C $ contain all elements of $ B$. Count the elements in
\[
U \cup \Bigg( \bigcup _{g \in U} g^{-1} C g \Bigg) .
\]
There are $ q $ conjugates of $ C $ and the non-identity elements are distinct from each other across all conjugates. The non-identity elements of $ U $ all have order $ p $, so cannot lie in a conjugate of $ C $. Hence
\[
\Bigg| U \cup \Bigg( \bigcup _{g \in E} g^{-1} C g \Bigg) \Bigg| 
= | U| + q ( | C | - 1 ) 
= q + q((q-1) / \delta - 1) 
= q(q-1) / \delta
= | B | .
\]
Therefore $ B = U \cup ( \bigcup _{g \in E} g^{-1} C g ) $. If $ A \leq B $ and $ A \cong C_{ (q-1) / \delta } $ then $ A = \langle a \rangle $ for some $ a \in B $ where $ o(a) = (q-1) / \delta $. It cannot be that $ a \in U $ because there are no elements with this order. Hence $ a $ lies in some conjugate of $ C $, showing $ A $ is a conjugate of $ C $ and there are exactly $ q $ subgroups isomorphic to $ C_{ (q-1) / \delta } $ in $ B $. All other elements lie in $ U $, in particular all elements of order $ p $, making $ U $ the only $ p $-group of order $ q $ in $ B $. This implies $ U $ is normal in $ B $.

\medskip
If $ A_{1} $ and $ A_{2} $ are conjugates of $ C $, the fact that there is only one element of $ U $ that sends $ C $ to $ A_{1} $ (and similarly for $ A_{2} $) implies there is exactly one $ y \in U $ such that $ y^{-1} A_{1} y = A_{2} $. Finally if $ N_{B} (A_{1} ) \neq A_{1}  $ then $ N_{B} ( C ) \neq C $ and it follows that $ N_{B^{*} } ( C^{*} ) \neq C^{*} $. This contradicts Lemma \ref{brlndcyclc}. Hence $ N_{B} (A_{1} ) = A_{1}  $.
\end{proof}
\begin{lema}\label{egdfdvsp}[Schur-Zassenhaus Lemma]
Let $ G $ be a finite group. Suppose $ K \trianglelefteq G $ and that $ | K | $ and $ |G | / | K | $ are coprime. Then $ G \cong K \rtimes G/K $.
\end{lema}
\begin{proof}
See \cite{ROTMAN1}, Theorem 7.41, Chapter 7, page 190. 
\end{proof}
\begin{lema}\label{sbgrpfbrlq}
Put $ q = p^{m} $ where $ p $ is prime and $ m \in \mathbb{N} $. Let $ G $ be $ PSL_{2} (q) $ or $ PGL_{2} (q) $. Let $ B < G $ be a maximal Borel subgroup and write $ B =   U \rtimes T $ where $ U $ is an elementary abelian $ p $-group of order $ q $ and $ T \cong C_{(q-1) / \delta } $. If $ B_{0} \leq B $ then $ B_{0} = U_{0} \rtimes T_{0} $ where $ U_{0} \leq U $ and $ T_{0} $ is a subgroup of some $ C_{(q-1) / \delta } $ in $ B $.
\end{lema}
\begin{proof}
If $ B_{0} $ is trivial, an elementary abelian $ p $-group or a subgroup of some $ C_{(q-1) / \delta } $ there is nothing to show. So suppose $ B_{0} $ is none of these.

\medskip
Let $ W = B_{0} \cap U $. Then $ W \trianglelefteq B_{0} $ because $ U \trianglelefteq B $. By isomorphism theorems for groups we have $ B_{0} / W \cong B_{0} U / U $. Since $ B_{0} U / U \leq B/U $ and $ B/U \cong T $, it follows that $ B_{0} / W  $ is isomorphic to a subgroup of $ C_{(q-1) / \delta } $. Therefore $ | B_{0} | / | W | $ divides $ q-1 $. As $ | W | $ divides $ | U | = q $, the numbers $ | B_{0} | / | W | $ and $ | W | $ are coprime.

\medskip
The Schur-Zassenhaus Lemma can now be applied to give $ B_{0} \cong W \rtimes B_{0} / W  $.
\end{proof}
\begin{corl}\label{crlrsmllbrl}
Let $ G $ be $ PSL_{2} (q) $ or $ PGL_{2} (q) $. Let $ B \leq G $ be a Borel subgroup and write $ B = U \rtimes T $ where $ U $ is an elementary abelian group of order $ q $ and $ T \cong C_{ (q-1) / \delta } $. Let $ B_{1} \leq B $ and write $ B_{1} = U_{1} \rtimes T_{1} $ where $ U_{1} \leq U $ and $ T_{1} $ is a subgroup of some $ C_{ (q-1) / \delta } $ in $ B $. If $ T_{1} ^{ \prime } \leq B $ and $ T_{1} ^{ \prime } \cong T_{1} $ then there exists $ B_{1} ^{ \prime }  \leq B $ such that $ B_{1} ^{ \prime } = U_{1} \rtimes T_{1} ^{ \prime}  $ and $ B_{1} ^{ \prime } = u^{-1} B_{1} u $ for some $ u \in U $.
\end{corl}
\begin{proof}
The orders of $  T_{1} $ and $ T_{1} ^{ \prime } $ divide $ q - 1 $, hence do not divide $ q $. Therefore $ T_{1} , T_{1} ^{ \prime } \not\leq U $. So by Lemma \ref{brlcyclcprts}, there exist subgroups $ C , C^{ \prime } \leq B $ such that $ C \cong C^{ \prime } \cong C_{ (q-1) / \delta } $ and $ T_{1} \leq C $ and $ T_{1} ^{ \prime } \leq C ^{ \prime } $. Using the same lemma, $ C^{ \prime } = u^{-1} C u $ for some $ u \in U $. Since $ C $ and $ C ^{ \prime } $ are cyclic, they each contain exactly one subgroup of order $ | T_{1} | $. Therefore $ T_{1} ^{ \prime} = u^{-1} T_{1} u $. The fact $ U $ is elementary abelian means $ u^{-1} U u = U $. Thus
\[
u^{-1} B_{1} u = u^{-1} (U_{1} \rtimes T_{1} ) u = (u^{-1} U_{1} u) \rtimes (u^{-1} T_{1}  u ) = U_{1} \rtimes T_{1} ^{ \prime } .
\]
\end{proof}
\begin{lema}\label{pgrpnrmlbrl}
Let $ G $ be $ PSL_{2} (q) $ or $ PGL_{2} (q) $. Let $ B \leq G $ be a maximal Borel subgroup and write $ B = U \rtimes T $ where $ U \cong E _{q} $ and $ T \cong C_{ (q-1) / \delta } $. Suppose $ U^{ * } \trianglelefteq B $ and $ \{ 1_{G} \} < U^{ * } \leq U $. Then $ U^{ * } = U $.
\end{lema}
\begin{proof}
Let $ G^{ \prime } := SL_{2} (q) $ if $ G $ is $ PSL_{2} (q) $ or $ G^{ \prime } := GL_{2} (q) $ if $ G $ is $ PGL_{2} (q) $. Let $ B ^{ \prime } \leq G ^{ \prime } $ be the Borel subgroup corresponding to $ B $. Since all maximal Borel subgroups of $ G^{ \prime } $ are conjugate, we can assume $ B ^{ \prime } $ is the group of upper triangular matrices in $ G^{ \prime } $.

\medskip
Let $ U ^{ \prime } \trianglelefteq B ^{ \prime } $ be the subgroup corresponding to $ U $. An element $ u \in U ^{ * } \setminus \{ 1 \} $ corresponds to the matrices in the coset $ Z( G^{ \prime } ) u^{ \prime } $ where $ u^{ \prime } \in U ^{ \prime } $ has the form
\[
u^{ \prime } =
\begin{pmatrix*}[l]
1 & a \\
0 & 1
\end{pmatrix*}
\ \ \ \ \
\text{for some $ a \in \mathbb{F}_{q} ^{*} $.}
\]
The following element exists in $ B ^{ \prime } $;
\[
g =
\begin{pmatrix*}[l]
y & x \\
0 & y^{ -1}
\end{pmatrix*}
\ \ \ \ \
\text{where $ x ,y \in \mathbb{F}_{q} ^{*} $ and $ y \neq 1  $.}
\]
Now
\[
g^{-1} u^{ \prime } g = 
\begin{pmatrix*}[l]
1 & ay^{ -2 } \\
0 & 1
\end{pmatrix*}
\]
So by Lemma \ref{sqnmbrf} there are at least $ \tfrac{1}{2} (q-1) $ conjugates of $ u $ in $ B $. Including the identity there are at least $ \tfrac{1}{2} (q-1) + 1 =  \tfrac{1}{2} (q + 1) $ elements in $ U ^{*} $. We have $ q = p^{m} $ for some prime $ p $ and $ m \in \mathbb{N} $. Therefore the largest possible order for a non-trivial proper subgroup of $ U $ is $ p^{ m - 1} \leq \tfrac{1}{2} p^{m} = \tfrac{1}{2} q < \tfrac{1}{2} (q + 1) $. Hence $ U^{*} = U $.
\end{proof}
\begin{lema}\label{nrmlzrnbrlx}
Let $ G $ be $ PSL_{2} (q) $ or $ PGL_{2} (q) $ where $ q \geq 3 $. Suppose $ G \neq PSL_{2} (3) $. Write $ q = p^{m} $ where $ p $ is prime and $ m \in \mathbb{N} $. Let $ P < G $ be a non-trivial $ p $-group. Then $ N_{G} (P) $ is a subgroup of exactly one maximal subgroup $ M < G $, which is a Borel subgroup.
\end{lema}
\begin{proof}
Let $ S \leq G $ be a Sylow $ p $-subgroup with $ P \leq S $. The Sylow $ p $-subgroups of $ G $ have order $ q $. Each Borel subgroup has the form $ E_{q} \rtimes C_{ (q-1) / \delta } $ where $ U $ is an elementary abelian $ p $-group of order $ q $. Each $ E_{q} $ is a Sylow $p$-subgroup. As all Sylow $ p $-subgroups are conjugate, $ S \leq B $ where $ B $ is a Borel subgroup of $ G $. As $ S $ is abelian, we have $ S \leq N_{G} ( P ) $. Also $ S $ is not normal in $ G $, so $ N_{G} ( P ) $ is a subgroup of at least one maximal subgroup of $ G $ that contains $ S $.

\medskip
If $ G  = PGL_{2} (3) \cong S_{4} $ then $ P $ can only have order $ 3 $. The Sylow $3$-subgroups are cyclic of order $ 3 $ and are normalized by a $ D_{6} $, which are Borel subgroups in this case.

\medskip
If $ G  = PGL_{2} (5) \cong S_{5} $ then the $ 5 $-subgroups of $ G $ are all cyclic of order $ 5 $. They are normal subgroups of Borel subgroups of the form $ C_{5} \rtimes C_{4} $ in $ G $, which are maximal so must be the normalizers of the $ C_{5} $ they contain, as required.

\medskip
Now suppose $ G \neq PGL_{2} (3) $ and $ G \neq PGL_{2} (5) $. The maximal $ D_{2 ( q \pm 1 ) / \delta } $ do not have order divisible by $ q $, so $ S $ will not be in one of these subgroups.

\medskip
Maximal $ A_{4} $, $ S_{4} $ and $ A_{5} $ only occur when $ q $ is odd and these only have subgroups of odd order $ 3 $ or $ 5 $, which could only be Sylow subgroups when $ q \in \{ 3 , 5 \} $. The case for $ q = 3 $ has been dealt with. Out of the above only $ A_{5} $ has a subgroup of order $ 5 $, however $ PSL_{2} (5) $ does not have a maximal $ A_{5} $. Therefore $ S $ is not a subgroup of an $ A_{4} $, $ S_{4} $ or $ A_{5} $.

\medskip
The only other non-Borel subgroup to consider are subfield subgroups. However the largest $ p $-groups in a subfield subgroup have order $ q_{0} $ where $ q = q_{0} ^{r} $ for some prime $ r $. Thus a subfield subgroup could not contain $ S $.

\medskip
Hence the only maximal subgroups containing $ N_{G} ( P ) $ are Borel subgroups. To show there is only one such Borel, suppose $ B_{1} , B_{2} < G $ are Borel subgroups containing $ N_{G} ( P ) $. If $ B_{1} \neq B_{2} $ then $ B_{1} \cap B_{2} \cong C_{ (q-1) / \delta } $ by Lemma \ref{brlntrsccyclc}. But this intersection has no element of order $ p $, implying $ P \not\leq B_{1} \cap B_{2} $, hence $ N_{G} ( P ) \not\leq B_{1} \cap B_{2} $. This is a contradiction. Thus $ B_{1} = B_{2} $.
\end{proof}
\begin{lema}\label{prmsbgrps}
Write $ q = p^{m} $ where $ p $ is prime and $ m \in \mathbb{N} $. Let $ G $ be $ PSL_{2} (q) $ or $ PGL_{2} (q) $. Let $ P_{1} , P_{2} < G $ be Sylow $ p $-subgroups. Then either $ P_{1} = P_{2} $ or $ P_{1} \cap P_{2} = \{ 1 _{G} \} $. 
\end{lema}
\begin{proof}
Looking at the order of $ G $, the Sylow $ p $-subgroups have order $ q $. Lemma \ref{brlcyclcprts} shows there happens to be exactly one elementary abelian $ p $-group of order $ q $ in each Borel subgroup of $ G $, which are Sylow $ p $-subgroups. Hence $ P_{1} < B_{1} < G $ and $ P_{2} < B_{2} < G $ where $ B_{1} $ and $ B_{2} $ are Borel subgroups. 

\medskip
Suppose $ P_{1} \neq P_{2} $. Then $ B_{1} \neq B_{2} $. By Lemma \ref{brlntrsccyclc} we have $ B_{1} \cap B_{2} \cong C_{ (q-1) / \delta } $. Therefore $ B_{1} \cap B_{2} $ has no elements of order $ p $ and it follows that $ P_{1} \cap P_{2} = \{ 1_{G} \} $.
\end{proof}

\newpage
\section{Action Description}
Throughout the rest of this chapter, let $ p $ be a prime, $ q$ a power of $ p $, suppose $ q \geq 4 $ and $q$ is not prime. Then $ q = q_{0}^{r} $ where $ q_{0} $ is again some power of $ p $ and $ r $ is a prime. In Table \ref{tableone}, we see that when $ q $ is even there exists a conjugacy class of the following maximal subgroups in $ PSL_{2} (q) $;
\begin{itemize}
\item $ PSL _{2} (q_{0} ) $, if $ q = q_{0} ^{r} $ for some prime $ r $ and $ q_{0} \neq 2 $ (one conjugacy class).
\end{itemize}
When $ q $ is odd, Table \ref{tablethree}, shows that we get similar conjugacy classes of maximal subgroups for $ PSL_{2} (q) $;
\begin{itemize}
\item $ PSL _{2} (q_{0} ) $, if $ q = q_{0} ^{r} $ for some odd prime $ r $ (one conjugacy class),
\item $ PGL _{2} (q_{0} ) $, if $ q = q_{0} ^{2} $ (two conjugacy classes).
\end{itemize}
From Table \ref{tablethree}, we see that when $ q $ is odd $ PGL_{2} (q) $ has maximal subgroups
\begin{itemize}
\item $ PGL _{2} (q_{0} ) $, if $ q = q_{0} ^{r} $ for some odd prime $ r $ (one conjugacy class).
\end{itemize}
For the remainder of this section let $ q_{0} $ be defined as above, where the context will be made clear depending on the maximal subgroup we are referring to. Also let $ G $ be $ PSL_{2} (q) $ or $ PGL_{2} (q) $ and let $ H $ be a maximal $ PSL_{2} (q_{0}) $ or $ PGL_{2} (q_{0}) $.

\medskip
It follows from Lemma \ref{cstquivtcnj} that rather than  describing the action in terms of right cosets of a maximal $ PSL_{2} (q_{0}) $ or $ PGL_{2} (q_{0}) $, we can look at the action by conjugation on $ \Omega = \{ gHg^{-1} : g \in G \} $ since the two are equivalent.
\section{Height and Relational Complexity of the Subfield Actions}
There are two special cases that do not fit the general proof for finding relational complexity and height. These are $ PSL_{2} (9) $ acting on maximal $ PGL_{2} (3) $ and $ PSL_{2} (q) $ acting on maximal $ PSL_{2} (3) $ when $ q = 3^{r} $ and $ r \geq 3 $. At the end of the section these are going to be dealt with. So until said otherwise, assume $ G \neq PSL_{2} (9) $. Also assume $ H \not\cong PSL_{2} (3)  $.

\medskip
Through this section let $ H_{1} \in H^{G} \setminus H $. Before calculating height and relational complexity, a lot of this section is devoted to finding what group $ H \cap H_{1} $ might be. This takes quite a bit of work, with some particularly long lemmas later on. The next two results narrow down some of the possibilities.
\begin{lema}\label{pgrpqzero}
If $ P \leq  H \cap H_{1} $ is a non-trivial $ p $-group then there exists a $ p $-group $ P^{ * } \leq  H \cap H_{1} $ with $ P \leq P^{ * } $ and $ | P^{ * } | = q_{0} $.
\end{lema}
\begin{proof}
Let $ P ^{ \prime } $ be a $ p $-group of greatest order in $ H \cap H_{1} $ that contains $ P $. The Sylow $ p $-subgroups of $ H $ have order $ q_{0} $ since that is the greatest power of $ p $ that divides $ |H | $. Each Borel subgroup of $ H $ contains a Sylow $ p $-group. So there exists a Borel subgroup $ B_{0} \leq H $ and a $ p $-subgroup $ U_{0} \leq B_{0} $ of order $ q_{0} $ where $ P ^{ \prime } \leq U_{0} $. Also $ U_{0} $ is the only $ p $-group of this order in $ B_{0} $ by Lemma \ref{brlcyclcprts}. So we can write $ B_{0} = U_{0} \rtimes T_{0} $ where $ T_{0} \cong C_{ (q_{0} - 1 ) / \delta } $.

\medskip
Similarly there exists a Borel subgroup $ B_{1} $ in $ H_{1} $ and $ p $-subgroup $ U_{1} \leq B_{1} $ of order $ q_{0} $ such that $ P ^{ \prime } \leq U_{1} $. Again we can write $ B_{1} = U_{1} \rtimes T_{1} $ where $ T_{1} \cong C_{ (q_{0} - 1 ) / \delta } $.

\medskip
Every subgroup of $ U_{0} \cap U_{1} $ is a $ p $-group. So by the way $ P ^{ \prime } $ is defined it must be that $ U_{0} \cap U_{1} = P ^{ \prime } $.

\medskip
Since $ U_{0} $ and $ U_{1} $ are elementary abelian, we have $ U_{0} , U_{1} \leq N_{G} ( P ^{ \prime } ) $. Lemma \ref{nrmlzrnbrlx} shows that $ U_{0} $ and $ U_{1} $ both lie in a Borel subgroup $ B $ of $ G $. By Lemma \ref{brlcyclcprts} there exists a Sylow $ p $-subgroup of $ U \leq G $ that is normal in $ B $ and $ U_{0} , U_{1} \leq U $. Lemma \ref{nrmlzrnbrlx} tells us that $ B $ is the only Borel subgroup of $ G $ that could contain $ U $. Hence $ U_{0} $ and $ U_{1} $ do not lie in any other Borel subgroup of $ G $ by Lemma \ref{prmsbgrps}. So using Lemma \ref{nrmlzrnbrlx} again, $ N_{G} ( U_{0} ) , N_{G} ( U_{1} ) \leq B $. Also $ T_{0} , T_{1} \leq B $ because $ T_{0} \leq N_{G} ( U_{0} ) $ and $ T_{1} \leq N_{G} ( U_{1} ) $. From Corollary \ref{crlrsmllbrl} we then have $ U_{0} \rtimes T_{1} \leq B $. Therefore $ T_{1} \leq N_{G} (U_{0} ) $. So for $ t_{1} \in T_{1} $ we have
\[
t_{1} ^{-1} P ^{ \prime } t_{1} 
= t_{1} ^{-1} (U_{0} \cap U_{1} ) t_{1} 
= ( t_{1} ^{-1} U_{0} t_{1} ) \cap (t_{1} ^{-1} U_{1} t_{1} ) 
= U_{0} \cap U_{1}
= P ^{ \prime } .
\]
This implies $ T_{1} \leq N_{G} (P ^{ \prime } ) $. It follows that $ P ^{ \prime } \trianglelefteq U_{1} \rtimes T_{1} = B_{1} $. Applying Lemma \ref{pgrpnrmlbrl} gives $ P ^{ \prime } = U_{1} $ and we get $ | P ^{ \prime }  | = q_{0} $.
\end{proof}
\begin{lema}\label{nnnafafsf}
Suppose $ q $ is odd. Then
\begin{itemize}
\item $ H \cap H_{1} \not\cong A_{5} $,
\item $ H \cap H_{1} \not\cong S_{4} $,
\item $ H \cap H_{1} \not\cong A_{4} $.
\end{itemize}
\end{lema}
\begin{proof}
We start with a couple of special cases. Suppose $ H \cong PGL_{2} (3) \cong S_{4} $. Then $ H \cap H_{1} $ cannot be isomorphic to $ S_{4} $ or $ A_{5} $ because $ H $ does not have such a proper subgroup. If $ H \cap H_{1} \cong A_{4} $ then $ H \cap H_{1} $ is normal in both $ H $ and $ H_{1} $, which is not possible by Lemma \ref{ntnrmlmxml}.

\medskip
Next suppose $ H \cong PSL_{2} (5) \cong A_{5} $. Then $ G \cong PSL_{2} (5^{m} ) $ for some odd prime $ m $. For the same reasons as above we can rule out $ H \cap H_{1} $ being isomorphic to $ S_{4} $ or $ A_{5} $. Suppose $ H \cap H_{1} \cong A_{4} $. Observe $ G $ is simple and so $ N_{G} (  H \cap H_{1} ) $ is contained in some maximal subgroup $ M < G $. We cannot have $ M $ being a Borel subgroup because then $ H \cap H_{1} \leq M $ and Borel subgroups do not have Klein four-subgroups by Lemma \ref{brlklnfrsbgrp}. Also $ M $ cannot be a $ D_{ q \pm 1 } $ because dihedral groups have at most two elements of order $ 3 $, but $ A_{4} $ has more. The only other maximal subgroups $ G $ can have are subfield subgroups that are conjugate to $ H $. These are isomorphic to $ A_{5} $ and an $ A_{4} $ subgroup is self normalizing in $ A_{5} $. Hence $  N_{G} (  H \cap H_{1} ) = H \cap H_{1} $.

\medskip
Now $ H_{1} = g^{-1} H g $ for some $ g \in G $. All $ A_{4} $ subgroups are conjugate in $ A_{5} $, so we can assume that $ g^{-1} ( H \cap H_{1} ) g = H \cap H_{1} $. Therefore $ g \in N_{G} (H \cap H_{1} ) \leq H $ and $ g^{-1} H g = H \neq H_{1} $. This is a contradiction. Thus $ H \cap H_{1} \not\cong A_{4} $ when $ H  \cong PSL_{2} (5) $.

\medskip
With the special cases out of the way, the general case will now be dealt with. Suppose $ H \cap H_{1} $ is isomorphic to $ A_{4} $, $ S_{4} $ or $ A_{5} $. Then $ H $ is not isomorphic to $ PGL_{2} (3) $ or $ PSL_{2} (5) $.

\medskip
If $ q_{0} $ is a power of $ 3 $ then $ q_{0} \geq 9 $ and since $ H \cap H_{1} $ contains a subgroup of order $ 3 $ it also has a subgroup of order $ q_{0} $ by Lemma \ref{pgrpqzero}. No subgroup of this order exists in $ H \cap H_{1} $. So $ q_{0} $ cannot be a power of $ 3 $.

\medskip
Suppose $ H \cong PSL_{2} (q_{0} ) $. Then $ q_{0} \geq 7 $. Whichever of $ A_{4} $, $ S_{4} $ or $ A_{5} $ the group $ H \cap H_{1} $ is isomorphic to, it contains an involution $ h^{*} \in H \cap H_{1} $ whose centralizer is a $ K_{4} $. Lemma \ref{nrmlzr} shows that $ C_{H} ( h^{*} ) $ is either a $ D_{ q_{0} -1 } $ or $ D_{ q_{0} +1 } $ and its order must be divisible by $ 4 $. Hence $ | C_{H} ( h^{*} ) | \geq 7 + 1 = 8 $. This means there exists $ c \in C_{H} ( h^{*} ) $ with $ o(c) > 2 $, in the rotational subgroup of $ C_{H} ( h^{*} ) $. Following the same reasoning, there exists an element of order $ o(c) $ in $ C_{H_{1} } ( h^{*} ) $. We can use the same lemma to show $ C_{G } ( h^{*} ) $ is dihedral. There is only one cyclic subgroup of order $ o(c) $ in $ C_{G } ( h^{*} ) $, which must be $ \langle c \rangle $. Therefore $ \langle c \rangle \leq C_{H_{1} } ( h^{*} ) $ and $ c $ centralizes $ h^{*} $ in $ H \cap H_{1} $. This is a contradiction. Thus $ H \not\cong PSL_{2} (q_{0} ) $.

\medskip
The only possibility left is that $ H \cong PGL_{2} (q_{0} ) $ with $ q_{0} \geq 5 $. Observe that $ H \cap H_{1} $ has a cyclic subgroup $ C $ of order $ 3 $ that is normalized in $ H \cap H_{1} $ by either a $ D_{6} $ or $ C $ itself. As $ q_{0} $ is not divisible by $ 3 $, either $ N_{H} ( C ) \cong D_{2(q-1) } $ or $ N_{H} ( C ) \cong D_{2(q_{0} +1) } $ by Lemma \ref{nrmlzr}. Same applies to $  N_{H_{1} } ( C )  $. As only one of $ q-1 $ or $ q+1 $ is divisible by $ 3 $, we have $ N_{H } ( C ) \cong N_{H_{1} } ( C ) $. The rotational subgroups of $  N_{H} ( C ) $ and $  N_{H_{1} } ( C ) $ are cyclic of order say $ r $ where $ r $ is divisible by $ 3 $ and $ r \geq 6 $. Using the same lemma again, $  N_{G } ( C ) $ has only one cyclic subgroup of order $ r $. So $  N_{H} ( C ) $ and $  N_{H_{1} } ( C ) $ share the same rotational subgroup. Hence $ H \cap H_{1} $ contains this subgroup. But $ H \cap H_{1} $ has no element of order $ 6 $ or more, revealing a contradiction. Thus $ H \not\cong PGL_{2} (q_{0} ) $.

\medskip
As $ H $ cannot be any possible subfield subgroup, we have a contradiction and the assumption that $ H \cap H_{1} $ is isomorphic to $ A_{4} $, $ S_{4} $ or $ A_{5} $ is wrong.
\end{proof}
Next we will see that $ H \cap H_{1} \not\cong K_{4} $ when $ q $ is odd, which takes several lemmas.
\begin{lema}\label{nvntnntrsctin}
Let $ G $ be $ PGL_{2} (q) $ or $ PSL_{2} (q) $ where $ q $ is odd, with $ q \geq 7 $ if $ G = PGL_{2} (q) $ and $ q \geq 11 $ if $ G = PSL_{2} (q) $. Suppose there exists an involution $ h \in H \cap H_{1} $. Also suppose there exists $ c \in C_{H} (h) $ and $ c_{1} \in C_{H_{1} } (h) $ with $ o(c) = o(c_{1} ) \geq 3 $. Then $ H \cap H_{1} \not\cong K_{4} $.
\end{lema}
\begin{proof}
We know that $ C_{G} (h) $ is dihedral of order $ | C_{G} (h) | \geq 10 $ by Lemma \ref{nrmlzr}. Both $ c $ and $  c_{1} $ lie in the rotational subgroup of $  C_{G} (h) $. Since the rotational subgroup is cyclic, $ \langle c  \rangle =  \langle c_{1} \rangle $. It follows that $ c \in H \cap H_{1} $ and $ H \cap H_{1} \not\cong K_{4} $.
\end{proof}
\begin{lema}\label{kfrsbgrpone}
Suppose $ G = PSL_{2} (5^{r}) $ with $ r $ an odd prime and $ r \geq 3 $. Suppose $ H \cong PSL_{2} (5) $. Then $ H \cap H_{1} \not\cong K_{4} $.
\end{lema}
\begin{proof}
Every $ K_{4} $ in $ A_{5} $ is normalized by an $ A_{4} $. So the fact $ PSL_{2} (5) \cong  A_{5} $ means that if $ H \cap H_{1} \cong K_{4} $ then $ N_{G} ( H \cap H_{1} ) $ contains at least two subgroups isomorphic to $ A_{4} $. As $ r $ is prime, the only subfield subgroups it has are those conjugate to $ H $. Therefore $ N_{G} ( H \cap H_{1} ) $ would have to be contained in one of the maximal subgroups of $ PSL_{2} (5^{r} ) $, which are of the one of the following types; Borel, $ D_{ (5^{r} - 1) } $, $ D_{ (5^{r} + 1) } $ or $ PSL_{2} (5) \cong  A_{5} $, which we now show is not possible.

\medskip
Lemma \ref{brlklnfrsbgrp} shows that $ N_{G} ( H \cap H_{1} ) $ cannot be a Borel subgroup because they do not have any Klein four-subgroups.

\medskip
Also $ N_{G} ( H \cap H_{1} ) $ does not lie in a $ D_{ (5^{r} - 1) } $ or $ D_{ (5^{r} + 1) } $ as each of these have at most two elements of order $ 3 $, so do not have any $ A_{4} $ subgroups.

\medskip
The only other possibility is that $ N_{G} ( H \cap H_{1} ) $ is a subgroup of a maximal $ A_{5} $. However each Klein four-subgroup in $ A_{5} $ is normalized by only one $ A_{4} $. Thus $ N_{G} ( H \cap H_{1} ) $ is not contained in an $ A_{5} $.
\end{proof}
\begin{lema}\label{kfrsbgrptwo}
Suppose $ G = PSL_{2} (q) $ where $ q $ is odd and $ q \geq 11 $. Suppose $ H \cong PSL_{2} (q_{0} ) $ where $ q_{0} \geq 7 $. Then $ H \cap H_{1} \not\cong K_{4} $.
\end{lema}
\begin{proof}
If $ H \cap H_{1} \cong K_{4} $ then there exists an involution $ h \in H \cap H_{1} $. Lemma \ref{nrmlzr} tells us all involutions in $ H $ have isomorphic centralizers of order at least $ 8 $. So there exists elements of order at least $ 3 $ in these centralizers, allowing us to apply Lemma \ref{nvntnntrsctin}.
\end{proof}
\begin{lema}\label{kfrsbgrpthree}
Suppose $ q $ is odd. Suppose one of the following is true:
\begin{itemize}
\item $ H \cong PGL_{2} ( 5 ) $ or
\item $ G = PGL_{2} (q) $ and $ H \cong PGL_{2} (3) $, with $ q = 3^{m} $ for some $ m \geq 3 $.
\end{itemize}
Then $ H \cap H_{1} \not\cong K_{4} $.
\end{lema}
\begin{proof}
Note that $ PGL_{2} (5) \cong S_{5} $ and $ PGL_{2} (3) \cong S_{4} $. In $ S_{5} $ and $ S_{4} $ there are two types of involution; the single-transpositions $ ( x_{1} \ x_{2} ) $ and double-transpositions $ ( x_{1} \ x_{2} )( x_{3} \ x_{4} ) $.

\medskip
The Klein four-subgroups of $ S_{5} $ and $ S_{4} $ either have two single-transpositions and are of the form
\[
 \{ e , ( x_{1} \ x_{2} ) , \ ( x_{3} \ x_{4} ) , \ ( x_{1} \ x_{2} )( x_{3} \ x_{4} ) \}
 \]
 or have three double-transpositions and are of the form
 \[
 \{ e , ( x_{1} \ x_{2} )( x_{3} \ x_{4} ) , \ ( x_{1} \ x_{3} )( x_{2} \ x_{4} ) , \ ( x_{1} \ x_{4} )( x_{2} \ x_{3} ) \}.
 \]
Suppose $ H \cap H_{1} \cong K_{4} $.

\medskip
In both $ S_{5} $ and $ S_{4} $, the double-transpositions are centralized by an element of order $ 4 $. If the involutions in $ H \cap H_{1} $ are three double-transpositions in $ H $, then there exists $ h \in H \cap H_{1} $ such that $ h $ is also a double-transposition in $ H_{1} $ and is centralized by elements of order $ 4 $ in both $ H $ and $ H_{1} $. It follows from Lemma \ref{nvntnntrsctin} that $ H \cap H_{1} \not\cong K_{4} $, which would be a contradiction. Thus the involutions in $ H \cap H_{1} $ must consist of two single-transpositions and a double-transposition in $ H $.

\medskip
By the same reasoning $ H \cap H_{1} $ must consist of two single-transpositions and a double-transposition in $ H_{1} $. So there exists $ h_{1}  \in  H \cap H_{1} $ that is a single-transposition in both $ H $ and $ H_{1} $.

\medskip
Suppose $ H \cong PGL_{2} (5) $. In $ S_{5} $, single-transpositions are centralized by an element of order $ 3 $. So $ h_{1} $ is centralized by an element of order $ 3 $ in $ H $ and similarly in $ H_{1} $. Again we get $ H \cap H_{1} \not\cong K_{4} $ using Lemma \ref{nvntnntrsctin}. Another contradiction, showing that the assumption $ H \cap H_{1} \cong K_{4} $ must be wrong in this case.

\medskip
Next suppose $ G = PGL_{2} (q) $ and $ H \cong PGL_{2} (3) $. Write the elements of the intersection as $ H \cap H_{1} = \{ 1_{G} , h_{1} , h_{2} , h_{3} \} $. Suppose $ h_{2} $ is the double-transposition in $ H $. Again using Lemma \ref{nvntnntrsctin}, $ h_{2} $ cannot be the double-transposition in $ H_{1} $. Therefore $ h_{3} $ is the double-transposition in $ H_{1} $ (and must be a single transposition in $ H $).

\medskip
Now Lemma \ref{pslsnpgl} shows there exists $ S \trianglelefteq G $ with $ S \cong PSL_{2} (q) $. Notice that $ H $ is not normal in $ G $ by Lemma \ref{nrmlsbgrpspgl} and $ q $ is non-prime with $ q \geq 3^{3} = 27 $, so $ H \cap S $ is maximal in $ S $ and $ | H \cap S | = \tfrac{1}{2} | H | $ by Lemma \ref{mxmlsbpslntr}.

\medskip
The only subgroup of $ S_{4} $ of index $ 2 $ is $ A_{4} $, which contains all the double-transpositions of $ S_{4} $ but none of the single-transpositions. Hence $ h_{2} \in H \cap S $ and $ h_{1} , h_{3} \notin H \cap S $, in particular $ h_{3} \notin S $.

\medskip
Following the same method as above with $ H_{1} $ instead of $ H $ we get $ h_{3} \in H_{1} \cap S $, which is a contradiction. Thus $ H \cap H_{1} \not\cong K_{4} $ in this case too.
\end{proof}
\begin{lema}\label{kfrsbgrpfive}
Suppose $ q $ is odd, $ G = PSL_{2} (q) $ or $ G = PGL_{2} (q) $ and $ H \cong PGL_{2} ( q_{0} ) $. If $ q_{0} \geq 7 $ then $ H \cap H_{1} \not\cong K_{4} $.
\end{lema}
\begin{proof}
Suppose $ H \cap H_{1} \cong K_{4} $. Let $ h_{1} , h_{2} , h_{3} \in H \cap H_{1} $ be the three involutions in the group. By Lemma \ref{nrmlzr}, for each $ i \in \{ 1, 2 , 3 \} $ the subgroups $ C_{H} (h_{i} ) $ and $ C_{H_{1} } (h_{i} ) $ are dihedral of order at least $  14 $.

\medskip
If $ | C_{H} (h_{j} )| = | C_{H_{1} } (h_{j} ) | $ for some $ j \in \{ 1, 2 , 3 \}  $ then there exists $ c \in C_{H} (h_{j} ) $ and $ c_{1} \in C_{H_{1} } (h_{j} ) $ with $ o(c) = o(c_{1} ) \geq 3 $. Using Lemma \ref{nvntnntrsctin} we find $ H \cap H_{1} \not\cong K_{4} $, a contradiction. Hence $ | C_{H} (h_{j} )| \neq | C_{H_{1} } (h_{j} ) | $.

\medskip
Let $ A_{H} $ and $ B_{H} $ be the two conjugacy classes of maximal dihedrals in $ H $, one containing subgroups of order $ 2( q_{0} + 1 ) $ and the other containing subgroups of order $ 2( q_{0} - 1 ) $. One of the conjugacy classes contains at least two of the centralizers. So we can assume $ C_{H} (h_{1} ) , C_{H} (h_{2} ) \in A_{H} $.

\medskip
Let $ A_{H_{1} } $ and $ B_{H_{1} } $ be the two conjugacy class of maximal dihedrals in $ H_{1} $, where the subgroups in $ A_{H_{1} } $ have the same order as those in $ A_{H} $ and the subgroups in $ B_{H_{1} } $ have the same order as those in $ B_{H} $. It must be that $ C_{H_{1}} (h_{1} ) , C_{H_{1}} (h_{2} ) \in B_{H_{1}} $.

\medskip
Since $ q $ is odd, $ q_{0} $ is also odd. Only one of $ q_{0} - 1 $ or $ q_{0} + 1 $ is divisible by $ 4 $, so exactly one of $ A_{H} $ or $ B_{H} $ contains dihedrals whose order is divisible by $ 8 $.

\medskip
Assume the subgroups in $ A_{H} $ have order divisible by $ 8 $. As $ h_{2} $ and $ h_{3} $ are reflections in $ C_{H} (h_{1} ) $, they are conjugate in $ C_{H} (h_{1} ) $ by Lemma \ref{refconjug}. Therefore they are conjugate in $ H $. Hence $ C_{H} (h_{3} ) \in A_{H} $. It follows that $ C_{H_{1}} (h_{3} ) \in B_{H_{1}} $ and so $ h_{1} , h_{2} $ and $ h_{3} $ are all conjuate in $ H_{1} $.

\medskip
Both $ h_{2} $ and $ h_{3} $ are reflections in $ C_{H_{1}} (h_{1} ) $. The subgroups in $ B_{H_{1} } $ are not divisible by $ 8 $, so Lemma \ref{refconjug} shows $ h_{2} $ and $ h_{3} $ are not conjugate in $  C_{H_{1}} (h_{1} ) $. By Lemma \ref{pglcnjclssrfl}, the two conjugacy classes of reflections in $ C_{H_{1}} (h_{1} ) $ are contained in separate conjugacy classes in $ H_{1} $. In particular, $ h_{2} $ and $ h_{3} $ are not conjugate to each other in $ H_{1} $. A contradiction. Thus the assumption that the subgroups in $ A_{H} $ have order divisible by $ 8 $ must be wrong.

\medskip
The only possibility left is the subgroups in $ A_{H} $ are not divisible by $ 8 $, which implies the subgroups in $ B_{H_{1} } $ are divisible by $ 8 $. The same reasoning as above can be applied with the roles of $ A_{H} $ and $ B_{H_{1} } $ swapped, which again leads to a contradiction. Hence the original assumption $ H \cap H_{1} \cong K_{4} $ must be wrong. 
\end{proof}
Collecting Lemmas \ref{kfrsbgrpone}, \ref{kfrsbgrptwo}, \ref{kfrsbgrpthree} and \ref{kfrsbgrpfive} together, we get;
\begin{lema}\label{kfrsbgrpcllctn}
Suppose $q $ is odd, $ G \neq PSL_{2} (9) $ and $ H \not\cong PSL_{2} (3) $. Then $ H \cap H_{1} \not\cong K_{4} $.
\end{lema}
With the possibility of $ H \cap H_{1} $ being a Klein four-group ruled out, a few more lemmas look at what other group it could be.
\begin{lema}\label{cyclccntrlz}
Let $ c \in H \cap H_{1} $. Suppose $ o(c) > 2 $ and $ o(c) $ is coprime to $ q $. Then there exists $ C \leq H \cap H_{1} $ with $ c \in C $ and $ C \cong C_{ (q_{0} \pm 1) / \delta _{0} }  $.
\end{lema}
\begin{proof}
Observe $ o(c) $ is coprime to $ q_{0} $. Therefore $ N_{H} ( \langle c \rangle ) \cong D_{ 2( q_{0} \pm 1 ) / \delta _{0} } $ by Lemma \ref{nrmlzr}. Similarly $ N_{H_{1} } ( \langle c \rangle ) $ is either a $ D_{ 2( q_{0} - 1 )  / \delta  _{0} } $ or $ D_{ 2( q_{0} + 1 )  / \delta  _{0} } $. Also $ c $ belongs in the rotational subgroups of these normalizers. If $ N_{H} ( \langle c \rangle ) \not\cong N_{H_{1}} ( \langle c \rangle ) $ then $ o(c) $ divides both $ | C_{ ( q_{0} - 1 )  / \delta } | $ and $ | C_{ ( q_{0} + 1 )  / \delta } | $, in particular divides $  q_{0} - 1  $ and $  q_{0} + 1  $, which is not possible. Therefore $ N_{H} ( \langle c \rangle ) \cong N_{H_{1}} ( \langle c \rangle ) $.

\medskip
Using Lemma \ref{nrmlzr} again we have $ N_{G} ( \langle c \rangle ) \cong D_{ 2( q \pm 1 ) / \delta } $. Let $ C $ be the rotational subgroup of $ N_{H} ( \langle c \rangle ) $. Then $ C $ is cyclic, $ | C | \geq o(c) \geq 3 $ and $ C $ centralizes $ \langle c \rangle $. So $ C $ is a subgroup of the rotational subgroup of $ N_{G} ( \langle c \rangle ) $. This is the only subgroup of $ N_{G} ( \langle c \rangle ) $ of order $ | C | $. Also $ N_{G} ( \langle c \rangle ) $ contains the rotational subgroup of $ N_{H_{1}} ( \langle c \rangle ) $, which is cyclic of order $ | C | $ and thus must be $ C $.
\end{proof}
\begin{lema}\label{ntadhdrlntr}
If $ |H \cap H_{1} | > 2 $ and $ H \cap H_{1} $ is a subgroup of some $ D < H $ where $ D \cong D_{ 2(q_{0} \pm 1) / \delta _{0} } $, then either $ H \cap H_{1} \cong D_{ 2(q_{0} \pm 1) / \delta _{0} } $ or $ H \cap H_{1} \cong C_{ (q_{0} \pm 1) / \delta _{0} } $. If $ H \cap H_{1} \cong D_{ 2(q_{0} \pm 1) / \delta _{0} } $ then $ q $ is odd.
\end{lema}
\begin{proof}
Put $ K := H \cap H_{1} $. If $ | D | = 4 $ then the result is true as the only possibility is $ K = D $. So assume $ |D | > 4 $.

\medskip
Let $ C $ be the rotational subgroup of $ D$. Put $ C_{K} := K \cap C $. It must be that $ | C_{K} | > 1 $, otherwise $ | K | \leq 2 $. The only way that $  | C_{K} | = 2 $ is if $ K \cong K_{4} $. The subgroup $ D $ can only have $ K_{4} $ subgroups if $ q_{0} $ is odd, however we then have $ K \not\cong K_{4} $ by Lemma \ref{kfrsbgrpcllctn}. So $  | C_{K} | > 2 $.

\medskip
The subgroup $ C_{K} $ is cyclic, so $ C_{K} =  \langle c \rangle $ for some $ c \in K $. Note that $ o(c) $ is coprime to $ q_{0} $ because it divides $ | C | = (q_{0} \pm 1 ) / \delta _{0} $. Since $ o(c) > 2 $, there exists a $ C^{*} \leq K $ with $ c \in C^{*} $ and $ C^{*} \cong C_{ ( q_{0} + 1 )  / \delta _{0} } $ or $ C^{*} \cong C_{ ( q_{0} - 1 )  / \delta _{0} } $ by Lemma \ref{cyclccntrlz}. Clearly $ C^{*} \leq C $.

\medskip
The only way $ C^{*} \neq C $ would be if $ C^{*} \cong C_{ ( q_{0} - 1 )  / \delta _{0} } $ and $ C \cong C_{ ( q_{0} + 1 )  / \delta _{0} } $. For this to happen, $ C^{*} = \langle c^{*} \rangle $ for some non-identity $ c^{*} $ whose order divides both $ q_{0} -1 $ and $ q_{0} + 1 $. Therefore $ | C^{*} | = o(c^{*} ) = 2 $ and $ c \not\in C^{*} $, a contradiction. Hence $ C = C^{*} \leq K $. Thus either $ K \cong C $ or $ K \cong D $.

\medskip
Suppose $ K \cong D_{ 2(q_{0} \pm 1) / \delta _{0} }  $. If $ q_{0} $ is even then $ q $ is also even. Also $ q_{0} \geq 4 $, so $ K $ is maximal in both $ H $ and $ H_{1} $ by Lemma \ref{nrmlzr} and the maximal $ D_{ 2(q_{0} \pm 1) / \delta _{0} } $ in these groups lie in a single conjugacy class. Therefore there exists $ g \in G $ such that $ H_{1} = g^{-1} H g $ and $ g^{-1} K g = K $. Let $ r \in K $ be a reflection. Let $ R $ be the rotational subgroup of $ K $. This has odd order $ | R | = q_{0} \pm 1 $. So all reflections in $ K $ are conjugate and there exists $ s \in R $ such that $ r = s^{-1} (g^{-1}  r g)s = (gs) ^{-1} r (gs) $. Since $ s \in H_{1} $ and $ s \in K $, it normalizes both $ H_{1} $ and $ K $. So we may as well assume $ g^{-1} r g = r $. Using Lemma \ref{nrmlzr} again, we have $ N_{G} (R) \cong D_{ 2(q + 1) / \delta } $ or $ N_{G} (R) \cong D_{ 2(q - 1) / \delta } $ and this must contain $ g $. Since $ q -1 $ and $ q+1$ are odd, $ \langle r \rangle $ is a non-normal Sylow $2$-subgroup of $ N_{G} (R) $ and so it stabilizes itself. Therefore $ g = r $ or $ g = 1_{G} $. Either way, $ g \in K \leq H $ and $ g^{-1} H g = H \neq H_{1} $, which is a contradiction. Thus $ q_{0} $ is not even.
\end{proof}

\newpage
\begin{lema}\label{lttlbrlnrm}
Let $ B_{0} $ be a maximal Borel subgroup of $ H $. Then $ N_{G} (B_{0} ) = B_{0} $.
\end{lema}
\begin{proof}
Write $ B_{0} = U_{0} \rtimes T_{0} $ where $ U_{0} $ is elementary abelian of order $ q_{0} $ and $ T_{0} \cong C_{ (q_{0} - 1) / \delta _{0} } $. By Lemma \ref{brlcyclcprts}, $ U_{0} $ is the only subgroup of $ B_{0} $ of order $ q_{0} $. Hence $ U_{0} \trianglelefteq N_{G} (B_{0} ) $. So Lemma \ref{nrmlzrnbrlx} tells us that there exists a Borel subgroup $ B $ that is maximal in $ G $ with $ N_{G} (B_{0} ) \leq B $.

\medskip
Lemma \ref{brlcyclcprts} shows there exist $ q $ conjugate subgroups $ A_{1} , \dots , A_{q} \leq B $ isomorphic to $ C_{ (q-1) / \delta } $ and that these subgroups intersect each other trivially. The elements not in these $ C_{ (q-1) / \delta } $ lie in an elementary abelian subgroup $ U \leq B $ of order $ q $. Since $ q_{0} $ does not divide $ (q-1) / \delta $, we have $ U_{0} \leq U $.

\medskip
Similarly there exist $ q_{0} $ conjugate subgroups $ A_{1} ^{ \prime } , \dots , A_{q_{0} } ^{ \prime } \leq B_{0} $ that are isomorphic to $ C_{ (q_{0} - 1) / \delta _{0} } $ and intersect each other trivially. The elements of $ B_{0} $ not in these $ C_{ (q_{0} -1) / \delta _{0} } $ lie in $ U_{0} $. The generators of the $ C_{ (q_{0} -1) / \delta _{0} } $ have order $ (q_{0} -1) / \delta _{0} $, which does not divide $ q $. Hence they lie in one of the $ C_{ (q-1) / \delta } $ in $ B $. So we may assume these cyclic subgroups are labelled such that $ A_{i} ^{ \prime } \leq A_{i} $ for each $ i \in \{ 1 , \dots , q_{0} \} $.

\medskip
Let $ g \in  N_{G} (B_{0} ) $. It will be shown that $ g \notin B \setminus B_{0} $. This will be covered in three cases, first being $ g \notin U \setminus U _{0} $, second we will see $ g \notin A_{i} \setminus A_{i} ^{ \prime } $ for $ i \in \{ 1 , \dots , q_{0} \} $ and third it will be shown $ g \notin A_{j} \setminus \{ 1 _{G} \} $ for $ j \in \{ q_{0} + 1 , \dots , q \} $.

\medskip
\textbf{Case (i).} By Lemma \ref{brlcyclcprts}, each element of $ U_{0} $ sends $ A_{1} ^{ \prime } $ to exactly one element of $ \{ A_{1} ^{ \prime } , \dots , A_{q_{0} } ^{ \prime } \} $ under conjugation. The same applies to $ A_{1} $ being sent to each of the subgroups in $ \{ A_{1} , \dots , A_{q_{0} } \} $. So, using Lemma \ref{brlcyclcprts} again, if $ g \in U \setminus U_{0} $ we would have $ g^{-1} A_{1} g \notin \{ A_{1} , \dots , A_{q_{0} } \} $ and $ g^{-1} A_{1} ^{ \prime } g \notin \{ A_{1} ^{ \prime } , \dots , A_{q_{0} } ^{ \prime } \} $. But this implies $ g^{-1} B_{0} g \neq B_{0} $, a contradiction. Therefore $ g \notin U \setminus U_{0} $.

\medskip
\textbf{Case (ii).} For this case, first suppose $ H \cong PGL_{2} (q_{0} ) $. Then $ \delta _{0} = 1 $ and $ | A_{i} ^{ \prime } | = q_{0} - 1  $ for each $ i \in \{ 1 , \dots , q_{0} \} $. Let $ m \in \{ 1 , \dots , q_{0} \} $. Let $ a_{1} , a_{2} \in A_{m} $ and $ u \in U_{0} \setminus \{ 1 _{G} \} $. Suppose $ a^{-1} _{1} u a_{1} = a^{-1} _{2} u a_{2} $. Then $ a_{1} a_{2} ^{-1} = u^{-1} a_{1} a^{-1} _{2} u $.

\medskip
Let $ u_{1} \in U \setminus \{ 1 _{G} \} $. Lemma \ref{brlcyclcprts} shows that $ u^{-1} a_{1} a^{-1} _{2} u \in  u_{1} ^{-1} C u_{1} $ if and only if either $ u = u_{1} $ or $ u^{-1} a_{1} a^{-1} _{2} u = 1_{G} $.

\medskip
So $ a_{1} a_{2} ^{-1} = u^{-1} a_{1} a^{-1} _{2} u $ implies $ a_{1} a^{-1} _{2} \in u^{-1} A_{m} u $. Therefore $ a_{1} a_{2} ^{-1} = 1_{G} $ and $ a_{1} = a_{2} $.

\medskip
What we can take from this is that conjugating $ u $ by the elements of $ A_{m} $ provides us with a set of $ | A_{m} | = (q-1) / \delta $ non-identity elements of $ U $. Out of these, $ | A_{m} ^{ \prime } | = q_{0} - 1 $ of the elements are obtained by conjugating by elements of $ A_{m} ^{ \prime } $. This is the same number of non-identity elements in $ U_{0} $. So the elements in $ A_{m}  \setminus A_{m} ^{ \prime } $ send $ u $ outside of $ U_{0} $ and thus do not normalize $ U_{0} $. Since $ U_{0} $ is the only subgroup of order $ q_{0} $ in $ B_{0} $, we have $ g^{-1} U_{0} g = U_{0} $. Hence $ g \notin A_{m}  \setminus A_{m} ^{ \prime } $.

\medskip
Next suppose $ H \cong PSL_{2} (q_{0} ) $ and $ q_{0} $ is odd. Then $ G \cong PSL_{2} (q ) $ and $ q $ is odd. Also $ \delta = \delta _{0} = 2 $ and $ | A_{i}  | = (q - 1) / 2  $ and $ | A_{i} ^{ \prime } | = (q_{0} - 1) / 2  $ for each $ i \in \{ 1 , \dots , q_{0} \} $, so the above method will not work here.

\medskip
Put $ G^{ *} := PGL_{2} (q) $. Using Lemma \ref{pslsnpgl} we may suppose $ G \leq G^{ * } $. For these values of $ q$ and $ q_{0} $, there exists a subfield subgroup $ H^{ \prime  } \leq  G^{ * } $ with $ H^{ \prime } \cong PGL_{2} (q_{0} ) $. Now $ q_{0} \geq 5 $, so $ q \geq 5^{3} > 13 $. Also $ q $ is not prime otherwise $ PSL_{2} (q) $ would have no subfield subgroups. Therefore $ H^{ \prime } \cap G $ is a maximal subgroup of $ G $ of order $ | H^{ \prime } \cap G | = \tfrac{1}{2} | H^{ \prime } | = \tfrac{1}{2} q_{0} (q_{0} + 1)(q_{0} - 1) $ by Lemma \ref{mxmlsbpslntr}.

\medskip
Since $ q > 13 $, the only maximal subgroups of $ PSL_{2} (q) $ who order is divisible by $ p $, where $ p $ is the prime dividing $ q $, are Borel subgroups or subfield subgroups. Borel subgroups have order divisible by $ q $, whereas $ H^{ \prime } \cap G $ does not. Hence $ H^{ \prime } \cap G $ is a subfield subgroup of $ G $.

\medskip
Since all subfield subgroups in $ G $ are conjugate, we can conjugate $ H^{ \prime } \cap G $ by some element of $ G $ to get $ H $. Conjugating $  H^{ \prime } $ by the same element gives some subfield subgroup $ H^{*} \leq G^{*} $ such that $ H^{ * } \cap G = H $.

\medskip
Since $ U_{0} \leq H \leq H^{*} $, Lemmas \ref{brlcyclcprts} and \ref{nrmlzrnbrlx} together show that $ N_{H^{*} } ( U_{0} ) $ is a Borel subgroup of $ H^{*} $. Denote $ B_{0} ^{*} := N_{H^{*} } ( U_{0} ) $. Note $ U_{0} \trianglelefteq B_{0} $. So $ B_{0} \leq B_{0} ^{*} $.

\medskip
By the same reasoning there exists a Borel subgroup $ B^{*} $ of $ G^{*} $ such that $ B^{*} \cap G = B $. As in the above discussion, there exist $q $ distinct subgroups $ F_{1} , \dots , F_{q} \leq B^{*} $ each being isomorphic to $ C_{q-1} $ and intersecting trivially. By Lemma \ref{brlcyclcprts} the group $ A_{i} $ must be a subgroup of one of these $ C_{q-1} $, so we may assume (relabelling if needed) that $ A_{i} \leq F_{i} $.

\medskip
Lemma \ref{prmsbgrps} shows there is only one Sylow $p$-subgroup of $ G ^{*} $ that contains $ U_{0} $ as a subgroup and this must be $ U $. Also by Lemma \ref{nrmlzrnbrlx} we have $ N_{G^{*} } ( U) = B^{*} $ and $ B^{*} $ is the only Borel subgroup containing $ U $ (and therefore $ U_{0} $). The same lemma shows that $ N_{G^{*} } ( U_{0} ) \leq B^{*} $. Therefore $ B_{0} ^{*} \leq B ^{*} $.

\medskip
Once again $ B_{0} ^{*} $ has $ q_{0} $ subgroups isomorphic to $ C_{ q_{0} - 1 } $ and they intersect each other trivially. Using Lemma \ref{brlcyclcprts} again, $ A_{i} ^{ \prime } $ is contained in one of these $ C_{ q_{0} - 1} $, say $ F_{i} ^{ \prime } \leq B_{0} ^{*}  $. Since $ A_{i} ^{ \prime } \leq A_{i} \leq F_{i} $ and Lemma \ref{brlcyclcprts} shows $ F_{i} ^{ \prime } $ and $ A_{i} $ are contained in exactly one $ C_{ q-1} $ in $ B^{*} $, we have $ F_{i} ^{ \prime } \leq F_{i} $.

\medskip
Now $ q = q_{0} $ where $ r $ is an odd prime. The number $ q_{0} ^{r-1} + q_{0} ^{r-2} + \dots + q_{0} + 1 $ is odd because it is the sum of an odd number of odd numbers. Therefore $ | A_{i} | = \tfrac{1}{2} (q-1) = \tfrac{1}{2} (q_{0} - 1)(q_{0} ^{r-1} +  q_{0} ^{r-2} + \dots + q_{0} + 1 ) $ is not divisible by $ | F_{i} ^{ \prime } | = q_{0} - 1 $. Hence $ F_{i} ^{ \prime } \cap A_{i} $ is a proper subgroup of $ F_{i} ^{ \prime } $ that contains $ A_{i} ^{ \prime } $. Since $ A_{i} ^{ \prime } $ has index $ 2 $ in $ F_{i} ^{ \prime } $, we have $ F_{i} ^{ \prime } \cap A_{i} = A_{i} ^{ \prime }  $. Hence
\begin{align*}
g \in A_{i} \setminus A_{i} ^{ \prime } 
= ( F_{i} \cap A_{i} ) \setminus  (F_{i} ^{ \prime } \cap A_{i} )
\subseteq F_{i} \setminus F_{i} ^{ \prime } .
\end{align*}
We saw toward the start this case, when dealing with subfield subgroups isomorphic to $ PGL_{2} (q_{0} ) $, that if $ g \in F_{i} \setminus F_{i} ^{ \prime }  $ then $ g $ does not normalize $ U_{0} $. However all elements of $ N_{G} ( B_{0} ) $ normalize $ U_{0} $ because this is the only subgroup of $ B_{0} $ of order $ q_{0} $. So we get a contradiction. Thus $ g \notin A_{i} \setminus A_{i} ^{ \prime } $.

\medskip
\textbf{Case (iii).} Let $ n \in \{ q_{0} +1 , \dots , q \} $. Suppose $ g \in A_{n} \setminus \{ 1_{G} \} $. Then $ g \notin B_{0} $. Also $ g^{-1} A_{1} g = A_{k} $ for some $ k \in \{ 1 , \dots , q_{0} \} $. There exists some $ v \in U_{0} $ such that $ v^{-1} A_{k} v = A_{1} $ by Lemma \ref{brlcyclcprts}. So $ gv \in N_{G} ( A_{1} ) $. Also $ U_{0} \leq N_{G} (B_{0} ) $, so $ v \in N_{G} (B_{0} ) $. Hence $ gv \in N_{G} (B_{0} ) $. From the discussion in the above cases either $ gv \in B_{0} $ or $ gv \in A_{t} \setminus \{ 1_{G} \} $ for some $ t \in \{ q_{0} +1 , \dots , q \}  $.

\medskip
It must be that $ gv \notin B_{0} $, otherwise $ g = gvv^{-1} \in B_{0} $.  Therefore $ gv \in A_{t} $.

\medskip
Now $ N_{G} ( A_{1} ) \cong D_{ 2( q \pm 1 ) / \delta } $ by Lemma \ref{nrmlzr}. We have $ g \neq v ^{-1} $ otherwise $ g \in B_{0} $. Hence $ o(gv) > 1 $. Note $ A_{1} $ also has an element of order $ o(gv) $. If $ o(gv) > 2 $ then $ gv $ is in the rotational subgroup of $  N_{G} ( A_{1} ) $. Since the rotational subgroup is cyclic it only contains one subgroup of order $ o(gv) $, which is $ \langle gv \rangle $. This subgroup must be contained in $ A_{1} $. But Lemma \ref{brlcyclcprts} shows $ A_{1} \cap A_{t} = \{ 1_{G} \}  $, so we have a contradiction. Therefore $ o(gv) = 2 $.

\medskip
Now it must be that $ | A_{t} | = (q_{0} - 1) / \delta _{0} $ is divisible by $ 2 $ and so $ (q - 1) / \delta $ is too. Hence $ q_{0} $ and $ q $ are odd. Also $ (q - 1) / \delta > (q_{0} - 1) / \delta _{0}  \geq 2 $. So $ A_{1} $ is in the rotational subgroup of $ N_{G} ( A_{1} )  $. In particular the involution $ a_{1} \in A_{1} $ is a rotation. So $ gv $ is a reflection and $ \langle a_{1} , gv \rangle \cong K_{4} $. By Lemma \ref{brlklnfrsbgrp}, there are no subgroups of $ B $ isomorphic to $ K_{4} $. So we have a contradiction again and the assumption $ g \in A_{n} \setminus \{ 1_{G} \} $ must be wrong.

\medskip
With all three cases dealt with, the only possibility is $ g \in B_{0} $. Thus $ N_{G} (B_{0} ) = B_{0} $.
\end{proof}
\begin{lema}\label{sdgiulvsldkcn}
Suppose $ | H \cap H_{1} | > 2 $ and $ H \cap H_{1} $ is a subgroup of a Borel subgroup $ B_{0} < H $. Write $ B_{0} = U_{0} \rtimes T_{0} $ where $ U_{0} $ is an elementary abelian group of order $ q_{0} $ and $ T_{0} \cong C_{(q_{0} - 1) / \delta _{0} } $. Then $ H \cap H_{1} = U_{0} $ or $ H \cap H_{1} \cong C_{(q_{0} - 1) / \delta _{0} } $.
\end{lema}
\begin{proof}
By Lemma \ref{sbgrpfbrlq} we can write $ H \cap H_{1} = U_{1} \rtimes T_{1} $ where $ U_{1} \leq U_{0} $ and $ T_{1} $ is isomorphic to a subgroup of $ C_{(q_{0} - 1) / \delta _{0} } $.

\medskip
Two special cases will be gotten out of the way first. Suppose $ H \cong PGL_{2} (3) $ or $ H \cong PSL_{2} (5) $. Note $ U_{0} \cong C_{3} $ if $ H \cong PGL_{2} (3) $ or $ U_{0} \cong C_{5} $ if $ H \cong PSL_{2} (5) $. Also $  C_{(q_{0} - 1) / \delta _{0} } \cong C_{2} $ in both cases. If $ T_{1} $ is trivial then $ H \cap H_{1} = U_{1} = U_{0} $ since $ U_{0} $ is the only non-trivial subgroup of itself.

\medskip
So suppose $ T_{1} $ is not trivial. Then we have $ T_{1} \cong C_{2} $. Also $ U_{1} $ is non-trivial otherwise we have $ | H \cap H_{1} | = | T_{1} | = 2 $. So again we get $ U_{1} = U_{0} $. It follows that  $ H \cap H_{1} = B_{0} $.

\medskip
All Borel subgroups in $ H $ are conjugate. Same for the Borel subgroups of $ H_{1} $. So there exists $ g \in G $ such that $ g^{-1} H g = H_{1} $ and $ g^{-1} B_{0} g = B_{0} $. From Lemma \ref{lttlbrlnrm} we see that $ N_{G} (B_{0} ) = B_{0} $. Hence $ g \in B_{0} $. But then $ g \in H $ and $ g^{-1} H g = H $, implying $ H = H_{1} $. This is a contradiction. Thus $ T_{1} $ is trivial in these special cases.

\medskip
Now for the general case. Suppose $ H \not\cong PGL_{2} (3) $ and $ H \not\cong PSL_{2} (5) $. Then $ | (q_{0} - 1 ) / \delta _{0} | > 2 $. If $ U_{1} $ is trivial then $ H \cap H_{1} = T_{1} $. It follows that $ | T_{1} | > 2 $, so only divides one of $ q_{0} - 1 $ or $ q_{0} + 1 $. It is not possible that $ | T_{1} | $ divides $ q_{0} + 1 $ because $  T_{1}   $ is isomorphic to a subgroup of $ C_{q_{0} - 1} $. Also $ | T_{1} | $ is coprime to $ q_{0} $, so we have  $ N_{H} (H \cap H_{1} ) \cong  D_{2(q_{0} -1) /\delta _{0} } $ by Lemma \ref{nrmlzr}. From Lemma \ref{ntadhdrlntr} we see that $ H \cap H_{1} \cong C_{ (q_{0} -1) / \delta _{0} } $.

\medskip
Now suppose $ U_{1} $ is non-trivial. Then by Lemma \ref{pgrpqzero} there exists $ P \leq H \cap H_{1} $ with $ U_{1} \leq P $ and $ | P | = q_{0} $. Lemma \ref{brlcyclcprts} shows that $ U_{0} $ is the only subgroup of $ B_{0} $ of order $ q_{0} $. Hence $ P = U_{0} $ and $ U_{0} \trianglelefteq H \cap H_{1} $.

\medskip
Suppose for a contradiction that $ T_{1} $ is non-trivial. Observe that $ H \cap H_{1} \leq N_{H_{1} } (U_{1} ) $. Hence $ H \cap H_{1} $ is a subgroup of some Borel subgroup $ B_{1} < H_{1} $ by Lemma \ref{nrmlzrnbrlx}. As $ | T_{1} | $ divides $ q_{0} - 1 $, it is coprime to $ q_{0} $. It follows from Lemma \ref{brlcyclcprts} that there exists $ A_{0} \leq B_{0} $ and $ A_{1} \leq B_{1}  $ that both contain $ T_{1} $, where $ A_{0} \cong A_{1}  \cong C_{ (q_{0} - 1 ) / \delta _{0} } $.

\medskip
Lemma \ref{nrmlzr} shows that $ N_{G} ( T_{1} ) \cong D_{2 (q \pm 1 ) / \delta } $. This normalizer contains both $ A_{0} $ and $ A_{1} $. Since $ | ( q_{0} -1 ) / \delta _{0} | > 2 $, there is only one $ C_{ ( q_{0} -1 ) / \delta _{0} } $ subgroup of $ N_{G} ( T_{1} )  $. Hence $ A_{0} = A_{1} $ and this is a subgroup of $ H \cap H_{1} $.

\medskip
Lemma \ref{brlcyclcprts} shows that
\[
\bigcup _{u \in U_{0} } u^{-1} A_{0} u
\]
contains $ | U_{0} | ( ( q_{0} -1 ) / \delta _{0} - 1 ) =  | U_{0} |  ( q_{0} -1 ) / \delta _{0} - | U_{0} | $ non-identity elements, none of which are in $ U_{0} $. As $ U_{0} \leq H \cap H_{1} $, we can count the elements above along with those in $ U_{0} $ to see that
\[
 | H \cap H_{1} | \geq  | U_{0} |  ( q_{0} -1 ) / \delta - | U_{0} | + | U_{0} | =  | U_{0} |  ( q_{0} -1 ) / \delta = | B_{0} | \geq  | H \cap H_{1} |.
\]
Hence $ | H \cap H_{1} | = | B_{0} | $ and we have $ H \cap H_{1} = B_{0}  $. Following the reasoning from the $ PGL_{2} (3) $ and $ PSL_{2} (5) $ special cases earlier in the proof, we cannot have $ H \cap H_{1} = B_{0}  $. This is a contradiction. Thus $ T_{1} $ is trivial. It follows that $ H \cap H_{1} = U_{1} $. As $ U_{1} \leq U_{0} \leq  H \cap H_{1} $ we have $ H \cap H_{1} = U_{0} $.
\end{proof}
The previous results are now collected together.
\begin{lema}\label{dffrntsbgrps}
If $ H \cap H_{1} $ is non-trivial then it is isomorphic to one of the following;
\begin{itemize}
\item $ C_{2} $,
\item $ D_{ 2(q_{0} \pm 1 ) / \delta _{0} } $ (only if $ q $ is odd),
\item $ C_{ (q_{0} \pm 1 ) / \delta _{0} } $ or
\item $ E_{q_{0}} $. 
\end{itemize}
\end{lema}
\begin{proof}
Once again we begin by dealing with the case $ H \cong PGL_{2} (3) \cong S_{4} $ separately. As it is being assumed $ G \neq PSL_{2} (9) $, we have $  H \cap H_{1} \not\cong K_{4} $ by Lemma \ref{kfrsbgrpcllctn}. Lemma \ref{nnnafafsf} shows that $ H \cap H_{1} \not \cong A_{4} $. So the only possible subgroups $ H \cap H_{1} $ could isomorphic to are $ C_{2} $, $ C_{3} = E_{ q_{0} } $, $ D_{6} $ (which is a Borel subgroup here), $ C_{4} = C_{  q_{0} + 1  } $ or a $ D_{8} = D_{  2(q_{0} + 1 ) } $. So suppose from now on $ H \not\cong PGL_{2} (3) $. Then $ q_{0} \neq 3 $.

\medskip
If $ H \cap H_{1} \cong C_{2} $ there is nothing further to show. So suppose not.

\medskip
We have $ H \cap H_{1} \neq H $, so $ H \cap H_{1} \leq M $ where $ M $ is a maximal subgroup of $ H $. If $ M $ is a maximal $  D_{ 2(q_{0} \pm 1 ) / \delta _{0} } $ then $ H \cap H_{1} \cong C_{ (q_{0} \pm 1 ) } $ or $ H \cap H_{1} \cong D_{ 2(q_{0} \pm 1 ) / \delta _{0} } $ by Lemma \ref{ntadhdrlntr}. The same lemma shows that $ H \cap H_{1} \cong D_{ 2(q_{0} \pm 1 ) / \delta _{0} } $ only if $ q $ is odd.

\medskip
If $ M $ is a Borel subgroup of $ H $ then $ H \cap H_{1} \cong  C_{ (q_{0} - 1 ) / \delta _{0} } $ or $ H \cap H_{1} \cong E_{q_{0}} $ by Lemma \ref{sdgiulvsldkcn}.

\medskip
If $ q_{0} $ is odd and $ H \cap H_{1} $ is a subgroup of some (not necessarily maximal) $ K \leq H $ that is isomorphic to $ A_{4} $, $ S_{4} $ or $ A_{5} $ then Lemma \ref{nnnafafsf} shows that $ H \cap H_{1} \neq K $ and $ H \cap H_{1} \not \cong A_{4} $. Also $  H \cap H_{1} \not\cong K_{4} $ by Lemma \ref{kfrsbgrpcllctn}. So $ H \cap H_{1} $ is isomorphic to one of the following;
\begin{itemize}
\item $ C_{3} $,
\item $ D_{6} $,
\item $ C_{4} $,
\item $ D_{8} $,
\item $ C_{5} $,
\item $ D_{10} $.
\end{itemize}
Each of the above groups normalizes a non-trivial cyclic subgroup. If $ H \cap H_{1} $ is isomorphic to one of these groups and does not contains a non-identity element of with order that divides $ q_{0} $, then $ H \cap H_{1} $ is a subgroup of some $ D_{ 2( q_{0} \pm 1 ) / \delta _{0} } $ by Lemma \ref{nrmlzr}. Hence $ H \cap H_{1} \cong D_{ 2( q_{0} \pm 1 ) / \delta _{0} } $ or $ H \cap H_{1} \cong C_{ ( q_{0} \pm 1 ) / \delta _{0} } $ as seen in the earlier cases. So we just need to consider what happens if $ H \cap H_{1} $ is isomorphic to one of the above groups and has a non-identity element whose order divides $ q_{0} $.

\medskip
Suppose $ q_{0} $ is a power of $ 3 $. Then the $  D_{6} $ or $ C_{3} $ normalize a $ C_{3} $ and so are subgroups of a maximal Borel subgroup of $ H $ by Lemma \ref{nrmlzrnbrlx}. The case where $ H \cap H_{1} $ is contained in a maximal Borel has already been dealt with earlier, from which we see that that only possibility is $ H \cap H_{1} \cong C_{3} = E_{q_{0} } $ (which only happens if $ q_{0} = 3 $).

\medskip
If $ q_{0} $ is a power of $ 5 $. Then the $  D_{10} $ or $ C_{5} $ normalize a $ C_{5} $ and so are subgroups of a maximal Borel subgroup of $ H $ by Lemma \ref{nrmlzrnbrlx}. Again we see the only possibility is $ H \cap H_{1} \cong C_{5} = E_{q_{0} } $ (which only happens if $ q_{0} = 5 $).

\medskip
The only other subgroups that $ M $ could be than those already given or those that $ K$ might be are subfield subgroups or $ PSL_{2} (q_{0} ) $ (if $ q_{0} $ is odd and $ H \cong PGL_{2} (q_{0} ) $). Suppose $M $ is one of these.

\medskip
If $ M $ is a subfield subgroup and $ p $ is the prime that divides $ q_{0} $ then the largest $ p $-subgroup of $ M$ has order $ q_{1} $ where $ 1 < q_{1} < q_{0} $ and $ q_{1} $ divides $ q_{0} $. Lemma \ref{pgrpqzero} shows that if $ H \cap H_{1} $ contains a $ p $-group then it has a subgroup of order $ q_{0} $. Therefore $ M \neq H \cap H_{1} $.

\medskip
If $ M = PSL_{2} (q_{0} ) $ then since $ q_{0} $ is odd and $ H \cong PGL_{2} (q_{0} ) $. So in this case we have $ M \trianglelefteq H$. Also $ M \trianglelefteq H_{1} $. This is not possible by Lemma \ref{ntnrmlmxml}. Thus $ M \neq H \cap H_{1} $ again.

\medskip
So in either of these two cases there exists a maximal subgroup $ M_{1} $ of $ M $ with $ H \cap H_{1} \leq M_{1} $. We can go through a similar process as above. If $ M_{1} $ is a Borel subgroup then it is in the normalizer of a non-trivial $ p$-group, so $ M _{1} $ and $ H \cap H_{1} $ are contained in a maximal Borel subgroup of $ H $ by Lemma \ref{nrmlzrnbrlx}. Hence $ H \cap H_{1} $ is of the required form by the earlier discussion.

\medskip
If $ M _{1} $ is a maximal $ D_{ q_{1} \pm 1  } $ (if $ M \cong PSL_{2} (q_{1} ) $) or $ D_{ 2(q_{1} \pm 1 ) } $ (if $ M  \cong PGL_{2} (q_{1} ) $) then $ M_{1} $ normalizes a rotational subgroup whose order is coprime to $ q_{0} $. Hence $ M_{1} $ and $ H \cap H_{1} $ lie in a maximal dihedral of $ H $ by Lemma \ref{nrmlzr}. Thus $ H \cap H_{1} $ is of the required form by the earlier discussion.

\medskip
If $ M_{1} $ is a maximal $ A_{4} $, $ S_{4} $ or $ A_{5} $ then $ H \cap H_{1} $ is of the required form by the earlier discussion.

\medskip
If $ M_{1} $ another maximal subgroup then it is either a subfield subgroup of $ M $ or isomorphic to $ PSL_{2} (q_{1} ) $. In either of these two cases we have $ H \cap H_{1} \neq M_{1} $ because the maximum power of $ p $ that divides $ |M _{1} | $ is strictly between $ 1 $ and $ q_{0} $, but Lemma \ref{pgrpqzero} shows $ H \cap H_{1} $ would contain a subgroup of order $ q_{0} $. Thus $ H \cap H_{1} \neq M_{1} $ and $ H \cap H_{1} \leq M_{2} $ where $ M_{2} $ is a maximal subgroup of $ M_{1} $. 

\medskip
This process can be repeated to give a chain of $ m \in \mathbb{N} $ subgroups
\[
H > M_{1} > M_{2} > \dots > M_{m}
\]
such that for $ i \in \{ 2 , \dots , m \} $ each $ M_{i} $ is a subfield subgroup or a maximal $ PSL_{2} (q_{i} ) $ in $M _{i-1} $, where $ q_{i} $ divides $ q_{0} $. This chain cannot be extended indefinitely, so assume it cannot be extended beyond $ M_{m} $. For the same reasons as the other subgroups in the chain, $ H \cap H_{1} \neq M_{m} $. Therefore $ H \cap H_{1} $ is a subgroup of a Borel, maximal dihedral, $ A_{4} $, $A_{5} $ or $ S_{4} $ in $ M_{m} $, from which we can use induction or have already been shown earlier that $ H \cap H_{1} $ is isomorphic to one of the required groups.
\end{proof}
For the next lemma we want a fact about the product of some involutions in $ A_{5} $. Suppose $ K_{1} , K_{2} \leq A_{5} $ with $ K_{1} \cong D_{6} \cong K_{2} $. Suppose $ K_{1} \cap K_{2} \cong C_{2} $. Then we can write
\[
K_{1} = \{ e , \ (r \ s \ t ) , \ (r \ t \ s ) , \ (r \ s) (u \ v) , \ (r \ t) (u \ v) , \ (s \ t) (u \ v) \}
\]
and
\[
K_{2} = \{ e , \  (t \ u \ v ) , \ (t \ v \ u ) , \ (r \ s) (u \ v) , \ (r \ s) (t \ u) , \ (r \ s) (t \ v) \} ,
\]
where $ e $ is the identity and $ \{ r,s,t,u,v \} = \{ 1, 2,3,4,5 \} $. The point here is that the product of an involution in $ K_{1} \setminus ( K_{1} \cap K_{2} ) $ with an involution in $ K_{2} \setminus ( K_{1} \cap K_{2} ) $ gives an element of order $ 5 $. For example $ (r \ t) (u \ v) \cdot (r \ s)  (t \ u) = ( r \ u \ v \ t \ s ) $ has order order $ 5 $. The other products can be checked.
\begin{lema}\label{hghtsbfldac}
The height of the action of $ G $ on $ \Omega $ has bound $ Ht( G, \Omega ) \leq 3 $.
\end{lema}
\begin{proof}
If $ Ht( G, \Omega ) > 3 $ then there exists and independent set of size $ 4 $ in $ \Omega $ by Corollary \ref{indsub}. Suppose, for a contradiction, that there exists an independent set $ \Delta := \{ H_{1} , H_{2} , H_{3} , H_{4} \} \subseteq \Omega $.

\medskip
For distinct $ r , s, t \in \{ 1, 2 , 3 , 4 \} $ with $ r \neq s \neq t \neq r $, we will consider each of the possible choices for $ H_{r} \cap H_{s} $ given in Lemma \ref{dffrntsbgrps} and then show in each case $ \Delta $ cannot be an independent set. First note that $ H_{r} \cap H_{s} \neq \{ 1_{G} \} $, otherwise $ G_{  H_{1} , H_{2} , H_{3} , H_{4} }  = G_{ H_{r} , H_{s} } = \{ 1_{G} \} $ and $ \Delta $ would not be an independent set by Definition \ref{stdfnind}. The same reasoning shows $ H_{r} \cap H_{s} \cap H_{t} \neq \{ 1 _{G} \} $.

\medskip
\textbf{Case 1:} Suppose $ H_{r} \cap H_{s} \cong C_{2} $. We have $ H_{r} \cap H_{s} \cap H_{t} < H_{r} \cap H_{s} $ by Lemma \ref{stach}. Therefore $ H_{r} \cap H_{s} \cap H_{t} = \{ 1_{G} \} $, a contradiction. Thus $ H_{r} \cap H_{s} \not\cong C_{2} $.

\medskip
\textbf{Case 2:} Suppose $ H_{r} \cap H_{s} \cong  E_{q_{0} }  $. If $ H_{s} \cap H_{t} \cong  C_{ (q_{0} \pm 1 ) / \delta _{0} } $ then $ H_{s} \cap H_{t} $ has no elements of order $ p $, so $ H_{r} \cap H_{s} \cap H_{t} = \{ 1_{G} \} $, a contradiction. By Lemma \ref{ntadhdrlntr}, we cannot have $ H_{s} \cap H_{t} \cong  D_{ 2(q_{0} \pm 1 ) / \delta _{0} } $ if $ p $ is even. If $ q $ is odd and $ H_{s} \cap H_{t} \cong  D_{ 2(q_{0} \pm 1 ) / \delta _{0} } $ then it has no element of order $ p $ and $ H_{r} \cap H_{s} \cap H_{t} = \{ 1_{G} \} $, anther contradiction. The only other possibility is $ H_{s} \cap H_{t} \cong E_{q_{0}} $. In this case both $ H_{r} \cap H_{s} $ and $ H_{s} \cap H_{t} $ are Sylow $ p $-subgroups of $ H_{s} $, so either $ H_{r} \cap H_{s} = H_{s} \cap H_{t} $ or $ (H_{r} \cap H_{s} ) \cap (H_{s} \cap H_{t} ) = H_{r} \cap H_{s} \cap H_{t} = \{ 1 _{G} \} $ by Lemma \ref{prmsbgrps} - again a contradiction. Therefore $ H_{r} \cap H_{s} \not\cong  E_{q_{0} }  $.

\medskip
\textbf{Case 3:} Suppose $ H_{r} \cap H_{s} \cong C_{ (q_{0} \pm 1) / \delta _{0} } $ or $ H_{r} \cap H_{s} \cong D_{ 2(q_{0} \pm 1) / \delta _{0} } $, since these are the only remaining possibilities for $ H_{r} \cap H_{s} $. We will consider what $ H_{r} \cap H_{s} \cap H_{t}  $ could be isomorphic to.

\medskip
Suppose there exists $ h \in H_{r} \cap H_{s} \cap H_{t} $ with $ o(h) \geq 3 $. Then $o(h) $ divies $ q_{0} \pm 1 $ and is coprime to $ q_{0} $. So, by Lemma \ref{nrmlzr}, the normalizers of $ \langle h \rangle $ in $ H_{r} , H_{s} $ and $ H_{t} $ are isomorphic to either $ D_{ 2(q_{0} - 1) / \delta _{0} } $ or $ D_{ 2(q_{0} + 1) / \delta _{0} } $. Only one of these could have an element of order $ o(h) $. Hence $ N_{H_{r} } ( \langle h \rangle ) \cong N_{H_{s} } ( \langle h \rangle ) \cong N_{H_{t} } ( \langle h \rangle ) $. Lemma \ref{nrmlzr} also shows $ N_{G} ( \langle h \rangle ) $ is dihedral, so it only has one subgroup with the same order as the rotational subgroup of $ N_{H_{r} } ( \langle h \rangle ) $. Therefore the rotational subgroups of $ N_{H_{r} } ( \langle h \rangle ) $, $ N_{H_{s} } ( \langle h \rangle ) $ and $ N_{H_{t} } ( \langle h \rangle ) $ are equal.

\medskip
Let $ R $ be the rotational subgroup of $ N_{H_{r} } ( \langle h \rangle ) $. Then $ R \leq H_{r} \cap H_{s} \cap H_{t} $. If $ H_{r} \cap H_{s} = R $ then $ H_{r} \cap H_{s} \cap H_{t} = R $ and $ \Delta $ is not independent by Lemma \ref{indsbstsntqual}. So $ H_{r} \cap H_{s} \neq R $. Since $ R \leq H_{r} \cap H_{s} $ and the order of $ R $ divides only one of $ | C_{ (q_{0} - 1) / \delta _{0} } | $ or $ | C_{ (q_{0} + 1) / \delta _{0} } | $, it must be that $  H_{r} \cap H_{s} $ is not cyclic and is instead dihedral. As $ R$ would have to be in the rotational subgroup of a dihedral subgroup of $ H_{r} $, the only dihedral of large enough order in $ H_{r} $ that contains $ R $ is $ N_{H_{r} } ( \langle h \rangle ) $. Hence $ H_{r} \cap H_{s} = N_{H_{r} } ( \langle h \rangle ) $. The same method shows that $ H_{r} \cap H_{t} = N_{H_{r} } ( \langle h \rangle ) $. But $  H_{r} \cap H_{s} = H_{r} \cap H_{t} $ implies $ \Delta $ is not independent by Lemma \ref{indsbstsntqual}. So we have a contradiction and it follows that all elements of $ H_{r} \cap H_{s} \cap H_{t} $ have order at most $ 2 $.

\medskip
Since $ H_{r} \cap H_{s} \cap H_{t}  \neq \{ 1_{G} \} $, there exists an involution in $ H_{r} \cap H_{s} \cap H_{t} $.

\medskip
Let $ u \in \{ 1,2,3,4 \} \setminus \{ r,s,t \} $. Everything in the proof so far regarding $  H_{r} \cap H_{s} \cap H_{t} $ applies to $ H_{r} \cap H_{s} \cap H_{u} $. Therefore this intersection contains an involution and has no elements of order $ 3 $ or more. Lemma \ref{indsbstsntqual} tells us that $  H_{r} \cap H_{s} \cap H_{t} \neq H_{r} \cap H_{s} \cap H_{u} $. Using the same lemma again, $  H_{r} \cap H_{s} \cap H_{t} $ and $ H_{r} \cap H_{s} \cap H_{u} $ cannot be subgroups of each other, otherwise we get $  H_{r} \cap H_{s} \cap H_{t} \cap H_{u} = H_{r} \cap H_{s} \cap H_{t} $ or $ H_{r} \cap H_{s} \cap H_{t} \cap H_{u} = H_{r} \cap H_{s} \cap H_{u} $. So there exists involutions $ a_{rst} \in H_{r} \cap H_{s} \cap H_{t} $ and $ a_{rsu} \in H_{r} \cap H_{s} \cap H_{u} $ such that $ a_{rst} \not\in H_{r} \cap H_{s} \cap H_{u} $ and $ a_{rsu} \not\in H_{r} \cap H_{s} \cap H_{t} $. Hence $ a_{rst} \neq a_{rsu} $. As $ a_{rst} , a_{rsu} \in H_{r} \cap H_{s} $, the subgroup $ H_{r} \cap H_{s} $ cannot be cyclic. Thus $ H_{r} \cap H_{s} \cong D_{ 2(q_{0} \pm 1) / \delta _{0} } $. We can similarly show every intersection of a pair of stabilizers in $ \Delta $ is isomorphic to either $ D_{ 2(q_{0} - 1) / \delta _{0} } $ or $ D_{ 2(q_{0} + 1) / \delta _{0} } $.

\medskip
Lemma \ref{dffrntsbgrps} tells us that $ q $ is odd and so $ q_{0} $ must be odd.

\medskip
Just as the involutions above were defined, there also exists an involution $ a_{rtu} \in H_{r} \cap H_{t} \cap H_{u} $ that is distinct from $ a_{rst} $ and $  a_{rsu} $. We will make use of this involution later.

\medskip
From Lemma \ref{nrmlzr} we see the normalizers of $ \langle a_{rst} \rangle $ in $ H_{r} , H_{s} $ and $ H_{t} $ are isomorphic to either $ D_{ 2(q_{0} - 1) / \delta _{0} } $ or $ D_{ 2(q_{0} + 1) / \delta _{0} } $. At least two of these normalizers are isomorphic, say $ N_{H_{r} } (  \langle a_{rst} \rangle ) \cong N_{H_{s} } (  \langle a_{rst} \rangle ) $.

\medskip
Now $ \langle a_{rst} \rangle $ lies in the rotational subgroups of these normalizers and so the order of the normalizers is divisible by $ 4 $. If $ H_{r} \cong PSL_{2} (q_{0} ) $ then only one of $ D_{ q_{0} - 1} $ or $ D_{ q_{0} + 1 } $ has order divisible by $ 4 $. Hence $ N_{H_{r} } (  \langle a_{rst} \rangle ) \cong N_{H_{s} } (  \langle a_{rst} \rangle ) \cong N_{H_{t} } (  \langle a_{rst} \rangle ) $. If we also have $ q_{0} > 5 $ then the rotational subgroup $ N_{H_{r} } (  \langle a_{rst} \rangle ) $ has an element $ b $ of order $ 3 $ or more. Lemma \ref{nrmlzr} shows that $ N_{ G } (  \langle a_{rst} \rangle ) $ is dihedral, so contains only one subgroup of order $ o(b) $, namely $ \langle b \rangle $. Hence $ \langle b \rangle $ is a subgroup of $  N_{H_{s} } (  \langle a_{rst} \rangle ) $ and $  N_{H_{t} } (  \langle a_{rst} \rangle ) $ as well. It follows that $ b \in H_{r} \cap H_{s} \cap H_{t} $. But we have determined earlier that no element of order $ o(b) $ can exist in this intersection. Thus $ H \not\cong PSL_{2} (q_{0} ) $ when $ q_{0} \geq 7 $.

\medskip
If $ q_{0} = 5 $ and $ H_{r} \cong PSL_{2} (5) \cong A_{5} $ then $ N_{H_{r} } (  \langle a_{rst} \rangle ) \cong K_{4} $ and the above argument does not work. By Lemma \ref{kfrsbgrpcllctn} we have $ H_{r} \cap H_{s} \not\cong D_{q_{0} - 1} \cong K_{4} $. Hence $ H_{r} \cap H_{s} \cong D_{q_{0} + 1 } = D_{6} $. Similarly we have $ H_{r} \cap H_{t} \cong H_{r} \cap H_{u} \cong D_{6} $. Observe $ H_{r} \cap H_{s} $ and $ H_{r} \cap H_{t} $ have an involution in common, namely $ a_{rst} $. Since $ a_{rsu} \in  H_{r} \cap H_{s} $ and $ a_{rtu} \in  H_{r} \cap H_{t} $ we have $ o(a_{rsu} a_{rtu} ) = 5 $ from the discussion before this lemma. But $ a_{rsu} , a_{rtu} \in H_{r} \cap H_{u} $ and this subgroup contains no elements of order $ 5 $. Thus we have a contradiction and it follows that $ H_{r} \not\cong PSL_{2} (5)  $ when $ q_{0} = 5 $.

\medskip
The only possibility is that $ H_{r} \cong PGL_{2} (q_{0} ) $ with $ q_{0} $ odd.

\medskip
Suppose $ q_{0} = 3 $. Then $ H_{r} \cong PGL_{2} ( 3 ) \cong S_{4} $. Then $ H_{r} \cap H_{S} $ is isomorphic to either $ K_{4} $ or $ D_{8} $. Lemma \ref{kfrsbgrpcllctn} shows that $ H_{r} \cap H_{S} \not\cong K_{4} $. Hence $ H_{r} \cap H_{S} \cong D_{8} $. Same goes for $ H_{r} \cap H_{t} $ and $ H_{r} \cap H_{u} $. There exists $ K \leq H_{r} $ with $ K \cong K_{4} $ and $ K $ is normal in $ H_{r} $. Looking at the subgroup structure of $ PGL_{2} (3) $, we see $ K $ is contained in every $ D_{8} $ in $ H_{r} $. By Lemma \ref{stach}, we have $ H_{r} \cap H_{s} \cap H_{t} < H_{r} \cap H_{s} $ and since $ K \leq H_{r} \cap H_{s} \cap H_{t} $ we must have $ K = H_{r} \cap H_{s} \cap H_{t} $. Also $ K \leq H_{r} \cap H_{s} \cap H_{t} \cap H_{u} $. Thus $ H_{r} \cap H_{s} \cap H_{t} = H_{r} \cap H_{s} \cap H_{t} \cap H_{u} $. Lemma \ref{indsbstsntqual} then shows that $ \Delta $ is not an independent set. A contradiction. Therefore $ q_{0} \neq 3 $.

\medskip
Now suppose $ q_{0} \geq 5 $. Then $ N_{H_{r} } (  \langle a_{rst} \rangle ) $, $ N_{H_{s} } (  \langle a_{rst} \rangle ) $ and $ N_{H_{t} } (  \langle a_{rst} \rangle ) $ are isomorphic to $ D_{ 2(q_{0} - 1)  } $ or $ D_{ 2(q_{0} + 1) } $ and have order at least $ 8 $. At least two of these, say $ N_{H_{r} } (  \langle a_{rst} \rangle ) $ and $ N_{H_{s} } (  \langle a_{rst} \rangle ) $, contain rotational subgroups that are isomorphic to each other with order at least $ 4 $. Let $ m $ be the order of these rotational subgroups. Lemma \ref{nrmlzr} shows that $ N_{G } (  \langle a_{rst} \rangle ) $ is dihedral so only contains one subgroup of order $ m $. So the rotational subgroups of $ N_{H_{r} } (  \langle a_{rst} \rangle ) $ and $ N_{H_{s} } (  \langle a_{rst} \rangle ) $ are equal and it lies in $ H_{r} \cap H_{s} $. Only one of $ D_{ 2(q_{0} - 1)  } $ or $ D_{ 2(q_{0} + 1) } $ has a subgroup of order $ m $, so $ H_{r} \cap H_{s} \cong N_{H_{r} } (  \langle a_{rst} \rangle ) $. Hence $ N_{H_{r} } (  \langle a_{rst} \rangle ) = H_{r} \cap H_{s} = N_{H_{s} } (  \langle a_{rst} \rangle ) $.

\medskip
Now $ H_{r} \cap H_{t} $ is isomorphic to $ D_{ 2(q_{0} - 1)  } $ or $ D_{ 2(q_{0} + 1) } $ and has an involution $ c_{rt} $ in its centre. If $ c_{rt} = a_{rst} $ then $ H_{r} \cap H_{t} $ contains elements of order $ 3 $ or more that normalize $ \langle a_{rst} \rangle $. Such elements are in $ N_{H_{r} } (  \langle a_{rst} \rangle ) $, so would lie in $ H_{r} \cap H_{s} \cap H_{t} $. But we determined earlier that no elements of order greater that $ 2 $ exist here. Hence $ c_{rt} \neq  a_{rst} $. As $ a_{rst}  \in  H_{r} \cap H_{t} $, the elements $ c_{rt} $ and $  a_{rst} $ commute with each other. Thus $ c_{rt} \in H_{r} \cap H_{s} $ and $ \langle a_{rst} , c_{rt} \rangle \cong K_{4} $. But then $ c_{rt} \in H_{r} \cap H_{s} \cap H_{t} $ and we have $ \langle a_{rst} , c_{rt} \rangle \leq H_{r} \cap H_{s} \cap H_{t} $.

\medskip
Since $ H_{r} \cap H_{s} \cap H_{t} $ is a subgroup of a dihedral and consists of involutions and the identity, it is isomorphic to $ C_{2} $ or $ K_{4} $. Therefore $ \langle a_{rst} , c_{rt} \rangle = H_{r} \cap H_{s} \cap H_{t} $.

\medskip
There exists an involution $ c_{st} $ in the centre of $ H_{s} \cap H_{t} $. Following the above discussion $ c_{st} \neq a_{rst} $ and $ c_{st} $ commutes with $ a_{rst} $. Therefore $ c_{st} \in H_{r} \cap H_{s} \cap H_{t} $. Now $ N_{H_{t}} ( c_{st} ) $ is isomorphic to $ D_{ 2(q_{0} - 1)  } $ or $ D_{ 2(q_{0} + 1) } $ by Lemma \ref{nrmlzr} and so is $ H_{s} \cap H_{t} $. Since $ N_{H_{t}} ( c_{st} ) $ cannot have subgroups isomorphic to both $ C_{q_{0} - 1} $ and $ C_{q_{0} + 1} $, we have $ H_{s} \cap H_{t} = N_{H_{t}} ( c_{st} ) $. By the same reasoning, if $ c_{rt} = c_{st} $ then $ H_{r} \cap H_{t} = H_{s} \cap H_{t} $. But $  H_{r} \cap H_{t} \neq H_{s} \cap H_{t} $, otherwise $ \Delta $ is not independent by Lemma \ref{indsbstsntqual}. Hence $ c_{rt} \neq c_{st} $. So $ a_{rst} , c_{rt} $ and $ c_{st} $ are the three involutions of $ H_{r} \cap H_{s} \cap H_{t} $.

\medskip
Following the same notation, for $ i, j \in \{ 1, 2, 3, 4 \} $ with $ i \neq j $, let $ c_{ij} $ be the involution in the centre of $ H_{i} \cap H_{j} $. This above discussion tells us that $ c_{12} , c_{13} \in H_{1} \cap H_{2} \cap H_{3} $ and $ H_{1} \cap H_{2} \cap H_{3} \cong K_{4} $. Also $ c_{12} , c_{14}  \in H_{1} \cap H_{2} \cap H_{4} $ and $ c_{13} , c_{14} \in H_{1} \cap H_{3} \cap H_{4} $; again these intersections are Klein four-groups. Now $ c_{13} $ and $ c_{14} $ are distinct and commute with $ c_{12} $, so are reflections in $ N_{H_{1} } ( c_{12} ) = H_{1} \cap H_{2} $. Observe $ o( c_{13} c_{14} ) = o(c_{34} ) = 2 $. But $ c_{13} c_{14} $ is a rotation in $ N_{H_{1} } ( c_{12} ) $. Hence $ c_{34} = c_{12} $. So we have
\[
H_{1} \cap H_{2} \cap H_{3} 
= \langle c_{12} , c_{13} \rangle 
= \langle c_{34} , c_{13} \rangle 
= H_{1} \cap H_{3} \cap H_{4} .
\]
This implies $ \Delta $ is not an independent set by Lemma \ref{indsbstsntqual}. So the assumption that $ H_{r} \cap H_{s} \cong C_{ (q_{0} \pm 1) / \delta _{0} } $ or $ H_{r} \cap H_{s} \cong D_{ 2(q_{0} \pm 1) / \delta _{0} } $ made at the start of Case 3 must be wrong.

\medskip
The conclusions of Cases 1, 2 and 3, together with Lemma \ref{dffrntsbgrps}, give us a contradiction. Thus the original assumption that there exists an independent set of size $ 4 $ is wrong.
\end{proof}
At last we are ready for the final theorem of the chapter (which is unfortunately another long one). For this theorem, the choice $ G = PSL_{2} (9) $ is allowed. Also the choice $ H \cong PSL_{2} (3) $ is allowed.
\begin{thrm}\label{fnalthrmdndit}
If $ H \cong PSL_{2} (3) $ then $ Ht(G, \Omega ) = 2 $ and $ RC(G, \Omega ) = 3 $. For all other choices of $ H $ the height of the action of $ G $ is $ 3 $ and relational complexity is $ 4 $.
\end{thrm}
\begin{proof}
Let's get some special cases out of the way first. Suppose $ H \cong PSL_{2} (3) $. Then $ G = PSL_{2} (q) $ with $ q \geq 3^{3} = 27 $. Observe $ PSL_{2} (3) \cong A_{4} $. The height of $ PSL_{2} (q) $ acting on a maximal $ A_{4} $ was given in Lemma \ref{afrctnhght} and the relational complexity in Lemma \ref{afrrcthrem}.

\medskip
Next suppose $ G = PSL_{2} (9) $. The tables in \cite{WISCONS1} show $ Ht(G, \Omega ) = 3 $ and $ RC(G, \Omega ) = 4 $ in this case.

\medskip
Now on to the general case. Suppose $ G \neq PSL_{2} (9) $. Also suppose $ H \not\cong PSL_{2} (3) $. The aim here is to construct tuples in $ \Omega ^{4} $ that are $ 3 $-subtuple complete but not $ 4 $-subtuple complete.

\medskip
Let $ H_{1} \in \Omega $. Let $ B_{1} \leq H_{1} $ be a Borel subgroup. Then we can write $ B_{1} = U_{1} \rtimes T_{1} $, where $ U_{1} $ is elementary abelian of order $q_{0} $ and $ T_{1} \cong C_{ (q_{0} - 1) / \delta _{0} } $. Since $ B_{1} $ normalizes $ U_{1} $, by Lemma \ref{nrmlzrnbrlx} we have $ B_{1} \leq B $ where $ B $ is a Borel subgroup of $ G $.

\medskip
By Lemma \ref{brlcyclcprts}, there exists $ A_{1} \leq B_{1} $ with $ A_{1} \cong C_{ (q_{0} - 1) / \delta _{0} } $. The same lemma shows there exist $ A_{1} ^{ \prime } \leq B $ with $  A_{1} ^{ \prime } \cong C_{ (q - 1) / \delta } $ such that $ A_{1} < A_{1} ^{ \prime } $.

\medskip
Let $ a \in A_{1} ^{ \prime } \setminus A_{1} $. Put $ H_{2} := a^{-1} H_{1} a $. Since $ A_{1} ^{ \prime } $ is cyclic we have $ A_{1} = a^{-1} A_{1} a \in H_{2} $.

\medskip
If $ H_{1} = H_{2} $ then $ a \in H_{1} $ by Lemma \ref{mxmlnrmlizr}. Observe that $ a \not\in B_{1} $ because Lemma \ref{brlcyclcprts} tells us $ B_{1} $ has no elements of order $ o(a ) $ since $ o(a ) = (q-1) / \delta > (q_{0} -1) / \delta _{0} $. As $ B_{1} $ is maximal in $ H_{1} $, we would have $ \langle B_{1} , a \rangle = H_{1} $. But $ \langle B_{1} , a \rangle \leq B $ and $ H_{1} $ is maximal in $ G $, so $ H_{1} = B $. This is not possible because $ B $ has a subgroup of order $ q $ and $ H_{1} $ does not. Therefore $ H_{1} \neq H_{2} $.

\medskip
Let $ B_{2} := a^{-1} B_{1} a $. Then $ B_{2} \leq B $ and $ A_{1} \leq B_{2} \leq H_{2} $. By Lemma \ref{brlcyclcprts} there exists $ A_{2} \leq B_{2} $ with $ A_{2} \cong A_{1} $ and $ A_{2} \cap A_{1} = \{ 1_{G} \} $. We want to show $ A_{2} \not\leq H_{1} \cap H_{2} $.

\medskip
If $ | A_{1} | \geq 3 $ then from Lemma \ref{dffrntsbgrps} we see that $ | A_{1} | $ only divides $ | H_{1} \cap H_{2} | $ if $ H_{1} \cap H_{2} \cong C_{ (q_{0} - 1) / \delta _{0} } $ or $ H_{1} \cap H_{2} \cong D_{ 2(q_{0} - 1) / \delta _{0} } $. Whichever of these groups $ H_{1} \cap H_{2}  $ is isomorphic to, it only has one $ C_{ (q_{0} - 1) / \delta _{0} } $ subgroup. Therefore $ A_{2} \not\leq H_{1} \cap H_{2}  $.

\medskip
Suppose for a moment that $ | A_{1} |  = (q_{0} - 1) / \delta _{0} = 2 $. Then $ q_{0} $ is odd, $ | A_{2} | = 2 $ and there exists $ a_{1} \in A_{1} $ and $ a_{2} \in A_{2} $ both of order $ 2 $. If $  A_{2} \leq H_{1} \cap H_{2} $ then $ a_{1} , a_{2} \in H_{1} \cap H_{2} $ and by Lemma \ref{dffrntsbgrps} we have $ H_{1} \cap H_{2} \cong D_{ 2(q_{0} \pm 1) / \delta _{0} } $. If $ o(a_{1} a_{2} ) = 2 $ then $ \langle a_{1} , a_{2} \rangle \cong K_{4} $. Lemma \ref{brlklnfrsbgrp} tells us that $ B_{2} $ has no Klein four-subgroups. Hence $ o(a_{1} a_{2} ) > 2 $. Note $ D_{ 2(q_{0} - 1) / \delta _{0} } \cong K_{4} $, which has no elements of order $ o(a_{1} a_{2} ) $. So $ H_{1} \cap H_{2} \cong  D_{ 2(q_{0} + 1) / \delta _{0} } $ and $ o(a_{1} a_{2} ) $ divides $ (q_{0} + 1) / \delta _{0} $. But then $  o(a_{1} a_{2} ) $ does not divide $ q_{0} $ nor $  (q_{0} - 1) / \delta _{0} $. Therefore $ a_{1} a_{2} \notin B_{2} $ by Lemma \ref{brlcyclcprts}, which is a contradiction. So $  A_{2} \not\leq H_{1} \cap H_{2} $ again in this case.

\medskip
Since $ B_{2} $ is a Borel subgroup of $ H_{2} $, we can write $ B_{2} = U_{2} \rtimes T_{2} $, where $ U_{2} $ is elementary abelian of order $ q_{0} $ and $ T_{2} \cong C_{ (q_{0} - 1) / \delta _{0} } $. By Lemma \ref{brlcyclcprts} there exists $ u_{2} \in U_{2} $ such that $ A_{2} = u_{2}^{-1} A_{1} u_{2} $.

\medskip
Put $ H_{3} := u_{2}^{-1} H_{1} u_{2} $ and $ B_{3} := u_{2}^{-1} B_{1} u_{2} $. Lemma \ref{brlcyclcprts} shows that there exists an elementary abelian group $ U \leq B $ with $ U_{1} , U_{2} \leq U $. Hence $ u_{2}^{-1} U_{1} u_{2} = U_{1} $ and $ U_{1} \leq H_{1} \cap H_{3} $. Also $ A_{2} \leq H_{3}  $, implying $ H_{1} \neq H_{3} $. Thus $ U_{1} = H_{1} \cap H_{3} $ by Lemma \ref{dffrntsbgrps}. Since $ A_{2} \leq H_{2} \cap H_{3} $, the same lemma tells us $ H_{2} \cap H_{3} \not\cong E_{q_{0}} $, in particular $ H_{1} \cap H_{2} \cap H_{3} = ( H_{1} \cap H_{3} ) \cap ( H_{2} \cap H_{3} ) = \{ 1_{G} \} $.

\medskip
Note $ B_{3} $ is a Borel subgroup of $ H_{3} $. So by Lemma \ref{brlcyclcprts} there exists $ u_{1} \in U_{1} \setminus \{ 1_{G} \} $ and $ A_{3} \leq H_{3} $ such that $ A_{2} \cap A_{3} = \{ 1_{G} \} $ and $ u_{1}^{-1} A_{2} u_{1} = A_{3} $.

\medskip
Put $ H_{4} := u_{1} ^{-1} H_{2} u_{1} $ and $ A_{4} := u_{1} ^{-1} A_{1} u_{1} $. By similar reasoning as above, $ u_{1} ^{-1} U_{2} u_{1} = U_{2} $. Hence $ U_{2} , A_{3} , A_{4} \leq H_{4} $.

\medskip
Let $ b \in A_{1} \setminus \{ 1_{G} \} $. Note that $ A_{4} \leq B_{1} $. Lemma \ref{brlcyclcprts} shows $ A_{1} \cap A_{4} = \{ 1 _{G} \} $. Hence $ b^{-1} A_{4} b \neq A_{4} $ by Lemma \ref{brlcyclcprts}. Also observe $ b^{-1} A_{4} b \neq A_{1} $, otherwise $ A_{4} = b A_{1} b^{-1} = A_{1} $. Using the same lemma again, there exists $ v \in U_{1} \setminus \{ 1_{G} \} $ such that $ v^{-1} A_{4} v = b^{-1} A_{4} b $. Since $ v \in U $ and $ U_{2} \leq U $ we have $ v^{-1} U_{2} v = U_2 $.

\medskip
Put $ A_{3} ^{*} := v^{-1} A_{3} v $ and $ A_{4} ^{*} := v^{-1} A_{4} v $. Let $ H_{4} ^{*} = v^{-1} H_{4} v $. Then $ A_{3} ^{*} , A_{4} ^{*} , U_{2} \leq H_{4} ^{*} $.

\medskip
If $ H_{4}^{*} = H_{4} $ then $ v \in H_{4} $ by Lemma \ref{mxmlnrmlizr}. Since $ v \in U_{1} \leq H_{3} $, we would have $ v \in H_{3} \cap H_{4} $. We have $ o(v) $ dividing $ | U_{1} | = q_{0} $, so $ H_{3} \cap H_{4} \cong E_{q_{0} } $ by Lemma \ref{dffrntsbgrps}. But $ A_{3} \leq H_{3} \cap H_{4} $ and $ | A_{3} | = ( q_{0} - 1) / \delta _{0} $ does not divide $ q_{0} $, a contradiction. Hence $ H_{4}^{*} \neq H_{4} $.

\medskip
With these subgroups constructed, we can finally use them to define some suitable $ 4 $-tuples. Let
\newline $ I = ( H_{1} , H_{2} , H_{3} , H_{4} ) $ and $ J = ( H_{1} , H_{2} , H_{3} , H_{4} ^{*} ) $. First it will be shown that $ I \sim _{3} J $.

\medskip
The identity can be used to send $ (H_{1} , H_{2} , H_{3} ) $ to itself.

\medskip
Any element of $ G $ that sends $ (H_{1} , H_{2} , H_{4} ) $ to $ (H_{1} , H_{2} , H_{4}^{*} ) $ stabilizes $ H_{1} $ and $ H_{2} $, so lies in $ H_{1} \cap H_{2} $. We have $ b \in H_{1} \cap H_{2} $, so let's see if $ b^{-1} H_{4} b = H_{4} ^{*} $. We know that $ b^{-1} A_{4} b = A_{4} ^{*} \leq H_{4}^{*} $. Since $ U_{2} , A_{1} \leq B  $, Corollary \ref{crlrsmllbrl} shows there exists a conjugate $ B_{2} ^{*} $ of $ B_{2} $ in $ B$ where $ B_{2} ^{*} = U_{2} \rtimes A_{1} $. Therefore $ A_{1} $ normalizes $ U_{2} $, in particular $ b $ normalizes $ U_{2} $. Hence $ U_{2} \leq   b^{-1} H_{4} b $. From Lemma \ref{dffrntsbgrps} we see that $ A_{4} ^{*} , U_{2} \leq b^{-1} H_{4} b \cap H_{4}^{*} $ implies $ b^{-1} H_{4} b = H_{4}^{*} $. Thus $ (H_{1} , H_{2} , H_{4} ) ^{b} = (H_{1} , H_{2} , H_{4}^{*} ) $.

\medskip
To send $ (H_{1} , H_{3} , H_{4} ) $ to $ (H_{1} , H_{3} , H_{4}^{*} ) $ we have to pick an element from $ H_{1} \cap H_{3} $. For this we can choose $ v \in U_{1} = H_{1} \cap H_{3} $ since $ v^{-1} H_{4} v = H_{4}^{*} $. It follows that $ (H_{1} , H_{3} , H_{4} ) ^{v} = (H_{1} , H_{3} , H_{4}^{*} ) $.

\medskip
Lastly we want to send $ (H_{2} , H_{3} , H_{4} ) $ to $ (H_{2} , H_{3} , H_{4}^{*} ) $. Such an element must be in $ H_{2} \cap H_{3} $. We have $ A_{2} \leq H_{2} \cap H_{3}  $ and $ A_{2} = u_{2}^{-1} A_{1} u_{2} $. Put $ c := u_{2} ^{-1} b u_{2} $. Then $ c \in A_{2} $ and
\begin{align*}
c ^{-1} H_{4} c  & = (  u_{2} ^{-1} b u_{2} )^{-1} H_{4}  (u_{2} ^{-1} b u_{2} ) \\
& =   u_{2} ^{-1} b^{-1} u_{2}  H_{4}  u_{2} ^{-1} b u_{2}  \\
& =   u_{2} ^{-1} b^{-1}   H_{4}   b u_{2}  \tag{since $u_{2} \in U_{2} \leq H_{4} $} \\
& =   u_{2} ^{-1}  H_{4} ^{*}   u_{2}  \tag{since $b^{-1}   H_{4}   b =  H_{4} ^{*} $ from case above} \\
& = H_{4} ^{*} \tag{since $u_{2} \in U_{2} \leq H_{4} ^{*} $}.
\end{align*}
Thus $ (H_{2} , H_{3} , H_{4} ) ^{c } = (H_{2} , H_{3} , H_{4}^{*} ) $. Therefore $ I \sim _{3} J $.

\medskip
We saw earlier that $ H_{1} \cap H_{2} \cap H_{3} = \{ 1_{G} \} $. Any element that sends $ I $ to $ J$ has to stabilize $ H_{1} $, $H_{2} $ and $ H_{3} $, of which only the identity does. But we saw $ H_{4} \neq H_{4} ^{*} $. So no element sends $ I $ to $ J $. Therefore $ I \not \sim _{4} J $.

\medskip
If $ RC(G, \Omega ) < 4 $ then, by Definition \ref{first}, the fact $ I \sim _{3} J $ (and therefore also $ I \sim _{2} J $) would imply $ I \sim _{4} J $, which is not true. Thus $ RC(G, \Omega ) \geq 4 $. Theorem \ref{rchgt} and Lemma \ref{hghtsbfldac} together show $ RC(G, \Omega )  \leq 4 $. Hence $ RC(G, \Omega )  = 4 $.

\medskip
Using Theorem \ref{rchgt} again we have $ Ht(G, \Omega ) \geq 3 $. It follows from Lemma \ref{hghtsbfldac} that $ Ht(G, \Omega ) = 3 $.
\end{proof}
Although it was not pointed out during the proof, Lemma \ref{ntndpndntntrs} tells us that any proper subset of $ \{ H_{1} , H_{2} , H_{3} , H_{4} \} $ is independent, giving us some explicitly constructed independent sets of size $ 3 $. In fact, since $ Ht(G, \Omega ) = 3 $, no independent sets of size four exist, meaning $ \{ H_{1} , H_{2} , H_{3} , H_{4} \} $ is an almost independent set.

\medskip
With this chapter now complete, Theorems \ref{thermintrofrst} and \ref{thermintroscnd} have finally been proved.

\newpage
\begin{appendices}
\chapter{GAP Code for the Height of the \texorpdfstring{$ S_{4} $}{S4} Action}\label{appa}
This code is designed to study actions of a group $ X $ on cosets of a subgroup $ H_{1} $ that is isomorphic to $ S_{4} $. We want to find out what group $ X $ can be if it is generated by the point stabilizers of an independent set $ \Delta $ of size 4. Write $ \Delta = \{ \delta _{1} , \delta _{2} , \delta _{3} , \delta _{4} \} $. It is not difficult to show $ \Delta ^{x} $ is independent for each $ x \in X $. So we may assume $ H_{1} $ stabilizes $ \delta _{1} $. For $ i \in \{ 1, 2, 3, 4 \} $ let $ H_{i} $ be the stabilizer of $ \delta _{i} $. Let $ \Delta ^{*} = \{ H_{1} , H_{2} , H_{3} , H_{4} \} $.

\medskip
If $ H $ is conjugate to $ H_{1} $ and $ K $ is the normal $ K_{4} $ in $ H $, then assume $ N_{X} (K) = H $ (as is the case for maximal $ S_{4} $ in $ PSL_{2} (q) $ and $ PGL_{2} (q) $). The following facts can then be made use of;
\begin{itemize}
\item all double intersections of stabilizers in $ \Delta ^{*} $ have order at least $ 4 $ (see Lemma \ref{stach} and subgroup structure of $ S_{4} $),
\item double intersections of stabilizers in $ \Delta ^{*} $ are not isomorphic to $ A_{4} $, hence have order less than $ 12 $ (see Lemma \ref{afrntrsctsfr}),
\item all triple intersections of stabilizers in $ \Delta ^{*} $ have order $ 2 $ (see Lemma \ref{cntthnkfnm}),
\item the intersection $ H_{1} \cap H_{2} \cap H_{3} \cap H_{4} $ is trivial (see Lemma \ref{cntthnkfnm}).
\end{itemize}

\medskip
To simulate double intersections that include $ H_{1} $, we first find subgroups of $ H_{1} $ of order $ 4 $ or more and less than $ 12 $ (the notation $ H1 $ will be used from now on to match the GAP code, and similarly for the other stabilizers).
\begin{lstlisting}
H1:=SymmetricGroup(4);
Poss1:=AllSubgroups(H1);;
Poss:=Filtered(Poss1, x->Order(x)>=4 and Order(x)<12);
\end{lstlisting}
Let $ i, j , k \in \{ 1,2,3,4 \} $. We will express the intersection of $ Hi $ and $ Hj $ as $ Hij $. Also we write $ Hijk $ for the intersection of $ Hi $, $ Hj $ and $ Hk $.

\medskip
Recall from Lemma \ref{indsbstsntqual} that the point stabilizers of any two distinct subsets of $ \Delta $ are not equal. Intersecting a pair of double intersections of stabilizers, both including $ H1 $, gives a triple intersection, which has order $ 2 $. We find all possible combinations of double intersection of stabilizers in $ \Delta ^{*} $ with $ H1 $ that satisfy these conditions and put them in a list of 3-tuples called ``Tups".
\begin{lstlisting}
Tups:=[];
for H12 in Poss do
for H13 in Poss do
for H14 in Poss do
H123:=Intersection(H12, H13);
H124:=Intersection(H12, H14);
H134:=Intersection(H13, H14);
if Order(H123)=2 and Order(H124)=2 and Order(H134)=2 then
if not(H123=H124) and not(H123=H134) and not(H124=H134) and not(H12=H123) and not(H13=H123) and 
not(H12=H124) and not(H14=H124) and not(H13=H134) and not(H14=H134) then
Add(Tups, [H12, H13, H14]);
fi;
fi;
od;
od;
od;
\end{lstlisting}
It can easily be shown $ \Delta ^{h} $ is independent for each $ h \in H1 $, meaning we only need to look at one element for each orbit that the $ 3 $-tuples in ``Tups" belong to. We reduce the size of ``Tups" by retaining one representative from each orbit under $ H$.
\begin{lstlisting}
TupsSorted:=[];
TupsOrb:=Orbits(H1,Tups,OnTuples);;
for i in [1..Size(TupsOrb)] do
	Add(TupsSorted, TupsOrb[i][1]);
od;
\end{lstlisting}
We end up with 13 possible configurations.

\medskip
For each 3-tuple in ``TupsSorted", the intersection of any two entries represents a triple intersection of elements of $ \Delta ^{*} $, including $ H1 $. Since triple intersections have order $ 2 $,  the non-identity element will be chosen from each triple intersection and then information on the order of products of these elements extracted to use for presentations later.

\medskip
Since we are looking at a group $ X $ generated by four $ S_{4} $, we better make sure in each case the elements we pick generate $ S_{4} $.
\begin{lstlisting}
for x in TupsSorted do
g1:=Elements(Intersection(x[1], x[2]))[2];
g2:=Elements(Intersection(x[1], x[3]))[2];
g3:=Elements(Intersection(x[3], x[2]))[2];
Print(StructureDescription(Group([g1, g2, g3])), " ");
od;
\end{lstlisting}
We see that they do generate $ S_{4} $.

\medskip
In addition to the orders of products of elements, we will want to keep track of the double intersections the elements ultimately come from, for use later. There is no particular method to the products chosen below, only that enough different products were tried until they worked to get the results we wanted later.
\begin{lstlisting}
Presentation:=[];
for x in TupsSorted do
g1:=Elements(Intersection(x[1], x[2]))[2];
g2:=Elements(Intersection(x[1], x[3]))[2];
g3:=Elements(Intersection(x[2], x[3]))[2];
IntStructure:=[];
Add(IntStructure, StructureDescription(x[1]));
Add(IntStructure, StructureDescription(x[2]));
Add(IntStructure, StructureDescription(x[3]));
A:=[];
Add(A, Order(g1));
Add(A, Order(g2));
Add(A, Order(g3));
Add(A, Order(g1*g2));
Add(A, Order(g1*g3));
Add(A, Order(g2*g3));
Add(A, Order(g1*g2*g3*g2));
Add(A, Order(g2*g3*g1*g3));
Add(A, Order(g3*g1*g2*g1));
Add(A, Order(g1*g2*g3));
Add(A, Order(g1*g3*g2*g1*g3));
Add(A, IntStructure);
Add(Presentation, A);
od;
Presentation;
\end{lstlisting}
The list ``Presentation" is displayed at the end so we can check to see that picking the second element from each triple intersection did indeed give non-identity elements as we wanted.

\medskip
Shortly we will need to be sure that for each of the orders of products in Presentation, a group generated by three elements satisfying the above relations generates an $ S_{4} $ (also a useful check to see if we need to add more relations).
\begin{lstlisting}
for x in Presentation do
f:=FreeGroup("a", "b", "c");;
rels:=[
f.1^x[1], f.2^x[2], f.3^x[3], (f.1*f.2)^x[4], (f.1*f.3)^x[5], (f.2*f.3)^x[6], 
(f.1*f.2*f.3*f.2)^x[7], (f.2*f.3*f.1*f.3)^x[8], (f.3*f.1*f.2*f.1)^x[9], (f.1*f.2*f.3)^x[10], 
(f.1*f.3*f.2*f.1*f.3)^x[11]
];;
g := f / rels;
ct := TryCosetTableInWholeGroup(TrivialSubgroup(g) : silent);
if ct=fail then
Print("Error! No group generated for ", x, " \n");
else
Print(StructureDescription(g), " \n");
fi;
od;
\end{lstlisting}
We see that each of the sets of relations generates an $ S_{4} $. Take notice of how the relations above are the same relations that had been used to create presentations earlier, but with g1 replaced by f.1, g2 replaced by f.2 and g3 replaced by f.3. This method is going to be repeated shortly, but using four generators instead of three.

\medskip
Note the fact that $ H1 \cap H2 \cap H3 \cap H4 $ is trivial means that for $ \{ r,s,t,u \} = \{ 1,2,3,4 \} $, if $ h_{rst} \in Hr \cap Hs \cap Ht $ and $ h_{rsu} \in Hr \cap Hs \cap Hu $ are non-identity elements, then $ h_{rst}  \neq h_{rsu} $. Combining this with the fact that any three such elements generated an $ S_{4} $, picking an element from each of the four triple-intersections must generate a group that properly contains each $ S_{4} $ generated by three elements.

\medskip
If we pick \textit{any} stabilizer $ H \in \Delta ^{*} $, the double intersections of stabilizers that include $ H $ must be isomorphic to the subgroups in one of the $3$-tuples in TupsSorted. Also, for a particular $ 3 $-tuple, corresponding triple intersections must have elements which can form products in the same way as the above 12 relations.

\medskip
To represent what is happening, we look at a group generated by four elements and for each subset of three elements apply one of the sets of relations from ``Presentations" to them.

\medskip
I have tried to put this code in a format so you can copy and paste it into notepad (or some other text editor) before putting it into GAP. Toward the end of the below code, there are two cases where information is printed and it spans multiple lines. If there are any issues when trying to put this into GAP, please delete the $\backslash$ at the end of those lines and combine the whole print command into one long line in notepad.
\begin{lstlisting}
for i in [1..Length(Presentation)] do
for j in [i..Length(Presentation)] do
for k in [j..Length(Presentation)] do
for l in [k..Length(Presentation)] do

P:=Presentation[i];
Q:=Presentation[j];
R:=Presentation[k];
S:=Presentation[l];

# Each presentation simulates what is happening in one of the groups H1, H2, H3 or H4. 
# Attached to each presentation in Presentation is a triple containing the  
# subgroup structure of each double intersection corresponding to the presentation from earlier.
# We will take P to represent H1, Q to be H2, R to be H3 and S to be H4.
# We then take P[12][1] to be H12, P[12][2] to be H13 and P[12][3] to be H14.
# Similarly Q[12][1] will be chosen to be H12, Q[12][2] as H23 and Q[12][3] as H24
# Also R[12][1] is H13, R[12][2] is H23 and R[12][3] is H34.
# Finally S[12][1] is H14, S[12][2] is H24 and S[12][3] is H34.
# For any group to be generated we want to make sure corresponding double
# intersections match.

if 

P[12][1]=Q[12][1] and P[12][2]=R[12][1] and P[12][3]=S[12][1] and 
Q[12][2]=R[12][2] and Q[12][3]=S[12][2] and R[12][3]=S[12][3] 

then

# Since we are running through all combinations of presentation
# and since a presentation for every possible triple of 
# subgroups earlier was created, it does not matter the subgroups 
# are matched in the way above.
# Next the process of generating a group on four elements is started.

# In a similar process to how we assigned subgroups in P,Q,R and S to other subgroups, 
# we want to say which triple intersections the generator involutions lie in,
# then apply the appropriate relations from the presentation to the elements
# and make sure they match up correctly with presentations from other subgroups.
# Our generators will be labeled f.1, f.2, f.3 and f.4.

# Going back to how the list "Presentation" was created, if y is in Presentation
# the element g1 was the involution corresponding to the intersection of y[12][1]
# and y[12][2].

# The element g2 was the involution corresponding to the intersection of y[12][1]
# and y[12][3].

# The element g3 was the involution corresponding to the intersection of y[12][2]
# and y[12][3].

# For each of the presentations P,Q,R and S we will determine which elements
# "g1", "g2" and "g3" need to be substituted for out of f.1, f.2, f.3 and f.4
# and then reformulate the products of these elements and their orders to create
# the relations for the group that is being generated.

# Take f.1 to be the involution in H123. This lies in P[12][1]=Q[12][1], P[12][2]=R[12][1] 
# and Q[12][2]=R[12][2]. So when matching the products of elements of in our setting up of
# "Presentation" with the presentation P, we want to replace "g1" with f.1. Similarly for Q we  
# want to replace "g1" with f.1 again. And for R we also want to replace "g1" with f.1.

# Take f.2 to be the involution in H124. This lies in P[12][1]=Q[12][1], P[12][3]=S[12][1]
# and Q[12][3]=S[12][2]. So when matching the products of elements of in our setting  up
# of "Presentation" with the presentation P, we want to replace "g2" with f.2. Similarly 
# for Q we want to replace "g2" with f.2 again. And for S we want to replace "g1" with f.2.

# Take f.3 to be the involution in H134. This lies in P[12][2]=R[12][1], P[12][3]=S[12][1] 
# and R[12][3]=S[12][3]. So when matching the products of elements of in our setting up of 
# "Presentation" with the presentation P, we want to replace "g3" with f.3. Similarly for 
# R we want to replace "g2" with f.3. And for S we also want to replace "g2" with f.3.

# Take f.4 to be the involution in H234. This lies in Q[12][2]=R[12][2], Q[12][3]=S[12][2]
# and R[12][3]=S[12][3]. So when matching the products of elements of in our setting up
# of "Presentation" with the presentation Q, we want to replace "g3" with f.4. Similarly
# for R we want to replace "g3" with f.4. And for S we also want to replace "g3" with f.4.

# Relations are created for each presentation, then a group is attempted to be created.

f:=FreeGroup("a", "b", "c", "d");;

rels:=[

f.1^P[1], f.2^P[2], f.3^P[3], (f.1*f.2)^P[4], (f.1*f.3)^P[5], (f.2*f.3)^P[6], 
(f.1*f.2*f.3*f.2)^P[7], (f.2*f.3*f.1*f.3)^P[8], (f.3*f.1*f.2*f.1)^P[9], 
(f.1*f.2*f.3)^P[10], (f.1*f.3*f.2*f.1*f.3)^P[11],

f.1^Q[1], f.2^Q[2], f.4^Q[3], (f.1*f.2)^Q[4], (f.1*f.4)^Q[5], (f.2*f.4)^Q[6], 
(f.1*f.2*f.4*f.2)^Q[7], (f.2*f.4*f.1*f.4)^Q[8], (f.4*f.1*f.2*f.1)^Q[9], 
(f.1*f.2*f.4)^Q[10], (f.1*f.4*f.2*f.1*f.4)^Q[11],

f.1^R[1], f.3^R[2], f.4^R[3], (f.1*f.3)^R[4], (f.1*f.4)^R[5], (f.3*f.4)^R[6], 
(f.1*f.3*f.4*f.3)^R[7], (f.3*f.4*f.1*f.4)^R[8], (f.4*f.1*f.3*f.1)^R[9], 
(f.1*f.3*f.4)^R[10], (f.1*f.4*f.3*f.1*f.4)^R[11],

f.2^S[1], f.3^S[2], f.4^S[3], (f.2*f.3)^S[4], (f.2*f.4)^S[5], (f.3*f.4)^S[6], 
(f.2*f.3*f.4*f.3)^S[7], (f.3*f.4*f.2*f.4)^S[8], (f.4*f.2*f.3*f.2)^S[9], 
(f.2*f.3*f.4)^S[10], (f.2*f.4*f.3*f.2*f.4)^S[11]

];;

g := f / rels;

ct := TryCosetTableInWholeGroup(TrivialSubgroup(g) : silent);

if ct=fail then

# If no group is generated, an error is returned with information on the double intersection
# of stabilizers, to be analysed later.

Print("Error! No group generated when we have H12 isomorphic to ", P[12][1], ", have H13 \ 
isomorphic to ", P[12][2],", have H14 isomorphic to ", P[12][3], ", have H23 isomorphic to " , \
Q[12][2], ", have H24 isomorphic to " , Q[12][3], " and have H34 isomorphic to ", R[12][3], ".", \
 "\n", "\n");
 
else

if Order(g)>24 then

# We are only interested in groups of order more than 24 since we trying to find a group X that
# properly contains an S4

Print("Group of order ", Order(g), " with group ID ",  IdGroup(g), ".", "\n",  "\n");

fi;
fi;
fi;
od;
od;
od;
od;
\end{lstlisting}
The output gives the groups
\begin{lstlisting}
Group of order 192 with group ID [ 192, 1493 ].

Group of order 168 with group ID [ 168, 42 ].

Group of order 120 with group ID [ 120, 34 ].
\end{lstlisting}
In GAP, the code
\begin{lstlisting}
IdGroup(SymmetricGroup(5));
\end{lstlisting}
gives group ID [120,34], telling us the group of order 120 generated is $ S_{5} $. The code
\begin{lstlisting}
IdGroup(PSL(2,7));
\end{lstlisting}
gives group ID [168,42], telling us the group of order 120 generated is $ PSL_{2}(7) $. For investigating the $ S_{4} $ actions of $ PSL_{2} (q) $ and $ PGL_{2} (q) $ we will not need to know anything more than the order of the group of order 192 that was generated.

\medskip
The output also gives the following errors;
\begin{lstlisting}
Error! No group generated when we have H12 isomorphic to C2 x C2, have H13 isomorphic to S3, have H14 isomorphic to S3, have H23 isomorphic to S3, have H24 isomorphic to S3 and have H34 isomorphic to D8.

Error! No group generated when we have H12 isomorphic to C2 x C2, have H13 isomorphic to S3, have H14 isomorphic to S3, have H23 isomorphic to S3, have H24 isomorphic to S3 and have H34 isomorphic to C2 x C2.
\end{lstlisting}
Let $ X^{ \prime } $ be one of the groups generated above that did not give an error with the relations we tried. It could be that there are additional relations that had not been considered. In the GAP code we chose four generators from the intersection of our stabilizers, so label these as $ a,b,c,d $. Lemma \ref{gnrtrhmmrphsm} shows that there exists a surjective homomorphism $ \phi : X ^{ \prime } \rightarrow \langle a , b , c , d \rangle $. Thus $  \text{im} ( \phi )  =  \langle a , b_ , c, d \rangle $.

\medskip
If $  X ^{ \prime } $ is one of $ S_{5} $ or $ PSL_{2} (7) $ then the only normal subgroups $ \ker ( \phi ) $ could potentially be are the trivial subgroup, the whole group or $ A_{5} $ (in the case of $ S_{5} $). In the latter two cases we would have $ | \text{im} ( \phi )  | \leq 2 $, which would prevent $ \text{im} ( \phi )  $ having an $ S_{4} $ subgroup. Hence $ \ker ( \phi ) $ is trivial and $ \langle a , b_ , c, d \rangle = X ^{ \prime }  $.

\medskip
If $ X ^{ \prime } $ is the group of order $ 192 $ then $ | \text{im} ( \phi ) | $ divides $ 192 $. For the purposes of dealing with the $ S_{4} $ actions of $ PSL_{2} (q) $ and $ PGL_{2} (q) $ this case does not need investigating further.

\medskip
The two errors that appeared will be handled in the main chapter on the $ S_{4} $ actions.

\newpage
\chapter{GAP Code for the Relational Complexity of the \texorpdfstring{$ S_{4} $}{S4} Action}\label{appb}
A function is created here to determine whether the relational complexity of the $ S_{4} $ actions of $ PSL_{2} (p) $ or $ PGL_{2} (p) $ is $ 3 $ or $ 4 $, where $ p $ is a prime with $ p\geq 7 $. If the relational complexity is $ 4 $ then output will be provided on the intersection of entries of some suitable $ 4 $-tuples to be more closely examined.

\medskip
Let $ \Omega $ be the conjugacy class of maximal $ S_{4} $ being acted on. Some facts are laid out that will be required for the function. We know from Theorem \ref{eopwcnhg} that $ RC(G, \Omega ) \in \{ 3,4 \} $. The comments directly before Definition \ref{almstndstdfn} tell us $ RC( G , \Omega ) = 4 $ if and only if there exists $ I, J \in \Omega ^{4} $ such that $ I \sim _{3} J $ and $ I \not\sim _{4} J $.

\medskip
Necessary conditions for $ I \not\sim _{4} J $:
\begin{itemize} 
\item The remarks before Definition \ref{almstndstdfn} point out the entries of $ I $ form an almost independent set.
\item By Lemma \ref{indsbstsntqual}, none of the entries of $ I $ are equal.
\item  Put $ I: = ( I_{1} , I_{2} , I_{3} , I_{4} ) $ and $ J: = ( J_{1} , J_{2} , J_{3} , J_{4} ) $. For $ r,s,t \in \{ 1,2,3,4 \} $ we have $ \{ I_{r} , I_{s} , I_{t} \} $ is independent if and only if $ I_{r} \cap I_{s} \cap I_{t} $ is not equal to any of $ I_{r} \cap I_{s}  $, $ I_{r} \cap I_{t}  $ or $ I_{s} \cap I_{t}  $ by Lemma \ref{scndfnt}.
\item Since every proper subset of $ \Delta := \{ I_{1} , I_{2} , I_{3} , I_{4}  \} $ is independent, it follows from Lemma \ref{stach}  the intersection of any two entries must be non-trivial.
\item Using Lemma \ref{frsnrtytql} we can assume $ I_{i} = J_{i} $ for each $ i \in \{ 1,2,3 \} $. We have $ I_{4} \neq J_{4} $, otherwise the identity sends $ I $ to $ J $.
\item There exists an element that sends $ (I_{r} , I_{s} , I_{t} ) $ to $ (J_{r} , J_{s} , J_{t} ) $.
\item If $ r,s \neq 4 $ then any element that sends $ (I_{r} , I_{s} , I_{4} ) $ to $ (J_{r} , J_{s} , J_{4} ) = (I_{r} , I_{s} , J_{4} ) $ stabilizes the first two entries, so is in $ I_{r} \cap I_{s} $.
\end{itemize}
With  $ I \sim _{3} J $ and all of the above in place, we will know any element that sends $ I $ to $ J $ must stabilize $ I_{1} $, $ I_{2} $ and $ I_{3} $. So a sufficient condition for $ I \not\sim _{4} J $ and therefore $ RC(G , \Omega ) = 4 $ is:
\begin{itemize}
\item There is no element in $ I_{1} \cap I_{2} \cap I_{3} $ that sends $ I_{4} $ to $ J_{4} $.
\end{itemize}
Each of the above conditions is going to be used in the code below to try find four suitable maximal $ S_{4} $ that could be used for the entries of $ I $ and $ J $ so that $ I \sim _{3} J $ and $ I \not \sim _{4} J $.

\medskip
I have tried to put this code in a format so you can copy and paste it into notepad (or some other text editor) before putting it into GAP. Toward the end of the below code, there are is a case where information is printed and it spans multiple lines. If there are any issues when trying to put this into GAP, please delete the $\backslash$ at the end of those lines and combine the whole print command into one long line in notepad.
\begin{lstlisting}
FindInt:=function(n)
local a,b,c,C,f,g,h,i,j,k,P,X;;

# The set X will later store information on the intersection of S4

X:=[];
Set(X);

# A conjugacy class of maximal S4 is looked for

C:=ConjugacyClassesMaximalSubgroups(n);
a:=0;
repeat
a:=a+1;
until StructureDescription(C[a][1])="S4" or a=Size(C);
if StructureDescription(C[a][1])="S4" then

# Using the earlier assumptions, four distinct S4 will be chosen from C.
# C[a][1] is fixed, corresponding tot he first entry of I. This is because
# the action is transitive, so we can act on I and J as needed to pick whatever
# element we want for the first entry of I.
# Then we run through the maximal S4 to correspond to the other entries, ensuring
# four distinct entries are picked each time.

i:=1;
repeat
i:=i+1;
if Size(Intersection(C[a][1],C[a][i]))>1 then
j:=i;
repeat
j:=j+1;
if

# C[a][i] represents the second entry of I and C[a][j] the third entry.
# It must be checked that the triple intersections of these entries
# are not equal to the double intersections.

not(Intersection(C[a][1],C[a][i],C[a][j])=Intersection(C[a][1],C[a][i])) and 
not(Intersection(C[a][1],C[a][i],C[a][j])=Intersection(C[a][i],C[a][j])) and 
not(Intersection(C[a][1],C[a][i],C[a][j])=Intersection(C[a][1],C[a][j])) 

then

k:=j;
repeat
k:=k+1;

# C[a][k] represents the fourth entry of I.
# Again triple intersections of S4 including C[a][k] cannot be equal to double
# intersections including the same S4.

if

not(Intersection(C[a][1],C[a][i],C[a][k])=Intersection(C[a][1],C[a][i])) and 
not(Intersection(C[a][1],C[a][i],C[a][k])=Intersection(C[a][i],C[a][k])) and 
not(Intersection(C[a][1],C[a][i],C[a][k])=Intersection(C[a][1],C[a][k])) and 
not(Intersection(C[a][1],C[a][j],C[a][k])=Intersection(C[a][1],C[a][j])) and 
not(Intersection(C[a][1],C[a][j],C[a][k])=Intersection(C[a][j],C[a][k])) and 
not(Intersection(C[a][1],C[a][j],C[a][k])=Intersection(C[a][1],C[a][k])) and 
not(Intersection(C[a][i],C[a][j],C[a][k])=Intersection(C[a][i],C[a][j])) and 
not(Intersection(C[a][i],C[a][j],C[a][k])=Intersection(C[a][j],C[a][k])) and 
not(Intersection(C[a][i],C[a][j],C[a][k])=Intersection(C[a][i],C[a][k]))

then

# Next a potential S4 is looked for to use as the fourth entry of J.
# This entry will be labelled as P. The entry P must be a conjugate of
# C[a][k] by a non-trivial element from the double intersections of C[a][1], 
# C[a][i], and C[a][j]. Also P cannot be equal to C[a][k].

for g in Elements(Intersection(C[a][1],C[a][i])) do
if not(ConjugateGroup(C[a][k],g)=C[a][k]) then
P:=ConjugateGroup(C[a][k],g);
for h in Elements(Intersection(C[a][1],C[a][j])) do
if ConjugateGroup(C[a][k],h)=P then
for f in Elements(Intersection(C[a][i],C[a][j])) do
if ConjugateGroup(C[a][k],f)=P then

# Now we want to check whether or not any element in the triple interaction 
# of the first three entries of I sends C[a][k] to P

c:=0;
for b in Elements(Intersection(C[a][1],(Intersection(C[a][i],C[a][j])))) do
if not(ConjugateGroup(C[a][k],b)=P) then
c:=c+1;
fi;
od;

# If there does not exist an element in the triple intersection sending
# C[a][k] to P, then the information on the structure of the double
# intersection of entries is recorded in X and if not previously recorded
# is printed to analyse later.

if c=Order(Intersection(C[a][1],(Intersection(C[a][i],C[a][j])))) then

if

not([StructureDescription(Intersection(C[a][1],C[a][i])) ,
StructureDescription(Intersection(C[a][1],C[a][j])) ,
StructureDescription(Intersection(C[a][i],C[a][j])) ,
StructureDescription(Intersection(C[a][1],C[a][k])) ,
StructureDescription(Intersection(C[a][i],C[a][k])) ,
StructureDescription(Intersection(C[a][j],C[a][k])) ] in X) 

then

AddSet(X, [StructureDescription(Intersection(C[a][1],C[a][i])) ,
StructureDescription(Intersection(C[a][1],C[a][j])) ,
StructureDescription(Intersection(C[a][i],C[a][j])) ,
StructureDescription(Intersection(C[a][1],C[a][k])) ,
StructureDescription(Intersection(C[a][i],C[a][k])) ,
StructureDescription(Intersection(C[a][j],C[a][k])) ] );

Print("\n", [StructureDescription(Intersection(C[a][1],C[a][i])) , \
StructureDescription(Intersection(C[a][1],C[a][j])) , \
StructureDescription(Intersection(C[a][i],C[a][j])) , \
StructureDescription(Intersection(C[a][1],C[a][k])) , \
StructureDescription(Intersection(C[a][i],C[a][k])) , \
StructureDescription(Intersection(C[a][j],C[a][k])) ] );

fi;
fi;
fi;
od;
fi;
od;
fi;
od;
fi;
until k=Size(C[a]);
fi;
until j=Size(C[a])-1;
fi;
until i=Size(C[a])-2;
fi;

# After running through all possible choices for the entries of I,
# if X is empty then there are no choices for the entries of J such
# that (I,J) is 3-subtuple complete but not 4-subtuple complete.
# Hence the relational complexity is 3.

if X=[] then

Print("\n No suitable almost independent sets found. The relational complexity is 3.");

# If X is not empty then suitable entries for I and J have been found
# so (I,J) is 3-subtuple complete but not 4-subtuple complete, showing
# the relational complexity is 4.

else

Print("\n The relational complexity is 4.");

fi;
end;
\end{lstlisting}
After putting the function into GAP, set a group with a conjugacy class of maximal $ S_{4} $ to be $ G $ then run the function. For example;
\begin{lstlisting}
G:=PSL(2,23);
FindInt(G);
\end{lstlisting}
If there are no almost independent sets satisfying the conditions covered in the function, then we get a message saying the relational complexity is $ 3 $. Otherwise we get output showing the structure description of the six intersections of stabilizers in an almost independent set that pass the tests.

\medskip
Note that if the relational complexity is 4 then the fact the program displays the intersection of entries of $ I $ before completing the whole calculation means you do not need to wait for it to finish checking everything to be certain it is 4.

\medskip
Unfortunately when the relational complexity is 3 the whole calculation must be finished, which can take quite a while for larger groups.

\newpage
\chapter{GAP Code for the Independent Set with \texorpdfstring{$ C_{3} $}{C3} Intersection.}\label{appc}
A similar process to Appendix \ref{appa} will be followed. This code is designed to study actions of a group $ X $ on cosets of a subgroup $ H_{1} $ that is isomorphic to $ A_{5} $. We want to find out what group $ X $ can be if it is generated by the point stabilizers of an independent set $ \Delta $ of size 3 whose point stabilizers have pairwise intersection isomorphic to $ A_{4} $ and triple intersection isomorphic to $ C_{3} $.

\medskip
Write $ \Delta = \{ \delta _{1} , \delta _{2} , \delta _{3} \} $. It is not difficult to show $ \Delta ^{x} $ is independent for each $ x \in X $. So we may assume $ H_{1} $ stabilizes $ \delta _{1} $. For $ i \in \{ 1, 2, 3\} $ let $ H_{i} $ be the stabilizer of $ \delta _{i} $. Let $ \Delta ^{*} = \{ H_{1} , H_{2} , H_{3} \} $.

\medskip
The general outline will be to pick two $ A_{4} $ subgroups of $ H_{1} $, use the fact they intersect in a $ C_{3} $ to choose a shared element of order $ 3 $ and then pick an element of order $ 2 $ from each $ A_{4} $. Combining the element of order $ 3 $ with one of the involutions generates an $ A_{4} $. Adding in the other involution generates $ H_{1} $.

\medskip
A presentation will be recorded for each of the triples of elements picked in the above way. These presentations will then be used to simulate how the $ A_{4} $ in each of the double intersections of stabilizers interact with each other and see what group is generated using four elements. 

\medskip
For double intersections of stabilizers that include $ H_{1} $, we first find subgroups of $ H_{1} $ of order $ 12 $ and pick out elements of order $ 2 $ and $ 3 $ from two $ A_{4} $. Any two $ A_{4} $ in $ A_{5} $ intersect in a $ C_{3} $. Looking at the structure of $ A_{5} $, if $ B_{1} , B_{2} , B_{3} \leq H_{1} $ are distinct $ A_{4} $ subgroups, then there exists elements of order $ 3 $ in $ B_{1} $ that sends $ B_{2} $ to $ B_{3} $ by conjugation. So it does not matter which two $ A_{4} $ we choose to look at in $ H_{1} $. The notation $ H1 $ will be used from now on to match the GAP code, and similarly for the other stabilizers.
\begin{lstlisting}
H1:=AlternatingGroup(5);
Poss1:=AllSubgroups(H1);;
Poss:=Filtered(Poss1, x->Order(x)=12);
Pair:=[];
Add(Pair, Poss[1]);
Add(Pair, Poss[2]);
E1:=Filtered(Elements(Pair[1]), x->Order(x)=2);
E2:=Filtered(Elements(Pair[2]), x->Order(x)=2);
E0:=Intersection(Elements(Pair[1]),Elements(Pair[2]));
E12:=Filtered(E0, x->Order(x)=3);
\end{lstlisting}
Now an element is picked from $ E1 $, $ E2 $ and $ E12 $ and then orders of products of these elements recorded. We run through all combinations of elements in these sets in this way and create a presentation.
\begin{lstlisting}
Presentation0:=[];

for g1 in E12 do;
for g2 in E1 do
for g3 in E2 do

B:=[];

Add(B, Order(g1));
Add(B, Order(g2));
Add(B, Order(g3));
Add(B, Order(g1*g2));
Add(B, Order(g1*g3));
Add(B, Order(g2*g3));
Add(B, Order(g1*g2*g3));
Add(B, Order(g1*g3*g2));
Add(B, Order(g1*g2*g1*g3));
Add(B, Order(g2*g1*g2*g3));
Add(B, Order(g3*g2*g3*g1));
Add(B, Order(g1*g3*g2*g1*g3));
Add(B, Order(g1*g2*g3*g1*g2));
Add(B, Order(g2*g3*g1*g2*g3));
Add(B, Order(g1*g2*g3*g2));
Add(B, Order(g1*g3*g2*g3));
Add(B, Order(g2*g1*g3*g1));

Add(Presentation0, B);

od;
od;
od;
\end{lstlisting}
Some of these presentations are duplicates, so can be removed.
\begin{lstlisting}
Presentation:=DuplicateFreeList(Presentation0);
\end{lstlisting}
This reduces the list to just three presentations. Picking any two of the two-point stabilizers $ H12 $, $ H13 $ or $ H23 $ gives two $ A_{4} $ that lies in a one-point stabilizer. For $ H2 $ or $H3 $, we can pick any two involutions and an element of order $ 3 $ in the same way as earlier and they will give the same presentations as created above with $ H1 $.

\medskip
So for each of the stabilizers we take a presentation, piece these presentations together and see what group is generated. This will be done for all combinations of presentation.
\begin{lstlisting}
for i in [1..Length(Presentation)] do
for j in [1..Length(Presentation)] do
for k in [1..Length(Presentation)] do

P:=Presentation[i];
Q:=Presentation[j];
R:=Presentation[k];

Unbind(f);
f:=FreeGroup("a", "b", "c", "d");;

Unbind(a);
Unbind(b);
Unbind(c);
Unbind(d);
Unbind(g1);
Unbind(g2);
Unbind(g3);

a:=f.1;
b:=f.2;
c:=f.3;
d:=f.4;

# A list called "rels0" will store our relations

rels0:=[];;

# In each of the presentations created earlier, the first entry listed represents
# our element of order 3. The variable "a" will be used for this element and needs
# to be assigned as the first element in each set of relations below.
# The elements of order two will be matched up arbitrarily between each presentation
# since we are running through all possible combinations of presentations.

for u in [1,2,3] do

if u=1 then
x:=P;
g1:=a;
g2:=b;
g3:=c;
fi;

if u=2 then
x:=Q;
g1:=a;
g2:=b;
g3:=d;
fi;

if u=3 then
x:=R;
g1:=a;
g2:=c;
g3:=d;
fi;

# The relations below match the relations from earlier that the
# presentations were taken from.

T:=[
g1^x[1],
g2^x[2],
g3^x[3],
(g1*g2)^x[4],
(g1*g3)^x[5],
(g2*g3)^x[6],
(g1*g2*g3)^x[7],
(g1*g3*g2)^x[8],
(g1*g2*g1*g3)^x[9],
(g2*g1*g2*g3)^x[10],
(g3*g2*g3*g1)^x[11],
(g1*g3*g2*g1*g3)^x[12],
(g1*g2*g3*g1*g2)^x[13],
(g2*g3*g1*g2*g3)^x[14],
(g1*g2*g3*g2)^x[15],
(g1*g3*g2*g3)^x[16],
(g2*g1*g3*g1)^x[17]
];

for t in [1..Length(T)] do

Add(rels0, T[t]);

od;
od;

# Duplicate relations are removed and a group attempted to be generated.

rels:=DuplicateFreeList(rels0);;

g:=f/rels;

# If a group fails to be generated then an error will be shown.
# Otherwise all groups of order greater than 60 printed.

ct:=TryCosetTableInWholeGroup(TrivialSubgroup(g) : silent);

if ct=fail then

Print("Error! No group generated. \n");

else

if Order(g)>60 then

# We are only interested in groups of order more than 60 since we trying to find a group X that
# properly contains an A5

Print("Group of order ", Order(g), " found with structure description ",  StructureDescription(g), ".", "\n", "\n");

fi;
fi;
od;
od;
od;
\end{lstlisting}
We see there are no combinations of relations that fail to generate a group and the only group generated with order more that $ 60 $ is $ A_{6} $. There may be other relations that have not been considered in this code. By Corollary \ref{gnrtrcrlry}, if any further relations were added, a group would be generated with order dividing $ | A_{6} | = 360 $.

\newpage
\chapter{GAP Code for the Height of the \texorpdfstring{$ A_{5} $}{A5} Action.}\label{appd}
This code is designed to study actions of a group $ X $ on cosets of a subgroup $ H_{1} $ that is isomorphic to $ A_{5} $. This follows the same idea as in Appendix \ref{appa}. We want to find out what group $ X $ can be if it is generated by the point stabilizers of an independent set $ \Delta $ of size 4. Write $ \Delta = \{ \delta _{1} , \delta _{2} , \delta _{3} , \delta _{4} \} $. It is not difficult to show $ \Delta ^{x} $ is independent for each $ x \in X $. So we may assume $ H_{1} $ stabilizes $ \delta _{1} $. For $ i \in \{ 1, 2, 3, 4 \} $ let $ H_{i} $ be the stabilizer of $ \delta _{i} $. Let $ \Delta ^{*} = \{ H_{1} , H_{2} , H_{3} , H_{4} \} $.

\medskip
The following assumptions will be made based on facts we know from the $ A_{5} $ actions of $ PSL_{2} (q) $ when $ q \geq 11 $;
\begin{itemize}
\item all double intersections of stabilizers in $ \Delta ^{*} $ have order at least $ 6 $ (see Lemma \ref{stach}, Corollary \ref{crlkfrnrmlz} and subgroup structure of $ A_{5} $), 
\item all triple intersections of stabilizers in $ \Delta ^{*} $ have order $ 2 $ (see Corollary \ref{indsub} and Corollary \ref{oiweuvw}),
\item the intersection $ H_{1} \cap H_{2} \cap H_{3} \cap H_{4} $ is trivial (see Lemma \ref{afvactntrvl}).
\end{itemize}

\medskip
To simulate double intersections that include $ H_{1} $, we first find proper subgroups of $ H_{1} $ of order $ 6 $ or more (the notation $ H1 $ will be used from now on to match the GAP code, and similarly for the other stabilizers).
\begin{lstlisting}
H1:=AlternatingGroup(5);
Poss1:=AllSubgroups(H1);;
Poss:=Filtered(Poss1, x->Order(x)>=6 and Order(x)<=12);
\end{lstlisting}
Let $ i, j , k \in \{ 1,2,3,4 \} $. We will express the intersection of $ Hi $ and $ Hj $ as $ Hij $. Also we write $ Hijk $ for the intersection of $ Hi $, $ Hj $ and $ Hk $.

\medskip
Recall from Lemma \ref{indsbstsntqual} that the point-wise stabilizers of any two distinct subsets of $ \Delta $ are not equal. Intersecting a pair of double intersections of stabilizers, both including $ H1 $, gives a triple intersection, which has order $ 2 $. We find all possible triple intersection of stabilizers in $ \Delta ^{*} $ with $ H1 $ and put the corresponding double intersections in a list of 3-tuples called ``Tups".
\begin{lstlisting}
Tups:=[];
for H12 in Poss do
for H13 in Poss do
for H14 in Poss do
H123:=Intersection(H12, H13);
H124:=Intersection(H12, H14);
H134:=Intersection(H13, H14);
if Order(H123)=2 and Order(H124)=2 and Order(H134)=2 then
if not(H123=H124) and not(H123=H134) and not(H124=H134) and not(H12=H123) and not(H13=H123) and not(H12=H124) and not(H14=H124) and not(H13=H134) and not(H14=H134) then
Add(Tups, [H12, H13, H14]);
fi;
fi;
od;
od;
od;
\end{lstlisting}
It can easily be shown $ \Delta ^{h} $ is independent for each $ h \in H1 $, meaning we only need to look at one element for each orbit that the $ 3 $-tuples in ``Tups" belong to. We reduce the size of ``Tups" by retaining one representative from each orbit under $ H$. The resulting list is called ``TupsSorted".
\begin{lstlisting}
TupsSorted:=[];
TupsOrb:=Orbits(H1,Tups,OnTuples);;
for i in [1..Size(TupsOrb)] do
	Add(TupsSorted, TupsOrb[i][1]);
od;
\end{lstlisting}
We end up with 38 possible configurations.

\medskip
For each 3-tuple in ``TupsSorted", the intersection of any two entries represents a triple intersection of elements of $ \Delta ^{*} $, including $ H1 $. Since triple intersections have order $ 2 $,  the non-identity element will be chosen from each triple intersection and then information on the order of products of these elements extracted to use for presentations later.

\medskip
In addition to the orders of products, we will want to keep track of the structure of the double intersections the elements ultimately come from, for use later. There is no particular method to the products chosen below, only that enough different products were tried until they worked to get the results we wanted later.
\begin{lstlisting}
Presentation0:=[];

for x in TupsSorted do


# First put the structure description of the groups in x in a set A

A:=[];

IntStructure:=[];
Add(IntStructure, StructureDescription(x[1]));
Add(IntStructure, StructureDescription(x[2]));
Add(IntStructure, StructureDescription(x[3]));
Add(A, IntStructure);

# Then pick the non-identity element from each group in x and record orders of products

g1:=Elements(Intersection(x[1], x[2]))[2];
g2:=Elements(Intersection(x[1], x[3]))[2];
g3:=Elements(Intersection(x[2], x[3]))[2];

B:=[];

Add(B, Order(g1));
Add(B, Order(g2));
Add(B, Order(g3));
Add(B, Order(g1*g2));
Add(B, Order(g1*g3));
Add(B, Order(g2*g3));
Add(B, Order(g1*g2*g3));
Add(B, Order(g1*g3*g2));
Add(B, Order(g1*g2*g1*g3));
Add(B, Order(g2*g1*g2*g3));
Add(B, Order(g3*g2*g3*g1));
Add(B, Order(g1*g3*g2*g1*g3));
Add(B, Order(g1*g2*g3*g1*g2));
Add(B, Order(g2*g3*g1*g2*g3));
Add(B, Order(g2*g1*g3*g2*g1));
Add(B, Order(g3*g2*g1*g3*g2));
Add(B, Order(g3*g1*g2*g3*g1));
Add(B, Order(g1*g2*g3*g2));
Add(B, Order(g1*g3*g2*g3));
Add(B, Order(g2*g1*g3*g1));

# Finally add B to A then A to Presentation0

Add(A, B);

Add(Presentation0, A);

od;
\end{lstlisting}
Now we have a list of presentations corresponding to the tuples in TupsSorted. There are many duplicate entries in Presentation0, so they can be removed.
\begin{lstlisting}
Presentation:=DuplicateFreeList(Presentation0);
\end{lstlisting}
The list ``Presentation" is displayed so we can check to see that picking the second element from each triple intersection did indeed give non-identity elements as we wanted (shown by the first three entries of each presentation).

\medskip
Also make note of the fact that for any triple of subgroups attached to a presentation, we can can permute the three subgroups and there is a presentation paired with the new triple. This will be important shortly.

\medskip
Shortly we will need to be sure that for each of the orders of products in Presentation, a group generated on
three elements satisfying the above relations generates an $ A_{5} $ (also a useful check to see if we need to add more relations).
\begin{lstlisting}
V:=[];
Set(V);

for x in Presentation do

f:=FreeGroup("a", "b", "c");;

Unbind(g1);
Unbind(g2);
Unbind(g3);

g1:=f.1;
g2:=f.2;
g3:=f.3;

rels:=[
g1^x[2][1],
g2^x[2][2],
g3^x[2][3],
(g1*g2)^x[2][4], 
(g1*g3)^x[2][5], 
(g2*g3)^x[2][6], 
(g1*g2*g3)^x[2][7], 
(g1*g3*g2)^x[2][8], 
(g1*g2*g1*g3)^x[2][9], 
(g2*g1*g2*g3)^x[2][10], 
(g3*g2*g3*g1)^x[2][11], 
(g1*g3*g2*g1*g3)^x[2][12],
(g1*g2*g3*g1*g2)^x[2][13],
(g2*g3*g1*g2*g3)^x[2][14],
(g2*g1*g3*g2*g1)^x[2][15],
(g3*g2*g1*g3*g2)^x[2][16],
(g3*g1*g2*g3*g1)^x[2][17],
(g1*g2*g3*g2)^x[2][18],
(g1*g3*g2*g3)^x[2][19],
(g2*g1*g3*g1)^x[2][20]
];;


g := f / rels;
ct := TryCosetTableInWholeGroup(TrivialSubgroup(g) : silent);

if ct=fail then
Print("Error! No group generated for ", x, " \n");

else

AddSet(V, StructureDescription(g));

fi;

od;

V;
\end{lstlisting}
We see that each of the sets of relations generates an $ A_{5} $.

\medskip
If we pick \textit{any} stabilizer $ H \in \Delta ^{*} $, the double intersections of stabilizers that include $ H $ must be isomorphic to the subgroups in one of the $3$-tuples in TupsSorted. Also, for a particular $ 3 $-tuple, corresponding triple intersections must have elements which can form products in the same way as the above 12 relations.

\medskip
Note the fact that $ H1 \cap H2 \cap H3 \cap H4 $ is trivial means that for $ \{ r,s,t,u \} = \{ 1,2,3,4 \} $, if $ h_{rst} \in Hr \cap Hs \cap Ht $ and $ h_{rsu} \in Hr \cap Hs \cap Hu $ are non-identity elements, then $ h_{rst}  \neq h_{rsu} $. Combining this with the fact that any three such elements generated an $ A_{5} $, picking an element from each of the four triple-intersections in $ \Delta^{*} $ must generate a group that properly contains each $ A_{5} $ generated by three elements.

\medskip
To represent what is happening, we look at a group generated by four elements and for each subset of three elements apply one of the sets of relations from ``Presentations" to them.

\medskip
This next piece of code takes a few hours to run. At two points it will generate a group of order more than 15000 which took around three hours to generate the first group and around one hour to generate the other. The rest of the groups generated do not take too long though. To let you see that GAP is working, I have included a counter that will be printed whenever GAP is trying to generate a group. The groups that take a while to generate occur when the counter hits 688 and 1679.

\medskip
I have tried to put this code in a format so you can copy and paste it into notepad (or some other text editor) before putting it into GAP. Toward the end of the below code, there are two cases where information is printed and it spans multiple lines. If there are any issues when trying to put this into GAP, please delete the $\backslash$ at the end of those lines and combine the whole print command into one long line in notepad.
\begin{lstlisting}
# The set V will store the structure of the groups generated

V:=[];
Set(V);

# The set W will store information on the presentations that fail to generate a group

W:=[];
Set(W);

# Start counter to track number of groups attempted to be generated

counter:=1;

# Run through each presentation and match elements up between presentations to see
# what group is generated.

for i in [1..Length(Presentation)] do
for j in [i..Length(Presentation)] do
for k in [j..Length(Presentation)] do
for l in [k..Length(Presentation)] do

P:=Presentation[i];
Q:=Presentation[j];
R:=Presentation[k];
S:=Presentation[l];

counter:=counter+1;

# Each presentation simulates what is happening in one of the groups H1, H2, H3 or H4. 
# Attached to each presentation in Presentation is a triple containing the  
# subgroup structure of each double intersection corresponding to the presentation from earlier.
# We will take P to represent H1, Q to be H2, R to be H3 and S to be H4.
# We then take P[1][1] to be H12, P[1][2] to be H13 and P[1][3] to be H14.
# Similarly Q[1][1] will be chosen to be H12, Q[1][2] as H23 and Q[1][3] as H24
# Also R[1][1] is H13, R[1][2] is H23 and R[1][3] is H34.
# Finally S[1][1] is H14, S[1][2] is H24 and S[1][3] is H34.
# For any group to be generated we want to make sure corresponding double
# intersections match.

if

P[1][1]=Q[1][1] and P[1][2]=R[1][1] and P[1][3]=S[1][1] and 
Q[1][2]=R[1][2] and Q[1][3]=S[1][2] and R[1][3]=S[1][3]

then

# Since we are running through all combinations of presentation
# and since a presentation for every possible triple of 
# subgroups earlier was created, it does not matter the subgroups 
# are matched in the way above.
# Next the process of generating a group on four elements is started.

Print(counter, "\n \n");

f:=FreeGroup("a", "b", "c", "d");;

Unbind(a);
Unbind(b);
Unbind(c);
Unbind(d);

# Set the generators as a,b,c and d

a:=f.1;
b:=f.2;
c:=f.3;
d:=f.4;

rels0:=[];;

# In a similar process to how we assigned subgroups in P,Q,R and S to other subgroups, 
# we want to say which triple intersections the involutions lie in,
# then apply the appropriate relations from the presentation to the elements
# and make sure they match up correctly with presentations from other subgroups.

# Take a to be the involution in H123.
# Take b to be the involution in H124.
# Take c to be the involution in H134.
# Take b to be the involution in H234.

# Going back to the way the lists "Presentation0" and "Presentation" were created,
# for a triple of subgroups Y, 
# element g1 was defined to be the involution in the intersection of Y[1] and Y[2],
# element g2 was defined to be the involution in the intersection of Y[1] and Y[3],
# element g3 was defined to be the involution in the intersection of Y[2] and Y[3].

# Taking each of our four presentations in turn, the same labels will be used
# again to create relations for our generated group, where the relations used
# earlier will be repeated here.
# For example, the relation P will be picked, set g1:=a since both a and g1 are the
# involution in the intersection of P[1][1] and P[1][2].
# Similarly g2:=b since both labels are for the involution in intersection
# of P[1][1] and P[1][3].
# Also g3:=c since both labels are for the involution in intersection
# of P[1][2] and P[1][3].
# Relations are then generated for P based on the presentation for P and stored in "rel0".
# Then the next presentation Q is dealt with in the same manner.
# The variables g1, g2, g3 are assigned differently for each presentation
# and is based on the equalities P[1][1]=Q[1][1] etc above.
# Once relations are created for each presentation, a group is attempted to be created.


for y in [1,2,3,4] do

if y=1 then
x:=P;
g1:=a;
g2:=b;
g3:=c;
fi;

if y=2 then
x:=Q;
g1:=a;
g2:=b;
g3:=d;
fi;

if y=3 then
x:=R;
g1:=a;
g2:=c;
g3:=d;
fi;

if y=4 then
x:=S;
g1:=b;
g2:=c;
g3:=d;
fi;

T:=[
g1^x[2][1],
g2^x[2][2],
g3^x[2][3],
(g1*g2)^x[2][4], 
(g1*g3)^x[2][5], 
(g2*g3)^x[2][6], 
(g1*g2*g3)^x[2][7], 
(g1*g3*g2)^x[2][8], 
(g1*g2*g1*g3)^x[2][9], 
(g2*g1*g2*g3)^x[2][10], 
(g3*g2*g3*g1)^x[2][11], 
(g1*g3*g2*g1*g3)^x[2][12],
(g1*g2*g3*g1*g2)^x[2][13],
(g2*g3*g1*g2*g3)^x[2][14],
(g2*g1*g3*g2*g1)^x[2][15],
(g3*g2*g1*g3*g2)^x[2][16],
(g3*g1*g2*g3*g1)^x[2][17],
(g1*g2*g3*g2)^x[2][18],
(g1*g3*g2*g3)^x[2][19],
(g2*g1*g3*g1)^x[2][20]
];

for t in [1..Length(T)] do

Add(rels0, T[t]);

od;
od;

rels:=DuplicateFreeList(rels0);;

g := f / rels;

# Information on any failed attempts is printed, along with any groups
# generated of order 60 or more. This is to see progress is being made.
# Errors and suitable groups are stored and once the process is done
# the will be printed together so we can see final results.

ct := TryCosetTableInWholeGroup(TrivialSubgroup(g) : silent);

if ct=fail then

AddSet(W, [P[1][1], P[1][2], P[1][3], Q[1][2], Q[1][3], R[1][3] ]);

Print("Error! No group generated when we have H12 isomorphic to ", P[1][1], ", have H13 \
isomorphic to ", P[1][2],", have H14 isomorphic to ", P[1][3], ", have H23 isomorphic to " , \
 Q[1][2], ", have H24 isomorphic to " , Q[1][3], " and have H34 isomorphic to ", R[1][3], \
  ".", "\n", "\n");

else

V1:=[Order(g),StructureDescription(g)];

if V1[1]>60 then

# We are only interested in groups of order more than 60 since we trying to find a group X that
# properly contains an A5

Print("Group of order ", V1[1], " with structure description ",  V1[2], ".", "\n \n");

AddSet(V,V1);

fi;
fi;
fi;

od;
od;
od;
od;

# The final results are printed together with each group or error shown only once
# to make the output easier to see.

for v in V do

Print("Group of order ", v[1], " with structure description ",  v[2], ".", "\n \n");

od;

for w in W do

Print("Error! No group generated when we have H12 isomorphic to ", w[1], ", have H13 isomorphic \
 to ", w[2],", have H14 isomorphic to ", w[3], ", have H23 isomorphic to " , w[4], ", \
 have H24 isomorphic to " , w[5], " and have H34 isomorphic to ", w[6], ".", "\n", "\n");

od;
\end{lstlisting}
The output shows the groups generated are;
\begin{lstlisting}
Group of order 660 with structure description PSL(2,11).

Group of order 960 with structure description (C2 x C2 x C2 x C2) : A5.

Group of order 1080 with structure description C3 . A6.

Group of order 3420 with structure description PSL(2,19).

Group of order 15360 with structure description ((C2 x C2 x Q8) : Q8) : A5.
\end{lstlisting}
There are two configuration of presentation that do not generate a group, which are shown in the output as;
\begin{lstlisting}
Error! No group generated when we have H12 isomorphic to A4, have H13 isomorphic to S3, have H14 isomorphic to D10, have H23 isomorphic to D10, have H24 isomorphic to D10 and have H34 isomorphic to A4.

Error! No group generated when we have H12 isomorphic to A4, have H13 isomorphic to S3, have H14 isomorphic to D10, have H23 isomorphic to D10, have H24 isomorphic to S3 and have H34 isomorphic to A4.
\end{lstlisting}
Suppose $ q \geq 11 $ and $ PSL_{2} (q) $ has maximal $ A_{5} $ subgroups of the type defined in Section \ref{sctionafv}. If $ PSL_{2} (q) $ acts on the cosets of a maximal $ A_{5} $ subgroup and has an independent set of size $ 4 $, then either the stabilizers of the points intersect in the same way as one of the errors or Corollary \ref{gnrtrcrlry} tells us $ | PSL_{2} (q) | $ divides the order of one groups that have been generated.

\medskip
Our choice of $ q $ means $ q $ is a power of some prime $ p \geq 7 $ and $ | PSL_{2} (q) | $ is divisible by $ p $. The prime factors of $ 960 $, $ 1080 $ and $ 15360 $ are $2$, $ 3 $ and $ 5 $. Hence $ | PSL_{2} (q) | $ does not divide these orders.

\medskip
Similarly the only choices of $ q $ where $ | PSL_{2} (q) | $ divides $ 660 $ or $ 3420 $ are $ q = 11 $ or $ q = 19 $. The tables in \cite{WISCONS1} show the height of the $ A_{5} $ actions of $ PSL_{2} (11) $ and $ PSL_{2} (19) $ is $ 3 $. So no independent set of size $ 4 $ exists here either (the reason these two groups pop out of the GAP code is that there are two conjugacy classes of $ A_{5} $ in each group, which was not taken into account when setting this up).

\medskip
From this we conclude the only way the $ A_{5} $ action of $ PSL_{2} (q) $ can have an independent set of size $ 4 $ is if the stabilizers intersect in the way given by one of the errors above.
\end{appendices}

\newpage
\bibliographystyle{plain}
\bibliography{bibli}

\end{document}